\documentclass{amsart}
\usepackage{a4wide}
\usepackage{amssymb}
\usepackage{amsmath}
\usepackage{pifont}
\usepackage{amsfonts}
\usepackage{latexsym}
\usepackage{epsfig}
\usepackage{graphicx}
\usepackage{verbatim}
\usepackage{float}
\usepackage{exscale}

\newcommand{\BIGOP}[1]{\mathop{\mathchoice%
{\raise-0.22em\hbox{\huge $#1$}}%
{\raise-0.05em\hbox{\Large $#1$}}{\hbox{\large $#1$}}{#1}}}
\newcommand{\bigtimes}{\BIGOP{\times}}
\newcommand{\BIGboxplus}{\mathop{\mathchoice%
{\raise-0.35em\hbox{\huge $\boxplus$}}%
{\raise-0.15em\hbox{\Large $\boxplus$}}{\hbox{\large $\boxplus$}}{\boxplus}}}

\renewcommand{\ldots}{\ensuremath{\dotsc}}

\newcommand{\Rplus}{{\mathbb R}_{>0}}

\def\epsilon{\varepsilon}
\def\hat{\widehat}

\def\undertilde#1{{\baselineskip=0pt\vtop
{\hbox{$#1$}\hbox{$\scriptscriptstyle\sim$}}}{}}
\def\underdtilde#1{{\baselineskip=0pt\vtop
{\hbox{$#1$}\hbox{$\scriptscriptstyle\approx$}}}{}}
\def\epsilon{\varepsilon}
\def\hat{\widehat}

\def\nabq{\undertilde{\nabla}_{q}}
\def\nabqi{\undertilde{\nabla}_{q_i}}
\def\nabqj{\undertilde{\nabla}_{q_j}}
\def\nabx{\undertilde{\nabla}_{x}}
\def\nabr1{\undertilde{\nabla}_{r_1}}
\def\nabr2{\undertilde{\nabla}_{r_2}}

\def\lae{\ell_{a}}

\def\ut{\undertilde{u}}
\def\utae{\undertilde{u}_{L}}
\def\utaed{\undertilde{u}_{L,\delta}}
\def\uta{\undertilde{u}_{L}}

\def\utaeD{\utae^{\Delta t}}
\def\utaeDm{\utae^{\Delta t,-}}
\def\utaeDp{\utae^{\Delta t,+}}

\def\vt{\undertilde{v}}
\def\wt{\undertilde{w}}
\def\xt{\undertilde{x}}
\def\qt{\undertilde{q}}

\def\bt{\undertilde{b}}
\def\ft{\undertilde{f}}
\def\gt{\undertilde{g}}
\def\rt{\undertilde{r}}
\def\nt{\undertilde{n}}

\def\zt{\undertilde{z}}

\def\Bt{\undertilde{B}}
\def\Ct{\undertilde{C}}
\def\Ft{\undertilde{F}}

\def\Ht{\undertilde{H}}

\def\Lt{\undertilde{L}}

\def\Xt{\undertilde{X}}
\def\Vt{\undertilde{V}}
\def\Wt{\undertilde{W}}
\def\Ut{\undertilde{U}}

\def\dq{\,{\rm d}\undertilde{q}}
\def\dx{\,{\rm d}\undertilde{x}}

\def\dt{\,{\rm d}t}
\def\zerot{\undertilde{0}}
\def\tautt{\underdtilde{\tau}}

\def\Att{\underdtilde{A}}
\def\Btt{\underdtilde{B}}
\def\Ctt{\underdtilde{C}}

\def\Dtt{\underdtilde{D}}
\def\Gtt{\underdtilde{G}}
\def\Itt{\underdtilde{I}}

\def\Ltt{\underdtilde{L}}

\def\unabtt{\underdtilde{\nabla}_{x}\,\ut}

\def\vnabtt{\underdtilde{\nabla}_{x}\,\vt}

\def\sigtt{\underdtilde{\sigma}}

\def\pae{p_{L}}

\def\hpsiaet{\widetilde{\psi}_{L}}

\def\hpsiaedt{\widetilde{\psi}_{L,\delta}}
\def\psiae{\psi_{L}}
\def\psia{\widetilde \psi_{L}}

\def\dd {{\,\rm d}}

\newcommand{\nabxtt}{\underdtilde{\nabla}_{x}\,}

\newcommand{\bet}{\noalign{\vskip6pt plus 3pt minus 1pt}}

\def\Xint#1{\mathchoice
{\XXint\displaystyle\textstyle{#1}}%
{\XXint\textstyle\scriptstyle{#1}}%
{\XXint\scriptstyle\scriptscriptstyle{#1}}%
{\XXint\scriptscriptstyle\scriptscriptstyle{#1}}%
\!\int}
\def\XXint#1#2#3{{\setbox0=\hbox{$#1{#2#3}{\int}$}
\vcenter{\hbox{$#2#3$}}\kern-.5\wd0}}

\def\dashint{\Xint-}

\newtheorem{example}{Example}[section]
\newtheorem{definition}{Definition}[section]
\newtheorem{lemma}{Lemma}[section]
\newtheorem{theorem}{Theorem}[section]
\newtheorem{remark}{Remark}[section]
\newtheorem{corollary}{Corollary}[section]

\renewcommand{\theequation}{\arabic{section}.\arabic{equation}}

\newcounter{ind}
\def\eqlabstart{%
 \setcounter{ind}{\value{equation}}\addtocounter{ind}{1}%
 \setcounter{equation}{0}%
 \renewcommand{\theequation}{\arabic{section}.\arabic{ind}\alph{equation}}%
}
\def\eqlabend{%
 \renewcommand{\theequation}{\arabic{section}.\arabic{equation}}%
 \setcounter{equation}{\value{ind}}%
}

\newcounter{appendix}
\renewcommand\appendix{\par
        \refstepcounter{appendix}
           \setcounter{section}{0}
       \setcounter{theorem}{0}
           \setcounter{equation}{0}
\renewcommand\thesection{\appendixname ~\Alph{section}}
\renewcommand\thesubsection{\Alph{section}.\arabic{subsection}}%
\renewcommand\theequation{\Alph{section}.\arabic{equation}}}%

\begin{document}

\markboth{John W. Barrett and Endre S\"{u}li}
{Existence 
of Global Weak Solutions for Dilute Polymers}

%
%

\title[Existence 
of Global Weak Solutions for Dilute Polymers]
{Existence 
of global weak solutions to finitely extensible nonlinear
bead-spring chain models for dilute polymers with variable density and viscosity}

\author{JOHN W. BARRETT}

\address{\footnotesize Department of Mathematics, Imperial College London\\
London SW7 2AZ, UK\\
{\tt jwb@imperial.ac.uk}}

\author{ENDRE S\"ULI}

\address{
Mathematical Institute, University of Oxford\\ Oxford OX1 3LB, UK\\
{\tt endre.suli@maths.ox.ac.uk}}

\maketitle


\begin{abstract}
We show the existence of global-in-time weak solutions to a general class of coupled
bead-spring chain models that arise from the kinetic theory of dilute
solutions of nonhomogeneous polymeric liquids with noninteracting polymer chains,
with 
finitely extensible nonlinear elastic (FENE) 
spring potentials.
The class of models under consideration involves the unsteady incompressible
Navier--Stokes equations with
variable density and density-dependent dynamic viscosity in a bounded domain in $\mathbb{R}^d$, $d = 2$ or $3$, for the
density, the velocity and the pressure of the fluid, with an elastic extra-stress tensor appearing
on the right-hand side in the momentum equation. The extra-stress tensor stems from the
random movement of the polymer chains and is defined by the Kramers expression through
the associated probability density function
that satisfies a Fokker--Planck-type parabolic equation, a crucial feature of which is
the presence of a centre-of-mass diffusion term and a nonlinear density-dependent drag coefficient.
We require no structural assumptions
on the drag term in the Fokker--Planck equation; in particular, the drag term need not
be corotational. With
initial density $\rho_0 \in [\rho_{\rm min},\rho_{\rm max}]$ for the continuity equation,
where $\rho_{\rm min}>0$; a square-integrable and divergence-free initial velocity datum
$\undertilde{u}_0$ for the Navier--Stokes equation; and a nonnegative initial
probability density function $\psi_0$ for the Fokker--Planck equation, which has finite
relative entropy with respect to the Maxwellian $M$ associated with the spring potential
in the model, we prove, {\em via} a limiting procedure on certain regularization parameters,
the existence of a global-in-time weak solution $t \mapsto (\rho(t),\undertilde{u}(t), \psi(t))$
to the coupled Navier--Stokes--Fokker--Planck system, satisfying the initial condition
$(\rho(0),\undertilde{u}(0), \psi(0)) = (\rho_0,\undertilde{u}_0, \psi_0)$,
such that $t\mapsto \rho(t) \in [\rho_{\rm min},\rho_{\rm max}]$,
$t\mapsto \undertilde{u}(t)$ belongs to the classical Leray space and
$t \mapsto \psi(t)$ has bounded relative entropy with respect to $M$ and
$t \mapsto \psi(t)/M$ has integrable Fisher information (w.r.t. the Gibbs measure
${\rm d}\nu:= M(\undertilde{q})\,{\rm d}\undertilde{q}\,{\rm d}\undertilde{x}$)
over any time interval $[0,T]$, $T>0$.

\medskip


\end{abstract}

\section{Introduction}
\label{sec:1}

This paper establishes the existence of global-in-time weak solutions
to a large class of bead-spring chain models with finitely
extensible nonlinear elastic (FENE) 
type spring potentials, ---
a system of nonlinear partial differential equations that arises
from the kinetic theory of dilute polymer solutions. The solvent is an
incompressible, viscous, isothermal Newtonian fluid with variable density and viscosity
confined to a bounded open Lipschitz domain $\Omega \subset \mathbb{R}^d$, $d=2$ or $3$, with
boundary $\partial \Omega$. For the sake of simplicity of presentation,
we shall suppose that $\Omega$ has a `solid boundary' $\partial \Omega$;
the velocity field $\ut$ will then satisfy the no-slip boundary condition
$\ut=\zerot$ on $\partial \Omega$. The polymer chains, which are suspended
in the solvent, are assumed not to interact with each other. The
equations of continuity, balance of linear momentum and incompressibility then
have the form of the incompressible Navier--Stokes equations with
variable density and viscosity (cf. Antontsev, Kazhikhov \& Monakhov \cite{AKM},
Feireisl \& Novotn\'y \cite{FN}, Lions \cite{Lions1} or Simon \cite{Simon-density}) in
which the elastic {\em extra-stress} tensor $\tautt$ (i.e., the
polymeric part of the Cauchy stress tensor) appears as a source
term:

Given $T \in \mathbb{R}_{>0}$, find
$ \rho\,:\,(\xt,t)
\in \Omega \times [0,T] \mapsto
\rho(\xt,t) \in {\mathbb R}$,
$\ut\,:\,(\xt,t)\in
\overline\Omega \times [0,T] \mapsto
\ut(\xt,t) \in {\mathbb R}^d$ and $p\,:\, (\xt,t) \in \Omega \times (0,T]
\mapsto p(\xt,t) \in {\mathbb R}$ such that
\begin{subequations}
\begin{alignat}{2}
\frac{\partial \rho}{\partial t}
+ \nabx\cdot(\ut \,\rho)
&= \zerot \qquad &&\mbox{in } \Omega \times (0,T],\label{ns0a}\\
\rho(\xt,0)&=\rho_{0}(\xt)    \qquad &&\forall \xt \in \Omega,\label{ns00a} \\
\frac{\partial(\rho\,\ut)}{\partial t}
+ \nabx\cdot(\rho\,\ut \otimes \ut) - \nabx \cdot (\mu(\rho)  \,\Dtt(\ut)) + \nabx p
&= \rho\,\ft + \nabx \cdot \tautt \qquad &&\mbox{in } \Omega \times (0,T],\label{ns1a}\\
\nabx \cdot \ut &= 0        \qquad &&\mbox{in } \Omega \times (0,T],\label{ns2a}\\
\ut &= \zerot               \qquad &&\mbox{on } \partial \Omega \times (0,T],\label{ns3a}\\
(\rho\,\ut)(\xt,0)&=(\rho_0\,\ut_{0})(\xt)    \qquad &&\forall \xt \in \Omega.\label{ns4a}
\end{alignat}
\end{subequations}
It is assumed that each of the equations above has been written in its nondimensional form;
$\rho$ denotes a nondimensional solvent density,
$\ut$ a nondimensional solvent velocity, defined as the velocity field
scaled by the characteristic flow speed $U_0$.
Here
$\Dtt(\vt) := \frac{1}{2}\,(\nabxtt \vt + (\nabxtt \vt)^{\rm T})$
is the rate of strain tensor,
with  $(\vnabtt)(\xt,t) \in {\mathbb
R}^{d \times d}$ and $\{\vnabtt\}_{ij} = \textstyle
\frac{\partial v_i}{\partial x_j}$.
The scaled dynamic viscosity of the solvent,
$\mu(\cdot) \in \mathbb{R}_{>0}$, is density-dependent;
in addition, $p$ is
the nondimensional pressure and $\ft$ is the nondimensional density of body forces.

In a {\em bead-spring chain model}, consisting of $K+1$ beads coupled with $K$ elastic
springs to represent a polymer chain, the extra-stress tensor
$\tautt$ is defined by the \textit{Kramers expression}
as a weighted average of $\psi$, the probability
density function of the (random) conformation vector
$\qt := (\qt_1^{\rm T},\dots, \qt_K^{\rm T})^{\rm T}
\in \mathbb{R}^{Kd}$ of the chain (see equation (\ref{tau1}) below), with $\qt_i$
representing the $d$-component conformation/orientation vector of the $i$th spring.
The Kolmogorov equation satisfied by $\psi$ is a second-order parabolic equation,
the Fokker--Planck equation, whose transport coefficients depend
on the velocity field $\ut$, and the viscous drag coefficient appearing in the Fokker--Planck
equation is a linear function of the dynamic viscosity $\mu$ (Stokes drag is assumed), which,
in turn, is a nonlinear function of the density $\rho$.
The domain $D$ of admissible conformation
vectors $D \subset \mathbb{R}^{Kd}$ is a $K$-fold
Cartesian product $D_1 \times \cdots \times D_K$ of balanced convex
open sets $D_i \subset \mathbb{R}^d$, $i=1,\dots, K$; the term
{\em balanced} means that $\qt_i \in D_i$ if, and only if, $-\qt_i \in D_i$.
Hence, in particular, $\undertilde{0} \in D_i$, $i=1,\dots,K$.
Typically $D_i$ is the whole of $\mathbb{R}^d$ or a bounded open $d$-dimensional ball
centred at the origin $\zerot \in \mathbb{R}^d$ for each $i=1,\dots,K$.
When $K=1$, the model is referred to as the {\em dumbbell model.}

Let $\mathcal{O}_i\subset [0,\infty)$ denote the image of $D_i$
under the mapping $\qt_i \in D_i \mapsto
\frac{1}{2}|\qt_i|^2$, and consider the {\em spring potential}~$U_i
\!\in\! C^1(\mathcal{O}_i; \mathbb{R}_{\geq 0})\cap W^{2,\infty}_{\rm loc}(\mathcal{O}_i;\mathbb {R}_{\geq 0})$, $i=1,\dots, K$.
Clearly, $0 \in \mathcal{O}_i$. We shall suppose that $U_i(0)=0$
and that $U_i$ is unbounded on
$\mathcal{O}_i$ for each $i=1,\dots, K$.
The elastic spring-force $\Ft_i\,:\, D_i \subseteq \mathbb{R}^d \rightarrow \mathbb{R}^d$
of the $i$th spring in the chain is defined by
\begin{equation}\label{eqF}
\Ft_i(\qt_i) := U_i'(\textstyle{\frac{1}{2}}|\qt_i|^2)\,\qt_i, \qquad i=1,\dots,K.
\end{equation}

The partial Maxwellian $M_i$, associated with the spring potential $U_i$, is defined by
\[
M_i(\qt_i) := \frac{1}{\mathcal{Z}_i} {\rm e}^{-U_i(\frac{1}{2}\,|\qt_i|^2)}, \qquad
\mathcal{Z}_i:= {\displaystyle \int_{D_i} {\rm e}^{-U_i(\frac{1}{2}|\qt_i|^2)} \dq_i},\qquad
i=1,\dots,K.
\]
The (total) Maxwellian in the model is then
\begin{align}
M(\qt) := \prod_{i=1}^K M_i(\qt_i) \qquad \forall
\qt:=(\qt_1^{\rm T},\ldots,\qt_K^{\rm T})^{\rm T} \in D
:= \bigtimes_{i=1}^K D_i.
\label{MN}
\end{align}
Observe that, for $i=1, \dots, K$,
\begin{equation}
M(\qt)\,\nabqi [M(\qt)]^{-1} = - [M(\qt)]^{-1}\,\nabqi M(\qt) =
\nabqi U_i(\textstyle{\frac{1}{2}}|\qt_i|^2)
=U_i'(\textstyle{\frac{1}{2}}|\qt_i|^2)\,\qt_i, \label{eqM}
\end{equation}
and, by definition,
\[ \int_D M(q) \dq = 1.\]

\begin{example}
\label{ex1.1} \em
In the Hookean dumbbell model $K=1$, and the spring force
is defined by ${\Ft}({\qt}) = {\qt}$, with ${\qt} \in {D}=\mathbb{R}^d$,
corresponding to ${U}(s)= s$, $s \in \mathcal{O} = [0,\infty)$.
More generally, in a Hookean bead-spring chain model, $K \geq 1$, $\Ft_i(\qt_i) = \qt_i$,
corresponding to $U_i(s) = s$, $i=1,\dots,K$, and $D$ is
the Cartesian product of $K$ copies of $\mathbb{R}^d$. The associated Maxwellian is
\[ M(\qt) = M_1(\qt_1)\cdots M_K(\qt_K) = \frac{1}{\mathcal{Z}} {\rm e}^{-\frac{1}{2}|\qt|^2},\]
with $|\qt|^2:= |\qt_1|^2 + \cdots +|\qt_K|^2$ 
and $\mathcal{Z} := \mathcal{Z}_1 \cdots \mathcal{Z}_K = (2\pi)^{{Kd}/{2}}$.
Hookean dumbbell and Hookean bead-spring chain models are physically unrealistic
as they admit arbitrarily large extensions.
$\quad\diamond$
\end{example}

A more realistic class of models assumes that the springs in the bead-spring chain have
finite extension: the domain $D$ is then taken to be a Cartesian product of $K$
{\em bounded} open balls $D_i \subset \mathbb{R}^d$, centred at the
origin $\zerot \in \mathbb{R}^d$, $i=1,\dots, K,$ with $K \geq 1$. The spring potentials
$U_i\, :\, s \in [0,b_i/2) \mapsto U_i(s) \in [0,\infty)$, with $b_i>0$, $i=1,\dots, K$, are in that case
nonlinear and unbounded functions, and the associated bead-spring chain model is referred
to as a FENE (finitely extensible nonlinear elastic) model; in the case of $K=1$, the
corresponding model is called a FENE dumbbell model.

Here we shall be concerned with finitely extensible nonlinear 
bead-spring chain
models, with
$D:= B(\zerot,b_1^{\frac{1}{2}}) \times \cdots \times
B(\zerot, b_K^{\frac{1}{2}})$, where $b_i > 0$, $i=1,\dots,K$, and $B(\zerot, b_i^{\frac{1}{2}})$
is a bounded open ball in $\mathbb{R}^d$ of radius $b_i^{\frac{1}{2}}$, centred at $\zerot \in \mathbb{R}^d$. We shall adopt the following structural hypotheses on the spring potentials
$U_i$ and the associated partial Maxwellians $M_i$, $i=1,\dots, K$.

\medskip

We shall suppose that for $i=1,\dots, K$ there exist constants $c_{ij}>0$, $j=1, 2, 3, 4$, and $\gamma_i > 1$ such that
the spring potential $U_i$ satisfies%
\eqlabstart
\begin{eqnarray}
&& ~\hspace{6mm}c_{i1}\,[\mbox{dist}(\qt_i, \,\partial D_i)]^{\gamma_i}
\leq  M_i(\qt_i) \,
\leq
c_{i2}\,[\mbox{dist}(\qt_i, \,\partial D_i)]^{\gamma_i} \quad \forall \qt_i \in
D_i, \label{growth1}\\
&& ~\hspace{15mm}c_{i3} \leq [\mbox{dist}(\qt_i,\,\partial D_i)] \,U_i'
(\textstyle{\frac{1}{2}}|\qt_i|^2)
\leq c_{i4}\quad \forall \qt_i \in D_i. \label{growth2}
\end{eqnarray}
\eqlabend
Since
$[U_i(\textstyle{\frac{1}{2}}|\qt_i|^2)]^2 = (-\log M_i(\qt_i) + {\rm Const}.)^2$,
it follows from (\ref{growth1},b) that (if $\gamma_i >1$,
as has been assumed here,)
\begin{equation}\label{additional-1}
\int_{D_i} \left[1 + [U_i(\textstyle{\frac{1}{2}}|\qt_i|^2)]^2
+ [U_i'(\textstyle{\frac{1}{2}}|\qt_i|^2)]^2\right] M_i(\qt_i) \, \dd
\qt_i < \infty, \qquad i=1, \dots, K.
\end{equation}

\medskip

\begin{example}
\label{ex1.2} \em
In the FENE (finitely extensible nonlinear elastic)
dumbbell model, introduced by Warner \cite{Warner:1972}, $K=1$ and the spring force is given by
$
{\Ft}({\qt}) = (1 - |{\qt}|^2/b)^{-1}\,{\qt}$,
$\qt \in D = B(\zerot,b^{\frac{1}{2}})$,
corresponding to ${U}(s) = - \frac{b}{2}\log \left(1-\frac{2s}{b}\right)$,
$s \in \mathcal{O} = [0,\frac{b}{2})$, $b>2$.
%
More generally, in a FENE bead spring chain, one considers $K+1$ beads linearly
coupled with $K$ springs, each with a FENE spring potential.
Direct calculations show that the partial Maxwellians $M_i$ and the
elastic potentials $U_i$, $i=1,\dots, K$, of the FENE bead spring chain satisfy
the conditions (\ref{growth1},b) with and $\gamma_i:= \frac{b_i}{2}$,
provided that $b_i>2$, $i=1,\dots, K$. Thus, (\ref{additional-1}) also holds
and $b_i>2$, $i=1,\dots, K$.

It is interesting to note that in the (equivalent)
stochastic version of the FENE dumbbell model ($K=1$)
a solution to the system of
stochastic differential equations
associated with the Fokker--Planck equation exists and has
trajectorial uniqueness if, and only if, $\gamma = \frac{b}{2}\geq 1$;
(cf. Jourdain, Leli\`evre \& Le Bris \cite{JLL2} for details).
Thus, in the general class of FENE-type bead-spring chain models considered here,
the assumption $\gamma_i > 1$, $i=1,\dots, K$, is the weakest reasonable
requirement on the decay-rate of $M_i$  in (\ref{growth1}) as
$\mbox{dist}(\qt_i,\partial D_i) \rightarrow 0$.

Another example of a finitely extensible nonlinear elastic potential
is Cohen's Pad\'e approximant to
the inverse Langevin function (CPAIL) (cf. \cite{Cohen}), with
\begin{equation}\label{CPAIL-model}
U_i(s) := \frac{s}{3} - \frac{b_i}{3}\ln\left(1 - \frac{2 s}{b_i}\right) \qquad\text{and}\qquad \undertilde{F}_i(\qt_i) = \frac{ 1 - \frac{|\qt_i|^2}{3 b_i} }{ 1 - \frac{|\qt_i|^2}{b_i} }\qt_i,
\end{equation}
where $b_i$ is again a positive parameter. Again,  direct calculations show that
the Maxwellian $M$ and the elastic potential $U$ of the CPAIL dumbell model satisfy the
conditions (\ref{growth1},b) with $\gamma_i:= \frac{b_i}{3}$,
provided that $b_i>3$, $i=1,\dots, K$. Thus, (\ref{additional-1}) also holds
and $b_i>3$, $i=1,\dots, K$.

We note in passing that both of these force laws
are approximations to the \emph{inverse Langevin force law}
\[ \undertilde{F}_i(\qt_i) := \frac{\sqrt{b_i}}{3} L^{-1}\left( \frac{|\qt_i|}{\sqrt{b_i}} \right) \frac{\qt_i}{|\qt_i|},\]
where the $\emph{Langevin function}$ $L$ is defined by $L(t) := \coth(t)-1/t$
on $[0,\infty)$. As $L$ is strictly monotonic increasing on $[0,\infty)$ and tends to
$1$ as its argument tends to $\infty$, it follows that the function
$|\qt_i| \in [0,\sqrt{b_i}\,) \mapsto L^{-1}(|\qt_i|/\sqrt{b_i}\,)
 \in [0,\infty)$ is strictly monotonic increasing, with a vertical asymptote at $|\qt_i| = \sqrt{b_i}$.
$\quad\diamond$
\end{example}

\medskip

The governing equations of the general nonhomogeneous bead-spring chain
models with centre-of-mass diffusion considered in this paper are
(\ref{ns1a}--d), where the extra-stress tensor $\tautt$ is defined by the \textit{Kramers
expression}:
\begin{equation}\label{tau1}
\tautt(\xt,t) = k\left( \sum_{i=1}^K \int_{D}\psi(\xt,\qt,t)\, \qt_i\,
\qt_i^{\rm T}\, U_i'\left(\textstyle \frac{1}{2}|\qt_i|^2\right)
{\dd} \qt - K
\varrho(\xt,t)\,\Itt\right),
\end{equation}
with the density of polymer chains located at $\xt$ at time $t$ given by
\begin{equation}\label{rho1}
\varrho(\xt,t) = \int_{D} \psi(\xt, \qt,t)
{\dd}\qt,
\end{equation}
(not to be confused with the solvent density $\rho(\xt,t)$).
The probability density function $\psi$ is a solution of the Fokker--Planck (forward Kolmogorov) equation
\begin{align}
\label{fp0}
&\frac{\partial \psi}{\partial t} + (\ut \cdot\nabx) \psi +
\sum_{i=1}^K \nabqi
\cdot \left(\sigtt(\ut) \, \qt_i\,\psi \right)
\nonumber
\\
\bet
&\hspace{0.1in} =
\epsilon\,\Delta_x\left(\frac{\psi}{\zeta(\rho)}\right) +
\frac{1}{4 \,\lambda}\,
\sum_{i=1}^K \sum_{j=1}^K
A_{ij}\,\nabqi \cdot \left(
M
\,\nabqj \left(\frac{\psi}{\zeta(\rho)\, M}\right)\right) \quad \mbox{in } \Omega \times D \times
(0,T],
\end{align}
with $\sigtt(\vt) \equiv \nabxtt \vt$ and a density-dependent scaled drag coefficient $\zeta(\cdot)\in \Rplus$.
A concise derivation of the Fokker--Planck equation \eqref{fp0} is presented below.

The nondimensional constant $k>0$ featuring in \eqref{tau1} is a constant multiple
of the product of the polymer number density (the number of polymer molecules
per unit volume),
the Boltzmann constant $k_B$, and the absolute temperature $\mathtt{T}$.
In \eqref{fp0}, $\varepsilon>0$ is the centre-of-mass diffusion coefficient defined
as $\varepsilon := (\ell_0/L_0)^2/(4(K+1)\lambda)$ with
$L_0$ a characteristic length-scale  of the solvent flow,
$\ell_0:=\sqrt{k_B \mathtt{T}/\mathtt{H}}$ signifying the characteristic microscopic length-scale
and $\lambda :=(\zeta_0/4\mathtt{H})(U_0/L_0)$,
where $\zeta_0>0$ is a characteristic drag coefficient and $\mathtt{H}>0$ is a spring-constant.
The nondimensional parameter $\lambda \in \Rplus$, called the Weissenberg number
(and usually denoted by $\mathsf{Wi}$), characterizes
the elastic relaxation property of the fluid, and
$\Att=(A_{ij})_{i,j=1}^K$ is the symmetric positive definite
\textit{Rouse matrix}, or connectivity matrix;
for example, $\Att = {\tt tridiag}\left[-1, 2, -1\right]$ in the case
of a linear chain; see, Nitta \cite{Nitta}. Concerning these scalings and notational
conventions, we remark that the factor $1/(4\lambda)$ in equation \eqref{fp0} above
appears as a factor $1/(2\lambda)$ in the Fokker--Planck equation in our earlier papers
\cite{BS2011-fene,BS2010-hookean,BS2010,BS2011-feafene}.

%
\begin{definition} The collection of equations and structural hypotheses
{\rm(\ref{ns0a}--f)}--\eqref{fp0} will be referred to throughout the paper as model
$({\rm P})$, or as the {\em general nonhomogeneous FENE-type bead-spring chain model with
centre-of-mass diffusion}.
\end{definition}

We continue with a brief derivation of the Fokker--Planck equation \eqref{fp0} in the general nonhomogeneous FENE-type bead-spring chain model $({\rm P})$ with centre-of-mass diffusion. Henceforth, up to and including
equation \eqref{nondim-defs} below where our nondimensionalization is introduced,
the symbols $\rho$, $\ut$, $\psi$, $\varrho$, $\mu(\rho)$, $\zeta(\rho)$, $\Ut_i$, $\Ft_i$, $\xt$, $\qt$, $t$
will refer to the dimensional forms of these functions and variables, unless otherwise stated.
For the sake of simplicity of the exposition,
we shall confine ourselves to the case of linear chains, when the Rouse matrix $\Att =  {\tt tridiag}\left[-1, 2, -1\right]$; we emphasize however that the results that we prove later on in the paper are completely independent of the actual structure of $\Att$: we shall only require that $\Att$ is symmetric and positive definite.
Thus, each polymer molecule is modelled by a linear arrangement of $K+1$ beads with mass $m$ that are joined by $K$
massless elastic springs. For each $K \in \mathbb{N}_{\geq 1}$, the state of a bead-spring chain in such an
arrangement is fully determined by the position vectors of its $K+1$ beads, denoted by $\undertilde{r}_\ell$
for $\ell=1,\dots, K+1$. Equivalently the state of a bead-spring chain can be described by the center of mass of the bead
system,
$\undertilde{r}_c := \frac{1}{K +1}\sum_{\ell=1}^{K+1} \undertilde{r}_\ell$,
together with the $K$
connector vectors $\qt_i := \undertilde{r}_{i+1} - \rt_i$, $i=1,\dots, K$, which are to be understood as column vectors in $\mathbb{R}^d$ (cf. Fig. \ref{modelPoly}).
%
\begin{figure}[t]\label{modelPoly}
    \includegraphics[width=0.62\textwidth]{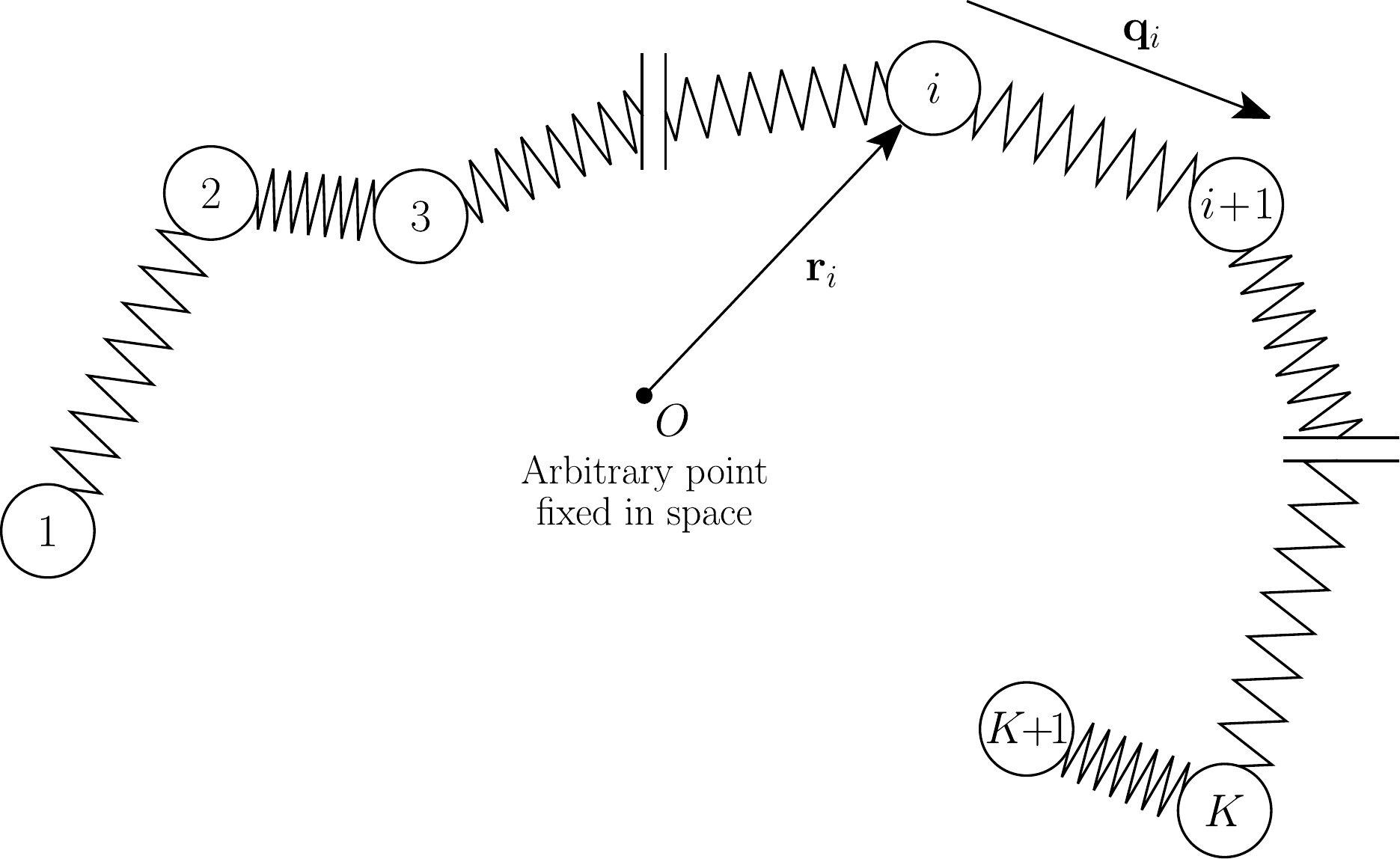}
    \caption{\vspace{0in}Rouse chain with $K$ springs and $K+1$ beads.}
\end{figure}
%

\smallskip

In what follows we shall write $\undertilde{r}:=(\undertilde{r}_1^{\rm T}, \dots , \undertilde{r}_{K+1}^{\rm T})^{\rm T}$ and $\qt:=(\qt_1^{\rm T}, \dots, \qt_K^{\rm T})^{\rm T}$. Each spring exerts an elastic conservative force on the beads it connects along the corresponding connector vector and has a magnitude that depends isotropically on it. We model the spring force by a spring force function $\Ft_i$, which has the form $\Ft_i(\qt_i) = {\tt H}\, \Ut_i'(\frac{1}{2}|\qt_i|^2)\,\qt_i$,
where ${\tt H}>0$ is a spring constant, and $U_i$ is a spring potential, $i=1,\dots,K$. It will be assumed that $|\qt_i| < q^{\rm max}_i$, where
$q^{\rm max}_i$ denotes the maximum admissible length of the $i$th spring. We let $D_i :=\{\qt_i \in \mathbb{R}^d\,:\,
|\qt_i|<q^{\rm max}_i\}$. It is immediate that $\Ft_i(\qt_i) = -\Ft(-\qt_i)$ for all $\qt_i\in D_i$,
$i=1,\dots, K$.

In the absence of external forces and neglecting inertial effects the Langevin equation for the $\ell$-th bead in this model is
\begin{align}\label{langevin}
\zeta_0 \left( \dd \rt_\ell - \ut(\rt_\ell, \cdot) \dt\right) = \Bt_\ell \dt
+ \sum_{i=1}^K G_{\ell i}\,\Ft_i(\qt_i) \dt,\qquad \ell=1,\dots, K+1.
\end{align}
Here $\zeta_0$ is a characteristic drag coefficient, which we shall assume for the moment to be a fixed positive constant; $\Bt_\ell$ denotes a Brownian force acting on the $\ell$th bead; and the $(K+1) \times K$ matrix $\Gtt$ with entries $G_{ij}$, $i=1,\dots,K+1$, $j=1,\dots, K$, called the \textit{graph incidence matrix}, which (in our case here of a linear bead-spring chain with $K+1$ beads and $K$ springs) is defined by
\[ \Gtt := \left( \begin{array}{cccc}
   1 & ~ & ~ & ~  \\
  -1 & 1 & ~ & ~  \\
  ~  & \ddots & \ddots & ~ \\
  ~  &  ~     & -1 & 1\\
  ~  &  ~     & ~  & -1
  \end{array}
  \right) \in \mathbb{R}^{(K+1) \times K}; \!\quad{i.e.},\; G_{ij}:=\left\{\begin{array}{cl}
  + 1 & \mbox{if spring $j$ starts from bead $i$},\\
  - 1 & \mbox{if spring $j$ teminates in bead $i$},\\
  0   & \mbox{otherwise.}
  \end{array} \right.\]
The Rouse matrix $\Att \in \mathbb{R}^{K \times K}$ is related to the
matrix $\Gtt$ through the equality $\Att = \Gtt^{\rm T} \Gtt$.
The form of the matrix $\Gtt$ reflects the fact that the $\ell$th bead is pulled by the $(\ell - 1)$st spring
in the $-\qt_{\ell-1} = \rt_{\ell-1} - \rt_\ell$ direction and by the $\ell$th spring in the $\qt_\ell = \rt_{\ell+1}-\rt_\ell$
direction with proper provision for the beads at the ends of the chain, which are only pulled by a single
spring each. Note that $\qt_i = -\sum_{\ell=1}^{K+1} G_{\ell i}\, \rt_\ell$, $i=1,\dots, K$.
For $\ell=1,\dots,K+1$, the Brownian force $\Bt_\ell$ is defined by a $d$-component vectorial Wiener process $\Wt_\ell$
via $\Bt_\ell \dt = \sqrt{2\, k_B\, {\tt T}\,\zeta_0}\, \dd \Wt_\ell$, where $k_B = 1.38 \times 10^{-23}\, {\rm m^2 kg\,s^{-2} K^{-1}}$ is Boltzmann's constant and ${\tt T}$ is the absolute temperature measured in Kelvin, K. The coefficient $\sqrt{2 k_B\,{\tt T}\, \zeta_0}$ is due to the Einstein--Smoluchowski relation, which determines the diffusion coefficient in Brownian motion. Therefore, \eqref{langevin} can be written in the following form
\begin{eqnarray}
\label{eq.stochastic1}
 \dd\left[ \begin{array}{c}
            {\rt}_1(t)\\
            {\rt}_2(t)\\
            \vdots\\
            {\rt}_{K+1}(t)
           \end{array} \right]
 &=&
\left\{\left[ \begin{array}{c}
            \ut ({\rt}_1(t),t)\\
            \ut ({\rt}_2(t),t)\\
            \vdots\\
            \ut ({\rt}_{K+1}(t),t)
           \end{array} \right]
+ \zeta^{-1}_0\, \Gtt\left[
\begin{array}{c}
            \Ft_1 (\rt_2(t)-\rt_1(t))\\
            \Ft_2 (\rt_3(t)-\rt_2(t))\\
            \vdots\\
            \Ft_{K} (\rt_{K+1}(t)-\rt_{K}(t))
           \end{array} \right]\right\} \dt\nonumber\\
&&+\, \sqrt{\frac{2k_B{\tt T}}{\zeta_0}}\,\dd\left[ \begin{array}{c}
            {\Wt}_1(t)\\
            {\Wt}_2(t)\\
            \vdots\\
            {\Wt}_{K+1}(t)
           \end{array} \right].
\end{eqnarray}

By defining
\[
 \Xt(t) := \left[ \begin{array}{c}
            {\rt}_1(t)\\
            {\rt}_2(t)\\
            \vdots \\
            {\rt}_{K+1}(t)
           \end{array} \right], \qquad
\Wt(t) := \left[ \begin{array}{c}
            {\Wt}_1(t)\\
            {\Wt}_2(t)\\
            \vdots\\
            {\Wt}_{K+1}(t)
           \end{array} \right], \qquad
 \sigma := \sqrt{\frac{2k_B {\tt T}}{\zeta_0}}\, {\Itt},\]

\[ \bt(\Xt(t)) := \left[ \begin{array}{c}
            \ut ({\rt}_1(t),t)\\
            \ut ({\rt}_2(t),t)\\
            \vdots\\
            \ut ({\rt}_{K+1}(t),t)
           \end{array} \right]
+ \zeta^{-1}_0\, \Gtt\left[
\begin{array}{c}
            \Ft_1 (\rt_2(t)-\rt_1(t))\\
            \Ft_2 (\rt_3(t)-\rt_2(t))\\
            \vdots\\
            \Ft_{K} (\rt_{K+1}(t)-\rt_{K}(t))
           \end{array} \right], \]
equation \eqref{eq.stochastic1} can be rewritten as the following stochastic differential equation:
\begin{equation}
\label{eq.stochastic2}
\dd \Xt(t) = \bt(\Xt(t))\dt + \sigma(\Xt(t))\dd\Wt(t), \qquad \Xt(0) = \Xt,
\end{equation}
for the $(K+1)d$-component vectorial random variable $\Xt(t)$, $t \in [0,T]$.

By recalling a classical result from stochastic analysis stated in Theorem \ref{thm.kolmogorov} below, we can now write down the associated {\em forward Kolmogorov equation} (Fokker--Planck equation), a parabolic partial differential equation governing the evolution of the probability density function of the stochastic process $t \mapsto \Xt(t)$ (see, for example, Corollary 5.2.10 in \cite{MonteCarlo}).

\begin{theorem}
\label{thm.kolmogorov}
 Suppose that the $(K+1)d$-component vectorial random variable $\Xt(t)$ has a probability density function $({\zt},t) \mapsto \psi({\zt},t)$ of class $C^{2,1}(\mathbb{R}^{(K+1)d}\times[0,\infty))$ (i.e., twice continuously differentiable with respect to ${\zt} \in \mathbb{R}^{(K+1)d}$ and once with respect to $t$), and let $\Xt(0) = \Xt$ be a square-integrable random variable with probability density function $\psi_0 \in C^2(\mathbb{R}^{(K+1)d})$. Also, suppose that $b$ and $\sigma$ in \eqref{eq.stochastic2} are globally Lipschitz continuous, and $c({\zt}) = \sigma({\zt})\,\sigma({\zt})^{\rm T}$. Then,
\begin{equation}
\label{eq.kolmogorov}
 \frac{\partial\psi}{\partial t} + \sum_{j=1}^{(K+1)d} \frac{\partial}{\partial z_j}\left(b_{j} \psi\right) = \frac{1}{2}\sum_{i,j=1}^{(K+1)d} \frac{\partial^2}{\partial z_{i}\partial z_{j}}(c_{ij}\psi),
\end{equation}
in $\mathbb{R}^{(K+1)d}\times[0,\infty)$ where $\psi({\zt},0) = \psi_0({\zt})$ for ${\zt} \in \mathbb{R}^{(K+1)d}$.
\end{theorem}
\smallskip

The Hookean spring force satisfies the global Lipschitz continuity assumption in Theorem~\ref{thm.kolmogorov}, whereas FENE-type spring forces do not. Indeed, FENE-type forces are only locally Lipschitz on $D$ and are not defined on all of $\mathbb{R}^d$. For the purposes of the present informal derivation of equation \eqref{fp0} we shall ignore this technicality, and will proceed as if Theorem~\ref{thm.kolmogorov} held in the FENE case also.
Thus, on applying Theorem~\ref{thm.kolmogorov} to \eqref{eq.stochastic2}, we arrive at the following forward
Kolmogorov equation:
\begin{eqnarray}
\label{eq.fp12}
&&\frac{\partial\psi^{1,K+1}}{\partial t} + \sum_{\ell=1}^{K+1} \undertilde{\nabla}_{r_\ell} \cdot \left[ \ut ({\rt}_\ell,t)\,\psi^{1,K+1} +
\frac{1}{\zeta_0}\sum_{i=1}^K G_{\ell i}\Ft_i ({\rt}_{i+1} - {\rt}_i )\,\psi^{1,K+1} \right]
\nonumber\\
&&\qquad\qquad\,
= \sum_{\ell=1}^{K+1} \Delta_{r_\ell} \left(\frac{k_B\,{\tt T}}{\zeta_0}\,\psi^{1, K+1}\right),
\end{eqnarray}
where $\psi^{1,K+1}$ denotes the probability density function with respect to the $({\rt}_1,\dots, {\rt}_{K+1})$ coordinates. By performing the linear change of variables
\[ \rt:=(\rt_1,\dots, \rt_{K+1}) \in \mathbb{R}^{(K+1)d}
\mapsto (\rt_c, \qt):=(\rt_c,\qt_1, \dots, \qt_K) \in \mathbb{R}^{(K+1)d}\]
and letting
$\psi(\rt_c, \qt, t) := \psi^{1,K+1}(\rt({\rt}_c, \qt),t)$, we obtain
\begin{eqnarray}
 \label{eq.fp-xq}
  &&\frac{\partial\psi}{\partial t} + \undertilde{\nabla}_{r_c} \cdot \left(\frac{1}{K+1} \sum_{\ell=1}^{K+1}
  \ut(\rt_\ell,\cdot)\,\psi\right)\nonumber\\
  &&\qquad - \sum_{i=1}^K \undertilde{\nabla}_{q_i} \cdot \left(\left[\sum_{\ell=1}^{K+1} G_{\ell i}\, \ut(\rt_\ell,\cdot) +
  \frac{1}{\zeta_0}
  \sum_{j=1}^{K} A_{ij}\,\Ft_j(\qt_j)\right]\psi\right)\nonumber\\
  &&\qquad\qquad =  \Delta_{r_c}
  \left(\frac{k_B\,{\tt T}}{K+1}\, \frac{\psi}{\zeta_0}\right)
  +
  \sum_{i=1}^K \sum_{j=1}^K A_{ij} \nabla_{q_i} \cdot  \left(\nabla_{q_j}\left(k_B\, {\tt T}\,\frac{\psi}{\zeta_0}\right)\right),
\end{eqnarray}
where $\Att = \Gtt^{\rm T} \Gtt$ is the Rouse matrix. We shall use the approximations
\[ \frac{1}{K+1} \sum_{\ell=1}^{K+1} \ut(\rt_\ell,\cdot) \approx \ut(\rt_c,t), \quad
\sum_{\ell=1}^{K+1} G_{\ell i}\, \ut(\rt_\ell,\cdot) \approx \underdtilde{\nabla}_{r_c}\ut(\rt_c,t)\,(\rt_i -
\rt_{i+1}) =  - \underdtilde{\nabla}_{r_c}\ut(\rt_c,t)\,\qt_i,\]
for $i = 1,\dots, K$,
in the second and the third term on the left-hand side of \eqref{eq.fp-xq}, respectively, and suppose that the
associated approximation errors are negligible; we note in particular that if $\ut$ is linear in its spatial
variable, then the above approximations are exact. Otherwise, the approximation errors are nonzero and their
magnitudes depend on the extent to which $\xt \mapsto \ut(\xt,\cdot)$ deviates from a linear function on the
characteristic microscopic length scale $\ell_0$. Hence,
\begin{eqnarray}
 \label{eq.fp-xq-1}
  &&\frac{\partial\psi}{\partial t} + \undertilde{\nabla}_{r_c} \cdot \left(\ut(\rt_c,\cdot)\,\psi\right)
  + \sum_{i=1}^K \undertilde{\nabla}_{q_i} \cdot \left(\left(\underdtilde{\nabla}_{r_c}\ut(\rt_c,\cdot)\,\qt_i\right)\psi -
  \sum_{j=1}^{K} A_{ij}\,\Ft_j(\qt_j)\,\frac{\psi}{\zeta_0}\right)\nonumber\\
  &&\qquad\qquad =  \Delta_{r_c}
  \left(\frac{k_B\,{\tt T}}{K+1}\, \frac{\psi}{\zeta_0}\right)
  +
  \sum_{i=1}^K \sum_{j=1}^K A_{ij} \nabla_{q_i} \cdot \left(\nabla_{q_j}\left(k_B\, {\tt T}\,\frac{\psi}{\zeta_0}\right)\right).
\end{eqnarray}
More generally, since the (Stokes) drag coefficient depends linearly on the dynamic viscosity, which, in turn, has been assumed to depend nonlinearly on the (variable) density, we shall replace $\zeta_0$ in \eqref{eq.fp-xq-1} with $\zeta_0 \,\zeta(\rho(\xt,t))$, where $\zeta(\rho)$ is a smooth
nondimensional function of the (unknown) nondimensionalized density $\rho$.
After renaming the centre of mass $\rt_c$ into $\xt$, we thus have that
\begin{eqnarray}
 \label{eq.fp-xq-2}
  &&\frac{\partial\psi}{\partial t} + \undertilde{\nabla}_{x} \cdot \left(\ut(\xt,\cdot)\,\psi\right)
  + \sum_{i=1}^K \undertilde{\nabla}_{q_i} \cdot \left(\left(\underdtilde{\nabla}_{x}\ut(\xt,\cdot)\,\qt_i\right)\psi - \frac{1}{\zeta_0}
  \sum_{j=1}^{K} A_{ij}\,\Ft_j(\qt_j)\,\frac{\psi}{\zeta}\right)\nonumber\\
  &&\qquad\qquad =  \frac{k_B\,{\tt T}}{\zeta_0}\, \frac{1}{K+1}\, \Delta_{x}
  \left(\frac{\psi}{\zeta}\right)
  +
  \frac{k_B\, {\tt T}}{\zeta_0}\,\sum_{i=1}^K \sum_{j=1}^K A_{ij} \nabla_{q_i} \cdot \left(\nabla_{q_j}\left(\frac{\psi}{\zeta}\right)\right),
\end{eqnarray}
which, in the special case of $\zeta(\rho(\xt,t)) \equiv 1$ collapses to \eqref{eq.fp-xq-1}.
Concerning alternative models with variable,
configuration-dependent drag, $\zeta_0\, \zeta(\qt)$, where $\zeta$ is a certain fixed function of
$\qt$, we refer to the comments at the end of this section.

The quantity $\varrho(\xt,t)=\int_D \psi(\xt,\qt,t)\dq$, signifying the density of polymer chains at the point $\xt \in \Omega$ and time $t\in [0,T]$, has units of length to the power minus $d$.
We shall assume temporarily, for the purposes of nondimensionalization, that there are no initial polymer molecule concentration gradients; i.e., that $\varrho(\cdot,0)$ is constant throughout the spatial domain $\Omega$; we denote this constant, called the {\em polymer density number}, by $n_p$. It then follows on integrating
\eqref{eq.fp-xq-2} over $D$, using the divergence theorem and dropping boundary integrals (cf. (\ref{eqpsi2aa},b) below)
that $\varrho$ is constant through $\Omega$ and throughout time.
We emphasize, that our analysis in subsequent sections does not require that $\varrho$ is constant in space or in time: we shall only demand that $\varrho(\cdot,0) = \int_D \psi_0(\cdot,\qt) \dq \in L^\infty(\Omega)$, in fact (cf. \eqref{inidata}). In any case, with $n_p$ thus defined,
we introduce nondimensionalized (hatted) variables in terms of their nonhatted counterparts as follows:
\begin{equation}\label{nondim-defs}
 \xt := L_0\widehat\xt, \quad \qt_i := \ell_0 {\widehat \qt}_i, \quad
\ut := U_0\widehat\ut, \quad t := (L_0/U_0)\,\widehat{t},\quad \mbox{and}\quad \psi:= (n_p/\ell_0^{Kd})\,\widehat\psi,
\end{equation}
where $\ell_0:= \sqrt{k_B\, {\tt T}/{\tt H}}$ is the characteristic length scale of a spring and $L_0$
and $U_0$ are the characteristic macroscopic length and velocity, respectively. Upon nondimensionalization using the
hatted variables, the Fokker--Planck equation \eqref{eq.fp-xq-2} becomes
\begin{eqnarray*}
&&\frac{\partial \widehat \psi}{\partial \widehat{t}} +
\nabla_{\widehat x} \cdot (\widehat\ut\, \widehat \psi) + \sum_{i=1}^K \nabla_{\widehat{q}_i} \cdot\left(
(\nabx \widehat\ut)\,\widehat{\undertilde{q}}_i\, \widehat\psi - \frac{1}{4\lambda}
\sum_{j=1}^K A_{ij} \,\widehat{\Ft}_j(\widehat{\qt}_j)\left(\frac{\widehat\psi}{\zeta}\right)\right)\\
&&\qquad = \frac{1}{4\lambda\, (K+1)} \left(\frac{\ell_0}{L_0}\right)^2 \Delta_{\widehat x} \left(\frac{\widehat\psi}{\zeta}\right)
+ \frac{1}{4\lambda} \sum_{i=1}^K \sum_{j=1}^K A_{ij} \nabla_{\widehat q_i} \cdot
\nabla_{\widehat q_j} \left(\frac{\widehat\psi}{\zeta}\right),
\end{eqnarray*}
where $\lambda := (\zeta_0/4H)\,(U_0/L_0)$ is the characteristic relaxation time of a spring (the Weissenberg number,
usually denoted by $\mathsf{Wi}$:  the ratio of the microscopic to macroscopic time scales),
\[\widehat{\Ut}_i(s) = \ell_0^{-2}\, \Ut_i(\ell_0^2\,s),\qquad \widehat{\Ft}_i(\widehat{\qt}_i)
= \widehat{\Ut}_i'({\textstyle \frac{1}{2}}|\widehat{\qt}_i|^2)\,\widehat{\qt}_i =
({\tt H}\,\ell_0)^{-1}\, \Ft_i(\ell_0\, \widehat{\qt}_i),\]
and the spatial, configurational
and temporal variables $\widehat{\xt}$, $\widehat{\qt}$ and $\widehat{t}$ now belong
to the rescaled domains
\[ \widehat \Omega := \Omega/L_0,\qquad \widehat{D}:= D/\ell_0,\qquad \mbox{and}\qquad
[0,\widehat{T}]\quad \mbox{with}\quad \widehat{T}:=U_0 T/L_0,\]
respectively. Discarding the hats
and writing $\varepsilon$ for the centre-of-mass diffusion coefficient; i.e.,
\[ \varepsilon:= \frac{1}{4\,\lambda\, (K+1)} \left(\frac{\ell_0}{L_0}\right)^2,\]
we have
\begin{eqnarray*}
&&\frac{\partial\psi}{\partial {t}} +
\nabla_{x} \cdot (\ut\,\psi) + \sum_{i=1}^K \nabqi \cdot\left(
(\nabx\ut)\,\qt_i\, \psi - \frac{1}{4\lambda}
\sum_{j=1}^K A_{ij} \,{\Ft}_i(\qt_j)\left(\frac{\psi}{\zeta}\right)\right)\\
&&\qquad = \varepsilon \,\Delta_{x} \left(\frac{\psi}{\zeta}\right)
+ \frac{1}{4\,\lambda} \sum_{i=1}^K \sum_{j=1}^K A_{ij} \nabqi \cdot
\nabqj \left(\frac{\psi}{\zeta}\right).
\end{eqnarray*}
Finally, on recalling the definition \eqref{MN} of the Maxwellian $M$ and noting the identities
\eqref{eqM}, the two terms involving the entries
$A_{ij}$ of the Rouse matrix $\Att$ in the last equation can be merged into a single term, resulting in
the Fokker--Planck equation (\ref{fp0}). From this nondimensionalization procedure one also obtains the important nondimensional parameter
\[ b_i := \frac{(q^{\rm max}_i)^2\,{\tt H}}{k_B\,{\tt T}} = \frac{(q^{\rm max}_i)^2}{\ell_0^2},\qquad i=1,\dots, K,\]
which measures how the maximal admissible extension $q^{\rm max}_i$ of the $i$th spring in the chain compares with the
characteristic microscopic length-scale $\ell_0$. Having defined $b_i$
we can express the non-dimensionalized configuration domain for the $i$th spring as $D_i =
\{\qt_i \in \mathbb{R}^d\,:\, |\qt_i| < \sqrt{b_i}\}$.

A noteworthy feature of equation (\ref{fp0}) in the model $({\rm P})$
compared to classical Fokker--Planck equations for bead-spring models in the
literature is the presence of the
$\xt$-dissipative centre-of-mass diffusion term $\varepsilon
\,\Delta_x \psi$ on the right-hand side of the Fokker--Planck equation (\ref{fp0}).
We refer to Barrett \& S\"uli \cite{BS} for the derivation of (\ref{fp0})
in the case of $K=1$ and constant $\rho$;
see also the article by Schieber \cite{SCHI}
concerning generalized dumbbell models with centre-of-mass diffusion,
and the recent paper of Degond \& Liu \cite{DegLiu} for a careful justification
of the presence of the centre-of-mass diffusion term through asymptotic analysis.
In standard derivations of bead-spring models the centre-of-mass
diffusion term is routinely omitted on the grounds that it is
several orders of magnitude smaller than the other terms in the
equation. Indeed, when the characteristic macroscopic
length-scale $L_0\approx 1$, (for example, $L_0 = \mbox{diam}(\Omega)$), Bhave,
Armstrong \& Brown \cite{Bh} estimate the ratio $\ell_0^2/L_0^2$ to
be in the range of about $10^{-9}$ to $10^{-7}$. However, the
omission of the term $\varepsilon \,\Delta_x \psi$ from (\ref{fp0})
in the case of a heterogeneous solvent velocity $\ut(\xt,t)$ is a
mathematically counterproductive model reduction. When $\varepsilon
\,\Delta_x\psi$ is absent, (\ref{fp0}) becomes a degenerate
parabolic equation exhibiting hyperbolic behaviour with respect to
$(\xt,t)$. Since the study of weak solutions to the coupled problem
requires one to work with velocity fields $\ut$ that have very
limited Sobolev regularity (typically $\ut \in
L^\infty(0,T;\Lt^2(\Omega)) \cap L^2(0,T; \Ht^1_0(\Omega))$), one is
then forced into the technically unpleasant framework of
hyperbolically degenerate parabolic equations with rough transport
coefficients (cf. Ambrosio \cite{Am}, DiPerna \& Lions \cite{DPL}, Mucha \cite{Mucha}).
The resulting difficulties are further
exacerbated by the fact that, when $D$ is bounded, a typical spring
force $\Ft(\qt)$ for a finitely extensible model (such as FENE)
explodes as $\qt$ approaches $\partial D$; see
Example~\ref{ex1.2} above. For these reasons,
as in our earlier papers in this field (cf. \cite{BS,BS2,BS2011-fene,BS2010-hookean}), we
shall retain the centre-of-mass diffusion term in (\ref{fp0}).

We continue with a brief
literature survey. Unless otherwise stated, the centre-of-mass
diffusion term is absent from the model considered in the cited reference
(i.e., $\varepsilon$ is set to $0$) and the density $\rho$ of the solvent is assumed
to be constant; also, in all references cited $K=1$,
i.e., a simple dumbbell model is considered rather than a bead-spring
chain model.

An early contribution to the existence and uniqueness of
local-in-time solutions to a family of dumbbell type polymeric
flow models is due to Renardy \cite{R}. While the class of
potentials $\Ft(\qt)$ considered by Renardy \cite{R} (cf.\
hypotheses (F) and (F$'$) on pp.~314--315) does include the case of
Hookean dumbbells, it excludes the practically relevant case of the
FENE dumbbell model (see Example~\ref{ex1.2} above). More recently, E, Li \&
Zhang \cite{E} and Li, Zhang \& Zhang \cite{LZZ} have revisited the
question of local existence of solutions for dumbbell models. A further
development in this direction is the work of Zhang \& Zhang \cite{ZZ},
where the local existence of regular solutions to FENE-type dumbbell
models has been shown. All of these papers require high regularity of the initial data.
Constantin \cite{CON} considered the Navier--Stokes equations
coupled to nonlinear Fokker--Planck equations describing the
evolution of the probability distribution of the particles
interacting with the fluid.
Otto \& Tzavaras \cite{OT} investigated the Doi model (which is
similar to a Hookean model (cf. Example \ref{ex1.1} above), except
that $D=S^2$) for suspensions of rod-like molecules in the dilute regime.
Jourdain, Leli\`evre \& Le Bris \cite{JLL2} studied the existence of
solutions to the FENE dumbbell model in the case of a simple Couette flow.
By using tools from the theory of stochastic differential equations, they showed the existence
of a unique local-in-time solution to the FENE dumbbell model for $d=2$
when the velocity field $\ut$ is unidirectional and of the particular form $\ut(x_1,x_2)
= (u_1(x_2),0)^{\rm T}$.

In the case of Hookean dumbbells ($K=1$), and assuming $\varepsilon=0$ and constant solvent
density $\rho$, the coupled
microscopic-macroscopic model described above yields, formally,
taking the second moment of $\qt \mapsto \psi(\qt,\xt,t)$, the fully macroscopic,
Oldroyd-B model of viscoelastic flow. Lions \& Masmoudi \cite{LM}
have shown the existence of global-in-time weak solutions to the Oldroyd-B model
in a simplified corotational setting (i.e., with $\sigma(\ut) = \nabxtt \ut$
replaced by $\frac{1}{2}(\nabxtt \ut - (\nabxtt u)^{\rm T})$)
by exploiting the propagation in time of the compactness of the
solution (i.e., the property that if one takes a sequence of weak solutions that
converges weakly and such that the corresponding sequence of
initial data converges strongly, then the weak
limit is also a solution) and the DiPerna--Lions \cite{DPL} theory of
renormalized solutions to linear hyperbolic
equations with nonsmooth transport coefficients. It is not
known if an identical global
existence result for the Oldroyd-B model also holds in the
absence of the crucial assumption
that the drag term is corotational. With $\varepsilon>0$ and constant solvent density $\rho$,
the coupled microscopic-macroscopic model above yields,
taking the appropriate moments in the case of Hookean dumbbells, a dissipative
version of the Oldroyd-B model. In this sense,
the Hookean dumbbell model has a macroscopic closure:
it is the Oldroyd-B model when $\varepsilon=0$, and a dissipative version of Oldroyd-B
when $\varepsilon>0$ (cf. Barrett \& S\"uli \cite{BS}).
Barrett \& Boyaval \cite{barrett-boyaval-09} have proved a global existence result
for this dissipative Oldroyd-B model in two space dimensions.
In contrast, the FENE model is not known to have an
exact closure at the macroscopic level, though Du, Yu \& Liu \cite{DU} and Yu, Du \&
Liu \cite{YU} have recently considered the analysis of approximate closures of
the FENE dumbbell model.
Lions \& Masmoudi \cite{LM2} proved the global existence of weak
solutions for the \textit{corotational} FENE dumbbell model,
once again corresponding to the
case of $\varepsilon=0$, constant solvent density $\rho$, and $K=1$, and the Doi model,
also called the rod model;
see also the work of Masmoudi \cite{M}. Recently, Masmoudi \cite{M10}
has extended this analysis to the noncorotational case.

Previously, El-Kareh \& Leal \cite{EKL} had proposed a steady macroscopic model,
with added dissipation in the equation satisfied by the conformation tensor, defined as
$$\Dtt(\xt):=\int_D\qt\,\qt^{\rm T} U'({\textstyle\frac{1}{2}}|\qt|^2) \,\psi(\xt,\qt) {\dd}\qt,$$
in order to account
for Brownian motion across streamlines; the model
can be thought of as an approximate macroscopic closure of a
FENE-type micro-macro model with centre-of-mass diffusion.

Barrett, Schwab \& S\"uli \cite{BSS}
showed the existence of global weak solutions to the coupled microscopic-macroscopic
model (\ref{ns1a}--d),
(\ref{fp0}) with $\varepsilon=0$, $K=1$, constant solvent-density $\rho$, an $\xt$-mollified
velocity gradient in the Fokker--Planck
equation and an $\xt$-mollified probability density function $\psi$ in the Kramers
expression, admitting a large class of potentials $U$
(including the Hookean dumbbell model and general FENE-type dumbbell models);
in addition to these
mollifications, $\ut$ in the
$\xt$-convective term $(\ut\cdot\nabx) \psi$ in the
Fokker--Planck equation was also mollified.
Unlike Lions \& Masmoudi \cite{LM}, the arguments in
Barrett, Schwab \& S\"uli \cite{BSS} did {\em not} require that
the drag term $\nabq\cdot(\sigtt(\ut)\,\qt\, \psi)$ in the Fokker--Planck equation
was corotational in the FENE case.

In Barrett \& S\"uli \cite{BS},
we derived the coupled Navier--Stokes--Fokker--Planck model
with centre-of-mass diffusion stated above, in the case of $K=1$ and constant solvent-density $\rho$.
We established the existence of global-in-time weak solutions to a mollification
of the model for a general class of spring-force-potentials including in
particular the FENE potential. We justified also, through a rigorous
limiting process, certain classical reductions of this model appearing
in the literature that exclude the centre-of-mass diffusion term from the
Fokker--Planck equation on the grounds that the diffusion coefficient is
small relative to other coefficients featuring in the equation.
In the case of a corotational drag term we performed a
rigorous passage to the limit as the
mollifiers in the Kramers expression and the drag term converge to identity operators.

In Barrett \& S\"uli \cite{BS2} we showed the existence of global-in-time
weak solutions to the general class of noncorotational FENE type dumbbell models
(including the standard
FENE dumbbell model) with centre-of-mass
diffusion, in the case of $K=1$ and constant solvent-density $\rho$, with microsropic  cut-off
(cf. \eqref{cut1} and \eqref{betaLa} below) in the drag term
\begin{equation}\label{drag}
 \nabq\cdot(\sigtt(\ut)\,\qt\, \psi) =
\nabq\cdot \left[\sigtt(\ut) \, \qt\,M\,\zeta(\rho)\left(\frac{\psi}{\zeta(\rho)\,M}\right)\right].
\end{equation}
Subsequently, in \cite{BS2011-fene} \cite{BS2010-hookean}, we removed the presence of the
cut-off by passing to the limit $L \rightarrow \infty$, with $K \geq 1$,
and the solvent density $\rho$, the viscosity $\mu$ and the drag coefficient $\zeta$ kept constant.
%
%

In this paper we prove the existence of global-in-time weak
solutions to 
FENE-type 
models {\em without} cut-off or mollification,
in the general case of $K \geq 1$ and with variable solvent-density $\rho$, variable viscosity
$\mu(\rho)$ and variable drag $\zeta(\rho)$.
This is achieved
by replacing the use of Dubinski\u{\i}'s compactness theorem in 
\cite{BS2011-fene} with 
the application of the Div-Curl lemma in our proof of relative compactness of the sequence of
approximating solutions to the Fokker--Planck equation in the Maxwellian-weighted $L^1$ space
$L^1(0,T;L^1_M(\Omega \times D))$.
Since the argument is long and technical, we give a brief overview of the main steps of the proof.

{\em Step 1.} Following the approach in Barrett \& S\"uli \cite{BS2,BS2011-fene,BS2010-hookean}
and motivated by recent papers of Jourdain, Leli\`evre, Le Bris \& Otto \cite{JLLO} and
Lin, Liu \& Zhang \cite{LinLZ}
(see also  Arnold, Markowich, Toscani \& Unterreiter \cite{AMTU},
and Desvillettes \& Villani \cite{DV})
concerning the convergence of the probability density function $\psi$ to its equilibrium
value $\psi_{\infty}(\xt,\qt):=M(\qt)$
(corresponding to the equilibrium value $\ut_\infty(\xt) :=\zerot$
of the velocity field in the case of constant density) in the absence of body forces $\ft$,
we observe that if $\psi/(\zeta(\rho)\,M)$ is bounded above then, for $L \in \mathbb{R}_{>0}$
sufficiently large, the drag term (\ref{drag}) is equal to
\begin{equation}\label{cut1}
\nabq\cdot \left[\sigtt(\ut) \, \qt\,M \,\zeta(\rho)\,\beta^L\left
(\frac{\psi}{\zeta(\rho)\,M}\right)\right],
\end{equation}
where $\beta^L \in C({\mathbb R})$ is a cut-off function defined as
\begin{align}
\beta^L(s) := \min(s,L).
\label{betaLa}
\end{align}
More generally, in the case of $K \geq 1$,
in analogy with \eqref{cut1}, the drag term with cut-off is
defined by
\[\sum_{i=1}^K \nabqi \cdot \left(\sigtt(\ut) \, \qt_i\,M \,\zeta(\rho)\,\beta^L\!\left(\frac{\psi}{\zeta(\rho)\,M}\right)\right).\]
It then follows that, for $L\gg 1$, any solution $\psi$ of (\ref{fp0}),
such that $\psi/(\zeta(\rho)\,M)$ is bounded above by $L$, also satisfies
\begin{eqnarray}
\label{eqpsi1aa}
&&\hspace{-6.5mm}\frac{\partial \psi}{\partial t} + (\ut \cdot\nabx) \psi
+ \sum_{i=1}^K \nabqi \cdot \left(\sigtt(\ut) \, \qt_i\,M \,\zeta(\rho)\,\beta^L\!\left(\frac{\psi}{\zeta(\rho)\,M}\right)\right)
\nonumber
\\
\bet
&&=
\epsilon\,\Delta_x \left(\frac{\psi}{\zeta(\rho)}\right) +
\frac{1}{4 \,\lambda}\,
\sum_{i=1}^K \sum_{j=1}^K
A_{ij}\,\nabqi \cdot \left(
M
\,\nabqj \left(\frac{\psi}{\zeta(\rho)\,M}\right)\right)
\quad \mbox{in } \Omega \times D \times
(0,T]. ~~~
\end{eqnarray}
%
Let
$\partial\overline{D}_i := D_1 \times \cdots \times D_{i-1} \times \partial D_i \times D_{i+1} \times \cdots \times D_K$.
We impose the following boundary and initial conditions:
\begin{subequations}
\begin{align}
&\left[\frac{1}{4\,\lambda} \sum_{j=1}^K A_{ij}\,
M\,\nabqj\!\left(\frac{\psi}{\zeta(\rho)\,M}\right)
- \sigtt(\ut) \,\qt_i\,M\,\zeta(\rho)\,\beta^L\!\left(\frac{\psi}{\zeta(\rho)\,M}\right)
\right]\! \cdot \frac{\qt_i}{|\qt_i|}
=0\qquad ~~\nonumber\\
&~ \hspace{1.49in}\mbox{on }
\Omega \times \partial \overline{D}_i\times (0,T],
\mbox{~~for $i=1,\dots, K$,} \label{eqpsi2aa}\\
&\epsilon\,\nabx \left(\frac{\psi}{\zeta(\rho)}\right)\,\cdot\,\nt =0
\qquad\mbox{on }
\partial \Omega \times D\times (0,T],\label{eqpsi2ab}\\
&\psi(\cdot,\cdot,0)=M(\cdot)\,\zeta(\rho_0(\cdot))\beta^L\!\left({\psi_{0}
(\cdot,\cdot)}/{(\zeta(\rho_0(\cdot))\,M(\cdot))}\right) \geq 0
\qquad \mbox{on $\Omega\times D$},\label{eqpsi3ac}
\end{align}
\end{subequations}
where $\qt_i$ is normal to $\partial D_i$, as $D_i$
is a bounded ball centred at the origin,
and $\nt$ is normal to $\partial \Omega$.

The initial datum $\psi_0$ for the Fokker--Planck equation is nonnegative,
defined on $\Omega\times D$, with
\[\int_{D} \psi_0(\xt,\qt) \dd \qt \in L^\infty(\Omega), \qquad \int_{\Omega \times D} \psi_0(\xt,\qt) \dq \dx = 1,\]
and assumed to have finite relative entropy with
respect to the Maxwellian $M$; i.e.
\[\int_{\Omega \times D} \psi_0(\xt,\qt)
\log \frac{\psi_0(\xt,\qt)}{M(\qt)} \dq \dx < \infty.\]
As we shall suppose throughout that the range of the function $\zeta$ is a compact subinterval $[\zeta_{\rm min}, \zeta_{\rm max}]$ of $(0,\infty)$, the finiteness of the relative entropy with respect to the Maxwellian $M$ is equivalent to demanding that
\[\int_{\Omega \times D} \frac{\psi_0(\xt,\qt)}{\zeta(\rho_0(\xt))}
\log \frac{\psi_0(\xt,\qt)}{\zeta(\rho_0(\xt))\,M(\qt)} \dq \dx < \infty.
\]
Clearly, if there exists $L>0$ such that $0 \leq \psi_0 \leq L\, \zeta(\rho_0)\,M$,
then $M\,\zeta(\rho_0)\,\beta^L(\psi_{0}/(\zeta(\rho_0)\,M)) = \psi_0$. Henceforth $L >1$ is assumed.

\begin{definition}
The coupled problem {\rm (\ref{ns0a}--f), (\ref{tau1}), (\ref{rho1}),
(\ref{eqpsi1aa}), (\ref{eqpsi2aa}--c) 
}
will be referred to
as model $({\rm P}_{L})$, or as the {\em general nonhomogeneous FENE-type bead-spring chain
model with centre-of-mass diffusion and microscopic cut-off}, with cut-off parameter $L>1$.
\end{definition}

In order to highlight the dependence
on $L$, in subsequent sections the solution to
(\ref{eqpsi1aa}), (\ref{eqpsi2aa}--c) will be labelled $\psiae$.
Because of the coupling of
(\ref{eqpsi1aa}) to (\ref{ns1a}) through (\ref{tau1}), the density, velocity and the pressure
will also depend $L$ and we shall therefore
denote them in subsequent
sections by $\rho_L$, $\ut_{L}$ and $p_{L}$.

The cut-off $\beta^{L}$ has a convenient property: when the solvent density $\rho$ is constant,
the couple $(\ut_{\infty},\psi_{\infty})$, defined by
$\ut_{\infty}(\xt) := \zerot$ and $\psi_{\infty}(\xt,\qt):=M(\qt)$,
which is an equilibrium solution of the  system
(\ref{ns1a}--d), (\ref{tau1}), (\ref{rho1}), (\ref{fp0}), with $\ft =\zerot$,
is still an equilibrium solution of the cut-off version of the problem, when (\ref{fp0})
is replaced by (\ref{eqpsi1aa}) (with boundary and initial conditions (\ref{eqpsi2aa}--c)),
for all $L>0$.
Thus, unlike the truncation of
the (unbounded) spring-potential proposed in the work of El-Kareh \& Leal \cite{EKL}
in the case of constant-density flows, the introduction of
the cut-off function $\beta^L$ into the Fokker--Planck equation (\ref{fp0})
does not alter the equilibrium solution $(\ut_{\infty},\psi_\infty)$ of the
original Navier--Stokes--Fokker--Planck system.
In addition, the boundary conditions for $\psi$
on $\partial\Omega\times D\times(0,T]$ and $\Omega \times \partial D
\times (0,T]$ ensure that
\[\int_{\Omega \times D}\psi(\xt,\qt,t) \dd \qt  =
\int_{\Omega \times D}\psi(\xt,\qt,0) \dd \qt =1\]
for a.e. $t \in \mathbb{R}_{\geq 0}$, in agreement with the requirement
that $\psi$ is a probability density function.

\textit{Step 2.} Ideally, one would like to pass to the limit $L\rightarrow \infty$
in problem $({\rm P}_{L})$ to
deduce the existence of solutions to $({\rm P})$. Unfortunately,
such a direct attack at the problem is (except in the special case of $d=2$, or
in the absence of
convection terms from the model,) fraught with technical difficulties. Instead,
we shall first (semi)discretize problem $({\rm P}_{L})$
by an implicit Euler scheme with respect to $t$, with step size $\Delta t$;
this then results in a time-discrete version $({\rm P}^{\Delta t}_{L})$ of
$({\rm P}_{L})$. By using Schauder's fixed point theorem, we will show in
Section \ref{sec:existence-cut-off} the existence of solutions to $({\rm P}^{\Delta t}_{L})$.
In the course of the proof, for technical reasons, a further cut-off,
now from below, is required,
with a cut-off parameter $\delta \in (0,1)$, which we shall let pass to $0$ to
complete the proof of existence of solutions to $({\rm P}^{\Delta t}_{L})$ in the limit
of $\delta \rightarrow 0_+$ (cf.
Section \ref{sec:existence-cut-off}).
Ultimately, of course, our aim is to show existence of weak solutions to the
general nonhomogeneous FENE-type bead-spring chain model with centre-of-mass diffusion, $({\rm P})$,
and that demands passing to
the limits $\Delta t \rightarrow 0_+$ and $L \rightarrow \infty$; this then brings us
to the next step in our argument.

{\em Step 3.} We shall link the time step $\Delta t$ to the cut-off parameter $L>1$ by
demanding that $\Delta t = o(L^{-1})$, as $L \rightarrow \infty$, so that the
only parameter in the problem $({\rm P}^{\Delta t}_{L})$ is the cut-off parameter (the centre-of-mass
diffusion parameter $\epsilon$ being fixed). We shall show that $\rho^{\Delta t}_L$ can be bounded, independent of the
cut-off parameter $L$, in $L^\infty(0,T;L^\infty(\Omega))$. By using special energy estimates, based on testing the
Fokker--Planck equation in $({\rm P}^{\Delta t}_{L})$ with the derivative of the relative entropy with respect to
the Maxwellian of the general nonhomogeneous FENE-type bead-spring
chain model, we show that $\ut^{\Delta t}_L$
can also be bounded, independent of $L$.
Specifically, $\ut^{\Delta t}_{L}$ is bounded in the norm of
the classical Leray space, independent of $L$;
also, the $L^\infty$ norm in time of the relative entropy of
$\psi^{\Delta t}_{L}/\zeta(\rho^{\Delta t}_L)$ and the $L^2$ norm in time
of the Fisher information of  $\widetilde{\psi}^{\Delta t}_{L}
:=\psi^{\Delta t}_{L}/(\zeta(\rho^{\Delta t}_L)\,M)$ are bounded,
independent of $L$.
We then use these $L$-independent
bounds on the relative entropy and the Fisher information to derive an
$L$-independent bound on a fractional-order in time Nikol'ski\u{\i} norm
of $\ut^{\Delta t}_L$.

\textit{Step 4.} The collection of $L$-independent bounds from Step 3,
then enables us to extract a weakly convergent subsequence of solutions to
problem $({\rm P}^{\Delta t}_{L})$ as $L \rightarrow \infty$; and then
further strongly convergent subsequences $\{\ut^{\Delta t}_{L}\}_{L>1}$
and $\{\rho^{\Delta t}_{L}\}_{L>1}$. The extraction of a strongly
convergent subsequence from the weakly convergent sequence $\{\widetilde\psi^{\Delta t}_{L}\}_{L>1}$
is considerably more complicated: after some technical preparation, we apply the Div-Curl lemma,
to finally obtain a strongly convergent subsequence of solutions
$(\rho^{\Delta t_k}_{L_k}, \ut^{\Delta t_k}_{L_k}, \widetilde\psi^{\Delta t_k}_{L_k})$
to $({\rm P}^{\Delta t}_{L})$ with $\Delta t = o(L^{-1})$ as $L \rightarrow \infty$,
in $L^\infty(0,T;L^p(\Omega)) \times L^2(0,T;L^r(\Omega))
\times L^p(0,T; L^1_M(\Omega \times D))$ for any $p \in [1,\infty)$; any $r \in [1,\infty)$ when $d=2$ and any $r \in [1,6)$ when $d=3$; enabling us to pass to the limit
with the microscopic cut-off parameter $L$ in the model $({\rm P}^{\Delta t}_{L})$, with
$\Delta t = o(L^{-1})$,
as $L \rightarrow \infty$, to finally deduce the existence of a weak solution to model $({\rm P})$,
the general nonhomogeneous FENE-type 
bead-spring chain models with centre-of-mass diffusion.

The paper is structured as follows. We begin, in Section~\ref{sec:2}, by stating $({\rm P}_{L})$,
the coupled nonhomogeneous Navier--Stokes--Fokker--Planck system with centre-of-mass diffusion and microscopic
cut-off for a general class of FENE-type 
spring potentials. In Section~\ref{sec:existence-cut-off}
we establish the existence of solutions to the time-discrete problem $({\rm P}^{\Delta t}_{L})$.
In Section~\ref{sec:entropy} we derive an $L$-independent bound on the solvent density $\rho^{\Delta t}_L$
in $L^\infty(0,T;L^\infty(\Omega))$; we also derive a set of $L$-independent bounds on
$\ut_L^{\Delta t}$ in the classical Leray space, together with $L$-independent bounds
on the relative entropy of
$\psi^{\Delta t}_{L}/\zeta(\rho^{\Delta t}_L)$ with respect to the Maxwellian $M$, and the $L^2$ norm in time
of the Fisher information of  $\widetilde{\psi}^{\Delta t}_{L}
:=\psi^{\Delta t}_{L}/(\zeta(\rho^{\Delta t}_L)\,M)$. We then use these $L$-independent bounds on spatial norms to show that the Nikol'ski\u{\i} norm $N^\gamma(0,T;\Lt^2(\Omega))$ of $\ut^{\Delta t}_L$ is
bounded, independent of $L$ and $\Delta t=o(L^{-1})$, for a suitable value of $\gamma \in (0,1)$.
This allows us to prove, via Simon's extension of the Aubin--Lions compactness theorem
(cf. \cite{Simon}), strong-convergence of the sequence
$\{\ut^{\Delta t}_L\}_{L>1}$ in $L^2(0,T;\Lt^r(\Omega))$ for $r \in [1,\infty)$ when $d=2$
and $r \in [1,6)$ when $d=3$. We then use this strong convergence result together with the
DiPerna--Lions theory of renormalized solutions to linear transport equations with nonsmooth
transport velocities to deduce the strong convergence of the sequence of approximate densities $\{\rho^{\Delta t}_L\}_{L>1}$, and pass to the limit in our approximation to the continuity equation,
as $L\rightarrow \infty$, with $\Delta t = o(L^{-1})$. Weak convergence of the sequence
$\{\widetilde\psi^{\Delta t}_L\}_{L>1}$ in the Maxwellian-weighted $L^1$ space
$L^1(0,T;L^1_M(\Omega \times D))$ is an immediate consequence of our entropy estimate, via
de la Vall\'ee-Poussin's theorem and the Dunford--Pettis theorem. The proof of the strong
convergence of the sequence is however considerably more complicated; it is established
in Section \ref{Lindep-time}, by first
developing interior estimates in standard (unweighted) Lebesgue and
Sobolev norms, exploiting the fact that on nonempty open relatively compact subsets of $D$ the Maxwellian is bounded above and below by positive constants. We then use these interior estimates in conjunction with the
Div-Curl lemma to deduce weak convergence of the sequence $(1+\widetilde\psi^{\Delta t}_L)^{\alpha + 1}$
on nonempty open relatively compact subsets of
$(0,T) \times \Omega \times D$, where $\alpha \in (0,1)$. Thus we can make use of
the fact that the continuous functions $s \in [0,\infty) \mapsto (1+s)^{1+\alpha}$ and $s \in [0,\infty) \mapsto s^{\alpha}$ are, respectively, strictly convex and strictly concave, and therefore weakly lower (respectively, upper) semicontinuous, to deduce that $\{\widetilde\psi^{\Delta t}_L\}_{L>1}$ converges
to a limiting function $\widetilde\psi \in L^1(0,T;L^1_M(\Omega \times D))$, almost everywhere on compact subsets of $(0,T)\times\Omega \times D$; hence, by using a nested sequence of nonempty open relatively compact sets, we show that $\{\widetilde\psi^{\Delta t}_L\}_{L>1}$ converges
to $\widetilde\psi$ almost everywhere on $(0,T)\times\Omega \times D$. Thanks to the fact that
$M(\qt) \dq$ is a probability measure on $D$, and therefore $M(\qt) \dq \dx \dt$ is a finite measure
of $(0,T)\times\Omega \times D$,
Egoroff's theorem then implies almost uniform convergence of the sequence, and therefore also
convergence in measure; thus we can appeal to Vitali's theorem to finally deduce strong convergence in
$L^1(0,T;L^1_M(\Omega \times D))$ of (a subsequence of) the sequence $\{\widetilde\psi^{\Delta t}_L\}_{L>1}$.  This strong convergence result
then allows us in Section~\ref{5-sec:passage.to.limit} to pass to the limit
with the cut-off parameter $L$ in problem $({\rm P}^{\Delta t}_{L})$, with $\Delta t = o(L^{-1})$,
as $L \rightarrow \infty$, to deduce the existence of a weak solution $(\rho, \ut,\psi:=M\,\zeta(\rho)\, \widetilde\psi)$ to problem $({\rm P})$, the general nonhomogeneous FENE-type 
bead-spring chain models with centre-of-mass diffusion.
We shall operate within Maxwellian-weighted Sobolev spaces, which provide the
natural functional-analytic framework for the problem.
Our proofs require special
density and embedding results in these spaces that
are proved,  respectively, in Appendix C and Appendix D of \cite{BS2010}, which is an extended
version of our paper \cite{BS2011-fene} for FENE-type models.

Our techniques can be easily modified to prove large-data global existence of weak solutions to
kinetic models with configuration-dependent drag, where instead of being a nonlinear function of
the unknown density $\rho$, as in \eqref{fp0}, the drag coefficient $\zeta$ is a
given $C^1$ function of $\qt$,
bounded above and below by positive constants
(cf. Hinch \cite{Hinch}, Larson \cite{Larson}, Schr\"oder et al. \cite{Schroder}
and references therein). The idea behind these models, which have been developed to explain
physical mechanism by which large stresses rapidly build up in dilute polymer solutions,
is that of a bead friction coefficient that depends strongly on the inter-bead distance
through a nonlinear friction law. This principle of conformation-dependent hydrodynamic drag
assumes that as a chain becomes extended by the flow, the strength of the hydrodynamic friction
on the chain will also increase; see, \cite{Omowunmi} for a detailed survey.

\section{The polymer model $({\rm P}_{L})$}
\label{sec:2}
\setcounter{equation}{0}

Let $\Omega \subset {\mathbb R}^d$ be a bounded open set with a
Lipschitz-continuous boundary $\partial \Omega$, and suppose that
the set $D:= D_1\times \cdots \times D_K$ of admissible
conformation vectors $\qt := (\qt_1^{\rm T}, \ldots ,
\qt_K^{\rm T})^{\rm T}$ in (\ref{fp0}) is
such that $D_i$, $i=1, \dots, K$, is an open ball
in ${\mathbb R}^d$, $d=2$ or $3$, centred at the origin with boundary $\partial D_i$
and radius $\sqrt{b_i}$, $b_i>2$; let
\begin{align}
\partial D := \bigcup_{i=1}^K \partial \overline D_i,\qquad \mbox{where}\qquad \partial \overline D_i:= D_1 \times \cdots \times D_{i-1} \times \partial D_i \times D_{i+1} \times \cdots \times D_K.
\label{dD}
\end{align}
Collecting (\ref{ns0a}--f), (\ref{tau1}), (\ref{rho1}), (\ref{eqpsi1aa}) and (\ref{eqpsi2aa}--c),
we then consider the following initial-boundary-value problem,
dependent on the parameter $L > 1$. As has been already emphasized in the
Introduction, the centre-of-mass diffusion coefficient $\varepsilon>0$ is a
physical parameter and is regarded as being fixed.

(${\rm P}_{L}$)
Find
$ \rho_L\,:\,(\xt,t)
\in \Omega \times [0,T] \mapsto
\rho_L(\xt,t) \in {\mathbb R}$,
$\utae\,:\,(\xt,t)\in \overline{\Omega} \times [0,T]
\mapsto \utae(\xt,t) \in {\mathbb R}^d$ and $\pae\,:\, (\xt,t) \in
\Omega \times (0,T] \mapsto \pae(\xt,t) \in {\mathbb R}$ such that
\begin{subequations}
\begin{alignat}{2}
\frac{\partial \rho_L}{\partial t}
+ \nabx\cdot(\ut_L \,\rho_L)
&= \zerot \qquad &&\mbox{in } \Omega \times (0,T],\label{equ0}\\
\rho_L(\xt,0)&=\rho_{0}(\xt)    \qquad &&\forall \xt \in \Omega,\label{equ00} \\
\frac{\partial (\rho_L\,\ut_L)}{\partial t}
+ \nabx\cdot(\rho_L\,\ut_L \otimes \ut_L) - 
\nabx \cdot (\mu(\rho_L)  \,\Dtt(\ut_L))
+ \nabx p_L
&= \rho_L\,\ft + \nabx \cdot \tautt(\psi_L)
&&
\nonumber
\\
&&&\mbox{in } \Omega \times (0,T],\label{equ1}\\
\nabx \cdot \utae &= 0
&&\mbox{in } \Omega \times (0,T],
\label{equ2}\\
\utae &= \zerot  &&\mbox{on } \partial \Omega \times (0,T],
\label{equ3}\\
(\rho_L\,\utae)(\xt,0)&=(\rho_0\,\ut_{0})(\xt) &&\forall \xt \in \Omega,
\label{equ4}
\end{alignat}
\end{subequations}
where $\psiae\,:\,(\xt,\qt,t)\in
\overline{\Omega} \times \overline{D} \times [0,T]
\mapsto \psiae(\xt,\qt,t)
\in {\mathbb R}$, and
$\tautt(\psiae)\,:\,(\xt,t) \in \Omega \times (0,T] \mapsto
\tautt(\psiae)(\xt,t)\in \mathbb{R}^{d\times d}$ is the symmetric
extra-stress tensor defined as
\begin{equation}
\tautt(\psiae) := k \left[\left( \sum_{i=1}^K
\Ctt_i(\psiae)\right)
- K\,\varrho(\psiae)\, \Itt\right].
\label{eqtt1}
\end{equation}
Here $k \in \Rplus$,
$\Itt$ is the unit $d
\times d$ tensor,
\begin{subequations}
\begin{eqnarray}
\Ctt_i(\psiae)(\xt,t) &:=& \int_{D} \psiae(\xt,\qt,t)\, U_i'({\textstyle
\frac{1}{2}}|\qt_i|^2)\,\qt_i\,\qt_i^{\rm T} \dq,\qquad \mbox{and} \label{eqCtt}\\
\varrho(\psiae)(\xt,t) &:=& \int_{D} \psiae(\xt,\qt,t)\dq. \label{eqrhott}
\end{eqnarray}
\end{subequations}
The Fokker--Planck equation with microscopic cut-off satisfied by $\psiae$ is:
\begin{align}
\label{eqpsi1a}
&\hspace{-2mm}\frac{\partial \psiae}{\partial t} + (\utae \cdot\nabx) \psiae +
\sum_{i=1}^K \nabqi
\cdot \left[\sigtt(\utae) \, \qt_i\,M\,\zeta(\rho_L)\,\beta^L
\left(\frac{\psiae}{\zeta(\rho_L)\,M}\right)\right]
\nonumber
\\
&\hspace{0.01in} =
\epsilon\,\Delta_x\left(\frac{\psiae}{\zeta(\rho_L)}\right) +
\frac{1}{4 \,\lambda}\,
\sum_{i=1}^K \sum_{j=1}^K
A_{ij}\,\nabqi \cdot \left(
M \,\nabqj \left(\frac{\psiae}{\zeta(\rho_L)\,M}\right)\right)
\quad \mbox{in } \Omega \times D \times
(0,T].
\end{align}
Here, for a given $L > 1$, $\beta^L \in C({\mathbb R})$ is defined by (\ref{betaLa}),
$\sigtt(\vt) \equiv \nabxtt \vt$, and
\begin{align}
\!\!\!\!\!\Att \in {\mathbb R}^{K \times K} \mbox{ is symmetric positive definite
with smallest eigenvalue $a_0 \in {\mathbb R}_{>0}$.}
\label{A}
\end{align}
%
We impose the following boundary and initial conditions:
\begin{subequations}
\begin{align}
&\left[\frac{1}{4\,\lambda} \sum_{j=1}^K A_{ij}\,
M\,
\nabqj \left(\frac{\psiae}{\zeta(\rho_L)\,M}\right)
- \sigtt(\utae) \,\qt_i\,M\,\zeta(\rho_L)\,\beta^L\left(\frac{\psiae}{\zeta(\rho_L)\,M}\right)
\right]\cdot \frac{\qt_i}{|\qt_i|}
=0 \nonumber \\
&\hspace{2.105in}\mbox{on }
\Omega \times \partial \overline{D}_i \times (0,T],
%
\quad i=1, \dots, K,
\label{eqpsi2}\\
&\epsilon
\,\nabx \left(\frac{\psiae}{\zeta(\rho_L)}\right)\,\cdot\,\nt =0 \qquad \,\, \quad\qquad\, \mbox{on }
\partial \Omega \times D\times (0,T],\label{eqpsi2a}\\
&\psiae(\cdot,\cdot,0)=M(\cdot)\,\zeta(\rho_0(\cdot))\,
\beta^L(\psi_{0}(\cdot,\cdot)/(\zeta(\rho_0(\cdot))\,M(\cdot))) \geq 0 \qquad \mbox{on  $\Omega\times D$},\label{eqpsi3}
\end{align}
\end{subequations}
where $\nt$ is the unit outward normal to $\partial \Omega$.
The boundary conditions for $\psiae$
on $\partial\Omega\times D\times(0,T]$ and $\Omega \times \partial D
\times (0,T]$ have been chosen so as to ensure that
\begin{align}
\int_{\Omega \times D}\psiae(\xt,\qt,t) \dq \dx  =
\int_{\Omega \times D} \psiae(\xt,\qt,0) \dq \dx  \qquad \forall t \in (0,T].
\label{intDcon}
\end{align}
Henceforth, we shall write
\[\widetilde\psi_L = \frac{\psi_L}{\zeta(\rho_L)\,M},\quad
  \widetilde\psi_0 = \frac{\psi_0}{\zeta(\rho_0)\,M}.\]
Thus, for example, \eqref{eqpsi3} in terms of this compact notation becomes:
$\widetilde \psi_L(\cdot,\cdot,0) = \beta^L(\widetilde\psi_0(\cdot,\cdot))$ on $\Omega \times D$.

The notation $|\cdot|$ will be used to signify one of the following.
When applied to a real number $x$,
$|x|$ will denote the absolute value of the number $x$; when applied to a vector
$\vt$,  $|\vt|$ will
stand for the Euclidean norm of the vector $\vt$; and, when applied to a square matrix
$\Att$, $|\Att|$ will
signify the Frobenius norm, $[\mathfrak{tr}(\Att^{\rm T}\Att)]^{\frac{1}{2}}$, of the matrix
$\Att$, where, for a square matrix
$\Btt$, $\mathfrak{tr}(\Btt)$ denotes the trace of $\Btt$.

\section{Existence of a solution to the discrete-in-time problem}
\label{sec:existence-cut-off}
\setcounter{equation}{0}

Let
\begin{eqnarray}
&\Ht :=\{\wt \in \Lt^2(\Omega) : \nabx \cdot \wt =0\} \quad
\mbox{and}\quad \Vt :=\{\wt \in \Ht^{1}_{0}(\Omega) : \nabx \cdot
\wt =0\},&~~~ \label{eqVt}
\end{eqnarray}
where the divergence operator $\nabx\cdot$ is to be understood in
the sense of distributions on $\Omega$. Let $\Vt'$ be
the dual of $\Vt$.

For later purposes, we recall the following well-known
Gagliardo--Nirenberg inequality. Let $r \in [2,\infty)$ if $d=2$,
and $r \in [2,6]$ if $d=3$ and $\theta = d \,\left(\frac12-\frac
1r\right)$. Then, there is a constant $C=C(\Omega,r,d)$,
such that, for all $\eta \in H^{1}(\Omega)$:
\begin{equation}\label{eqinterp}
\|\eta\|_{L^r(\Omega)}
\leq C\,
\|\eta\|_{L^2(\Omega)}^{1-\theta}
\,\|\eta\|_{H^{1}(\Omega)}^\theta.
\end{equation}

Let $\mathcal{F}\in C(\mathbb{R}_{>0})$ be defined by
$\mathcal{F}(s):= s\,(\log s -1) + 1$, $s>0$.
As $\lim_{s \rightarrow 0_+} \mathcal{F}(s) = 1$,
the function $\mathcal{F}$ can be considered
to be defined and continuous on $[0,\infty)$,
where it is a nonnegative, strictly convex function
with $\mathcal{F}(1)=0$.
We assume the following:
\begin{align}\nonumber
&\partial \Omega \in C^{0,1}; \quad
\rho_0 \in [\rho_{\rm min},\rho_{\rm max}], \mbox{ with } \rho_{\rm min}>0;
\quad \ut_0 \in \Ht;\quad
\psi_0 \geq 0 \ {\rm ~a.e.\ on}\ \Omega \times D\\
&\mbox{with }
\mathcal{F}(\widetilde\psi_0)
\in L^1_M(\Omega \times D) 
\mbox{ and}
\int_{D} \psi_0(\xt,\qt)\,\dq \in L^\infty(\Omega); \quad \int_{\Omega \times D} \psi_0(\xt,\qt)\,\dq \dx = 1;
\nonumber\\
&\mbox{the Rouse matrix $\Att\in \mathbb{R}^{K \times K}$ satisfies (\ref{A})};
\nonumber
\\
&\mu \in C([\rho_{\rm min},\rho_{\rm max}], [\mu_{\rm min},\mu_{\rm max}]),
\quad \zeta \in C^1([\rho_{\rm min},\rho_{\rm max}], [\zeta_{\rm min},\zeta_{\rm max}]),
\mbox{ with } \mu_{\rm min},\,\zeta_{\rm min} >0, \nonumber
\\
&\ft \in L^{2}(0,T;\Lt^\varkappa(\Omega))
\quad \mbox{and} \quad
D_i = B(\zerot,b^{\frac{1}{2}}_i), \quad \gamma_i>1,\quad i= 1, \dots,
K \quad  \mbox{in (\ref{growth1},b)};
\label{inidata}
\end{align}
where $\varkappa >1$ if $d=2$ and $\varkappa = \frac{6}{5}$ if $d=3$.
For this range of $\varkappa$, we note from (\ref{eqinterp}) that there exists a constant
$C_\varkappa \in \Rplus$ such that
\begin{align}
\left|\int_{\Omega} \wt_1 \cdot \wt_2 \dx \right| \leq C_\varkappa\, \|\wt_1\|_{L^\varkappa(\Omega)}\,
\|\wt_2\|_{H^1(\Omega)} \qquad \forall \wt_1 \in \Lt^{\varkappa}(\Omega), \; \wt_2 \in \Ht^1_0(\Omega).
\label{fvkbd}
\end{align}

In (\ref{inidata}), $L^p_M(\Omega \times D)$, for $p\in [1,\infty)$,
denotes the Maxwellian-weighted $L^p$ space over $\Omega \times D$ with norm
\[
\| \varphi\|_{L^{p}_M(\Omega\times D)} :=
\left\{ \int_{\Omega \times D} \!\!M\,
|\varphi|^p \dq \dx
\right\}^{\frac{1}{p}}.
\]
Similarly, we introduce $L^p_M(D)$,
the Maxwellian-weighted $L^p$ space over $D$. Letting
\begin{eqnarray}
\| \varphi\|_{H^{1}_M(\Omega\times D)} &:=&
\left\{ \int_{\Omega \times D} \!\!M\, \left[
|\varphi|^2 + \left|\nabx \varphi \right|^2 + \left|\nabq
\varphi \right|^2 \,\right] \dq \dx
\right\}^{\frac{1}{2}}\!\!, \label{H1Mnorm}
\end{eqnarray}
we then set
\begin{eqnarray}
\quad X \equiv H^{1}_M(\Omega \times D)
&:=& \left\{ \varphi \in L^1_{\rm loc}(\Omega\times D): \|
\varphi\|_{H^{1}_M(\Omega\times D)} < \infty \right\}. \label{H1M}
\end{eqnarray}
It is shown in Appendix C of 
\cite{BS2010} (with the set $X$ denoted by $\widehat X$ there)
that
\begin{align}
C^{\infty}(\overline{\Omega \times D})
\mbox{ is dense in } X.
\label{cal K}
\end{align}

We have from Sobolev embedding that
\begin{equation}
H^1(\Omega;L^2_M(D)) \hookrightarrow L^s(\Omega;L^2_M(D)),
\label{embed}
\end{equation}
where $s \in [1,\infty)$ if $d=2$ or $s \in [1,6]$ if $d=3$.
Similarly to (\ref{eqinterp}) we have,
with $r$ and $\theta$ as there,
that there is a constant $C$, depending only  on
$\Omega$, $r$ and $d$, such that
\begin{equation}\label{MDeqinterp}
\hspace{-0mm}\|\varphi\|_{L^r(\Omega;L^2_M(D))}
\leq C\,
\|\varphi\|_{L^2(\Omega;L^2_M(D))}^{1-\theta}
\,\|\varphi\|_{H^1(\Omega;L^2_M(D))}^\theta \qquad
\mbox{$\forall\varphi \in
H^{1}(\Omega;L^2_M(D))$}.
\end{equation}
In addition, we note that the embeddings
\begin{subequations}
\begin{align}
 H^1_M(D) &\hookrightarrow L^2_M(D) ,\label{wcomp1}\\
H^1_M(\Omega \times D) \equiv
L^2(\Omega;H^1_M(D)) \cap H^1(\Omega;L^2_M(D))
&\hookrightarrow
L^2_M(\Omega \times D) \equiv L^2(\Omega;L^2_M(D))
\label{wcomp2}
\end{align}
\end{subequations}
are compact if $\gamma_i \geq 1$, $i=1, \dots,  K$, in (\ref{growth1},b);
see Appendix D
of \cite{BS2010}.

We recall the Aubin--Lions--Simon compactness theorem, see, e.g.,
Simon \cite[Theorem 5]{Simon}. Let $\mathcal{B}_0$, $\mathcal{B}$ and
$\mathcal{B}_1$ be Banach
spaces, $\mathcal{B}_i$, $i=0,1$, reflexive, with a compact embedding $\mathcal{B}_0
\hookrightarrow \mathcal{B}$ and a continuous embedding $\mathcal{B} \hookrightarrow
\mathcal{B}_1$. 
Then, any bounded closed subset $E$ of $L^2(0,T;{\mathcal B}_0)$, such that
%
\begin{align}
\int_\theta^{T}
\|\eta(t)-\eta(t-\theta)\|_{\mathcal B_1}^2 \dt \rightarrow 0 \mbox{\quad as $\theta
\rightarrow 0$,
\quad uniformly for $\eta \in E$},
\label{compact1}
\end{align}
is compact in $L^2(0,T;{\mathcal B})$.

We shall also require the following result (cf. Lemma 1.1, Chapter III, \P 1, in \cite{Temam}).
\begin{lemma}\label{Bochner}
Suppose that $\mathcal{B}$ is a Banach space with dual space
$\mathcal{B}'$, $(a,b)$ is a bounded open subinterval of $\mathbb{R}$,
and $u$ and $g$ are two functions in
$L^1(a,b;\mathcal{B})$. Then, the following three conditions are
equivalent:
\begin{itemize}
\item[(i)] $u$ is almost everywhere equal to the primitive of $g$; i.e.,
\[ u(t) = \xi + \int_a^t g(s) \dd s,\qquad \xi \in \mathcal{B},\quad \mbox{for a.e.
$t \in [a,b]\,;$}\]
\item[(ii)] For each test function $\varphi \in C^\infty_0(a,b)$,
\[\int_a^b u(t)\,\frac{\dd \varphi}{\dd t}(t) \dd t = - \int_a^b g(t) \, \varphi(t) \dd t;
\]
\item[(iii)] For each $\eta \in \mathcal{B}'$,
\[\frac{\rm d}{\dd t}\langle  u , \eta \rangle = \langle g , \eta \rangle\]
in the sense of scalar-valued distributions on $(a,b)$, where $\langle \cdot,\cdot \rangle$
is the duality pairing between $\mathcal{B}$ and $\mathcal{B}'$.
\end{itemize}
If any of the conditions (i)--(iii) holds, then $u$ is almost everywhere equal to a certain
continuous function from $[a,b]$ into $\mathcal{B}$.
\end{lemma}

\begin{corollary}\label{Bochner2}
Suppose that $(a,b)$ is a bounded open subinterval of $\mathbb{R}$ and let $\mathcal{B}$ be a Banach space.
Suppose further that $u\in W^{1,1}(a,b;\mathcal{B})$; then,
\begin{equation*}
u(t) = u(a) + \int_{a}^t \frac{\dd u}{\dd s}(s) \dd s\qquad \mbox{for all $t \in [a,b]$.}
\end{equation*}
\end{corollary}
\begin{proof}
Thanks to our assumption that $u \in W^{1,1}(a,b; \mathcal{B})$, and defining $g=\dd u/\!\dd t$,
we have that $u, g \in L^1(a,b;\mathcal{B})$;
thus by the implication (ii) $\Rightarrow$ (i) in Lemma \ref{Bochner} we deduce that
\begin{equation}\label{Bochner3a}
u(t) = \xi + \int_a^t \frac{\dd u}{\dd s}(s) \dd s, \qquad \xi \in \mathcal{B}, \quad \mbox{for a.e. $t \in
[a,b]$.}
\end{equation}
Since, again by Lemma \ref{Bochner}, $u$ is a continuous function from $[a,b]$  to $\mathcal{B}$,
it follows that \eqref{Bochner3a} holds for all $t \in [a,b]$. Hence, by taking $t=a$ in \eqref{Bochner3a}, we get
that $\xi = u(a)$. 
\end{proof}

The next result is stated without proof as Corollary 1, Ch. XVIII, p.470 in the book of Dautray \& Lions \cite{DL}.
\begin{corollary}\label{Bochner4}
Suppose that $g \in L^1(a,b; \mathcal{B})$ and $\eta \in \mathcal{B}'$. Then,
\[ \bigg \langle \int_a^b g(t) \dd t , \eta \bigg \rangle = \int_a^b \langle g(t) , \eta \rangle \dd t,\]
where $\langle \cdot,\cdot \rangle$
is the duality pairing between $\mathcal{B}$ and $\mathcal{B}'$.
\end{corollary}
\begin{proof}
We define the $\mathcal{B}$-valued distribution $u$ by
\[ u(t) = \int_a^t g(s) \dd s,\qquad \mbox{for a.e. $t \in [a,b]$}.\]
It then follows that $u, g \in L^1(a,b; \mathcal{B})$ and condition (i) of Lemma \ref{Bochner} holds
with $\xi = 0$. Hence, by Lemma \ref{Bochner}, condition (iii) of Lemma \ref{Bochner} also holds, which, after
integration over $[a,b]$, then yields
\begin{equation}\label{Bochner4a}
\langle u(b), \eta \rangle - \langle u(a) , \eta \rangle = \int_a^b \langle g(s) , \eta \rangle  \dd s.
\end{equation}
The definition of the duality pairing $\langle \cdot , \cdot \rangle$ implies that
$|\langle g(s) , \eta \rangle| \leq \|g(s)\|_{\mathcal{B}} \|\eta\|_{\mathcal{B}'}$ for all $s \in [a,b]$ and all $\eta \in \mathcal{B}'$; and, since $g \in L^1(a,b; \mathcal{B})$, the definition of the Bochner integral implies that $\int_a^b \|g(s)\|_{\mathcal{B}} \dd s
< \infty$. It thus follows that right-hand side of \eqref{Bochner4a} is finite for all $\eta \in \mathcal{B}'$.
By noting that $u(b) = \int_a^b g(s) \dd s$ and $u(a) = 0$, the identity \eqref{Bochner4a} becomes
\[\bigg \langle \int_a^b g(s) \dd s , \eta \bigg\rangle = \int_a^b \langle g(s) , \eta \rangle \dd s,\]
which, after renaming the dummy integration variable $s$ into $t$, yields the required result.
\end{proof}
Throughout we will assume that
(\ref{inidata}) hold,
so that (\ref{additional-1}) and (\ref{wcomp1},b) hold.
We note for future reference that (\ref{eqCtt}) and
(\ref{additional-1}) yield that, for
$\varphi \in L^2_M(\Omega \times D)$,
\begin{align}
\label{eqCttbd}
\int_{\Omega} |\Ctt_i( M\,\varphi)|^2\,\dx & =
\int_{\Omega} 
\left|
\int_{D} M\,\varphi \,U_i'\,\qt_{i}\,\qt_{i}^{\rm T} \dq \right|^2 \dx
\nonumber
\\
&\leq
\left(\int_{D} M\,(U_i')^2 \,|\qt_i|^4 \,\dq\right)
\left(\int_{\Omega \times D} M\,|\varphi|^2 \dq \dx\right)
\nonumber \\
& \leq C
\left(\int_{\Omega \times D} M\,|\varphi|^2 \dq \dx\right),
\qquad i=1, \dots,  K,
\end{align}
where $C$ is a positive constant.

We state a simple integration-by-parts formula.

\begin{lemma} Let $\varphi \in H^1_M(D)$ and suppose that $B \in
\mathbb{R}^{d\times d}$ is a square
matrix such that $\mathfrak{tr}(B)=0$; then,
\begin{equation}\label{intbyparts}
\int_D M\, \sum_{i=1}^K (B\qt_i) \cdot \nabqi \varphi \dd \qt = \int_D M\,
\varphi \sum_{i=1}^KU_i'(\textstyle{\frac{1}{2}|\qt_i|^2})\,  \qt_i  \qt_i^{\rm T}: B \dq.
\end{equation}
\end{lemma}
\begin{proof}
See the proof of Lemma 3.1 in \cite{BS2011-fene}.
\end{proof}

We now formulate our discrete-in-time approximation of problem
(P$_{L})$ for fixed parameters $\epsilon \in (0,1]$
and $L > 1$. For any $T>0$ and $N \geq 1$, let
$N \,\Delta t=T$ and $t_n = n \, \Delta t$, $n=0, \dots,  N$.
To prove existence of a solution under minimal
smoothness requirements on the initial datum $\ut_0 \in \Ht$ (recall (\ref{inidata})),
we assign to it the function $\ut^0 = \ut^0(\Delta t) \in \Vt$,  defined as the unique solution of
\begin{alignat}{2}
\int_{\Omega} \left[ \rho_0\,\ut^0 \cdot \vt + \Delta t\,
\nabxtt \ut^0 : \nabxtt \vt \right] \dx
&=   \int_{\Omega} \rho_0 \,\ut_0 \cdot \vt \dx \qquad
&&\forall \vt \in \Vt.
\label{proju0}
\end{alignat}
Hence,
\begin{equation}
\int_{\Omega} [\,\rho_0\,|\ut^0|^2 + \Delta t \,|\unabtt^0|^2 \,]\dx 
\leq
\int_{\Omega} \rho_0\,|\ut_0|^2 \dx
\leq C.
\label{idatabd}
\end{equation}
In addition, we have
that $\int_\Omega \rho_0\,(\ut^0-\ut_0)\cdot\vt \dx\,$ converges to $0$ for all $\vt \in\Ht$ in the limit of $\Delta t \rightarrow 0_+$.

Analogously to defining  $\ut^0 \in \Vt$ for a given initial velocity field
$\ut_0 \in \Ht$,
we shall assign a certain `smoothed' initial
datum,
\[\widetilde\psi^0 = \widetilde\psi^0(L,\Delta t) \in H^1_M(\Omega \times D),\]
to the given initial datum $\widetilde\psi_0 = \psi_0/(\zeta(\rho_0)\,M)$ such that
\begin{align}
&\int_{\Omega \times D} M \left[ \zeta(\rho_0)\,\widetilde \psi^0\, \varphi +
\Delta t\, \left( \nabx \widetilde \psi^0 \cdot \nabx \varphi +
\nabq \widetilde \psi^0 \cdot \nabq \varphi
\right) \right] \dq \dx \nonumber\\
&\hspace{2in}= \int_{\Omega \times D} M \,\zeta(\rho_0)\,\beta^{L}(\widetilde\psi_0)\, \varphi \dq \dx
\qquad \forall \varphi \in
H^1_M(\Omega \times D).
\label{psi0}
\end{align}
For $p\in [1,\infty)$, let
\begin{align}
Z_p &:= \left\{ \varphi \in L^p_M(\Omega \times D) :
\varphi
\geq 0 \mbox{ a.e.\ on } \Omega \times D
\mbox{ and }
\int_D M(\qt)\,\varphi(\xt,\qt) \dq \in L^\infty(\Omega)
\right\}.
\label{hatZ}
\end{align}
It is proved in the special case of $\zeta(\rho_0(\cdot)) \equiv 1$ in the Appendix of \cite{BS2011-feafene} that there exists a unique
$\widetilde\psi^0 \in H^1_M(\Omega \times D)$ satisfying \eqref{psi0}; furthermore,
$\widetilde\psi^0 \in Z_2$,
\begin{subequations}
\begin{align}
&\int_{\Omega \times D} M\,\zeta(\rho_0)\, \mathcal{F}(\widetilde\psi^0) \dq \dx
\nonumber \\
& \qquad\qquad\qquad\qquad + 4\,\Delta t \int_{\Omega \times D} M\,\left[
\big|\nabx \sqrt{\widetilde\psi^0} \big|^2 + \big|\nabq \sqrt{\widetilde\psi^{0}}\big|^2
\right]\!\dq \dx
\nonumber \\ & \qquad\qquad\qquad\qquad\qquad\qquad\qquad\qquad
\leq \int_{\Omega \times D} M \,\zeta(\rho_0)\,\mathcal{F}(\widetilde\psi_0) \dq \dx,
\label{inidata-1}
\end{align}
\begin{equation}
\mbox{ess.sup}_{\xt \in \Omega} \int_D M\, \widetilde{\psi}^0 \dq \leq
\mbox{ess.sup}_{\xt \in \Omega} \int_D M\, \widetilde{\psi}_0 \dq
\label{inidata-2}
\end{equation}
and
\begin{align}
\widetilde \psi^0 = \beta^L(\widetilde \psi^0) \rightarrow \widetilde \psi_0 \quad
\mbox{weakly in }L_M^1(\Omega \times D) \quad \mbox{as} \quad L \rightarrow \infty, \quad
\Delta t \rightarrow 0_+.
\label{psi0conv}
\end{align}
\end{subequations}
In the case of variable $\zeta(\rho_0(\cdot))$ the same properties hold under the assumptions on $\rho_0$ and $\zeta$ stated
in \eqref{inidata}. For example, the claim in \eqref{psi0conv}
that $\widetilde \psi^0 = \beta^L(\widetilde \psi^0)$ a.e. on $\Omega \times D$
follows from \eqref{psi0} on replacing $\widetilde\psi^0$
by $\widetilde\psi^0-L$ on the left-hand side of \eqref{psi0} and $\beta^L(\widetilde\psi_0)$ by $\beta^L(\widetilde\psi_0)-L$
on the right-hand side, which preserves the equality. We then take $\varphi = [\widetilde\psi^0-L]_+$ and note that $\beta^L(\widetilde\psi_0)-L\leq 0$
to deduce, thanks to the positivity of $\zeta(\rho_0)$ on $\Omega$, that $[\widetilde\psi^0-L]_+=0$ a.e. on $\Omega \times D$, which then implies that $\widetilde\psi^0 \leq L$ a.e. on $\Omega \times D$. An analogous argument shows that
$\widetilde\psi^0 \geq 0$ a.e. on $\Omega \times D$. In particular, $\widetilde \psi^0 = \beta^L(\widetilde \psi^0)$ a.e. on
$\Omega \times D$ and $\widetilde\psi^0 \in Z_2$, as was claimed in the line above \eqref{inidata-1}. In fact,
$\widetilde \psi^0 \in L^\infty(\Omega \times D) \cap H^1_M(\Omega \times D)$.

The proof of \eqref{inidata-2} is based on a similar cut-off argument.
By defining, for $\widetilde\psi^0 =\widetilde\psi^0(L,\Delta t)$, the function $\lambda^0_L$ by
\[ \lambda^0_L(x):= \int_D M(\qt) \widetilde\psi^0(\xt,\qt) \dq,\qquad \xt \in \Omega,\]
applying \eqref{psi0} with $\varphi(\xt,\qt) = \widetilde\varphi(\xt)
\otimes 1(\qt)$ and recalling Fubini's theorem, we have that
\[ \int_\Omega \left[\zeta(\rho_0)\, \lambda^0_L\,\widetilde\varphi + \Delta t\, \nabx \lambda^0_L \cdot \nabx \widetilde\varphi
\right] \dx = \int_\Omega \zeta(\rho_0)\, \widetilde \varphi \left[\int_D M(\qt)\, \beta^L(\widetilde\psi_0) \dq\right]\dx
\qquad \forall \widetilde \varphi \in H^1(\Omega),\]
and therefore, for each $\omega \in \mathbb{R}$ and all $\widetilde\varphi \in H^1(\Omega)$,
\begin{eqnarray}\label{indent-lambda}
&&\int_\Omega \left[\zeta(\rho_0)\, (\lambda^0_L-\omega)\, \widetilde\varphi + \Delta t\, \nabx (\lambda^0_L-\omega) \cdot \nabx \widetilde\varphi
\right] \dx \nonumber\\
&&\hspace{4cm}= \int_\Omega \zeta(\rho_0)\, \widetilde \varphi \left(\left[\int_D M(\qt)\, \beta^L(\widetilde\psi_0) \dq\right]-\omega\right)\dx.
\end{eqnarray}
Now thanks to \eqref{inidata} we have that
\[ 0 \leq \zeta(\rho_0)\int_D M(\qt)\, \beta^L(\widetilde\psi_0) \dq \leq \zeta(\rho_0) \int_D M(\qt)\, \widetilde\psi_0 \dq
= \int_D \psi_0(\xt,\qt) \dq \in L^\infty(\Omega),\]
and therefore by selecting
\begin{equation}\label{alpha}
\omega :=\mbox{ess.sup}_{x \in \Omega}
\int_D M(\qt) \widetilde\psi_0(\xt,\qt) \dq = \mbox{ess.sup}_{x \in \Omega} \left(\frac{1}{\zeta(\rho_0(\xt))}\int_D \psi_0(\xt,\qt) \dq  \right)
\end{equation}
we get that
\[ \zeta(\rho_0) \left(\left[\int_D M(\qt)\, \beta^L(\widetilde\psi_0) \dq\right]-\omega\right) \leq 0. \]
Thus, by choosing $\widetilde\varphi = [\lambda^0_L - \omega]_+$ in \eqref{indent-lambda}, we deduce that
\[ \int_\Omega \left[\zeta(\rho_0)\left([\lambda^0_L - \omega]_+\right)^2 + \Delta t\,
|\nabx  \left([\lambda^0_L - \omega]_+\right)|^2 \right]\dx \leq 0.\]
Hence,  $[\lambda^0_L - \omega]_+ = 0~$ a.e. on $\Omega$. In other words, $0 \leq \lambda^0_L(\xt) \leq \omega$ a.e. on $\Omega$, which then implies \eqref{inidata-2}.
We shall denote the mapping $\widetilde\psi_0 \mapsto \widetilde\psi^0$ by $S_{\Delta t, L}$; i.e.,
$\widetilde\psi^0= S_{\Delta t, L}\widetilde\psi_0$.

Let us define
\begin{align}
\Upsilon:= \{ \eta \in L^\infty(\Omega) : \eta \in [ \rho_{\rm min}, \rho_{\rm max}]
\mbox{ a.e.\ on } \Omega \}.
\label{denspace}
\end{align}

%

It follows 
for all 
$\vt,\,\wt \in \Ht^1(\Omega)$ that
\begin{align}
\vt \otimes \vt : \nabxtt \wt  &=
[(\vt \cdot \nabx) \wt] \cdot \vt
=  - [(\vt \cdot \nabx) \vt] \cdot \wt
+ (\vt \cdot \nabx)(\vt \cdot \wt)  \qquad \mbox{a.e.\ in } \Omega.
\label{ruwiden}
\end{align}
Noting the above, our discrete-in-time  approximation of (P$_{L}$) is then defined as follows.

{\boldmath $({\rm P}_{L}^{\Delta t})$} Let $N \in \mathbb{N}_{\geq 1}$ and define $\Delta t := T/N$; let, further,
$\rho^0_L := \rho_0 \in \Upsilon$, $\utae^0 := \ut^0 \in \Vt$ and $\hpsiaet^0 := \widetilde \psi^0 \in Z_2$.
For $n = 1,\dots, N$,  and given $\ut^{n-1}_L \in \Vt$, find
\begin{equation}\label{rho-reg}
\rho^{[\Delta t]}_L|_{[t_{n-1},t_n]} \in L^\infty(t_{n-1},t_n; L^\infty(\Omega)) \cap C([t_{n-1},t_n]; L^2(\Omega)) \cap W^{1,\infty}(t_{n-1},t_n; W^{1,\frac{q}{q-1}}(\Omega)'),
\end{equation}
where $q \in (2,\infty)$ when $d=2$ and $q \in [3,6]$ when $d=3$, such that
$\rho_L^{[\Delta t]}|_{[t_{n-1},t_n]}(\cdot,t_{n-1})= \rho^{n-1}_L$ 
and
\begin{subequations}
\begin{align}
& \int_{t_{n-1}}^{t_n}\! \left\langle\frac{\partial\rho^{[\Delta t]}_L}{\partial t} , \eta \right\rangle_{W^{1,\frac{q}{q-1}}(\Omega)} \!\!\!\dd t
- \int_{t_{n-1}}^{t_n}\int_{\Omega} \rho^{[\Delta t]}_L \,\utae^{n-1} \cdot \nabx \eta \dx \dd t
= 0
\nonumber\\
& \hspace{2.5in}
\qquad \!\!\!\!\!\!\!\!\forall \eta
\in L^1(t_{n-1},t_n;W^{1,\frac{q}{q-1}}(\Omega)),
\label{rhonL}
\end{align}
where the symbol $\langle \cdot , \cdot \rangle_{W^{1,\frac{q}{q-1}}(\Omega)}$ denotes the duality pairing between
$W^{1,\frac{q}{q-1}}(\Omega)'$ and $W^{1,\frac{q}{q-1}}(\Omega)$, with $q \in (2,\infty)$ when $d=2$ and $q \in [3,6]$ when $d=3$.
We then define
$\rho^{n}_L:=\rho^{[\Delta t]}_L|_{[t_{n-1},t_n]}(\cdot,t_n) \in \Upsilon$.
Given $(\rho^{n-1}_L,\utae^{n-1},\hpsiaet^{n-1}) \in \Upsilon \times \Vt \times Z_2$,
find
\[\left(\utae^n,\hpsiaet^n 
\right) \in \Vt \times (X \cap Z_2)\]
such that
\begin{align}
&\int_{\Omega} \left[
\frac{\rho^n_L\,\utae^{n}-\rho^{n-1}_L\,\utae^{n-1}}{\Delta t}
- \tfrac{1}{2}\,\frac{\rho^{n}_L-\rho^{n-1}_L}{\Delta t}
\,\ut_L^{n}
\right]
\cdot \wt \dx
+\int_{\Omega} \mu(\rho^n_L)\,
\Dtt(\utae^n)
: \Dtt(\wt) \dx
\nonumber
\\
& \qquad\qquad 
+ \tfrac{1}{2} \int_{\Omega}  \left(\frac{1}{\Delta t}\int_{t_{n-1}}^{t_n} \rho^{[\Delta t]}_L\dd t\right)
\, \left[
[(\utae^{n-1} \cdot \nabx) \utae^{n}] \cdot \wt
-[(\utae^{n-1} \cdot \nabx) \wt] \cdot \utae^n
\right]
\dx
\nonumber \\
&\qquad\qquad\qquad\qquad 
= \int_{\Omega} \rho^n_L\,\ft^n \cdot \wt \dx 
- k\,\sum_{i=1}^K \int_{\Omega} \Ctt_i(M\,\zeta(\rho^n_L)\,\hpsiaet^n): \nabxtt
\wt \dx
\qquad \forall \wt \in \Vt,
\label{Gequn}
\end{align}

\begin{align}
&\hspace{-1.6cm}\int_{\Omega \times D} M\,
\frac{\zeta(\rho^n_L)\,\hpsiaet^n
- \zeta(\rho^{n-1}_L)\,\hpsiaet^{n-1}}
{\Delta t}
\,\varphi \dq \dx
\nonumber
\\
\bet
&\hspace{-1cm}
-\int_{\Omega \times D} M \left(\frac{1}{\Delta t}\int_{t_{n-1}}^{t_n} \zeta(\rho^{[\Delta t]}_L)\dd t\right)\, \utae^{n-1} \cdot 
(\nabx\varphi)\,\hpsiaet^n
\dq \dx
\nonumber
\\
\bet
&\hspace{0.1cm} + \int_{\Omega \times D}
\sum_{i=1}^K  \left[\, \frac{1}{4\, \lambda}\,
\sum_{j=1}^K A_{ij}\,M\,\nabqj \hpsiaet^n
-[\,\sigtt(
\utae^n) \,\qt_i\,]\,M\,\zeta(\rho^n_L)\, \beta^L(\hpsiaet^{n})\right]\,\cdot\, \nabqi
\varphi \dq \dx\,
\nonumber \\
\bet
&\hspace{1.1cm} + \int_{\Omega \times D}
\epsilon\,M\,\nabx \hpsiaet^n \cdot\, \nabx \varphi\dq \dx=0
\qquad \forall \varphi \in
X;
\label{psiG}
\end{align}
\end{subequations}
where, for  $t \in [t_{n-1}, t_n)$ and $n=1, \dots,  N$,
\begin{align}
\ft^{\Delta t, +}(\cdot,t) =
\ft^n(\cdot) := \frac{1}{\Delta t}\,\int_{t_{n-1}}^{t_n}
\ft(\cdot,t) \dt \in \Lt^{\varkappa}(\Omega). 
\label{fn}
\end{align}
It follows from (\ref{inidata}) and (\ref{fn}) that
\begin{align}
\ft^{\Delta t, +} \rightarrow \ft \quad \mbox{strongly in } L^{2}(0,T;\Lt^\varkappa(\Omega))
\mbox { as } \Delta t \rightarrow 0_{+},
\label{fncon}
\end{align}
where $\varkappa >1$ if $d=2$ and $\varkappa =\frac{6}{5}$ if $d=3$.
Note that as the test function $\wt$ in \eqref{Gequn} is chosen to be divergence-free,
the term containing the density $\varrho$ of the polymer chains in the definition of
$\tautt$ (cf. \eqref{eqtt1})
is eliminated from \eqref{Gequn}.

For $n \in \{1, \dots, N\}$, and for the functions $\ut^{n-1}_L \in \Vt$ and $\rho^{n-1}_L \in \Upsilon$ fixed, the existence of a unique solution
\begin{equation}\label{fncon-1}
\rho^{[\Delta t]}_L|_{[t_{n-1},t_n]} \in L^\infty(t_{n-1},t_n; L^\infty(\Omega)) \cap C([t_{n-1},t_n]; L^2(\Omega))
\end{equation}
to \eqref{rhonL} satisfying the initial condition $\rho^{[\Delta t]}_L|_{[t_{n-1},t_n]}(\cdot,t_{n-1}) = \rho^{n-1}_L$ follows from Corollaries II.1 and II.2 and the discussion on p.546 in DiPerna \& Lions \cite{DPL} (with our $\ut^n_L \in \Vt$ here extended
from $\overline\Omega$ to $\mathbb{R}^d$ by $\zerot$). We refer to \ref{sec:apx} for a justification that the notion of
solution used in \eqref{rho-reg}, \eqref{rhonL} is equivalent to the notion of distributional solution, used by DiPerna \& Lions in \cite{DPL}. The statement
\[  \rho^{[\Delta t]}_L|_{[t_{n-1},t_n]} \in W^{1,\infty}(t_{n-1},t_n; W^{1,\frac{q}{q-1}}(\Omega)')\]
in \eqref{rhonL}, with $q \in (2,\infty)$ when $d=2$ and $q \in [3,6]$ when $d=3$, follows from the bound
\[ \left|\int_{t_{n-1}}^{t_n} \int_\Omega \ut^{n-1}_L\rho^{[\Delta t]}_L \cdot \nabx \varphi \dx \dd t \right|
\leq \|\ut^{n-1}_L\|_{L^q(\Omega)} \|\rho^{[\Delta t]}_L\|_{L^\infty(t_{n-1},t_n;L^\infty(\Omega))}
\|\nabx \varphi\|_{L^1(t_{n-1},t_n; L^{\frac{q}{q-1}}(\Omega))}
\]
and the fact that $\frac{\partial}{\partial t} \rho^{[\Delta t]}_L + \nabx \cdot (\ut^{n-1}_L\rho^{[\Delta t]}_L) = 0$
in the sense of distributions on $\Omega \times (t_{n-1},t_n)$; hence \eqref{rho-reg}.

A further relevant remark in connection with \eqref{Gequn} is that on noting that $\rho^n_L = \rho^{[\Delta t]}_L(\cdot,t_n)$
and $\rho^{n-1}_L= \rho^{[\Delta t]}_L(\cdot,t_{n-1})$, the second term on its left-hand side can be rewritten as
\begin{eqnarray}\label{Gequn-trans}
&&-\tfrac{1}{2} \int_\Omega \frac{\rho^n_L - \rho^{n-1}_L}{\Delta t}
\,\ut^n_L \cdot \wt \dx \dd t
= -\frac{1}{2 \Delta t}\int_{t_{n-1}}^{t_n}
\left\langle \frac{\partial \rho^{[\Delta t]}_L}{\partial t}  ,\ut^n_L \cdot \wt \right\rangle_{W^{1,\frac{q}{q-1}}(\Omega)} \dd t \nonumber\\
&&\qquad = - \frac{1}{2 \Delta t}\int_{t_{n-1}}^{t_n} \int_\Omega
\rho^{[\Delta t]}_L \ut^{n-1}_L \cdot \nabx (\ut^n_L \cdot \wt) \dx \dd t \nonumber\\
&&\qquad = - \tfrac{1}{2}  \int_\Omega \left(\frac{1}{\Delta t} \int_{t_{n-1}}^{t_n}
\rho^{[\Delta t]}_L\dd t \right) \left[\ut^{n-1}_L \cdot \nabx (\ut^n_L \cdot \wt)\right] \dx \qquad \forall \wt \in \Ht^1(\Omega),
\end{eqnarray}
where $q \in (2,\infty)$ when $d=2$ and $q \in [3,6]$ when $d=3$.
The first equality in \eqref{Gequn-trans} is a consequence of Corollary \ref{Bochner2} and Corollary \ref{Bochner4},
on noting that $\frac{\partial}{\partial t} \rho^{[\Delta t]}_L
\in L^\infty(t_{n-1},t_n; W^{1,\frac{q}{q-1}}(\Omega)') \subset L^1(t_{n-1},t_n; W^{1,\frac{q}{q-1}}(\Omega)')$,
with $q \in (2,\infty)$ when $d=2$ and $q \in [3,6]$ when $d=3$.

The identities \eqref{ruwiden} and \eqref{Gequn-trans} now motivate the form of the expression in the second
 line of  \eqref{Gequn}. We note here that
the requirements that $q>2$ when $d=2$ and $q \geq 3$ when $d=3$ are the consequence of our demand that the
scalar product $\ut^n_L \cdot \wt$ of the functions $\ut^n_L, \wt \in \Ht^1(\Omega)$ belongs to $W^{1, \frac{q}{q-1}}(\Omega)$,
which is required in \eqref{Gequn-trans}.

As $\rho^{n-1}_L \in \Upsilon$,  $\ut^{n-1}_L \in \Vt$ (extended from $\overline\Omega \subset \mathbb{R}^d$ to
 $\mathbb{R}^d$ by $\zerot$), $\rho^{[\Delta t]}_L \in C([t_{n-1},t_n]; L^2(\Omega))$, $\zeta \in C^1([\rho_{\rm min}, \rho_{\rm max}], [\zeta_{\rm min}, \zeta_{\rm max}])$, it follows from Corollary II.2 in the paper of DiPerna \& Lions \cite{DPL} that $\zeta(\rho^{[\Delta t]}_L)$ is a renormalized solution in the sense that
\begin{eqnarray}\label{Gequn-trans00}
&&\int_{t_{n-1}}^{t_n}
\left\langle \frac{\partial \zeta(\rho^{[\Delta t]}_L)}{\partial t}  , \varphi \right\rangle_{W^{1,\frac{q}{q-1}}(\Omega)} \dd t
\nonumber\\
&&\hspace{3cm}= \int_\Omega \left(\int_{t_{n-1}}^{t_n}
\zeta(\rho^{[\Delta t]}_L)\dd t \right) \left[\ut^{n-1}_L \cdot \nabx \varphi\right] \dx
\qquad \forall \varphi \in W^{1,\frac{q}{q-1}}(\Omega),
\end{eqnarray}
where $q \in (2,\infty)$ when $d=2$ and $q \in [3,6]$ when $d=3$.
Hence, on observing that $\zeta(\rho^n_L) = \zeta(\rho^{[\Delta t]}_L(\cdot,t_n))$ and $\zeta(\rho^{n-1}_L) =
\zeta(\rho^{[\Delta t]}_L(\cdot,t_{n-1}))$, we have that
\begin{eqnarray}\label{Gequn-trans0}
&&\int_\Omega \frac{\zeta(\rho^n_L) - \zeta(\rho^{n-1}_L)}{\Delta t}
\,\varphi \dx \dd t =
\frac{1}{\Delta t}\int_{t_{n-1}}^{t_n}
\left\langle \frac{\partial \zeta(\rho^{[\Delta t]}_L)}{\partial t}  , \varphi \right\rangle_{W^{1,\frac{q}{q-1}}(\Omega)} \dd t\nonumber\\
&&\qquad= \frac{1}{\Delta t}\int_{t_{n-1}}^{t_n} \int_\Omega
\zeta(\rho^{[\Delta t]}_L)\, \ut^{n-1}_L \cdot \nabx \varphi \dx \dd t \nonumber\\
&&\qquad=\int_\Omega \left(\frac{1}{\Delta t} \int_{t_{n-1}}^{t_n}
\zeta(\rho^{[\Delta t]}_L)\dd t \right) \left[\ut^{n-1}_L \cdot \nabx \varphi\right] \dx
\qquad \forall \varphi \in W^{1,\frac{q}{q-1}}(\Omega),
\end{eqnarray}
where $q \in (2,\infty)$ when $d=2$ and $q \in [3,6]$ when $d=3$.
%
The first equality in \eqref{Gequn-trans0} is a consequence of Corollary \ref{Bochner2} and Corollary \ref{Bochner4},
on noting that $\frac{\partial}{\partial t} \zeta(\rho^{[\Delta t]}_L)
\in L^1(t_{n-1},t_n; W^{1,\frac{q}{q-1}}(\Omega)')$,
with $q \in (2,\infty)$ when $d=2$ and $q \in [3,6]$ when $d=3$.
Since $\widetilde\psi^n_L \in X$, it follows from \eqref{Gequn-trans0} that
\begin{eqnarray}
&&-\tfrac{1}{2} \int_\Omega \frac{\zeta(\rho^n_L) - \zeta(\rho^{n-1}_L)}{\Delta t}
\,\hpsiaet^n\,\varphi \dx \dd t
= - \tfrac{1}{2}  \int_\Omega \left(\frac{1}{\Delta t} \int_{t_{n-1}}^{t_n}
\zeta(\rho^{[\Delta t]}_L)\dd t \right) \left[\ut^{n-1}_L \cdot \nabx (\hpsiaet^n\,\varphi)\right] \dx\nonumber\\
&&\hspace{10.3cm} \forall \varphi \in X,\quad \mbox{a.e. $\qt \in D$},\nonumber
\end{eqnarray}
and therefore we can rewrite \eqref{psiG}
in the following equivalent form:
\begin{align}
&\hspace{-1.6cm}\int_{\Omega \times D} M\,\left[
\frac{\zeta(\rho^n_L)\,\hpsiaet^n
- \zeta(\rho^{n-1}_L)\,\hpsiaet^{n-1}}
{\Delta t}
- \tfrac{1}{2}\,\frac{\zeta(\rho^n_L) - \zeta(\rho^{n-1}_L)}{\Delta t}\, \hpsiaet^n \right]
\,\varphi \dq \dx
\nonumber
\\
&\hspace{-1cm}
+ \tfrac{1}{2}
\int_{\Omega \times D} M
\left(\frac{1}{\Delta t}\int_{t_{n-1}}^{t_n} \zeta(\rho^{[\Delta t]}_L)\dd t\right)\, \utae^{n-1}
\cdot \left[(\nabx \hpsiaet^n)\,\varphi -
(\nabx\varphi)\,\hpsiaet^n
 \right]
\dq \dx
\nonumber
\\
&\hspace{-0.4cm} + \int_{\Omega \times D}
\sum_{i=1}^K  \left[\, \frac{1}{4\, \lambda}\,
\sum_{j=1}^K A_{ij}\,M\,\nabqj \hpsiaet^n
-[\,\sigtt(
\utae^n) \,\qt_i\,]\,M\,\zeta(\rho^n_L)\, \beta^L(\hpsiaet^{n})\right]\,\cdot\, \nabqi
\varphi \dq \dx\,
\nonumber \\
\bet
&\hspace{0.2cm} + \int_{\Omega \times D}
\epsilon\,M\,\nabx \hpsiaet^n \cdot\, \nabx \varphi\dq \dx=0
\qquad \forall \varphi \in
X.
\label{psiG-1a}
\end{align}
%

The following elementary result will play a crucial role in our proofs.

\begin{lemma}\label{basic}
Suppose that $S\subset \mathbb{R}$ is an open interval and let $F \in W^{2,1}_{\rm loc}(S)$. Let further $G$ denote the primitive function of $s\in S \mapsto s F''(s) \in \mathbb{R}$; i.e., $G'(s) = sF''(s)$, $s \in S$.
Then, the following statements hold.
\begin{itemize}
\item[a)] $s\in S \mapsto sF'(s) - F(s) - G(s) \in \mathbb{R}$ is a constant function on $\overline{S}$;
i.e., there exists $c_0 \in \mathbb{R}$ such that $sF'(s) - F(s) - G(s) = c_0$ for all $s \in \overline{S}$.
\item[b)] The following identity holds for any $a,b \in \overline{S}$ and any $A, B \in \mathbb{R}$:
\[ (Aa - Bb) F'(a) - (A-B) G(a) = A(F(a) + c_0) - B(F(b) + c_0) + B(b-a)^2 \int_0^1\! F''(\theta a + (1-\theta) b)\,\theta \dd \theta.\]
\item[c)] If in addition $B \geq 0$ and there exists a $d_0 \in \mathbb{R}$ such that ${\rm ess.inf}_{\theta \in [0,1]}
F''(\theta a + (1-\theta)b) \geq d_0$, then 
\[ (Aa - Bb) F'(a) - (A-B) G(a) \geq  A(F(a) + c_0) - B(F(b) + c_0) + \frac{1}{2}d_0 B (b-a)^2.\]
\end{itemize}
\end{lemma}

\begin{proof}
Before we embark on the proof of the lemma, we note that $G(s)$, the primitive function of $s \in S \mapsto s F''(s)
\in \mathbb{R}$, is only defined up to an additive constant. Altering the value of $G(s)$ by an additive constant $\gamma$,
say, does not affect the validity of the above statements: on replacing the constant $c_0$ by another constant, $c_0-\gamma$, the statements a), b), c) continue to hold.
\begin{itemize}
\item[a)] By hypothesis the function $s\in S \mapsto sF'(s) - F(s) - G(s) \in \mathbb{R}$ belongs to $ W^{1,1}_{\rm loc}(S)$, and it is therefore absolutely continuous on any compact subinterval of $\overline{S}$; hence it is differentiable almost everywhere on $S$. Upon differentiation and using that
    $G'(s) = sF''(s)$, we deduce that the first derivative of $s\in S \mapsto sF'(s) - F(s) - G(s) \in \mathbb{R}$
    is equal to $0$ almost everywhere on $S$. This implies the existence of a constant $c_0$ such that
    $sF'(s) - F(s) - G(s) = c_0$ for a.e. $s \in S$. Since the function $s\in S \mapsto sF'(s) - F(s) - G(s) \in \mathbb{R}$ is absolutely continuous on each compact subinterval of $\overline{S}$, it follows that $sF'(s) - F(s) - G(s) = c_0$ for all $s \in \overline{S}$.
\item[b)] Suppose that $a, b \in S$ and $A, B \in \mathbb{R}$. By Taylor series expansion with integral remainder,
%
\begin{align*}
~~~~~~~~~~~~~~~~ AF(a) - BF(b) = AF(a) - B\left[F(a) + (b-a)F'(a) + (b-a)^2 \int_0^1 F''(\theta a + (1-\theta) b)\, \theta \dd \theta\right].
\end{align*}
%
Hence,
\begin{eqnarray*}
&&AF(a) - BF(b) + B (b-a)^2 \int_0^1 F''(\theta a + (1-\theta) b) \theta \dd \theta\\
&&\qquad\quad  = (A-B)(F(a) - aF'(a)) + (Aa - Bb) F'(a)\\
&&\qquad\quad  = - (A-B) (aF'(a) - F(a) - G(a))  - (A-B)G(a) + (Aa - Bb) F'(a)\\
&&\qquad\quad  = - (A-B) c_0   - (A-B)G(a) + (Aa - Bb) F'(a).
\end{eqnarray*}
By transferring the first term on the right-hand side of the last equality to the left-hand side, the stated identity follows.
\item[c)] The inequality is a direct consequence of the identity stated in part b) of the lemma.
\end{itemize}
\end{proof}
\begin{remark}
Lemma \ref{basic} should be compared with the discussion between eqs. (2.27) and (2.28) in reference
\cite{surf2006}, which can be seen as a special case of Lemma \ref{basic}.
\end{remark}

In order to prove the existence of a solution to (P$_{L}^{\Delta t}$), we require the following convex regularization
$\mathcal{F}_{\delta}^L \in C^{2,1}({\mathbb R})$ of $\mathcal{F}$ defined, for any $\delta \in (0,1)$ and
$L>1$, by
\begin{align}
 &\mathcal{F}_{\delta}^L(s) := \left\{
 \begin{array}{ll}
 \textstyle\frac{s^2 - \delta^2}{2\,\delta}
 + s\,(\log \delta - 1) + 1
 \quad & \mbox{for $s \le \delta$}, \\
\mathcal{F}(s)\ \equiv
s\,(\log s - 1) + 1 & \mbox{for $\delta \le s \le L$}, \\
  \textstyle\frac{s^2 - L^2}{2\,L}
 + s\,(\log L - 1) + 1
 & \mbox{for $L \le s$}.
 \end{array} \right. \label{GLd}
\end{align}
Hence,
\begin{subequations}
\begin{align}
\quad &[\mathcal{F}_{\delta}^{L}]'(s) = \left\{
 \begin{array}{ll}
 \textstyle \frac{s}{\delta} + \log \delta - 1
 \quad & \mbox{for $s \le \delta$}, \\
 \log s & \mbox{for $\delta \le s \le L$}, \\
 \textstyle \frac{s}{L} + \log L - 1
& \mbox{for $L \le s$},
 \end{array} \right. \label{GLdp}\\
\quad
 &[\mathcal{F}_{\delta}^{L}]''(s) = \left\{
 \begin{array}{ll}
 {\delta}^{-1} \quad & \mbox{for $s \le \delta$}, \\
 s^{-1} & \mbox{for $\delta \le s \le L$}, \\
 L^{-1} & \mbox{for $L \le s$}. \,
\end{array} \right. \label{Gdlpp}
\end{align}
\end{subequations}
We note that
\begin{align}
\mathcal{F}^L_\delta(s) \geq \left\{
\begin{array}{ll}
\frac{s^2}{2\,\delta} &\quad \mbox{for $s \leq 0$},
\\
\frac{s^2}{4\,L} - C(L)&\quad \mbox{for $s \geq 0$};
\end{array}
\right.
\label{cFbelow}
\end{align}
and that
$[\mathcal{F}_{\delta}^{L}]''(s)$
is bounded below by $1/L$ for all $s \in \mathbb{R}$.
Finally, we set
\begin{align}
\beta^L_\delta(s) := ([\mathcal{F}_{\delta}^{L}]'')^{-1}(s)
= \max \{\beta^L(s),\delta\},
\label{betaLd}
\end{align}
and observe that $\beta^L_\delta(s)$
is bounded above by $L$ and bounded below by $\delta$ for all $s \in \mathbb{R}$.
Note also that both $\beta^L$ and $\beta^L_\delta$ are Lipschitz continuous on
$\mathbb{R}$, with Lipschitz constants equal to $1$.

\subsection{\boldmath Existence of a solution to  $({\rm P}_{L}^{\Delta t})$}
\label{sec:existence-cut-off.1}
On recalling the discussion following \eqref{fncon-1},
for $n\in \{1, \dots, N\}$ we define the function $\rho^n_L: = \rho^{[\Delta t]}_L|_{[t_{n-1},t_n]}(\cdot,t_{n}) \in L^\infty(\Omega)$; it will be shown below that $\rho^n_L \in \Upsilon$, in fact.
With $\rho^{[\Delta t]}_L|_{[t_{n-1},t_n]}$ thus fixed (and with its values $\rho^{n-1}_L$ and $\rho^n_L$ at
$t=t_{n-1}$ and $t=t_n$, respectively, also fixed,) we rewrite (\ref{Gequn}) as
\begin{equation}
b
(\utae^n,\wt) =
\ell_b(
\hpsiaet^n)(\wt)
\qquad \forall \wt \in \Vt;
\label{bLM}
\end{equation}
where, for all 
$\wt_i \in \Ht^{1}_{0}(\Omega)$, $i=1,2$,
\begin{subequations}
\begin{align}
b
(\wt_1,\wt_2) &:=
\int_{\Omega} \left[ \tfrac{1}{2}\,(\rho^n_L + \rho^{n-1}_L)\,\wt_1\cdot \wt_2
+ \Delta t\, \mu(\rho^n_L) \,\Dtt(\wt_1)
:\Dtt(\wt_2) \right] \dx
\nonumber \\
& \hspace{0cm} + \tfrac{1}{2}\,\Delta t
\int_{\Omega}\left(\frac{1}{\Delta t}\int_{t_{n-1}}^{t_n} \rho^{[\Delta t]}_L \dd t \right)\, \left[
[(\utae^{n-1} \cdot \nabx) \wt_1] \cdot \wt_2
-[(\utae^{n-1} \cdot \nabx) \wt_2] \cdot \wt_1
\right]
\dx
\label{bgen}
\end{align}
and, for all 
$\varphi \in L^2_M(\Omega \times D)$
and $\wt \in \Ht^1_0(\Omega)$,
\begin{align}
\ell_b(
\varphi)(\wt) &:=
\int_\Omega \left[
\rho^{n-1}_L\,
\utae^{n-1} \cdot \wt
+ \Delta t \, \rho^n_L\, \ft^n \cdot \wt
-\Delta t \,k\,\sum_{i=1}^K
\Ctt_i(M\,\zeta(\rho^n_L)\,\varphi) : \nabxtt
\wt
\right] \dx.
\label{lbgen}
\end{align}
\end{subequations}
It follows from Korn's inequality
\begin{align}
\int_{\Omega} |\Dtt(\wt)|^2 \dx \geq c_0 \,\|\wt\|_{H^1(\Omega)}^2
\qquad \forall \wt \in
\Ht_0^{1}(\Omega),
\label{Korn}
\end{align}
where $c_0>0$, that, for $\ut^{n-1}_L \in \Vt$ and $\rho^{n-1}_L, \rho^n_L \in \Upsilon$ fixed,
$b(\cdot,\cdot)$ is a continuous nonsymmetric coercive bilinear
functional on $\Ht^1_0(\Omega) \times \Ht^1_0(\Omega)$.
In addition, for $\ut^{n-1}_L \in \Vt$ and $\rho^{n-1}_L, \rho^n_L \in \Upsilon$ fixed,
thanks to (\ref{fvkbd}) and \eqref{eqCttbd}, $\ell_b(\varphi)(\cdot)$ is a continuous linear
functional on $\Ht^1_0(\Omega)$ for any $\varphi \in L^2_M(\Omega\times D)$.

%

It is also convenient to rewrite (\ref{psiG}) (or, equivalently, \eqref{psiG-1a}) as
\begin{align}
a(\hpsiaet^n,\varphi) = \lae(\utae^n,\beta^L(\hpsiaet^n))
(\varphi) \qquad \forall \varphi \in X,
\label{genLM}
\end{align}
where, for all $\varphi_i \in X$, $i=1,\,2$,
\begin{subequations}
\begin{align}
a(\varphi_1,\varphi_2)
&:=\int_{\Omega \times D} M\, \bigg[\zeta(\rho^n_L)\,
\varphi_1\,\varphi_2  + \Delta t\, \epsilon\,\nabx
\varphi_1\,\cdot\, \nabx \varphi_2 \nonumber \\
& \hspace{0.3in} - \Delta t
\left(\frac{1}{\Delta t}\int_{t_{n-1}}^{t_n}\zeta(\rho^{[\Delta t]}_L)\dd t\right)
\utae^{n-1}\,\varphi_1\,\cdot\, \nabx \varphi_2
\label{agen}
\nonumber\\
& \hspace{0.6in}
+\, \frac{\Delta t}{4\,\lambda} \,
\sum_{i=1}^K \sum_{j=1}^K A_{ij}\,
\nabqj \varphi_1 \, \cdot\, \nabqi
\varphi_2 \biggr] \dq \dx\nonumber\\
&= \int_{\Omega \times D} M\,\biggl[\tfrac{1}{2}(\zeta(\rho^n_L) +
\zeta(\rho^{n-1}_L))\,
\varphi_1\,\varphi_2  + \Delta t\, \epsilon\,\nabx
\varphi_1\,\cdot\, \nabx \varphi_2  \nonumber\\
& \hspace{0.3in}+ \tfrac{1}{2}\Delta t \left(\frac{1}{\Delta t}\int_{t_{n-1}}^{t_n}\zeta(\rho^{[\Delta t]}_L)\dd t\right)\left(\utae^{n-1}\,\varphi_2\,\cdot\, \nabx \varphi_1-
\utae^{n-1}\,\varphi_1\,\cdot\, \nabx \varphi_2 \right)
\nonumber\\
& \hspace{0.6in}
+\, \frac{\Delta t}{4\,\lambda} \,
\sum_{i=1}^K \sum_{j=1}^K A_{ij}\,
\nabqj \varphi_1 \, \cdot\, \nabqi
\varphi_2 \biggr] \dq \dx,
\end{align}
and, for all $\vt \in \Ht^1(\Omega)$, $\eta \in L^\infty(\Omega\times D)$
and $\varphi \in X$,
\begin{align}
\lae(\vt,\eta)(\varphi) &:=
\int_{\Omega \times D}
M \left[\zeta(\rho^{n-1}_L)\,\hpsiaet^{n-1}
\,\varphi
+ \Delta t\,\sum_{i=1}^K [\,\sigtt(\vt)
\,\qt_i\,]\,\zeta(\rho^n_L)\,\eta\, \cdot\, \nabqi
\varphi \right]\!\dq \dx.
\label{lgen}
\end{align}
\end{subequations}
%
Hence, on noting (\ref{A}), 
for $\ut^{n-1}_L \in \Vt$ and $\rho^{[\Delta t]}_L|_{[t_{n-1},t_n]}$ fixed (and therefore $\rho^n_L$ and
$\rho^{n-1}_L$ also fixed), $a(\cdot,\cdot)$ is a coercive bilinear functional on $ X \times X$.
In order to show that $a(\cdot,\cdot)$ is a continuous bilinear functional on $X \times X$, we shall
focus our attention on the case of $d=3$; in the case of $d=2$ the argument is completely analogous; we
begin by noting that, by H\"older's inequality and the Sobolev embedding theorem,
\begin{eqnarray*}
\int_{\Omega\times D} M |\ut^{n-1}_L|\, |\varphi_1|\, |\nabx \varphi_2| \dq \dx
&\leq& \int_D M \|\ut^{n-1}_L\|_{L^6(\Omega)} \|\varphi_1\|_{L^3(\Omega)} \|\nabx \varphi_2\|_{L^2(\Omega)} \dq
\\
&\leq& c(\Omega)\, \|\ut^{n-1}_L\|_{L^6(\Omega)} \int_D M  \|\varphi_1\|_{H^1(\Omega)} \,\|\varphi_2\|_{H^1(\Omega)} \dq
\\
& \leq & c(\Omega)\, \|\ut^{n-1}_L\|_{L^6(\Omega)}\, \|\varphi_1\|_{L^2_M(D; H^1(\Omega))}\, \|\varphi_2\|_{L^2_M(D; H^1(\Omega))}\\
& \leq & c(\Omega)\, \|\ut^{n-1}_L\|_{L^6(\Omega)} \,\|\varphi_1\|_{H^1_M(\Omega \times D)} \,\|\varphi_2\|_{H^1_M(\Omega \times D)}.
\end{eqnarray*}
This, obvious applications of the Cauchy--Schwarz inequality, and the fact that the range of the function
$\zeta$ is the compact subinterval $[\zeta_{\rm min}, \zeta_{\rm max}]$ of $(0,\infty)$, then imply that $a(\cdot,\cdot)$ is a continuous bilinear functional on $X \times X$.

In addition, for all $\vt \in \Ht^1(\Omega)$, $\eta \in L^\infty(\Omega \times D)$ and $\varphi \in X$, we have that
\begin{align}
|\lae(\vt,\eta)(\varphi)| &\leq
\|\zeta(\rho^{n-1}_L)\,\hpsiaet^{n-1}\|_{L^2_M(\Omega \times D)}
\,\|\varphi\|_{L^2_M(\Omega \times D)}
\nonumber \\
& \hspace{-0.5in}
+ \Delta t\,\left( \int_{D} M\,|\qt|^2 \dq \right)^{\frac{1}{2}}
\|\zeta(\rho^n_L)\,\eta\|_{L^\infty(\Omega \times D)}
\,\|\nabxtt \vt\|_{L^2(\Omega)}
\,\|\nabq \varphi\|_{L^2_M(\Omega \times D)}.
\label{lgenbd}
\end{align}
Therefore, by noting that $\zeta(\rho^{n-1}_L)\,\hpsiaet^{n-1} \in Z_2$,
$\zeta(\rho^n_L)\,\eta \in L^\infty(\Omega \times D)$
and recalling (\ref{MN}), it follows that $\ell_a(\vt,\eta)(\cdot)$
is a continuous linear functional on $X$
for any $\vt \in \Ht^1(\Omega)$ and $\eta \in L^\infty(\Omega \times D)$.

Before we prove existence of a solution to the problem (\ref{Gequn},c),
i.e., (\ref{bLM}) and (\ref{genLM}), let us first show by induction that the function $\rho^{[\Delta t]}_L|_{[t_{n-1},t_n]}$,
whose existence and uniqueness in the function space
\[ L^\infty(t_{n-1},t_n; L^\infty(\Omega)) \cap C([t_{n-1},t_n]; L^2(\Omega)) \cap W^{1,\infty}(t_{n-1},t_n; W^{1,\frac{q}{q-1}}(\Omega)')\]
has already been established, with $q \in (2,\infty)$ when $d=2$ and $q \in [3,6]$ when $d=3$, satisfies the two-sided bound $\rho_{\rm min} \leq \rho^{[\Delta t]}_L(\xt,t) \leq \rho_{\rm max}$ for a.e. $\xt \in \Omega$ and every $t \in [t_{n-1},t_n]$,
i.e., that $\rho^n_L \in \Upsilon$ for all $n \in \{0,\dots,N\}$. To this end,
for $\alpha \in (0,1)$, we consider the regularized problem
\begin{subequations}
\begin{align}
&\int_{t_{n-1}}^{t_n} \left\langle\frac{\partial\rho^{[\Delta t]}_{L,\alpha}}{\partial t} , \eta \right\rangle_{H^1(\Omega)}\!\! \dd t
- \int_{t_{n-1}}^{t_n}\int_{\Omega} \rho^{[\Delta t]}_{L,\alpha} \,\utae^{n-1} \cdot \nabx \eta \dx \dd t
\nonumber\\
&\hspace{1in} + \alpha \int_{t_{n-1}}^{t_n}\int_{\Omega} \nabx \rho^{[\Delta t]}_{L,\alpha} \cdot \nabx \eta \dx \dd t = 0
\qquad \forall \eta
\in L^2(t_{n-1},t_n;H^1(\Omega)),
\label{crhonLd}
\end{align}
subject to the initial condition
\begin{align}\label{crhonLd-ini}
\rho^{[\Delta t]}_{L,\alpha}|_{[t_{n-1},t_n]}(\cdot,t_{n-1}) = \rho^{n-1}_L \in \Upsilon,
\end{align}
\end{subequations}
where $\rho^{n-1}_L \in \Upsilon$ was assumed for the purposes of our inductive argument; clearly,
$\rho^{0}_L := \rho_0 \in \Upsilon$, so the basis of the induction is satisfied. We begin by
showing the existence and uniqueness of a solution $\rho^{[\Delta t]}_{L,\alpha}$ to (\ref{crhonLd},b) and
that $\rho_{\rm min} \leq \rho^{[\Delta t]}_{L,\alpha}(\xt,t) \leq \rho_{\rm max}$ for a.e. $x \in \Omega$ and
for all $t \in [t_{n-1},t_n]$; we shall then pass to the limit $\alpha \rightarrow 0_+$ to deduce that
the limiting function, which we shall show to coincide with $\rho^{[\Delta t]}_{L}$, satisfies the two-sided bound
$\rho_{\rm min} \leq \rho^{[\Delta t]}_{L}(\xt,t) \leq \rho_{\rm max}$ for a.e. $x \in \Omega$ and
for all $t \in [t_{n-1},t_n]$; hence in particular we shall deduce that $\rho^n_L = \rho^{[\Delta t]}_{L}(\cdot,t_n) \in \Upsilon$.

The existence of a unique weak solution
\[ \rho^{[\Delta t]}_{L,\alpha} \in C([t_{n-1},{t_n}]; L^2(\Omega))\cap L^2(t_{n-1},t_n;H^1(\Omega))\cap H^1(t_{n-1},{t_n}; H^1(\Omega)')\]
to (\ref{crhonLd},b)
is immediate; see, for example, Wloka \cite{Wloka}, Thm. 26.1. Further, on selecting, for $s \in (t_{n-1},t_n]$, the test function
$\eta = \chi_{[t_{n-1},s]}\, \rho^{[\Delta t]}_{L,\alpha}$ in equation \eqref{crhonLd}, where for a set $S \subset \mathbb{R}$, $\chi_S$ denotes the characteristic function of  $S$, and
noting that $\ut^{n-1}_L \in \Vt$, we obtain the energy identity
\[\|\rho^{[\Delta t]}_{L,\alpha}(s)\|^2_{L^2(\Omega)} + 2 \alpha \int_{t_{n-1}}^s \int_\Omega
|\nabx \rho^{[\Delta t]}_{L,\alpha}(s)|^2 \dx \dd t = \|\rho^{n-1}_L\|^2_{L^2(\Omega)},
\qquad s \in (t_{n-1},t_n], \]
which then implies that $\{\rho^{[\Delta t]}_{L,\alpha}\}_{\alpha \in (0,1)}$ is a bounded set in
the function space $L^\infty(t_{n-1},t_n;L^2(\Omega))$ and that $\{\sqrt{\alpha}\,\nabx \rho^{[\Delta t]}_{L,\alpha}\}_{\alpha \in (0,1)}$ is a bounded set in $L^2(t_{n-1},t_n;\Lt^2(\Omega))$. It then follows that
\[ \left\{\frac{\partial \rho^{[\Delta t]}_{L,\alpha}}{\partial t}\right\}_{\alpha \in (0,1)}
\qquad \mbox{is a bounded set in $L^2(t_{n-1},t_n;H^1(\Omega)')$}.\]

Hence there exists an element
$\rho^{[\Delta t]}_{L,0} \in L^\infty(t_{n-1},t_n;L^2(\Omega))\cap H^1(t_{n-1},t_n;H^1(\Omega)')$ and
a subsequence of $\{\rho^{[\Delta t]}_{L,\alpha}\}_{\alpha \in (0,1)}$ (not indicated) such that,
as $\alpha \rightarrow 0_+$,
\begin{subequations}
\begin{alignat}{3}
\rho^{[\Delta t]}_{L,\alpha}
&\rightarrow \rho^{[\Delta t]}_{L,0}
&&\qquad\mbox{weak$^\star$ in } L^\infty(t_{n-1},t_n;L^2(\Omega)),\label{alpha1}
\\
\alpha \nabx \rho^{[\Delta t]}_{L,\alpha}
&\rightarrow 0
&&\qquad\mbox{strongly in } L^2(t_{n-1},t_n;L^2(\Omega)),\label{alpha2}
\\
\frac{\partial}{\partial t}  \rho^{[\Delta t]}_{L,\alpha}
& \rightarrow \frac{\partial}{\partial t} \rho^{[\Delta t]}_{L,0}
&&\qquad\mbox{weakly in } L^2(t_{n-1},t_n;H^1(\Omega)').\label{alpha3}
\end{alignat}
\end{subequations}
By a weak parabolic maximum principle based on a cut-off argument, we also have that
\[ \rho_{\rm min} \leq \rho^{[\Delta t]}_{L,\alpha} \leq \rho_{\rm max},\qquad \forall L>1,\;\forall \alpha \in (0,1),\]
and therefore,
\[  \rho_{\rm min} \leq \rho^{[\Delta t]}_{L,0} \leq \rho_{\rm max},\qquad \forall L>1,\;\forall \alpha \in (0,1).\]

Thus, on passing to the limit in \eqref{crhonLd} we deduce that
\begin{align}\label{inter-xxx}
& \int_{t_{n-1}}^{t_n} \left\langle\frac{\partial\rho^{[\Delta t]}_{L,0}}{\partial t} , \eta \right\rangle_{H^1(\Omega)}\!\! \dd t
- \int_{t_{n-1}}^{t_n}\int_{\Omega} \rho^{[\Delta t]}_{L,0} \,\utae^{n-1} \cdot \nabx \eta \dx \dd t
= 0\nonumber\\
&\hspace{3in}
\forall \eta \in L^2(t_{n-1},t_n;H^1(\Omega)).
\end{align}
As
\[ \eta \in L^1(t_{n-1},t_n;W^{1,\frac{q}{q-1}}(\Omega)) \mapsto \int_{t_{n-1}}^{t_n}\int_{\Omega} \rho^{[\Delta t]}_{L,0} \,\utae^{n-1} \cdot \nabx \eta \dx \dd t \in \mathbb{R}\]
is a continuous linear functional for all $q \in (2,\infty)$ when $d=2$ and all $q \in [3,6]$ when $d=3$,
the application of a density argument to \eqref{inter-xxx} yields that
\begin{align*}
& \int_{t_{n-1}}^{t_n} \left\langle\frac{\partial\rho^{[\Delta t]}_{L,0}}{\partial t} , \eta \right\rangle_{W^{1,\frac{q}{q-1}}(\Omega)}\!\! \dd t
- \int_{t_{n-1}}^{t_n}\int_{\Omega} \rho^{[\Delta t]}_{L,0} \,\utae^{n-1} \cdot \nabx \eta \dx \dd t
= 0\nonumber\\
&\hspace{3in}
\forall \eta \in L^1(t_{n-1},t_n;W^{1,\frac{q}{q-1}}(\Omega)),
\end{align*}
where $q \in (2,\infty)$ when $d=2$ and $q \in [3,6]$ when $d=3$, and $\rho^{[\Delta t]}_{L,0}(\cdot,t_{n-1}) = \rho^{n-1}_L$.
As $\rho^{[\Delta t]}_L$ is already known
to be the unique weak solution to this problem by the argument from the beginning of this section, it follows that $\rho^{[\Delta t]}_{L,0} = \rho^{[\Delta t]}_L$, and therefore
\begin{equation}\label{rho-lower-upper}
\rho_{\rm min} \leq \rho^{[\Delta t]}_{L}|_{[t_{n-1},t_n]}(\xt,t) \leq \rho_{\rm max }\; \mbox{ for a.e. $\xt \in \Omega$, for all $t \in [t_{n-1},t_n]$ and $n=1,\dots,N$}.
\end{equation}
In particular, for $t=t_n$, $\rho_{\rm min} \leq \rho^n_L(\xt) :=\rho^{[\Delta t]}_{L}|_{[t_{n-1},t_n]}(\xt,t_n) \leq \rho_{\rm max}$ for a.e. $\xt \in \Omega$; hence, $\rho^n_L \in \Upsilon$, as was claimed in the first sentence of this section.

Thus, for any given $\rho^{n-1}_L \in \Upsilon$ and $\ut^{n-1}_L \in \Vt$, $n \in \{1,\dots, N\}$, we have shown the existence of a unique function
\[ \rho^{[\Delta t]}_{L}|_{[t_{n-1},t_n]} \in L^\infty(t_{n-1},t_n; L^\infty(\Omega)) \cap C([t_{n-1},t_n]; L^2(\Omega)) \cap W^{1,\infty}(t_{n-1},t_n; W^{1,\frac{q}{q-1}}(\Omega)')\]
such that $\rho^{[\Delta t]}_{L}|_{[t_{n-1},t_n]}(\cdot,t_{n-1}) = \rho^{n-1}_L$, with $\rho^0_L := \rho_0 \in \Upsilon$ when $n=1$, and
\begin{subequations}
\begin{align}
& \int_{t_{n-1}}^{t_n} \left\langle\frac{\partial\rho^{[\Delta t]}_L}{\partial t} , \eta \right\rangle_{W^{1,\frac{q}{q-1}}(\Omega)}\!\! \dd t
- \int_{t_{n-1}}^{t_n}\int_{\Omega} \rho^{[\Delta t]}_L \,\ut^{n-1}_{L} \cdot \nabx \eta \dx \dd t
= 0
\nonumber\\
&\hspace{3.5in}\forall \eta
\in L^1(t_{n-1},t_n;W^{1,\frac{q}{q-1}}(\Omega)),
\label{rhonL-2}
\end{align}
where $q \in (2,\infty)$ when $d=2$ and $q \in [3,6]$ when $d=3$, and $\rho^{[\Delta t]}_L|_{[t_{n-1},t_n]}
\in \Upsilon$.

 We now fix $\rho^n_L(\cdot):=\rho^{[\Delta t]}_L|_{[t_{n-1},t_n]}(\cdot,t_n) \in \Upsilon$, and we turn our attention to the proof of existence of solutions to (\ref{Gequn},c). To this end we consider the following regularized version of the
 system (\ref{Gequn},c): for a given $\delta \in (0,1)$, find
$(\utaed^{n},\widetilde{\psi}^n_{L,\delta}) \in \Vt \times X$ such that
\begin{alignat}{2}
b(\utaed^n,\wt) &= \ell_b(\hpsiaedt^n)(\wt)
\qquad &&\forall \wt \in \Vt,
\label{bLMd} \\
a(\widetilde{\psi}^n_{L,\delta},\varphi) &= \lae(\utaed^n,\beta^L_\delta(\widetilde{\psi}^n_{L,\delta}))
(\varphi) \qquad &&\forall \varphi \in X.
\label{genLMd}
\end{alignat}
\end{subequations}

We emphasize at this point that (\ref{rhonL-2}) decouples from (\ref{bLMd},c); indeed, given $\rho^{n-1}_{L}
\in \Upsilon$ and $\ut^{n-1}_{L} \in \Vt$,
one can solve (\ref{rhonL-2}) uniquely for
\[
 \rho^{[\Delta t]}_{L}|_{[t_{n-1},t_n]} \in L^\infty(t_{n-1},t_n; L^\infty(\Omega)) \cap C([t_{n-1},t_n]; L^2(\Omega)) \cap W^{1,\infty}(t_{n-1},t_n; W^{1,\frac{q}{q-1}}(\Omega)'),\]
where $q \in (2,\infty)$ when $d=2$ and $q \in [3,6]$ when $d=3$, and $\rho^{[\Delta t]}_L(\cdot,t_{n-1})
= \rho^{n-1}_L(\cdot)$; by defining $\rho^n_L(\cdot):=\rho^{[\Delta t]}_L|_{[t_{n-1},t_n]}(\cdot,t_n)$,
we can then consider the system (\ref{bLMd},c) for $\{\ut^n_{L,\delta},\widetilde{\psi}^n_{L,\delta}\}$
independently of (\ref{rhonL-2}).

The existence of a solution to (\ref{bLMd},c) will be proved by
using a fixed-point argument. Given
$\widetilde \psi \in L^2_M(\Omega \times D)$,
let $(\ut^{\star}, \widetilde \psi^\star) \in
\Vt \times X$ be such that
\begin{subequations}
\begin{alignat}{2}
\qquad b(\ut^{\star},\wt) &= \ell_b(\widetilde \psi)(\wt)
\qquad &&\forall \wt \in \Vt,
\label{fix4} \\
\qquad a(\widetilde \psi^{\star},\varphi) &=
\lae(\ut^\star,\beta^L_\delta(\widetilde \psi))(\varphi) \qquad
&&\forall \varphi \in X. \label{fix3}
\end{alignat}
\end{subequations}

The Lax--Milgram theorem yields the existence of a unique solution $\ut^\star \in \Vt$ to (\ref{fix4})
for a given $\widetilde\psi \in X$, and the existence of a unique solution $\widetilde \psi^\star \in
X$ to (\ref{fix3}) for a given pair $(\ut^\star, \widetilde \psi) \in \Vt \times X$. Therefore
the overall procedure (\ref{fix4},b) that maps a function $\widetilde\psi \in L^2_M(\Omega \times D)$
into $\widetilde\psi^\star \in X$ is well defined.

\begin{lemma}
\label{fixlem} Let ${\mathcal T}: L^2_M(\Omega \times D) \rightarrow X
\subset L^2_M(\Omega \times D)$ denote the
nonlinear map that takes the function $\widetilde{\psi}$ to $\widetilde \psi^{\star} = {\mathcal T}(\widetilde \psi)$
{\em via} the procedure
{\rm (\ref{fix4},b)}. Then, the mapping ${\mathcal T}$ has a fixed point. Hence, there exists a solution
$(\utaed^n,\hpsiaedt^n) \in \Vt \times X$ to {\rm (\ref{bLMd},c)}.
\end{lemma}

\begin{proof}
The proof is a simple adaption of the proof of Lemma 3.2 in
\cite{BS2011-fene}, where $\rho^{n}_{L} = \rho^{n-1}_L \equiv 1$
and $\zeta(\cdot)$ is identically equal to a positive constant.
Clearly, a fixed point of ${\mathcal T}$ yields a
solution of (\ref{bLMd},c).
In order to show that ${\mathcal T}$ has a fixed point, we apply
Schauder's fixed-point theorem; that is, we need to show that:
(i)~${\mathcal T}:
L^2_M(\Omega \times D) \rightarrow
L^2_M(\Omega \times D)$ is continuous; (ii)~${\mathcal T}$ is compact; and
(iii)~there exists a $C_{\star} \in {\mathbb R}_{>0}$ such that
\begin{eqnarray}
\|\widetilde{\psi}\|_{L^{2}_M(\Omega\times D)} \leq C_{\star}
\label{fixbound}
\end{eqnarray}
for every $\widetilde{\psi} \in L^2_M(\Omega \times D)$ and $\kappa \in (0,1]$
satisfying $\widetilde{\psi} = \kappa\, {\mathcal T}(\widetilde{\psi})$.

(i) Let $\{\widetilde{\psi}^{(p)}\}_{p \geq 0}$ be such that
\begin{eqnarray}
\widetilde{\psi}^{(p)}
\rightarrow \widetilde{\psi} \qquad \mbox{strongly in }
L^{2}_M(\Omega\times D)\qquad \mbox{as } p \rightarrow \infty.
\label{Gcont1}
\end{eqnarray}
It follows immediately from (\ref{betaLd}) and (\ref{eqCttbd}) that
\begin{subequations}
\begin{alignat}{1}
M^{\frac{1}{2}}\,\beta^L_\delta(\widetilde{\psi}^{(p)})
\rightarrow M^{\frac{1}{2}}\,
\beta^L_\delta(\widetilde{\psi}) \qquad
\mbox{strongly in }
L^r(\Omega\times D)\quad
\mbox{as } p \rightarrow \infty,
\label{betacon}
\end{alignat}
for all $r \in [1,\infty)$ and, for $i=1, \dots, K$,
\begin{alignat}{1}
&
\Ctt_i(M\,\zeta(\rho^n_L)\,\widetilde{\psi}^{(p)}) \rightarrow \Ctt_i(M\,\zeta(\rho^n_L)\,\widetilde{\psi}) \qquad
\mbox{strongly in } L^{2}(\Omega)\quad
\mbox{as } p \rightarrow \infty.
\label{Cttcon}
\end{alignat}
\end{subequations}
In order to prove that ${\mathcal T}: L^2_M(\Omega \times D) \rightarrow
L^2_M(\Omega \times D)$ is continuous, we need to show that
\begin{eqnarray}
\widetilde{\eta}^{(p)}:={\mathcal T}(\widetilde{\psi}^{(p)}) \rightarrow {\mathcal T}(\widetilde{\psi})
\qquad \mbox{strongly in } L^2_M(\Omega\times D)\quad \mbox{as } p
\rightarrow \infty. \label{Gcont2}
\end{eqnarray}
We have from the definition of ${\mathcal T}$ (see (\ref{fix4},b)) that, for all $p \geq 0$,
\begin{subequations}
\begin{align}
a(\widetilde \eta^{(p)}, \varphi)&= \ell_a(\vt ^{(p)},
\beta^L_\delta(\widetilde \psi^{(p)}))(\varphi)
\qquad \forall \varphi \in X,
\label{Gcont5}
\end{align}
and $\vt^{(p)} \in \Vt$ satisfies
\begin{align}
b(\vt^{(p)},\wt) &= \ell_b(\widetilde \psi^{(p)})(\wt)
\qquad \forall \wt \in \Vt.
\label{Gcont5a}
\end{align}
\end{subequations}

Choosing $\widetilde \varphi = \widetilde \eta^{(p)}$ in (\ref{Gcont5})
yields, on noting the simple identity
\begin{equation}
2\,(s_1-s_2)\,s_1 = s_1^2 + (s_1 -s_2)^2 -s_2^2 \qquad \forall
s_1, s_2 \in {\mathbb R}, \label{simpid}
\end{equation}
and  
(\ref{inidata}), 
(\ref{betaLd})
that, for all $p \geq 0$,
\begin{align}
&\int_{\Omega \times D}
M\,\left [\zeta(\rho^n_L)\, |\widetilde{\eta}^{(p)}|^2
+ \zeta(\rho^{n-1}_L)\,| \widetilde\eta^{(p)} - \hpsiaet^{n-1}|^2
\right.\nonumber\\
&\hspace{6cm} \left.+ \frac{a_0\,\Delta t}{2\,\lambda}\,|\nabq \widetilde\eta^{(p)}|^2
+ 2\,\epsilon \, \Delta t\,|\nabx \widetilde\eta^{(p)}|^2 \right] \dq \dx
\nonumber
\\
& \qquad \qquad \leq \int_{\Omega \times D} M\,\zeta(\rho^{n-1}_L)\,
|\hpsiaet^{n-1}|^2 \dq \dx + C(L,\lambda,\zeta_{\rm max},a_0^{-1})\,\Delta t\int_{\Omega}
|\nabxtt \vt^{(p)}|^2 \dx.
\label{Gcont6}
\end{align}
Choosing $\wt \equiv \vt^{(p)}$ in (\ref{Gcont5a}), and noting
(\ref{simpid}), 
(\ref{inidata}), (\ref{fvkbd}),
(\ref{eqCttbd}),
(\ref{Korn})
and (\ref{Gcont1}) yields, for
all $p\geq 0$, that
\begin{align}
&\int_{\Omega} \left[ \rho^n_{L}\,|\vt^{(p)}|^2 +
\rho^{n-1}_L\,|\vt^{(p)}-\utae^{n-1}|^2  \right] \dx +
\mu_{\rm min}\,\Delta t \, \int_{\Omega} |\Dtt(\vt^{(p)})|^2 \dx
\label{Gcont4}
\nonumber\\
& \qquad
\leq \int_\Omega \rho^{n-1}_L\,|\utae^{n-1}|^2 \dx + C\,\Delta t \,
\|\ft^n\|_{L^\varkappa(\Omega)}^2
+ C\, \Delta t \, \int_{\Omega \times D}
M\,|\widetilde{\psi}^{(p)}|^2 \dq \dx
\leq C(L).
\end{align}
Combining (\ref{Gcont6}) and (\ref{Gcont4}), noting that by \eqref{inidata} the range of $\zeta $ is the compact subinterval $[\zeta_{\rm min}, \zeta_{\rm max}]$ of $(0,\infty)$, and recalling (\ref{Korn}),
we have for all $p \geq 0$ that
\begin{eqnarray}
\|\widetilde\eta^{(p)}\|_{X}
+ \|\vt^{(p)}\|_{H^1(\Omega)} \leq C(L, (\Delta t)^{-1})\,.
\label{Gcont7}
\end{eqnarray}
It follows from
(\ref{Gcont7}), (\ref{embed}) and the compactness of the embedding (\ref{wcomp2})
that there exists a subsequence
$\{(\widetilde{\eta}^{(p_k)},\vt^{(p_k)})\}_{p_k \geq 0}$ and
functions $\widetilde{\eta}\in X$ and $\vt \in \Vt$ such that, as $p_k \rightarrow \infty$,
\begin{subequations}
\begin{alignat}{2}
\widetilde{\eta}^{(p_k)}
&\rightarrow
\widetilde{\eta}
\qquad
&&\mbox{weakly in }
L^{s}(\Omega ;L^2_M(D)),
\label{Gcont8a} \\
\bet
M^{\frac{1}{2}}\,\nabx \widetilde{\eta}^{(p_k)}
&\rightarrow M^{\frac{1}{2}}\,\nabx \widetilde{\eta} \qquad
&&\mbox{weakly in }
\Lt^{2}(\Omega \times D),
\label{Gcont8bx} \\
\bet
M^{\frac{1}{2}}\,\nabq \widetilde{\eta}^{(p_k)}
&\rightarrow M^{\frac{1}{2}}\,\nabq \widetilde{\eta}
\qquad
&&\mbox{weakly in }
\Lt^{2}(\Omega \times D),
\label{Gcont8b} \\
\bet
\widetilde{\eta}^{(p_k)}
&\rightarrow
\widetilde{\eta}
\qquad
&&\mbox{strongly in }
L^{2}_M(\Omega\times D),
\label{Gcont8as} \\
\bet
\vt^{(p_k)} &\rightarrow \vt \qquad
&&\mbox{weakly in }
\Ht^{1}(\Omega);
\label{Gcont8c}
\end{alignat}
\end{subequations}
where $s \in [1,\infty)$ if $d=2$ or $s \in [1,6]$ if $d=3$.
We deduce from (\ref{Gcont5a}), (\ref{bgen},b), (\ref{Gcont8c})
and (\ref{Cttcon})
that the functions $\vt \in \Vt$ and $\widetilde{\psi} \in X$ satisfy
\begin{eqnarray}\qquad
b(\vt,\wt) =\ell_b(\widetilde \psi)(\wt)
\qquad \forall \wt \in \Vt.
\label{Gcont9}
\end{eqnarray}
It follows from (\ref{Gcont5}), (\ref{agen},b), (\ref{Gcont8a}--e) and
(\ref{betacon}) that $\widetilde\eta,\,\widetilde \psi \in X$
and $\vt \in \Vt$ satisfy
\begin{align}
a(\widetilde{\eta},\varphi)
&= \lae(\vt,\beta^L_\delta(\widetilde \psi))(\varphi)
\qquad \forall \varphi \in C^\infty(\overline{\Omega \times D}).
\label{Gcont11}
\end{align}
Then, noting that $a(\cdot,\cdot)$
is a continuous bilinear functional on $X \times X$,
that $\lae(\vt,\beta^L_\delta(\widetilde \psi))(\cdot)$ is a continuous linear functional on $X$,
and recalling (\ref{cal K}), we deduce that (\ref{Gcont11}) holds
for all $\varphi \in X$.
Combining this $X$ version of (\ref{Gcont11}) and (\ref{Gcont9}), we have that
$\widetilde\eta = {\mathcal T}(\widetilde \psi)\in X$. Therefore
the whole sequence
\[\widetilde{\eta}^{(p)} \equiv {\mathcal T}(\widetilde{\psi}^{(p)})
\rightarrow {\mathcal T}(\widetilde{\psi})\qquad \mbox{strongly in $L^2_M(\Omega\times D)$},
\]
as $p \rightarrow \infty$, and so (i) holds.

(ii) Since the embedding $X \hookrightarrow L^{2}_M(\Omega \times D)$
is compact, we directly deduce that the mapping ${\mathcal T}:L^2_M(\Omega \times D) \rightarrow
L^2_M(\Omega \times D)$ is compact. It therefore remains to show that (iii) holds.

(iii) Let us suppose that $\widetilde{\psi} =
\kappa \, {\mathcal T}(\widetilde{\psi})$; then, the pair $(\vt,\widetilde{\psi}) \in
\Vt \times X$ satisfies
\begin{subequations}
\begin{alignat}{2}
b(\vt,\wt) &= \ell_b(\widetilde \psi)(\wt)
\quad &&\forall \wt \in \Vt,
\label{fix4sig}
\\
a(\widetilde{\psi},\varphi)&=\kappa\,
\lae(\vt,\beta^L_\delta(\widetilde \psi))(\varphi)\qquad
&&\forall \varphi \in X. \label{fix3sig}
\end{alignat}
\end{subequations}
Choosing $\wt \equiv \vt$ in (\ref{fix4sig})
yields, similarly to (\ref{Gcont4}), that
\begin{align}
&\tfrac{1}{2}\,\displaystyle
\int_{\Omega} \left[ \,\rho^n_{L}\,|\vt|^2 +\rho^{n-1}_L\,
|\vt-\utae^{n-1}|^2 - \rho^{n-1}_L\,|\utae^{n-1}|^2 \,\right]
\dx 
+ \Delta t\,
\int_{\Omega} \mu(\rho^n_{L})\,|\Dtt(\vt)|^2 \dx
\nonumber \\
&\hspace{1in}
= \Delta t \left[
\int_{\Omega}\rho^n_{L}\,\ft^n \cdot \vt \dx
- k\,\sum_{i=1}^K
\int_{\Omega}
 \Ctt_i(M\,\zeta(\rho^n_L)\,\widetilde{\psi}): \nabxtt \vt \dx \right].
\label{Gequnbhat}
\end{align}
Selecting $\varphi = [\mathcal{F}_\delta^L]'(\widetilde{\psi})$ in (\ref{fix3sig}), defining
$\mathcal{G}^L_\delta \in W^{1,1}_{\rm loc}(\mathbb{R})$ by
%
%
\begin{equation}\label{Gdl}
\mathcal{G}^L_\delta(s) := \left\{\begin{array}{ll}
\frac{1}{2\delta} (s^2 + \delta^2)-1         & \qquad\mbox{if $s \leq \delta$}, \\
s - 1                                        & \qquad\mbox{if $s \in [\delta, L]$},\\
\frac{1}{2L} (s^2+L^2)-1                             & \qquad\mbox{if $s \geq L$};
\end{array}
\right.
\end{equation}
using that, thanks to (\ref{betaLd}),  $[\mathcal{G}^L_\delta]'(s) = s/\beta^L_\delta(s) = s [\mathcal{F}^L_\delta]''(s)$;
and that, by virtue of \eqref{Gequn-trans0} with $\varphi = [\mathcal{G}^L_\delta](\widetilde{\psi})$, we have
\begin{align}
&- \Delta t \int_{\Omega \times D} M
\left(\frac{1}{\Delta t}\int_{t_{n-1}}^{t_n}\zeta(\rho^{[\Delta t]}_L)\dd t\right)
\utae^{n-1}\, \widetilde\psi \,\cdot\, \nabx [\mathcal{F}^L_\delta]'(\widetilde{\psi}) \dq \dx
\nonumber\\
&\qquad = - \int_{\Omega \times D} M
\left(\int_{t_{n-1}}^{t_n}\zeta(\rho^{[\Delta t]}_L)\dd t\right)
\utae^{n-1}\,\cdot\, \nabx [\mathcal{G}^L_\delta](\widetilde{\psi}) \dq \dx
\nonumber\\
&\qquad = - \int_{\Omega \times D} M
\left(\zeta(\rho^{n}_L) - \zeta(\rho^{n-1}_L)\right)
\, \mathcal{G}^L_\delta(\widetilde{\psi}) \dq \dx.
\end{align}
The convexity of $\mathcal{F}_\delta^L$ and Lemma \ref{basic}, with $c_0=0$ on noting that
$s [\mathcal{F}^L_\delta]'(s) - \mathcal{F}^L_\delta(s) - \mathcal{G}^L_\delta(s)=0$, then imply that
\begin{align}
&
\int_{\Omega \times D} M \left( \zeta(\rho^n_L)\,\mathcal{F}_\delta^L (\widetilde{\psi})
- \zeta(\rho^{n-1}_L)\,\mathcal{F}_\delta^L (\kappa \,\hpsiaet^{n-1}) \right) \dq \dx
\nonumber \\
& \qquad
+ \frac{\Delta t}{4\,\lambda}\,
\sum_{i=1}^K \sum_{j=1}^K A_{ij}
\int_{\Omega \times D} M\,
\nabqj
\widetilde{\psi} \cdot \nabqi ([\mathcal{F}_\delta^L]'(\widetilde{\psi}))
\dq \dx
\nonumber
\\
&\qquad\qquad
+ \varepsilon \,\Delta t
\int_{\Omega \times D}
 M\,  \nabx \widetilde\psi \cdot \nabx ([\mathcal{ F}_{\delta}^L]'(\widetilde{\psi}))
\dq \dx
\nonumber \\
&
\hspace{1in}
\leq \kappa\,\Delta t \,\sum_{i=1}^K
\int_{\Omega \times D}
M\,\zeta(\rho^n_L)\,\sigtt(\vt) \,\qt_i \cdot
\nabqi \widetilde{\psi}
\dq \dx
\nonumber
\\
&\hspace{1in}
= \kappa\,\Delta t \,\sum_{i=1}^K
\int_{\Omega}
\Ctt_i(M\,\zeta(\rho^n_L)\,\widetilde \psi) : \sigtt(\vt)  \dx,
\label{acorstab1}
\end{align}
where in the transition to the final inequality we applied \eqref{intbyparts} with
$B := \sigtt(\vt)$ (on account of it being independent of the variable $\qt$), together
with the fact that $\mathfrak{tr}(\sigtt(\vt))  = \nabx \cdot ~\vt = 0$, and recalled
\eqref{eqCtt}.

Combining (\ref{Gequnbhat}) and (\ref{acorstab1}),
and noting 
\eqref{Gdlpp}, (\ref{inidata}), (\ref{fvkbd}) and
(\ref{Korn}) 
yields that
\begin{align}
&\frac{\kappa}{2}\,\displaystyle
\int_{\Omega} \left[ \,\rho^n_L\,|\vt|^2 +\rho^{n-1}_{L}\,
|\vt-\utae^{n-1}|^2 \,\right]
\dx
+ \kappa\,\Delta t\,
\int_{\Omega} \mu(\rho^n_L)\,|\Dtt(\vt)|^2 \dx
\nonumber
\\
&\qquad
+k\,
\int_{\Omega \times D} M\,\zeta(\rho^n_L)\, \mathcal{F}_\delta^L (\widetilde{\psi})
\dq \dx
\nonumber
\\
&\qquad\qquad +k\,L^{-1}\,\Delta t\,
\int_{\Omega \times D} M\,
\left[\,\epsilon
\,|\nabx \widetilde{\psi}|^2 + \frac{a_0}{4\,\lambda} \,|\nabq
\widetilde{\psi}|^2 \right] \dq \dx
\nonumber \\
&\qquad\qquad\qquad\leq
\kappa\,\Delta t \,\int_{\Omega} \rho^n_L\,\ft^n \cdot \vt \dx +
\frac{\kappa}{2}\,\int_{\Omega} \rho^{n-1}_L\,|\utae^{n-1}|^2 \dx
\nonumber\\
&\qquad \qquad \qquad \qquad + k\,
\int_{\Omega \times D} M\,\zeta(\rho^{n-1}_L)\,\mathcal{F}_\delta^L (\kappa \,\hpsiaet^{n-1})
\dq \dx
\nonumber
\\
&\qquad\qquad \qquad \leq
\frac{\kappa \,\Delta t\, \mu_{\rm min}}{2} \,
\int_{\Omega} |\Dtt(\vt)|^2 \dx
+ \frac{\kappa \,\Delta t\,\rho_{\rm max}^2 C_{\varkappa}^2}{2\mu_{\rm min}\,c_0}\,
\|\ft^n\|_{L^{\varkappa}(\Omega)}^2 \nonumber\\
&\qquad\qquad\qquad\qquad
+ \frac{\kappa}{2}\,\int_{\Omega} \rho^{n-1}_{L}\,|\utae^{n-1}|^2 \dx
+ k\,
\int_{\Omega \times D} M\,\zeta(\rho^{n-1}_L)\,\mathcal{F}_\delta^L (\kappa \,\hpsiaet^{n-1})
\dq \dx.
\label{Ek}
\end{align}

It is easy to show that $\mathcal{F}^L_\delta(s)$ is nonnegative for all
$s \in \mathbb{R}$, with  $\mathcal{F}^L_\delta(1)=0$.
Furthermore, for any $\kappa \in (0,1]$,
$\mathcal{F}^L_\delta(\kappa\, s) \leq \mathcal{F}^L_\delta(s)$
if $s<0$ or $1 \leq \kappa\, s$, and also
$\mathcal{F}^L_\delta(\kappa\, s) \leq \mathcal{F}^L_\delta(0) \leq 1$
if $0 \leq \kappa\, s \leq 1$.
Thus we deduce that
\begin{equation}\label{deltaL}
\mathcal{F}_\delta^L(\kappa\, s)
\leq \mathcal{F}_\delta^L(s)+ 1\qquad \forall s \in {\mathbb R},\quad
\forall \kappa \in (0,1].
\end{equation}
Hence, the bounds (\ref{Ek}) and (\ref{deltaL}), on noting (\ref{cFbelow}),
give rise to the desired bound (\ref{fixbound}) with $C_*$ dependent only on
$\delta$, $L$, $k$, $\mu_{\rm min}$, $\rho_{\rm max}$, $\zeta_{\rm min}$, $\ft$, $\ut^{n-1}_L$ and $\hpsiaet^{n-1}$.
Therefore (iii) holds, and so ${\mathcal T}$ has a fixed point, proving
existence of a solution
to (\ref{bLMd},c).
\qquad\end{proof}

Choosing $\wt \equiv \utaed^n$ in (\ref{bLMd})
and $\varphi \equiv [\mathcal{F}_\delta^L]'(\hpsiaedt^n)$ in (\ref{genLMd}),
and combining and noting (\ref{inidata}), then yields, as in
(\ref{Ek}), with $C(L)$ a positive constant, independent of $\delta$ and $\Delta t$,
\begin{align}\label{E1}
&
\tfrac{1}{2}\,\displaystyle
\int_{\Omega} \left[ \,\rho^n_L\,|\utaed^n|^2 +\rho^{n-1}_{L}\,
|\utaed^n-\utae^{n-1}|^2 \,\right] \dx
+k\, \int_{\Omega \times D} M\,\zeta(\rho^n_L)\,\mathcal{F}_\delta^L (\hpsiaedt^n)
\dq \dx
\nonumber
\\
&\hspace{0.5in}
+ \Delta t\, \biggl[ \tfrac{1}{2} \,
\int_{\Omega} \mu(\rho^n_L)\,|\Dtt(\utaed^n)|^2 \dx
+k\,L^{-1}\,\epsilon\,
\int_{\Omega \times D} M\,
\,|\nabx \hpsiaedt^n|^2
\dq \dx
\nonumber \\
& \hspace{1in}+ \frac{k\,L^{-1}\,a_0}{4\,\lambda} \,
\int_{\Omega \times D}
M\, |\nabq
\hpsiaedt^n|^2
 \dq \dx \biggr] \nonumber
\\
&\qquad \leq
\tfrac{1}{2}\,\int_{\Omega} \rho^{n-1}_L\,|\utae^{n-1}|^2 \dx
+
\frac{\Delta t\,\rho_{\rm max}^2 C_{\varkappa}^2}{2\mu_{\rm min}\,c_0}\,\|\ft^n \|_{L^\varkappa(\Omega)}^2
\nonumber\\
& \hspace{1.0in}
+ k\,
\int_{\Omega \times D} M\,\zeta(\rho^{n-1}_L)\,\mathcal{F}_\delta^L (\hpsiaet^{n-1})
\dq \dx
\nonumber \\
& \qquad \leq C(L).
\end{align}

Equation \eqref{rhonL-2} being independent of $\delta$, we
are now ready to pass to the limit $\delta \rightarrow 0_+$ in (\ref{bLMd},c),
to deduce the existence of
a solution $\{(\rho^{[\Delta t]}_L|_{[t_{n-1},t_n]},\utae^n,\hpsiaet^n)\}_{n=1}^N$ to
({\rm P}$^{\Delta t}_{L}$), with $\rho^n_L =\rho^{[\Delta t]}_L(\cdot,t_n) \in \Upsilon$, $\utae^n \in \Vt$
and $\hpsiaet^n \in X \cap Z_2$, $n=1,\dots, N$.

\begin{lemma}
\label{conv}
There exists a subsequence (not indicated) of $\{(\utaed^{n},
\hpsiaedt^{n})\}_{\delta >0}$, and functions $\utae^{n} \in \Vt$
and $\hpsiaet^{n} \in X \cap Z_2$, $n \in \{1,\dots, N\}$,
such that, as $\delta \rightarrow 0_+$,
%
\begin{subequations}
\begin{alignat}{2}
&\utaed^{n} \rightarrow \utae^{n} \qquad &&\mbox{weakly in } \Vt, \label{uwconH1}\\
\bet
&\utaed^{n} \rightarrow \utae^{n}
\qquad &&\mbox{strongly in }
\Lt^{r}(\Omega), \label{usconL2}
\end{alignat}
\end{subequations}
where $r \in [1,\infty)$ if $d=2$ and $r \in [1,6)$ if $d=3$;
and
\begin{subequations}
\begin{alignat}{2}
\!\!\!\!M^{\frac{1}{2}}\,\hpsiaedt^{n} &\rightarrow
M^{\frac{1}{2}}\,
\hpsiaet^{n} &&\quad \mbox{weakly in }
L^2(\Omega\times D),~~~ \label{psiwconL2}\\
\bet
M^{\frac{1}{2}}\,\nabq \hpsiaedt^{n}
&\rightarrow M^{\frac{1}{2}}\,\nabq \hpsiaet^{n}
&&\quad \mbox{weakly in }
\Lt^2(\Omega\times D), \label{psiwconH1}\\
\bet
M^{\frac{1}{2}}\,\nabx \hpsiaedt^{n}
&\rightarrow M^{\frac{1}{2}}\,\nabx \hpsiaet^{n}
&&\quad \mbox{weakly in }
\Lt^2(\Omega\times D), \label{psiwconH1x}\\
\bet
M^{\frac{1}{2}}\,\hpsiaedt^{n} &\rightarrow
M^{\frac{1}{2}}\,\hpsiaet^{n}
&&\quad \mbox{strongly in }
L^{2}(\Omega\times D),\label{psisconL2}
\\
\bet
M^{\frac{1}{2}}\,\beta_\delta^L(\hpsiaedt^{n}) &\rightarrow
M^{\frac{1}{2}}\,\beta^L(\hpsiaet^{n})
&&\quad \mbox{strongly in }
L^s(\Omega\times D),\label{betaLdsconL2}
\end{alignat}
for all $s \in [1,\infty)$; and, for $i=1, \dots,  K$,
\begin{alignat}{1}
&\hspace{-4mm}\Ctt_i(M\,\zeta({\rho^n_L})\,\hpsiaedt^{n}) \rightarrow \Ctt_i(M\,\zeta({\rho^n_L})\,\hpsiaet^{n})
\qquad \!\!\!\!\!\mbox{ strongly in }
\Ltt^{2}(\Omega).\label{CwconL2}
\end{alignat}
\end{subequations}
Furthermore, $(\rho^{[\Delta t]}_L|_{[t_{n-1},t_n]},\utae^n, \hpsiaet^n)$ solves (\ref{rhonL}--c) for $n=1,\dots, N$. Hence, there exists a solution $\{(\rho^{[\Delta t]}_L|_{[t_{n-1},t_n]},\utae^n, \hpsiaet^n)\}_{n=1}^N$ to
{\em ({\rm P}$^{\Delta t}_{L}$)}, with $\rho^n_L=\rho^{[\Delta t]}_L(\cdot,t_n) \in \Upsilon$, $\utae^n \in \Vt$ and $\hpsiaet^n \in X \cap Z_2$ for all $n=1,\dots, N$.
\end{lemma}


\begin{proof}

The weak convergence results (\ref{uwconH1}), (\ref{psiwconL2}) and that
$\hpsiaet^n \ge 0$ a.e.\ on $\Omega \times D$ follow immediately from 
(\ref{E1}), on noting (\ref{Korn}) and (\ref{cFbelow}). The strong convergence
(\ref{usconL2}) for $\utaed^{n}$ follows from (\ref{uwconH1}),
on noting that $\Vt \subset \Ht^{1}_0(\Omega)$ is compactly embedded in $\Lt^r(\Omega)$ for
the stated values of $r$.

The results (\ref{psiwconH1},c)
follow from 
(\ref{E1});
see the proof of Lemma 3.3 in \cite{BS2011-fene} for details.
The strong convergence result (\ref{psisconL2}) for $\hpsiaedt^{n}$
follows directly from (\ref{psiwconL2}--c) and (\ref{wcomp2}).
In addition, 
(\ref{betaLdsconL2},f) follow from (\ref{psisconL2}), (\ref{betaLd}), (\ref{eqCtt}) and (\ref{eqCttbd}).

It follows from 
(\ref{uwconH1},b), (\ref{psiwconH1}--f), (\ref{bgen},b), (\ref{agen},b),
(\ref{lgenbd}) and
(\ref{cal K}) that we may pass to the limit $\delta \rightarrow 0_+$ in
(\ref{bLMd},c) to obtain that $(\utae^n,\hpsiaet^n) \in \Vt \times X$
with $\hpsiaet^n \geq 0$ a.e.\ on $\Omega \times D$ solves (\ref{bLM}),
(\ref{genLM}); that is, (\ref{Gequn},c).

Next we shall show that
\begin{equation}\label{lambda-Linfty}
\int_D M(\qt)\, \widetilde\psi^n_L(\xt, \qt) \dq \dx \in L^\infty(\Omega),
\end{equation}
uniformly with respect to $L>1$ for all $n \in \{1,\dots,N\}$. Hence we shall deduce in particular that
$\widetilde\psi^n_L\in Z_2$. We begin by selecting $\varphi(\xt,\qt) = \widetilde\varphi(\xt) \otimes 1(\qt)$
in \eqref{psiG} with $\widetilde\varphi \in H^1(\Omega)$, which then yields that
\begin{eqnarray*}
&&\int_{\Omega \times D} M \frac{\zeta(\rho^n_L)\, \widetilde\psi^n_{L} - \zeta(\rho^{n-1}_L)\, \widetilde\psi^{n-1}_{L}}
{\Delta t} \,\widetilde\varphi \dq \dx\\
&&\qquad \qquad - \int_{\Omega \times D} M \left( \frac{1}{\Delta t} \int_{t_{n-1}}^{t_n} \zeta(\rho^{[\Delta t]}) \dd t\right)\, \ut^{n-1}_{L}(\xt) \cdot (\nabx \widetilde\varphi)\, \widetilde \psi^n_{L} \dq \dx\\
&&\qquad \qquad \qquad \qquad +\, \varepsilon \int_{\Omega \times D} M\, \nabx \widetilde \psi^n_{L} \cdot \nabx \widetilde \varphi \dq \dx = 0
\qquad \forall \widetilde\varphi \in H^1(\Omega).
\end{eqnarray*}
We define
\[ \lambda^n_{L}(\xt) := \int_D M(\qt)\, \widetilde\psi^n_{L}(\xt, \qt) \dq, \qquad n=0,\dots,N,\]
with $\widetilde\psi^0_{L} := \widetilde\psi^0 = \beta^L(\widetilde\psi^0)$, and note that $\lambda^n_L \in H^1(\Omega)$.
By Fubini's theorem we then have that
\begin{eqnarray*}
&&\int_\Omega \frac{\zeta(\rho^n_L)\, \lambda^n_{L} - \zeta(\rho^{n-1}_L)\,\lambda^{n-1}_{L}}{\Delta t}\, \widetilde\varphi(\xt) \dx
- \int_\Omega \left( \frac{1}{\Delta t} \int_{t_{n-1}}^{t_n} \zeta(\rho^{[\Delta t]}) \dd t \right)
\ut^{n-1}_{L} \cdot (\nabx \widetilde \varphi) \, \lambda^n_{L} \dx\\
&& \qquad \qquad \qquad +\, \varepsilon \int_\Omega \nabx \lambda^n_{L} \cdot \nabx \widetilde \varphi \dx = 0
\qquad \forall \widetilde \varphi \in H^1(\Omega).
\end{eqnarray*}
This in particular implies, using the identity  \eqref{Gequn-trans0}
with $\varphi$ replaced by $\omega\,\varphi$, $\omega \in \mathbb{R}$, that, for each $\omega \in \mathbb{R}$,
\begin{eqnarray*}
&&\int_\Omega \frac{\zeta(\rho^n_L)\, (\lambda^n_{L}-\omega) - \zeta(\rho^{n-1}_L)\,(\lambda^{n-1}_{L}-\omega)}{\Delta t}\, \widetilde\varphi(\xt) \dx\\
&& \qquad\qquad - \int_\Omega \left( \frac{1}{\Delta t} \int_{t_{n-1}}^{t_n} \zeta(\rho^{[\Delta t]}) \dd t \right)
\ut^{n-1}_{L} \cdot (\nabx \widetilde \varphi) \, (\lambda^n_{L}-\omega) \dx\\
&& \qquad \qquad \qquad \qquad +\, \varepsilon \int_\Omega \nabx (\lambda^n_{L} - \omega) \cdot \nabx \widetilde \varphi \dx = 0
\qquad \forall \widetilde \varphi \in H^1(\Omega).
\end{eqnarray*}
On selecting $\widetilde\varphi = [\lambda^n_{L} - \omega]_{+}$ in this identity we then deduce that
%
%
\begin{eqnarray*}
&&\int_\Omega \frac{\zeta(\rho^n_L)\, (\lambda^n_{L}-\omega) - \zeta(\rho^{n-1}_L)\,(\lambda^{n-1}_{L}-\omega)}{\Delta t}\, [\lambda^n_{L} - \omega]_{+} \dx\\
&& \qquad \qquad- \frac{1}{2} \int_\Omega \left( \frac{1}{\Delta t} \int_{t_{n-1}}^{t_n} \zeta(\rho^{[\Delta t]}) \dd t \right)
\ut^{n-1}_{L} \cdot\nabx  \left([\lambda^n_{L} - \omega]_{+}\right)^2 \dx\\
&& \qquad \qquad \qquad \qquad  +\, \varepsilon \int_\Omega |\nabx \left([\lambda^n_{L} - \alpha
]_{+}\right)|^2 \dx = 0,
\end{eqnarray*}
and hence, by omitting the (nonnegative) last term from the left-hand side, we have that
\begin{eqnarray}\label{inter1a}
&&\int_\Omega \frac{\zeta(\rho^n_L)\, (\lambda^n_{L}-\omega) - \zeta(\rho^{n-1}_L)\,(\lambda^{n-1}_{L}-\omega)}{\Delta t}\, [\lambda^n_{L} - \omega]_{+} \dx\nonumber\\
&& \qquad \qquad- \frac{1}{2} \int_\Omega \left( \frac{1}{\Delta t} \int_{t_{n-1}}^{t_n} \zeta(\rho^{[\Delta t]}) \dd t \right)
\ut^{n-1}_{L} \cdot\nabx  \left([\lambda^n_{L} - \omega]_{+}\right)^2 \dx \leq  0.
\end{eqnarray}
As, once again using the identity \eqref{Gequn-trans0},
with $\varphi = \left( [\lambda^n_{L} - \omega]_+\right)^2$ this time, we have for each $\omega \in \mathbb{R}$ that
\begin{eqnarray*}
&&\int_\Omega \frac{\zeta(\rho^n_L) - \zeta(\rho^{n-1}_L)}{\Delta t}\, ([\lambda^n_{L} - \omega ]_{+})^2 \dx\\
&&\qquad\qquad - \int_\Omega \left(\frac{1}{\Delta t}\int_{t_{n-1}}^{t_n} \zeta(\rho^{[\Delta t]}) \dd t \right) \ut^{n-1}_L \cdot
\nabx \left( [\lambda^n_{L} - \omega]_+\right)^2 \dx = 0,
\end{eqnarray*}
we can rewrite the second term in \eqref{inter1a} to deduce that
\begin{eqnarray}\label{inter2a}
&&\int_\Omega \frac{\zeta(\rho^n_L)\, (\lambda^n_{L}-\omega) - \zeta(\rho^{n-1}_L)\,(\lambda^{n-1}_{L}-\omega)}{\Delta t}\, [\lambda^n_{L} - \omega]_{+} \dx\nonumber\\
&& \qquad \qquad- \frac{1}{2}
\int_\Omega \frac{\zeta(\rho^n_L) - \zeta(\rho^{n-1}_L)}{\Delta t}\, ([\lambda^n_{L} - \omega ]_{+})^2 \dx \leq 0.
\end{eqnarray}
The inequality \eqref{inter2a} can be restated in the following equivalent form:
\begin{eqnarray*}
&&\int_\Omega \zeta(\rho^{n-1}_L)\frac{(\lambda^n_{L}-\omega) - (\lambda^{n-1}_{L}-\omega)}{\Delta t}\, [\lambda^n_{L} - \omega]_{+} \dx \nonumber \\
&&\qquad\qquad+ \int_\Omega \frac{\zeta(\rho^n_L) - \zeta(\rho^{n-1}_L)}{\Delta t}\, (\lambda^n_{L} - \omega)\,[\lambda^n_{L} - \omega ]_{+} \dx \nonumber\\
&& \qquad \qquad\qquad \qquad - \frac{1}{2}
\int_\Omega \frac{\zeta(\rho^n_L) - \zeta(\rho^{n-1}_L)}{\Delta t}\, ([\lambda^n_{L} - \omega ]_{+})^2 \dx \leq 0,
\end{eqnarray*}
and hence, after simplifying the sum of the second and the third term on the left-hand side,
\begin{eqnarray}\label{inter3a}
&&\int_\Omega \zeta(\rho^{n-1}_L)\frac{(\lambda^n_{L}-\omega) - (\lambda^{n-1}_{L}-\omega)}{\Delta t}\, [\lambda^n_{L} - \omega]_{+} \dx \nonumber \\
&&\qquad\qquad + \frac{1}{2}
\int_\Omega \frac{\zeta(\rho^n_L) - \zeta(\rho^{n-1}_L)}{\Delta t}\, ([\lambda^n_{L} - \omega ]_{+})^2 \dx \leq 0.
\end{eqnarray}
Since $s \in \mathbb{R} \mapsto \frac{1}{2}\left([s-\omega]_+\right)^2 \in \mathbb{R}_{\geq 0}$ is a convex function,
we have that
\[ \frac{(\lambda^n_{L}-\omega) - (\lambda^{n-1}_{L}-\omega)}{\Delta t}\,[\lambda^n_{L} - \omega ]_{+}  \geq \frac{1}{2\,\Delta t}\left(([\lambda^n_{L}-\omega]_+)^2
-([\lambda^{n-1}_{L}-\omega]_+)^2\right).\]
Using this inequality in the first term of \eqref{inter3a},
on noting that $\zeta(\rho^{n-1}_L) \geq \zeta_{\rm min} >0$
and multiplying the resulting inequality by $2\Delta t$, we get that
\begin{eqnarray*}
\int_\Omega \zeta(\rho^{n-1}_L)\left\{([\lambda^n_{L}-\omega]_+)^2 - ([\lambda^{n-1}_{L}-\omega]_+)^2 \right\}\dx
+
\int_\Omega \left(\zeta(\rho^n_L) - \zeta(\rho^{n-1}_L)\right)\, ([\lambda^n_{L} - \omega ]_{+})^2 \dx \leq 0,
\end{eqnarray*}
and therefore, upon simplification,
\begin{eqnarray*}
\int_\Omega \zeta(\rho^n_L) \, ([\lambda^n_{L} - \omega ]_{+})^2 \dx \leq
\int_\Omega \zeta(\rho^{n-1}_L)\, ([\lambda^{n-1}_{L}-\omega]_+)^2 \dx,\qquad n=1,\dots,N.
\end{eqnarray*}
Thus, on recalling that by hypothesis $\rho^0_L = \rho_0$, we have that
\begin{eqnarray}\label{inter4a}
0 \leq \int_\Omega \zeta(\rho^n_L) \, ([\lambda^n_{L} - \omega ]_{+})^2 \dx \leq
\int_\Omega \zeta(\rho_0)\, ([\lambda^{0}_{L}-\omega]_+)^2 \dx,\qquad n=1,\dots,N.
\end{eqnarray}
Now we choose $\omega$ as in \eqref{alpha}, which yields $[\lambda^0_{L} - \omega ]_{+}=0$ a.e. on $\Omega$,
and therefore, by \eqref{inter4a}, also $[\lambda^n_L - \omega]_+ = 0~$
a.e. on $\Omega$ for all $n \in \{1,\dots,N\}$; in other words,
\begin{eqnarray}\label{lambda-bound} 0 \leq \lambda^n_L(\xt) \leq \omega :=
\mbox{ess.sup}_{x \in \Omega}\left(\frac{1}{\zeta(\rho_0(\xt))}
\int_D \psi_0(\xt,\qt) \dq \right)
\end{eqnarray}
for a.e. $x \in \Omega$ and all $n=0,\dots, N$.
Since $\omega$ here is independent of $L$, by recalling the definition of $\lambda^n_L$ we thus deduce that
$\widetilde\psi^n_L \in Z_2$, uniformly with respect to $L$, as was claimed in the line below \eqref{lambda-Linfty}.

Finally, as $(\rho^0_L,\utae^0,\hpsiaet^0) \in \Upsilon \times \Vt \times Z_2$, performing
the above existence proof at each time level $t_n$, $n=1,\ldots,N$,
yields a solution  $\{(\rho^{[\Delta t]}_L|_{[t_{n-1},t_n]},\utae^n,\hpsiaet^n)\}_{n=1}^N$
to (P$^{\Delta t}_{L}$) with $\rho^n_L = \rho^{[\Delta t]}_L(\cdot,t_n)$, $n=1,\dots,N$,
by noting that $\rho^{[\Delta t]}_L$ thus constructed is an element of $C([0,T];L^2(\Omega))$.
\end{proof}

\section{Entropy estimates}
\label{sec:entropy}
\setcounter{equation}{0}

Next, we derive bounds on the solution of $({\rm P}^{\Delta t}_{L})$, independent of $L$.
Our starting point is Lemma \ref{conv}, concerning the existence of a solution to the problem
$({\rm P}^{\Delta t}_{L})$. The model $({\rm P}^{\Delta t}_{L})$
includes `microscopic cut-off' in the drag term of the Fokker--Planck equation,
where $L>1$ is a
(fixed, but otherwise arbitrary,) cut-off parameter.
Our ultimate objective is to pass to the limits $L \rightarrow \infty$
and $\Delta t \rightarrow 0_+$ in the model $({\rm P}^{\Delta t}_{L})$,
with $L$ and $\Delta t$ linked by the condition
$\Delta t = o(L^{-1})$,
as $L \rightarrow \infty$.
To that end, we need to develop various bounds on sequences of weak solutions
of $({\rm P}^{\Delta t}_{L})$ that are uniform in the cut-off parameter
$L$ and thus permit the extraction of weakly convergent subsequences,
as $L \rightarrow \infty$, through the use of a weak compactness argument.
The derivation of such bounds, based on the use of the relative entropy associated
with the Maxwellian $M$, is our
main task in this section.

Let us introduce the following definitions, in line with (\ref{fn}):
\begin{subequations}
\begin{alignat}{1}
\hspace{-4mm}\utaeD(\cdot,t)&:=\,\frac{t-t_{n-1}}{\Delta t}\,
\utae^n(\cdot)+
\frac{t_n-t}{\Delta t}\,\utae^{n-1}(\cdot),
\;\, t\in [t_{n-1},t_n], \;\, n=1,\dots,N, \label{ulin}\\
\hspace{-2.5mm}\utaeDp(\cdot,t)&:=\ut^n(\cdot),\;\;
\utaeDm(\cdot,t):=\ut^{n-1}(\cdot),
\;\; t\in(t_{n-1},t_n], \;\; n=1,\dots,N. \label{upm}
\end{alignat}
\end{subequations}
We shall adopt $\ut_L^{\Delta t (,\pm)}$ as a collective symbol for $\ut_L^{\Delta t}$,
$\ut_L^{\Delta t,\pm}$. The corresponding notations
$\rho_L^{\Delta t}$, $\rho_L^{\Delta t,\pm}$ and $\rho_L^{\Delta t (,\pm)}$, and
$\psi_L^{\Delta t}$, $\psi_L^{\Delta t,\pm}$ and $\psi_L^{\Delta t (,\pm)}$ are defined analogously.
In addition, we define the products $(\rho_L\,\ut_L)^{\Delta t}$, $(\rho_L\,\ut_L)^{\Delta t,\pm}$
and $(\rho_L \, \ut_L)^{\Delta t (,\pm)}$; and  $(\zeta(\rho_L)\,\widetilde\psi_L)^{\Delta t}$, $(\zeta(\rho_L)\,\widetilde\psi_L)^{\Delta t,\pm}$
and $(\zeta(\rho_L)\,\widetilde\psi_L)^{\Delta t (,\pm)}$ similarly. The notation $\rho^{\Delta t}_L$ signifying the
piecewise linear interpolant of $\rho_L$ with respect to the variable $t$ is not to be confused with
the notation $\rho^{[\Delta t]}_L$, which denotes the function defined piecewise, over the union of
time slabs $\Omega \times [t_{n-1},t_n]$, $n=1,\dots,N$, as the unique solution of the equation
\eqref{rhonL} subject to the initial condition $\rho^{[\Delta t]}_L(\cdot,t_{n-1}) = \rho^{n-1}_L$,
$n=1,\dots, N$, with $\rho^0_L:= \rho_0$. Finally, we define the functions
$\rho_L^{\{\Delta t\}}$ and $\zeta^{\{\Delta t\}}_L$ by
\begin{equation}\label{zeta-time-average}
\rho^{\{\Delta t\}}_L|_{(t_{n-1},t_n)} := \frac{1}{\Delta t}
\int_{t_{n-1}}^{t_n} \rho_L^{[\Delta t]} \dd t,\;\;\; \zeta^{\{\Delta t\}}_L|_{(t_{n-1},t_n)} := \frac{1}{\Delta t}
\int_{t_{n-1}}^{t_n} \zeta(\rho_L^{[\Delta t]}) \dd t, \;\; n=1,\dots,N.
\end{equation}

%
%

Using the above notation,
(\ref{rhonL}--c) summed for $n=1, \dots,  N$ can be restated in the form:
find $(\rho^{[\Delta t]}_{L}(t),\ut^{\Delta t}_{L}(t),
\widetilde \psi^{\Delta t}_{L}(t))
\in \Upsilon \times \Vt \times (X \cap Z_2)$,
$t \in [0,T]$, such that
%
%
\begin{subequations}
\begin{align}
&\displaystyle\int_{0}^{T}\left\langle \frac{\partial \rho^{[\Delta t]}_L}{\partial t}\,,\eta
\right\rangle_{W^{1,\frac{q}{q-1}}(\Omega)} \dd t
- \int_0^T \int_\Omega \rho^{[\Delta t]}_L \,\ut^{\Delta t,-}_L
\cdot \nabx \eta
\dx \dt =0
\nonumber \\
& \hspace{3.4in}
\qquad
\forall \eta \in L^1(0,T;W^{1,\frac{q}{q-1}}(\Omega)),
\label{eqrhocon}
\end{align}
\begin{align}
&\displaystyle\int_{0}^{T}\!\! \int_\Omega \left[
\frac{\partial}{\partial t} (\rho_L\,\ut_L)^{\Delta t}
- \tfrac{1}{2}
\frac{\partial \rho_L^{\Delta t}}{\partial t}\, \ut_L^{\Delta t,+}
\right]
\cdot
\wt \dx \dt
+
\displaystyle\int_{0}^{T}\!\! \int_\Omega
\mu(\rho^{\Delta t, +}_L) \,\Dtt(\utaeDp)
:
\Dtt(\wt) \dx \dt
\nonumber \\
&\qquad+
\tfrac{1}{2} \int_{0}^T\!\! \int_{\Omega}
\rho^{\{\Delta t\}}_L \left[ \left[ (\utaeDm \cdot \nabx) \utaeDp \right]\cdot\,\wt
- \left[ (\utaeDm \cdot \nabx) \wt  \right]\cdot\,\utaeDp
\right]\!\dx \dt
\nonumber
\\
&\qquad\qquad=\int_{0}^T
\left[ \int_{\Omega}  \rho_L^{\Delta t,+}
\ft^{\Delta t,+} \cdot \wt \dx 
- k\,\sum_{i=1}^K \int_{\Omega}
\Ctt_i(M\,\zeta(\rho^{\Delta t,+}_L)\,\hpsiaet^{\Delta t,+}): \nabxtt
\wt \dx \right]\! \dt
\nonumber\\
&\hspace{3.9in}\forall \wt \in L^1(0,T;\Vt),
\label{equncon}
\end{align}
%
%
%
\begin{align}
\label{eqpsincon}
&\int_{0}^T \!\!\int_{\Omega \times D}
\left[
M\,\frac{ \partial}{\partial t} (\zeta(\rho_L)\,\hpsiaet)^{\Delta t}\,
\varphi 
+
\frac{1}{4\,\lambda}
\,\sum_{i=1}^K
 \,\sum_{j=1}^K A_{ij}\,
 M\,
 \nabqj \hpsiaet^{\Delta t,+}
\cdot\, \nabqi
\varphi
\right]
\dq \dx \dt
\nonumber \\
& \qquad 
+ \int_{0}^T \!\!\int_{\Omega \times D} \left[
\epsilon\, M\,
\nabx \hpsiaet^{\Delta
t,+} - \utae^{\Delta t,-}\,M\,\zeta^{\{\Delta t\}}_L\,\hpsiaet^{\Delta t,+} \right]\cdot\, \nabx
\varphi
\dq \dx \dt
\nonumber \\
&
\qquad\qquad 
- \int_{0}^T \!\!\int_{\Omega \times D} M\,\sum_{i=1}^K
\left[\sigtt(\utae^{\Delta t,+})
\,\qt_i\right]\,\zeta(\rho_L^{\Delta t,+})\,\beta^L(\hpsiaet^{\Delta t,+}) \,\cdot\, \nabqi
\varphi
\,\dq \dx \dt = 0
\nonumber\\
&\hspace{3.77in}\forall \varphi \in L^1(0,T;X);
\end{align}
\end{subequations}
with $q \in (2,\infty)$ when $d=2$ and $q \in [3,6]$ when $d=3$,
subject to the initial conditions
$\rho_L^{\Delta t}(0)=\rho_0 \in \Upsilon$, $\ut_L^{\Delta t}(0) = \ut^0
\in \Vt$ and
$\hpsiaet^{\Delta t}(0) = \widetilde \psi^0
\in X \cap Z_2$, where we recall (\ref{proju0}) and (\ref{psi0}).
We emphasize that (\ref{eqrhocon}--c) 
is an equivalent restatement of problem (${\rm P}^{\Delta t}_{L}$),
for which existence of a solution has been established (cf. Lemma \ref{conv}).

In conjunction with $\beta^L$, defined by \eqref{betaLa},
we consider the following cut-off version $\mathcal{F}^L$
of the entropy function $\mathcal{F}\,: s \in \mathbb{R}_{\geq 0} \mapsto \mathcal{F}(s) = s (\log s - 1) + 1
\in \mathbb{R}_{\geq 0}$:
\begin{equation}\label{eq:FL}
\mathcal{F}^L(s):= \left\{\begin{array}{ll}
s(\log s - 1) + 1,   &  ~~0 \leq s \leq L,\\
\frac{s^2 - L^2}{2L} + s(\log L - 1) + 1,  &  ~~L \leq s.
\end{array} \right.
\end{equation}
Note that
\begin{equation}\label{eq:FL1}
(\mathcal{F}^L)'(s) = \left\{\begin{array}{ll}
\log s,   &  ~~0 < s \leq L,\\
\frac{s}{L} + \log L - 1,  &  ~~L \leq s,
\end{array} \right.
\end{equation}
and
\begin{equation}\label{eq:FL2}
(\mathcal{F}^L)''(s) = \left\{\begin{array}{ll}
\frac{1}{s},   &  ~~0 < s \leq L,\\
\frac{1}{L},  &  ~~L \leq s.
\end{array} \right.
\end{equation}
Hence,
\begin{equation}\label{eq:FL2a}
\beta^L(s) = \min(s,L) = [(\mathcal{F}^L)''(s)]^{-1},\quad s \in \mathbb{R}_{\geq 0},
\end{equation}
with the convention $1/\infty:=0$ when $s=0$, and

\begin{equation}\label{eq:FL2b}
(\mathcal{F}^L)''(s) \geq \mathcal{F}''(s) = s^{-1},\quad s \in \mathbb{R}_{> 0}.
\end{equation}
We shall also require the following inequality, relating $\mathcal{F}^L$ to $\mathcal{F}$:
\begin{equation}\label{eq:FL2c}
\mathcal{F}^L(s) \geq \mathcal{F}(s),\quad s \in \mathbb{R}_{\geq 0}.
\end{equation}
For $0\leq s \leq 1$, \eqref{eq:FL2c} trivially holds, with equality. For $s\geq 1$, it
follows from \eqref{eq:FL2b}, with $s$ replaced by a dummy variable $\sigma$, after integrating
twice over $\sigma \in [1,s]$, and noting that $(\mathcal{F}^L)'(1)= \mathcal{F}'(1)$ and $(\mathcal{F}^L)(1)=\mathcal{F}(1)$.

\subsection{$L$-independent bounds on the spatial derivatives}
\label{Lindep-space}

We are now ready to embark on the derivation of the required bounds,
uniform in the cut-off parameter $L$,
on norms of $\rho^{\Delta t,+}_L(t) \in \Upsilon$,
$\uta^{\Delta t,+}(t)\in \Vt$ and $\psia^{\Delta t,+}(t) \in X \cap Z_2$, $t \in (0,T]$.
As far as
$\rho^{\Delta t,+}_L$ is concerned, it follows by tracing the constants in the argument leading to
inequality (17) in DiPerna \& Lions \cite{DPL} and recalling from \eqref{inidata} that
$\rho_0 \in \Upsilon$, that, for each $p \in [1,\infty]$,
\begin{subequations}
\begin{align}
&\|\rho_L^{[\Delta t]}(t)\|_{L^p(\Omega)}  \leq \|\rho_0\|_{L^p(\Omega)}, \qquad t \in (0,T],
\label{eq:energy-rho-0}
\end{align}
and therefore, for each $p \in [1,\infty]$,
\begin{align}
&\|\rho_L^{\Delta t,+}(t)\|_{L^p(\Omega)}  \leq \|\rho_0\|_{L^p(\Omega)}, \qquad t \in (0,T].
\label{eq:energy-rho}
\end{align}
\end{subequations}
Concerning $\uta^{\Delta t,+}$,
we select
$\wt = \chi_{[0,t]}\,\uta^{\Delta t,+}$ as test function in (\ref{equncon}), 
with $t$ chosen as $t_n$, $n \in \{1,\dots,N\}$.
We then deduce using the identity
(\ref{Gequnbhat}) with $\vt = \ut^{\Delta t,+}_L$ and $\widetilde\psi = \widetilde\psi^{\Delta t,+}_L$,
on noting 
(\ref{inidata}), (\ref{fvkbd}),
(\ref{Korn}) and
\eqref{idatabd} that, with $t = t_n$,
\begin{align}
& \int_{\Omega} \rho_{L}^{\Delta t, +}(t) \,|\uta^{\Delta t,+}(t)|^2 \dx +
\frac{\rho_{\rm min}}{\Delta t}\int_0^t
\|\uta^{\Delta t,+} - \uta^{\Delta t,-}\|^2 \dd s
\nonumber \\
& \hspace{1.5cm}
+ \int_0^t \int_{\Omega} \mu(\rho_L^{\Delta t,+})\,|\Dtt(\uta^{\Delta t,+})|^2 \dx \dd s \nonumber\\
&\qquad \leq \int_{\Omega} \rho_{0}\,|\ut_0|^2 \dx
+\frac{\rho_{\rm max}^2 C_{\varkappa}^2}{\mu_{\rm min}\,c_0}\,\,
\int_0^t\|\ft^{\Delta t,+}\|^2_{L^\varkappa(\Omega)} \dd s \nonumber\\
& \hspace{1.5cm} -2k \int_0^t \int_{\Omega \times D} M(\qt)\,\sum_{i=1}^K
\,U_i'\big({\textstyle \frac{1}{2}}|\qt_i|^2\big)\,\zeta(\rho^{\Delta t,+}_L)\,\widetilde\psi^{\Delta t,+}_L
\,\qt_i\, \qt_i^{\rm T}:
\nabxtt \uta^{\Delta t,+} \dq \dx \dd s,~~~~ \label{eq:energy-u}
\end{align}
where $\|\cdot\|$ denotes the $L^2$ norm over $\Omega$.
We intentionally {\em did not} bound the final term on
the right-hand side of \eqref{eq:energy-u}.
As we shall see in what follows,
this simple trick will prove helpful: our bounds on $\psia^{\Delta t, +}$ below will furnish
an identical term with the opposite sign, so then by
combining the bounds on $\uta^{\Delta t, +}$ and $\psia^{\Delta t, +}$
this pair of, otherwise dangerous, terms will be removed. This fortuitous cancellation reflects
the balance of total energy in the system.

Having dealt with $\uta^{\Delta t,+}$, we now embark on the less
straightforward task of deriving bounds on
norms of $\psia^{\Delta t,+}$ that are uniform in the cut-off parameter $L$.
The appropriate choice of test function in \eqref{eqpsincon} for this purpose
is $\varphi = \chi_{[0,t]}\,(\mathcal{F}^L)'(\psia^{\Delta t,+})$ with $t=t_n$, $n \in \{1,\dots,N\}$;
this can be seen by noting that with such a $\varphi$, at
least formally, the final term on the left-hand side of
\eqref{eqpsincon} can be manipulated to become identical to
the final term in \eqref{eq:energy-u}, but with the opposite sign.
While Lemma \ref{conv} guarantees that $\psia^{\Delta t,+}(t)$ belongs to $Z_2$ for all
$t \in [0,T]$, and is therefore nonnegative a.e. on $\Omega \times D \times [0,T]$, there is unfortunately
no reason why $\psia^{\Delta t,+}$ should be strictly positive on $\Omega\times D \times [0,T]$,
and therefore the
expression $(\mathcal{F}^L)'(\psia^{\Delta t,+})$ may in general
be undefined; the same is true of $(\mathcal{F}^L)''(\psia^{\Delta t,+})$,
which also appears in the algebraic manipulations.
We shall circumvent this problem by working
with $(\mathcal{F}^L)'(\psia^{\Delta t,+} + \alpha)$ instead of
$(\mathcal{F}^L)'(\psia^{\Delta t,+})$, where $\alpha>0$; since
$\psia^{\Delta t,+}$ is known to be nonnegative from Lemma \ref{conv},
$(\mathcal{F}^L)'(\psia^{\Delta t,+} + \alpha)$ and $(\mathcal{F}^L)''(\psia^{\Delta t,+} + \alpha)$
are well-defined.
After deriving the relevant bounds, which will involve
$\mathcal{F}^L(\psia^{\Delta t,+} + \alpha)$ only,
we shall pass to the limit $\alpha \rightarrow 0_{+}$, noting that, unlike
$(\mathcal{F}^L)'(\psia^{\Delta t,+})$ and $(\mathcal{F}^L)''(\psia^{\Delta t,+})$, the function
$(\mathcal{F}^L)(\psia^{\Delta t,+})$ is well-defined for any {\em nonnegative} $\psia^{\Delta t,+}$.
Thus, the core of the idea is to take any $\alpha \in (0,1)$, whereby $0 < \alpha < 1 < L$, and choose
$\varphi = \chi_{[0,t]}\,(\mathcal{F}^L)'(\psia^{\Delta t,+} + \alpha)$,
with $t = t_n$, $\;n \in \{1,\dots,N\}$, as test function in \eqref{eqpsincon}, and then pass to the
limit $\alpha \rightarrow 0_+$. An equivalent but slightly more transparent approach is to start from
\eqref{psiG} with the indices $n$ and $n-1$ in \eqref{psiG} replaced by $k$ and $k-1$, respectively,
choose $\varphi = (\mathcal{F}^L)'(\widetilde\psi^k_L + \alpha)$ as test function, sum the resulting
expressions through $k=1,\dots, n$, with $n \in \{1,\dots,N\}$, and then pass to the
limit $\alpha \rightarrow 0_+$. For reasons of clarity, we shall adopt the latter approach.

Thus, for $k = 1, \dots, n$ and $n \in \{1, \dots, N\}$, we arrive at the following identity
\begin{align}\label{z-terms}
&\int_{\Omega \times D} M\, \frac{\zeta(\rho^k_L)\,\widetilde \psi^k_L - \zeta(\rho^{k-1}_L)\, \widetilde
\psi^{k-1}_L}{\Delta t} \, [(\mathcal{F}^L)'(\widetilde\psi^k_L + \alpha)]\dq \dx\nonumber\\
&\quad- \int_{\Omega \times D} M \left(\frac{1}{\Delta t}\int_{t_{k-1}}^{t_k} \zeta(\rho^{[\Delta t]}_L) \dd t\right)
\ut^{k-1}_L \cdot \left(\nabx [(\mathcal{F}^L)'(\widetilde\psi^k_L + \alpha)]\right)\,\widetilde\psi^k_L\dq \dx\nonumber\\
&\quad + \int_{\Omega\times D} \sum_{i=1}^K \left[\frac{1}{4\lambda}\, \sum_{j=1}^K A_{ij}\, M\, \nabqj\widetilde\psi^k_L - [ \sigtt(\ut^k_L) \, \qt_i]\,M\, \zeta(\rho^k_L)\, \beta^L(\widetilde\psi^k_L) \right]\cdot \nabqi [(\mathcal{F}^L)'(\widetilde\psi^k_L + \alpha)]\dq \dx\nonumber\\
&\quad + \int_{\Omega \times D} \varepsilon\, M\, \nabx \widetilde\psi^k_L \cdot \nabx [(\mathcal{F}^L)'(\widetilde\psi^k_L + \alpha)]
\dq \dx = 0.\nonumber\\
\end{align}
We shall manipulate each of the terms on the left-hand side of \eqref{z-terms}. We begin by considering
\[ {\tt T}_1 := \int_{\Omega \times D} M\, \frac{\zeta(\rho^k_L)\,\widetilde \psi^k_L - \zeta(\rho^{k-1}_L)\, \widetilde
\psi^{k-1}_L}{\Delta t} \,
[(\mathcal{F}^L)'(\widetilde\psi^k_L + \alpha)
]\dq \dx\]
and
\[ {\tt T}_2 := - \int_{\Omega \times D} M \left(\frac{1}{\Delta t}\int_{t_{k-1}}^{t_k} \zeta(\rho^{[\Delta t]}_L) \dd t\right)
\ut^{k-1}_L \cdot \left(\nabx [(\mathcal{F}^L)'(\widetilde\psi^k_L + \alpha)]\right)\,\widetilde\psi^k_L\dq \dx\]
in tandem. By noting that, thanks to \eqref{eq:FL2a},
\begin{align*}\left(\nabx [(\mathcal{F}^L)'(\widetilde\psi^k_L + \alpha)]\right) \widetilde\psi^k_L &=
\widetilde\psi^k_L\,[\beta^L(\widetilde\psi^k_L + \alpha)]^{-1}\, \nabx \widetilde\psi^k_L\\
&= \left[ (\widetilde\psi^k_L+ \alpha) -\alpha\right]\,
[\beta^L(\widetilde\psi^k_L + \alpha)]^{-1}\, \nabx (\widetilde\psi^k_L + \alpha)
\\
&= \nabx \left[{\mathcal G}^L(\widetilde\psi^k_L + \alpha) - \alpha\, [({\mathcal F}^L)'(\widetilde\psi^k_L + \alpha)]\right],
\end{align*}
where
\[ {\mathcal G}^L(s) := \left\{\begin{array}{ll} s - 1, & s \leq L, \\
                                             \frac{s^2}{2L} + \frac{L}{2} - 1, & L \leq s,
                           \end{array} \right.\]
we have that
\[ {\tt T}_2 =  - \int_{\Omega} \left(\frac{1}{\Delta t}\int_{t_{k-1}}^{t_k} \zeta(\rho^{[\Delta t]}_L) \dd t\right)
\ut^{k-1}_L \cdot \nabx \left[\int_D M \left[{\mathcal G}^L(\widetilde\psi^k_L + \alpha) - \alpha\, [({\mathcal F}^L)'(\widetilde\psi^k_L + \alpha)]\right]\dq \right] \dx.\]
By applying \eqref{Gequn-trans0} with $\varphi = \int_D M [{\mathcal G}^L(\widetilde\psi^k_L + \alpha) - \alpha\, ({\mathcal F}^L)'(\widetilde\psi^k_L + \alpha)]\dq \in W^{1, \frac{q}{q-1}}(\Omega)$, where $q \in (2,\infty)$ when $d=2$ and $q \in [3,6]$ when $d=3$, we then have that
\[ {\tt T}_2 = - \int_{\Omega \times D} M \, \frac{\zeta(\rho^k_L) - \zeta(\rho^{k-1}_L)}{\Delta t} \left[{\mathcal G}^L(\widetilde\psi^k_L + \alpha) - \alpha\, [({\mathcal F}^L)'(\widetilde\psi^k_L + \alpha)] \right] \dq \dx.\]
We note in passing that the statement $\int_D M [{\mathcal G}^L(\widetilde\psi^k_L + \alpha) - \alpha\, ({\mathcal F}^L)'(\widetilde\psi^k_L + \alpha)]\dq \in W^{1, \frac{q}{q-1}}(\Omega)$ above follows, for all $\alpha \in (0,1)$, from
the following considerations. Since $\widetilde\psi^k_L \in X$, also $(\mathcal{F}^L)'(\widetilde\psi^k_L + \alpha)
\in X$, and hence $\int_D M\, (\mathcal{F}^L)'(\widetilde\psi^k_L + \alpha) \dq
\in H^1(\Omega) \subset W^{1,\frac{q}{q-1}}(\Omega)$ for all $q \in (2,\infty)$ when $d=2$ and $q \in [3,6]$ when $d=3$.
Furthermore, since $\widetilde\psi^k_L \in X$, also $\Gamma:=\int_D M {\mathcal G}^L(\widetilde\psi^k_L + \alpha) \dq \in L^1(\Omega)$
and, since $X \subset H^1(\Omega; L^2_M(D))$ and, by the Sobolev embedding \eqref{embed}, $H^1(\Omega; L^2_M(D))\subset L^q(\Omega; L^2_M(D))$ for the range of $q$ under consideration here, the definition of $\mathcal{G}^L$, a straightforward application of the Cauchy--Schwarz inequality to the integral over $D$ involved in the definition of $\Gamma$ and the application of H\"older's inequality to the integral over $\Omega$ involved in the definition of the $L^{\frac{q}{q-1}}(\Omega)$ norm, imply that $\nabx \Gamma \in L^{\frac{q}{q-1}}(\Omega)$. Finally, by a Gagliardo--Nirenberg inequality applied to
the function $\Gamma- \dashint_\Omega \Gamma \dx$ (cf. (2.9) and (2.10) on p.45 in \cite{LU-English}; or inequality (2.19)
on p.73 of \cite{LU-Russian} in conjunction with Poincar\'e's inequality in the $L^{\frac{q}{q-1}}(\Omega)$ norm; or Theorem 2.2
and Remark 2.1 in \cite{LSU-English}, where the proof of the Gagliardo--Nirenberg inequality can also be found):
\[ \bigg\|\,\Gamma - \dashint_\Omega \Gamma \dx \,\bigg\|_{L^{\frac{q}{q-1}}(\Omega)} \leq C(q,d)\, \|\Gamma\|_{L^1(\Omega)}^{\frac{q}{d+q}}\, \|\nabx \Gamma\|_{L^{\frac{q}{q-1}}(\Omega)}^{\frac{d}{d+q}}; \]
hence $\Gamma \in L^{\frac{q}{q-1}}(\Omega)$, which, together with $\nabx \Gamma \in L^{\frac{q}{q-1}}(\Omega)$, implies
that $\Gamma \in W^{1,\frac{q}{q-1}}(\Omega)$. Thus we deduce that
$\int_D M [{\mathcal G}^L(\widetilde\psi^k_L + \alpha) - \alpha\, ({\mathcal F}^L)'(\widetilde\psi^k_L + \alpha)]\dq \in W^{1, \frac{q}{q-1}}(\Omega)$, where $q \in (2,\infty)$ when $d=2$ and $q \in [3,6]$ when $d=3$, as was claimed above.

By rewriting ${\tt T}_1$ as
\begin{eqnarray*}
{\tt T}_1 &=& \int_{\Omega \times D} M\, \zeta(\rho^{k-1}_L)\, \frac{(\widetilde\psi^k_L+\alpha) - (\widetilde\psi^{k-1}_L+\alpha)}{\Delta t}
\,
[(\mathcal{F}^L)'(\widetilde\psi^k_L + \alpha)
] \dq \dx \\
&&+ \int_{\Omega \times D} M\, \frac{\zeta(\rho^k_L) - \zeta(\rho^{k-1}_L)}{\Delta t}\, \widetilde\psi^k_L \,
[({\mathcal F}^L)'(\widetilde\psi^k_L + \alpha)
]\dq \dx
\end{eqnarray*}
and adding this to the expression for ${\tt T}_2$ yields that
\begin{eqnarray*}
{\tt T}_1 + {\tt T}_2 &=& \int_{\Omega \times D} M\, \zeta(\rho^{k-1}_L)\, \frac{(\widetilde\psi^k_L+\alpha) - (\widetilde\psi^{k-1}_L+\alpha)}{\Delta t}
\,
[(\mathcal{F}^L)'(\widetilde\psi^k_L + \alpha)
] \dq \dx \\
&&+ \int_{\Omega \times D} M\, \frac{\zeta(\rho^k_L) - \zeta(\rho^{k-1}_L)}{\Delta t} \left[(\widetilde\psi^k_L + \alpha)[({\mathcal F}^L)'(\widetilde\psi^k_L + \alpha)] - {\mathcal G}^L(\widetilde\psi^k_L + \alpha)\right]\dq \dx\\
&=& \int_{\Omega \times D} M\, \frac{\zeta(\rho^{k}_L)(\widetilde\psi^k_L+\alpha) - \zeta(\rho^{k-1}_L)(\widetilde\psi^{k-1}_L+\alpha)}{\Delta t}
\,
[(\mathcal{F}^L)'(\widetilde\psi^k_L + \alpha)
] \dq \dx \\
&&- \int_{\Omega \times D} M\, \frac{\zeta(\rho^k_L) - \zeta(\rho^{k-1}_L)}{\Delta t} \,{\mathcal G}^L(\widetilde\psi^k_L + \alpha)\dq \dx.
\end{eqnarray*}
By applying part c) of Lemma \ref{basic} with $F(s) = \mathcal{F}^L(s)$, $G(s) = \mathcal{G}^L(s)$, $A = \zeta(\rho^k_L)$, $B=\zeta(\rho^{k-1}_L)$, $a = \widetilde\psi^k_L + \alpha$, $b = \widetilde\psi^{k-1}_L + \alpha$, noting that
$s\, (\mathcal{F}^L)'(s) - \mathcal{F}^L(s) - \mathcal{G}^L(s) = 0:= c_0$ for all $s \in (0,\infty)$, and
$\mbox{ess.inf}_{s>0}(\mathcal{F}^{L})''(s) = 1/L :=d_0$, it follows that
\begin{eqnarray}\label{t1t2}
{\tt T}_1 + {\tt T}_2 &\geq & \frac{1}{\Delta t}\left[\int_{\Omega \times D} M\, \zeta(\rho^k_L)\, \mathcal{F}^L(\widetilde\psi^k_L+\alpha)\dq \dx - \int_{\Omega \times D} M\, \zeta(\rho^{k-1}_L)\, \mathcal{F}^L(\widetilde\psi^{k-1}_L+\alpha)\dq \dx\right]\nonumber\\
&& +~ \frac{1}{2\Delta t\,L} \int_{\Omega \times D} M\, \zeta(\rho^{k-1}_L)\,(\widetilde\psi^k_L - \widetilde\psi^{k-1}_L)^2 \dq \dx.
\end{eqnarray}

We now move on to the next term in \eqref{z-terms}: thanks to \eqref{A}, we have that
\begin{alignat}{2}\label{t3}
{\tt T}_3\! :\!&= \frac{1}{4\lambda}\,\int_{\Omega \times D} \sum_{i=1}^K \sum_{j=1}^K A_{ij}\,M\,\nabqj \widetilde\psi^k_L \cdot
\nabqi [(\mathcal{F}^L)'(\widetilde \psi^k_L + \alpha)] \dq \dx
\nonumber \\
&= \frac{1}{4\lambda}\,\int_{\Omega \times D} \,M \, [(\mathcal{F}^L)''(\widetilde \psi^k_L + \alpha)] \, \sum_{i=1}^K \sum_{j=1}^K A_{ij}\,\nabqj \widetilde\psi^k_L \cdot
\nabqi \widetilde \psi^k_L \dq \dx
\nonumber \\
&\geq \frac{a_0}{4\lambda}\int_{\Omega \times D} M \,  [(\mathcal{F}^L)''(\widetilde \psi^k_L + \alpha)] \,
|\nabq \widetilde\psi^k_L|^2 \dq \dx.
\end{alignat}
It is tempting to bound $ [(\mathcal{F}^L)''(\widetilde \psi^k_L + \alpha)]$ below further by
$(\widetilde \psi^k_L + \alpha)^{-1}$ using \eqref{eq:FL2b}. We have refrained from doing so as the
precise form of \eqref{t3} will be required to absorb an extraneous term that the process of shifting
$\widetilde \psi^k_L$ by the addition of $\alpha>0$ generates in term ${\tt T}_5$ below (cf. the
last line in \eqref{t5a}). Once the extraneous term has been absorbed into the right-hand side of
\eqref{t3}, we shall apply inequality \eqref{eq:FL2b} to the resulting expression to bound it below further.

The next term that has to be dealt with, this time by a direct use of \eqref{eq:FL2b}, is
\begin{eqnarray}\label{t4}
{\tt T}_4:= \varepsilon \int_{\Omega \times D} M\, \nabx \widetilde \psi^k_L \cdot \nabx [({\mathcal F}^L)'(\widetilde{\psi}^k_L + \alpha)] \geq \varepsilon\int_{\Omega \times D} M \, \frac{|\nabx \widetilde\psi^k_L|^2}{\widetilde \psi^k_L + \alpha} \dq \dx.
\end{eqnarray}

It remains to consider the critical final term
\begin{align}\label{t5a}
{\tt T}_5 :\!&= - \int_{\Omega \times D} \sum_{i=1}^K\, [\sigtt(\ut^k_L)\,\qt_i]\, M\, \zeta(\rho^k_L)\, \beta^L(\widetilde\psi^k_L)\cdot
\nabqi [(\mathcal{F}^L)'(\widetilde \psi^k_L + \alpha)] \dq \dx \nonumber\\
&= - \int_{\Omega \times D} \sum_{i=1}^K\, [\sigtt(\ut^k_L)\,\qt_i]\, M\, \zeta(\rho^k_L)\, \beta^L(\widetilde\psi^k_L+\alpha)\cdot
\nabqi [(\mathcal{F}^L)'(\widetilde \psi^k_L + \alpha)] \dq \dx \nonumber\\
&\qquad + \int_{\Omega \times D} \sum_{i=1}^K\, [\sigtt(\ut^k_L)\,\qt_i]\, M\, \zeta(\rho^k_L)\, [\beta^L(\widetilde\psi^k_L+\alpha)-\beta^L(\widetilde\psi^k_L)]\cdot
\nabqi [(\mathcal{F}^L)'(\widetilde \psi^k_L + \alpha)] \dq \dx \nonumber\\
&= - \int_\Omega \zeta(\rho^k_L)\, \left[\int_D M [(\nabxtt\ut^k_L)\, \qt_i] \cdot \nabqi\widetilde\psi^k_L \dq \right] \dx\nonumber\\
&\qquad + \int_{\Omega \times D} M\, \zeta(\rho^k_L)
\left[1 - \frac{\beta^L(\widetilde\psi^k_L)}{\beta^L(\widetilde\psi^k_L+\alpha)} \right] \sum_{i=1}^K\,[(\nabxtt \ut^k_L)\,\qt_i]
\cdot \nabqi \widetilde \psi^k_L \dq \dx.
\end{align}
Thus, by applying the integration-by-parts formula \eqref{intbyparts} to the expression in the square brackets in the penultimate line of \eqref{t5a}, we deduce that
\begin{align}\label{t5}
{\tt T}_5 &=
- \int_{\Omega\times D} M \, \sum_{i=1}^K U'(\textstyle{\frac{1}{2}}|\qt_i|^2)\, \zeta(\rho^k_L)\,
\widetilde\psi^k_L \, [(\qt_i\, \qt_i^{\tt T}):\nabxtt \ut^k_L] \dq\dx \nonumber\\
&\qquad + \int_{\Omega \times D} M\, \zeta(\rho^k_L)
\left[1 - \frac{\beta^L(\widetilde\psi^k_L)}{\beta^L(\widetilde\psi^k_L+\alpha)} \right] \sum_{i=1}^K\,[(\nabxtt \ut^k_L)\,\qt_i]
\cdot \nabqi \widetilde \psi^k_L \dq \dx.
\end{align}

By summing \eqref{t1t2}, \eqref{t3}, \eqref{t4} and \eqref{t5} we obtain
\begin{align}\label{t-sum-1}
&\frac{1}{\Delta t}\left[\int_{\Omega \times D} M\, \zeta(\rho^k_L)\, \mathcal{F}^L(\widetilde\psi^k_L+\alpha)\dq \dx - \int_{\Omega \times D} M\, \zeta(\rho^{k-1}_L)\, \mathcal{F}^L(\widetilde\psi^{k-1}_L+\alpha)\dq \dx\right]\nonumber\\
&\qquad
+ \frac{1}{2\Delta t\,L} \int_{\Omega \times D} M\, \zeta(\rho^{k-1}_L)\,(\widetilde\psi^k_L - \widetilde\psi^{k-1}_L)^2 \dq \dx
\nonumber\\
&\qquad + \frac{a_0}{4\lambda}\int_{\Omega \times D} M \, [(\mathcal{F}^L)''(\widetilde \psi^k_L + \alpha)]\, |\nabq \widetilde\psi^k_L|^2 \dq \dx 
+ \varepsilon\int_{\Omega \times D} M \, \frac{|\nabx \widetilde\psi^k_L|^2}{\widetilde \psi^k_L + \alpha} \dq \dx\nonumber\\
&\leq  \int_{\Omega\times D} M \, \sum_{i=1}^K U'(\textstyle{\frac{1}{2}}|\qt_i|^2)\, \zeta(\rho^k_L)\,
\widetilde\psi^k_L \, [(\qt_i\, \qt_i^{\rm T}):\nabxtt \ut^k_L] \dq\dx \nonumber\\
&\qquad - \int_{\Omega \times D} M\, \zeta(\rho^k_L)
\left[1 - \frac{\beta^L(\widetilde\psi^k_L)}{\beta^L(\widetilde\psi^k_L+\alpha)} \right] \sum_{i=1}^K\,[(\nabxtt \ut^k_L)\,\qt_i]
\cdot \nabqi \widetilde \psi^k_L \dq \dx.
\end{align}
As each term in \eqref{t-sum-1} can be seen as the value of a piecewise constant function on the interval $(t_{k-1},t_k)$, multiplication of \eqref{t-sum-1} by $\Delta t$ and summation over the indices $k=1,\dots, n$, where $n \in \{1, \dots, N\}$,
yields on noting that $\widetilde\psi^{\Delta t}_L(0) = \widetilde\psi^0 = \beta^L(\widetilde\psi^0)$,
for $t=t_n$, that

\begin{eqnarray}\label{eq:energy-psi-summ1}
&&\hspace{-3mm}\int_{\Omega \times D} M\, \zeta(\rho^{\Delta t,+}_L(t))\,\mathcal{F}^L(\psia^{\Delta t,+}(t) + \alpha) \dq \dx
\nonumber\\
&&\qquad+\,\frac{1}{2 \Delta t\, L}\int_0^t\int_{\Omega \times D}
M\,\zeta(\rho^{\Delta t,-}_L)\,(\psia^{\Delta t,+} - \psia^{\Delta t,-})^2 \dq \dx \dd s
\nonumber\\
&&\qquad +\,\frac{a_0}{4\,\lambda}  \int_0^t \int_{\Omega \times D}\,
M\,
[(\mathcal{F}^L)''(\psia^{\Delta t,+} + \alpha)]\,|\nabq \psia^{\Delta t,+} |^2 \,\dq \dx \dd s\nonumber\\
&&\qquad
+\, 
\varepsilon\, \int_0^t \int_{\Omega \times D}\, M\,
\frac{|\nabx \psia^{\Delta t,+}|^2}{\psia^{\Delta t,+} + \alpha} \dq \dx \dd s\nonumber\\
&&\hspace{-3mm}\leq \int_{\Omega \times D} M\, \zeta(\rho_0)\, \mathcal{F}^L(\beta^L(\widetilde\psi^0) + \alpha)
\dq \dx\nonumber\\
&&\qquad+ \int_0^t \int_{\Omega \times D} M\,\sum_{i=1}^K \qt_i\,\qt_i^{\rm T}
\,U_i'(\textstyle{\frac{1}{2}|\qt|^2})\,\zeta(\rho^{\Delta t,+}_L)\,\psia^{\Delta t,+} :
\nabxtt \uta^{\Delta t,+}  \dq \dx \dd s
\nonumber\\
&&\qquad- \int_0^t \int_{\Omega \times D} M\,\zeta(\rho^{\Delta t,+}_L)\left[1 -\frac{\beta^L(\psia^{\Delta t,+})}{\beta^L(\psia^{\Delta t,+}
+ \alpha)}\right] \sum_{i=1}^K
\left[(\nabxtt \uta^{\Delta t,+})\,\qt_i\right]
\cdot \nabqi \psia^{\Delta t,+}\, \dq \dx \dd s.
\end{eqnarray}
We refer to \cite[(4.14)--(4.18)]{BS2011-fene} for the details of a similar, but somewhat simpler, argument in the case of $\zeta \equiv 1$.
The denominator in the prefactor of the second integral motivates us to link $\Delta t$ to $L$
so that $\Delta t\, L = o(1)$
as $\Delta t \!\rightarrow\! 0_{+}$ (or, equivalently, $\Delta t = o(L^{-1})$ as
$L \rightarrow \infty$), in order to drive the integral multiplied by the prefactor to $0$ in
the limit of $L \rightarrow \infty$,
once the product of the two has been bounded above by a constant, independent of $L$.

Comparing \eqref{eq:energy-psi-summ1} with \eqref{eq:energy-u} we see that after multiplying
\eqref{eq:energy-psi-summ1} by $2k$ and adding the resulting inequality to \eqref{eq:energy-u}
the final term in \eqref{eq:energy-u} is cancelled by $2k$ times the
second term on the right-hand side of \eqref{eq:energy-psi-summ1}. Hence, for any $t=t_n$, with
$n \in \{1,\dots,N\}$, we deduce that
\begin{eqnarray}\label{eq:energy-u+psi}
&&\hspace{-2mm}\int_{\Omega} \rho_L^{\Delta t, +}(t)\,|\uta^{\Delta t, +}(t)|^2 \dx
+ \frac{\rho_{\rm min}}{\Delta t}
\int_0^t \|\uta^{\Delta t, +} - \uta^{\Delta t,-}\|^2
\dd s + \int_0^t \int_{\Omega} \mu(\rho_L^{\Delta t,+})\,|\Dtt(\uta^{\Delta t, +})|^2 \dx \dd s\nonumber\\
&&\hspace{-1mm}+\,2k\int_{\Omega \times D} M\,\zeta(\rho^{\Delta t,+}_L(t))\,
\mathcal{F}^L(\psia^{\Delta t, +}(t) + \alpha) \dq \dx \nonumber\\
&&\hspace{7mm} + \frac{\zeta_{\rm min}\,k}{\Delta t\, L}
\int_0^t \int_{\Omega \times D} M\, (\psia^{\Delta t, +} - \psia^{\Delta t, -})^2
\dq \dx \dd s
\nonumber \\
&&\hspace{15mm}\quad +\, 
2k\,\varepsilon\,
\int_0^t \int_{\Omega \times D} \, M\,
\frac{|\nabx \psia^{\Delta t, +} |^2}{\psia^{\Delta t, +} + \alpha} \dq \dx \dd s\nonumber\\
&&\hspace{23mm}\quad\quad + \frac{a_0\, k}{2\,\lambda}\,
\int_0^t \int_{\Omega \times D} \,M\,
(\mathcal{F}^L)''(\psia^{\Delta t, +} + \alpha)\,|\nabq \psia^{\Delta t, +} |^2 \,\dq \dx \dd s\nonumber
\\
&&\hspace{-2mm}\leq \int_{\Omega} \rho_{0}\,|\ut_0|^2 \dx +
\frac{\rho_{\rm max}^2 \,C_{\varkappa}^2}{\mu_{\rm min}\,c_0}\,
\int_0^t
\|\ft^{\Delta t,+}\|^2_{L^\varkappa(\Omega)} \dd s + 2k \int_{\Omega \times D}
M\, \zeta(\rho_0)\,\mathcal{F}^L(\beta^L(\widetilde\psi^0) + \alpha) \dq \dx \nonumber
\\
&&\quad  -\, 2k \int_0^t \int_{\Omega \times D}\! M\,\zeta(\rho^{\Delta t,+}_L)\,
\sum_{i=1}^K \left[(\nabxtt \uta^{\Delta t, +})\,\qt_i\right]
\left[1 -\frac{\beta^L(\psia^{\Delta t, +})}{\beta^L(\psia^{\Delta t, +}
+ \alpha)}\right]\!\cdot \nabqi \psia^{\Delta t, +} \dq \dx \dd s.\nonumber\\
\end{eqnarray}
Let $\bt := (b_1,\ldots,b_K)$, recall (\ref{inidata}), and $b :=|\bt|_1 := b_1 +\ldots + b_K$,
then 
we can bound the magnitude of the last term on the right-hand side of \eqref{eq:energy-u+psi} by
\begin{align}
&\frac{a_0\, k}{4 \lambda} \left(\int_0^t \int_{\Omega \times D}
M (\mathcal{F}^L)''(\psia^{\Delta t, +} + \alpha)
\, |\nabq \psia^{\Delta t, +}|^2\, \dq \dx \dd s\right)
\nonumber \\
&\hspace{2.5in}
+ \alpha\,\frac{4\lambda\, k\, b\, \zeta_{\max}^2}{a_0}
\left(\int_0^t \int_{\Omega}  |\nabxtt \uta^{\Delta t, +}|^2 \dx \dd s\right),
\label{dragbd}
\end{align}
%
%
see \cite[(4.20)]{BS2010-hookean} for details in the case of $\zeta \equiv 1$.
Noting (\ref{dragbd}), (\ref{Korn}),
and using \eqref{eq:FL2b} to bound $(\mathcal{F}^L)''(\psia^{\Delta t, +} + \alpha)$
 from below by $\mathcal{F}''(\psia^{\Delta t, +} + \alpha)= (\psia^{\Delta t, +}
 + \alpha)^{-1}$ and
\eqref{eq:FL2c} to bound $\mathcal{F}^L(\psia^{\Delta t, +} + \alpha)$
by $\mathcal{F}(\psia^{\Delta t, +}+\alpha)$ from below yields, for all $t=t_n$,
$n \in \{1,\dots, N\}$, that
\begin{eqnarray}\label{eq:energy-u+psi1}
&&\hspace{-2mm}
\int_{\Omega}
\rho_{L}^{\Delta t, +}(t)\,|\uta^{\Delta t, +}(t)|^2 \dx + \frac{\rho_{\rm min}}{\Delta t}
\int_0^t \|\uta^{\Delta t, +} - \uta^{\Delta t,-}\|^2
\dd s + \int_0^t \int_{\Omega} \mu(\rho_L^{\Delta t,+})\,|\Dtt(\uta^{\Delta t, +})|^2 \dx \dd s\nonumber\\
&&\hspace{-2mm}+ \,2k\int_{\Omega \times D}\!\! M\, \zeta(\rho^{\Delta t,+}_L(t))\,
\mathcal{F}(\psia^{\Delta t, +}(t) + \alpha) \dq \dx + \frac{\zeta_{\rm min}\,k}{\Delta t\, L}
\int_0^t \int_{\Omega \times D}\!\! M\, (\psia^{\Delta t, +}
- \psia^{\Delta t, -})^2 \dq \dx \dd s
\nonumber \\
&&\hspace{-2mm}+\, 2k\,\varepsilon \int_0^t \int_{\Omega \times D}
\,M\,
\frac{|\nabx \psia^{\Delta t, +} |^2}{\psia^{\Delta t, +} + \alpha} \dq \dx \dd s
+\, \frac{a_0 \,k}{4\,\lambda}  \int_0^t \int_{\Omega \times D}
\,M\,
\frac{|\nabq \psia^{\Delta t, +}|^2}{\psia^{\Delta t, +} + \alpha} \,\dq \dx \dd s\nonumber\\
&&\hspace{4mm}\leq
\int_{\Omega} \rho_{0}\,|\ut_0|^2 \dx +
\,\frac{\rho_{\rm max}^2\,C_{\varkappa}^2}{\mu_{\rm min}\,c_0}\,
\int_0^t
\|\ft^{\Delta t,+}\|^2_{L^\varkappa(\Omega)} \dd s
\nonumber \\
&& \hspace{10mm}
+ \,2k \int_{\Omega \times D} M\, \zeta(\rho_0)\,\mathcal{F}^L(\beta^L(\widetilde\psi^0) + \alpha
) \dq \dx
+ \alpha\,\frac{4\lambda\, k\, b\, \zeta_{\rm max}^2}{a_0\,c_0} \int_0^t
\|\Dtt(\uta^{\Delta t, +})\|^2 \dd s.
\end{eqnarray}
Next we note that an analogous argument to the one that was used to derive \cite[(4.25)]{BS2011-fene} yields that
\begin{align}
\int_{\Omega \times D} M\, \zeta(\rho_0)\, \mathcal{F}^L(\beta^L(\widetilde\psi^0) + \alpha) \dq \dx
\leq
 \frac{3\alpha}{2}\, \zeta_{\rm max}\,|\Omega|
+ \int_{\Omega \times D} M\, \zeta(\rho_0)\, \mathcal{F}(\widetilde\psi^0 + \alpha) \dq \dx.
\label{FLbLbd}
\end{align}
The only restriction we have imposed on $\alpha$ so far is that it belongs to the open
interval $(0,1)$; let us now restrict the range of $\alpha$ further by demanding that, in fact,
\begin{equation}\label{alphacond}
0 < \alpha < \min \left(1 , \frac{\mu_{\rm min}\,a_0\,c_0}{\,4\lambda\,k\,b\,\zeta_{\rm max}^2}\right),
\end{equation}
where $c_0$ is the constant appearing in the Korn inequality (\ref{Korn}).
Then, the last term on the right-hand side of \eqref{eq:energy-u+psi1}
can be absorbed into the third term
on the left-hand side,
giving, on noting (\ref{Korn}) and (\ref{FLbLbd}), for $t=t_n$ and $n \in \{1,\dots, N\}$,
\begin{align}\label{eq:energy-u+psi2}
&
\int_{\Omega} \rho_{L}^{\Delta t,+}(t) \,|\uta^{\Delta t, +}(t)|^2 \dx
+ \frac{\rho_{\rm min}}{\Delta t}
\int_0^t \|\uta^{\Delta t, +} - \uta^{\Delta t,-}\|^2
\dd s
\nonumber \\
& \qquad +
\int_0^t \int_{\Omega}
\left(\mu(\rho^{\Delta t,+}_L) - \alpha\,\frac{4\lambda\,
 k\, b\, \zeta_{\rm max}^2}{a_0\,c_0}\right)\,
|\Dtt(\uta^{\Delta t, +})|^2 \dx \dd s\nonumber\\
&\qquad \qquad + \,2k\int_{\Omega \times D}\!\! M\,\zeta(\rho^{\Delta t,+}_L(t))\,
\mathcal{F}(\psia^{\Delta t, +}(t) + \alpha) \dq \dx \nonumber\\
&\qquad \qquad \qquad+ \frac{\zeta_{\rm min}\,k}{\Delta t\, L}
\int_0^t \int_{\Omega \times D}\!\! M\, (\psia^{\Delta t, +}
- \psia^{\Delta t, -})^2 \dq \dx \dd s
\nonumber \\
&\qquad \qquad \qquad \qquad + 2k\,\varepsilon\,\int_0^t \int_{\Omega \times D} \,
M\,
\frac{|\nabx \psia^{\Delta t, +} |^2}{\psia^{\Delta t, +} + \alpha} \dq \dx \dd s
\nonumber \\
& \qquad\qquad \qquad\qquad \qquad
+\, \frac{a_0\, k}{4\lambda}
\int_0^t \int_{\Omega \times D}\, M\,
\,\frac{|\nabq \psia^{\Delta t, +}|^2}
{\psia^{\Delta t, +} + \alpha} \,\dq \dx \dd s\nonumber\\
&\quad \leq \int_{\Omega}
\rho_{0}\,|\ut_0|^2 \dx +
\frac{\rho_{\rm max}^2\,C_{\varkappa}^2}{\mu_{\rm min}\,c_0}\,
\int_0^t
\|\ft^{\Delta t,+}\|^2_{L^\varkappa(\Omega)} \dd s
\nonumber \\
& \qquad \qquad
+ 3\,\alpha\,k\,\zeta_{\rm max}\,|\Omega|
+ 2k\, \int_{\Omega \times D} M\,\zeta(\rho_0)\, \mathcal{F}(\widetilde\psi^0 + \alpha) \dq \dx.
\end{align}
The key observation at this point is that the right-hand side of \eqref{eq:energy-u+psi2} is
completely independent of the cut-off parameter $L$.

On noting \cite[pp.\ 1243--44]{BS2011-fene},
we can pass to the limit $\alpha\rightarrow 0_+$ in
(\ref{eq:energy-u+psi2})
to obtain,
for all $t=t_n$, $n \in \{1,\dots, N\}$, that
\begin{subequations}
\begin{eqnarray}
&&\hspace{-2mm}
\int_{\Omega} \rho_{L}^{\Delta t, +}(t)\,|\uta^{\Delta t, +}(t)|^2 \dx
+ \frac{\rho_{\rm min}}{\Delta t}
\int_0^t \|\uta^{\Delta t, +} - \uta^{\Delta t,-}\|^2
\dd s + \int_0^t \int_{\Omega} \mu(\rho^{\Delta t,+}_L)\,|\Dtt(\uta^{\Delta t, +})|^2
\dx \dd s\nonumber\\
&&\hspace{4mm}+ \,2k\int_{\Omega \times D} M\,\zeta(\rho^{\Delta t,+}_L(t))\, \mathcal{F}(\psia^{\Delta t, +}(t))
\dq \dx + \frac{\zeta_{\rm min}\, k}{\Delta t\, L}
\int_0^t \int_{\Omega \times D}\!\! M\, (\psia^{\Delta t, +}
- \psia^{\Delta t, -})^2 \dq \dx \dd s
\nonumber \\
&&\hspace{6mm}\quad +\, 8k\,\varepsilon
\int_0^t \int_{\Omega \times D} M\,
|\nabx \sqrt{\psia^{\Delta t, +}} |^2 \dq \dx \dd s
\nonumber\\
&&\hspace{8mm} \quad\quad
+\, \frac{a_0 k}{\lambda} \, \int_0^t \int_{\Omega \times D}
M\,
|\nabq \sqrt{\psia^{\Delta t, +}}|^2 \,\dq \dx \dd s\nonumber\\
&&\hspace{10mm}\leq \int_{\Omega} \rho_{0}\,|\ut_0|^2 \dx +
\frac{\rho_{\rm max}^2\,C_{\varkappa}^2}
{\mu_{\rm min}\,c_0}\,
\int_0^t\|\ft^{\Delta t,+}
\|^2_{L^\varkappa(\Omega)} \dd s + 2k \int_{\Omega \times D} M\,\zeta(\rho_0)\, \mathcal{F}(\widetilde\psi^0)
\dq \dx\label{penultimate-line}\\
&&\hspace{10mm} \leq \int_{\Omega} \rho_{0}\,|\ut_0|^2 \dx
+
\frac{\rho_{\rm max}^2\,C_{\varkappa}^2}
{\mu_{\rm min}\,c_0}\,
\,\int_0^T\|\ft
\|^2_{L^\varkappa(\Omega)}
\dd s + 2k \int_{\Omega \times D} M\,\zeta(\rho_0)\, \mathcal{F}(\widetilde\psi_0) \dq \dx
\nonumber \\
&&\hspace{4.4in}
=:[{\sf B}(\ut_0,\ft, \widetilde\psi_0)]^2,
\label{eq:energy-u+psi-final2}
\end{eqnarray}
\end{subequations}
where, in the last line, we used \eqref{inidata-1} to bound the third term in \eqref{penultimate-line},
and that $t \in [0,T]$ together with the definition \eqref{fn}
of $\ft^{\Delta t,+}$ to bound the second term.

\subsection{$L$-independent bound on the time-derivative of $\rho^{[\Delta t]}_L$}
\label{sec:time-rho}

In this section we shall derive two $L$-independent bounds on $\partial \rho^{[\Delta t]}_L/\partial t$.
We begin by observing that
\begin{align}\label{rho-t-0}
&\,\int_0^T \left\langle \frac{\partial \rho^{[\Delta t]}_L}{\partial t}\,, \eta \right\rangle_{W^{1,\frac{q}{q-1}}(\Omega)}\dd t
= - \int_0^T \int_\Omega \rho^{[\Delta t]}_L \, \frac{\partial \eta}{\partial t} \dx \dd t\qquad
\forall \eta \in C^\infty_0(0,T;W^{1,\frac{q}{q-1}}(\Omega)),
\end{align}
which implies that $\partial \rho^{[\Delta t]}_L/\partial t$ coincides with the derivative  of
\[\rho^{[\Delta t]}_L
\in L^\infty(0,T;L^\infty(\Omega))\subset \mathcal{D}'(0,T;W^{1,\frac{q}{q-1}}(\Omega)')\]
with respect to $t$ in the sense of $W^{1,\frac{q}{q-1}}(\Omega)'$-valued distributions on $(0,T)$.

We deduce from \eqref{eqrhocon} that, for any $\eta \in L^2(0,T;W^{1,\frac{q}{q-1}}(\Omega))$,
\begin{align}\label{rho-t-1}
&\left|\,\int_0^T \left\langle \frac{\partial \rho^{[\Delta t]}_L}{\partial t}\,, \eta \right\rangle_{W^{1,\frac{q}{q-1}}(\Omega)}\dd t \,\right| \nonumber\\
&\qquad\qquad\qquad \leq
\|\rho^{[\Delta t]}_L\|_{L^\infty(0,T;L^\infty(\Omega))}\,\|\ut^{\Delta t,-}_L\|_{L^2(0,T;L^q(\Omega))}\,
\|\nabx \eta \|_{L^2(0,T;L^{\frac{q}{q-1}}(\Omega))},
\end{align}
where $q \in (2,\infty)$ when $d=2$ and $q \in [3,6]$ when $d=3$.

Next we shall show that the expression on the right-hand side of \eqref{rho-t-1} is bounded, independent of
$L$ and $\Delta t$. Indeed, we have by \eqref{rho-lower-upper} that $\rho^{[\Delta t]}_L \in \Upsilon$
and therefore $\|\rho^{[\Delta t]}_L\|_{L^\infty(0,T;L^\infty(\Omega))} \leq \rho_{\rm max}$.
Thus, by \eqref{Korn}, \eqref{inidata}  and \eqref{eq:energy-u+psi-final2}, we then have that
\begin{align}\label{u-t-2}
&\|\nabxtt\ut^{\Delta t,+}\|^2_{L^2(0,T;L^2(\Omega))}
\leq \frac{1}{c_0\, \mu_{\rm min}}\int_0^T \int_\Omega \mu(\rho^{\Delta t,+}_L)\, |\Dtt(\ut^{\Delta t,+})|^2 \dx \dd s
\leq \frac{1}{c_0\, \mu_{\rm min}}\,[{\sf B}(\ut_0, \ft, \widetilde\psi_0)]^2,\nonumber\\
\end{align}
and, by \eqref{idatabd},
\eqref{eq:energy-u+psi-final2} and \eqref{u-t-2}, we have
\begin{eqnarray}\label{u-t-3}
&&\|\nabxtt\ut^{\Delta t,-}_L\|^2_{L^2(0,T;L^2(\Omega))} = \Delta t\, \|\nabxtt\ut^0\|^2 + \int_{\Delta t}^T \|\nabxtt\ut^{\Delta t,-}\|^2 \dd s
\nonumber\\
&&\hspace{3.65cm}
= \Delta t\, \|\nabxtt\ut^0\|^2 + \int_0^{T-\Delta t} \|\nabxtt\ut^{\Delta t,+}\|^2 \dd s\nonumber\\
&&\hspace{3.65cm}
\leq  \int_\Omega \rho_0 |\ut_0|^2 \dd x + \|\nabxtt\ut^{\Delta t,+}\|^2_{L^2(0,T;L^2(\Omega))}\nonumber\\
&&\hspace{3.65cm}
\leq \left(1+ \frac{1}{c_0\, \mu_{\rm min}}\right) [{\sf B}(\ut_0, \ft, \widetilde\psi_0)]^2.
\end{eqnarray}
Thus we deduce from \eqref{u-t-2}, \eqref{u-t-3} and Poincar\'{e}'s inequality that $\|\ut^{\Delta t(,\pm)}_L\|_{L^2(0,T;H^1(\Omega))}
\leq C_\ast$, where $C_\ast$ is a positive constant that
depends solely on $\epsilon$,  $\rho_{\rm min}$, $\rho_{\rm max}$,
$\mu_{\rm \min}$, $\zeta_{\rm min}$, $\zeta_{\rm max}$, $T$, $|\Att|$, $a_0$, $c_0$, $C_\varkappa$, $k$,
$\lambda$, $K$ and $b$; in particular, $C_\ast$ is independent of $L$ and $\Delta t$.
Hence, by Sobolev's embedding theorem,
\begin{equation}\label{u-t-4}
\|\ut^{\Delta t(,\pm)}_L\|_{L^2(0,T;L^s(\Omega))} \leq C_\ast,
\qquad \mbox{for all}\; \left\{ \begin{array}{ll} s \in [1,\infty) & \mbox{if $d=2$},\\
                                                  s \in [1,6] & \mbox{if $d=3$},
                                \end{array}
                                \right.
\end{equation}
and hence also for all $s=q$, with $q$ as above, where
$C_\ast$ is independent of $L$ and $\Delta t$. Using these bounds, we deduce from \eqref{rho-t-1} that
\begin{eqnarray}\label{rho-t-3}
\left|\,\int_0^T \left\langle \frac{\partial \rho^{[\Delta t]}_L}{\partial t}\,, \eta \right\rangle_{W^{1,\frac{q}{q-1}}(\Omega)}\dd t \,\right| \leq C_\ast \|\nabx \eta \|_{L^2(0,T;L^{\frac{q}{q-1}}(\Omega))}\qquad \forall \eta \in L^2(0,T;W^{1,\frac{q}{q-1}}(\Omega)),\nonumber\\
\end{eqnarray}
which then implies that $\partial \rho^{[\Delta t]}_L/\partial t$ is bounded, independent of $L$ and $\Delta t$ in the function space
$L^2(0,T;W^{1,\frac{q}{q-1}}(\Omega)')$, where $q \in (2,\infty)$ when $d=2$ and $q \in [3,6]$ when $d=3$.

Note also that, for any $\eta \in L^1(0,T;H^1(\Omega))\subset L^1(0,T;W^{1,\frac{q}{q-1}}(\Omega))$, equation
\eqref{eqrhocon} yields
\begin{eqnarray}\label{rho-t-4}
\left|\,\int_0^T \left\langle \frac{\partial \rho^{[\Delta t]}_L}{\partial t}\,, \eta \right\rangle_{W^{1,\frac{q}{q-1}}(\Omega)}\dd t \,\right| \leq
\|\rho^{[\Delta t]}_L\|_{L^\infty(0,T;L^\infty(\Omega))}\,\|\ut^{\Delta t,-}_L\|_{L^\infty(0,T;L^2(\Omega))}\,
\|\nabx \eta \|_{L^1(0,T;L^2(\Omega))},\nonumber\\
\end{eqnarray}
where $q \in (2,\infty)$ when $d=2$ and $q \in [3,6]$ when $d=3$. Also, $\rho^{\Delta t,+}_L \geq \rho_{\rm min} (>0)$ a.e. on
$\Omega \times (0,T]$, which, together with \eqref{eq:energy-u+psi-final2},
implies that $\|u^{\Delta t,+}_L\|_{L^\infty(0,T;L^2(\Omega))} \leq C_\ast$, where $C_{\ast}$ denotes a positive constant that
is independent of $L$ and $\Delta t$. Since, by \eqref{proju0}, $\|u^0\|^2\leq \frac{\rho_{\rm max}}{\rho_{\rm min}}\|u_0\|^2$, it follows that
$\|u^{\Delta t,-}_L\|_{L^\infty(0,T;L^2(\Omega))} \leq C_\ast$, where
$C_\ast$ is a positive constant, independent of $L$ and $\Delta t$.
Thus, for any $\eta \in L^1(0,T;H^1(\Omega))\subset L^1(0,T;W^{1,\frac{q}{q-1}}(\Omega))$,
\begin{eqnarray}\label{rho-t-5}
\left|\,\int_0^T \left\langle \frac{\partial \rho^{[\Delta t]}_L}{\partial t}\,, \eta \right\rangle_{W^{1,\frac{q}{q-1}}(\Omega)}\dd t \,\right| \leq C_\ast\, \|\eta \|_{L^1(0,T;H^1(\Omega))},
\end{eqnarray}
where $q \in (2,\infty)$ when $d=2$ and $q \in [3,6]$ when $d=3$; i.e.,
$\partial \rho^{[\Delta t]}_L/\partial t$, when considered as an element of  $[L^1(0,T;H^1(\Omega))]' = L^\infty(0,T;H^1(\Omega)')$, is bounded, independent of $L$ and $\Delta t$.

In summary then, we have shown that $\partial \rho^{[\Delta t]}_L/\partial t \in \mathcal{D}'(0,T;W^{1,\frac{q}{q-1}}(\Omega)')$ is bounded, independent of $L$ and $\Delta t$:
\begin{itemize}
\item[(i)] when considered as an element of the function space $L^2(0,T;W^{1,\frac{q}{q-1}}(\Omega)')$,
where $q \in (2,\infty)$ when $d=2$ and $q \in [3,6]$ when $d=3$; and
\item[(ii)] when considered as an element of the function space $L^\infty(0,T;H^1(\Omega)')$.
\end{itemize}

\subsection{$L$-independent bound on the time-derivative of $\ut^{\Delta t}_L$}
\label{utsec}
We begin by deriving some preliminary $L$-independent bounds on the time-derivatives of $\ut_L^{\Delta t}$ and $(\rho_L \ut_L)^{\Delta t}$.
On noting from \eqref{ulin} and \eqref{upm} that
\[ \ut^{\Delta t}_L = \frac{t-t_{n-1}}{\Delta t}\, \ut^{\Delta t,+}_L + \frac{t_n-t}{\Delta t}\, \ut^{\Delta t,-}_L, \quad t \in (t_{n-1},t_n],\quad
n=1,\dots,N,\]
an elementary calculation yields that
\begin{eqnarray}\label{u-t-1}
&&\frac{1}{6}\int_0^T \left(\|\nabxtt\ut^{\Delta t,+}_L\|^2+ \|\nabxtt\ut^{\Delta t,-}_L\|^2\right) \dd s
\leq \int_0^T \|\nabxtt\ut^{\Delta t}_L\|^2 \dd s \nonumber\\
&&\hspace{6.19cm} \leq \frac{1}{2}\int_0^T\left(\|\nabxtt\ut^{\Delta t,+}_L\|^2
+ \|\nabxtt\ut^{\Delta t,-}_L\|^2\right)\!\dd s.
\end{eqnarray}

Note that, by \eqref{rho-lower-upper}, $\rho_{\rm min} \leq \rho^{\Delta t(,\pm)}_L \leq \rho_{\rm max}$ a.e. on $\Omega \times [0,T]$ for all
$\Delta t$ and $L$. Hence we have that $\|\rho^{\Delta t(,\pm)}_L\|_{L^\infty(0,T;L^\infty(\Omega))} \leq \rho_{\rm max}$. As
\[ (\rho_L\ut_L)^{\Delta t}(\cdot,t)  =
\rho^n_L(\cdot) \left[\frac{t-t_{n-1}}{\Delta t} \ut^n_L(\cdot) + \frac{t_n  - t}{\Delta t} \ut^{n-1}_L(\cdot) \right]
+ \frac{t_n - t}{\Delta t}\left(\rho^{n-1}_L(\cdot) - \rho^n_L(\cdot)\right) \ut^{n-1}_L(\cdot)\]
for all $t \in [t_{n-1},t_n]$ and $n=1,\dots,N$, which in turn implies that
\begin{eqnarray*}
\|(\rho_L\ut_L)^{\Delta t}(\cdot,t)\|_{L^s(\Omega)} &\leq &\|\rho^{\Delta t,+}_L(\cdot,t)\|_{L^\infty(\Omega)}\, \|\ut^{\Delta t}_L(\cdot,t)\|_{L^s(\Omega)}
\nonumber\\
&& + \left(\|\rho^{\Delta t,-}_L(\cdot,t)\|_{L^\infty(\Omega)} + \|\rho^{\Delta t,+}_L(\cdot,t)\|_{L^\infty(\Omega)}\right) \|\ut^{\Delta t,-}_L(\cdot,t)\|_{L^s(\Omega)}
\nonumber\\
&\leq & \rho_{\rm max} \|\ut^{\Delta t}_L(\cdot,t)\|_{L^s(\Omega)} + 2 \rho_{\rm max} \|\ut^{\Delta t,-}_L(\cdot,t)\|_{L^s(\Omega)}
\end{eqnarray*}
for all $t \in (t_{n-1},t_n]$ and $n=1,\dots, N$. By squaring both sides, using the algebraic-geometric mean inequality
and integrating the resulting inequality over $t \in [0,T]$, we deduce on noting \eqref{u-t-4} that
\begin{equation}
\label{u-t-5}
\|(\rho_L \ut_L)^{\Delta t}\|_{L^2(0,T;L^s(\Omega))} \leq C_\ast,\qquad \mbox{for }
\left\{\begin{array}{ll} s \in [1,\infty) & \mbox{if $d=2$},\\
s \in [1,6] & \mbox{if $d=3$},\end{array}\right.
\end{equation}
where $C_\ast$ is a positive constant, independent of $L$ and $\Delta t$;
\eqref{u-t-4} and \eqref{u-t-5} then imply that
\begin{eqnarray*}
\frac{\partial u^{\Delta t}_L}{\partial t} \; \mbox{and} \; \frac{\partial(\rho_L u_L)^{\Delta t}}{\partial t}
\; \mbox{are bounded in $H^{-1}(0,T; L^s(\Omega))$, where }
\left\{\begin{array}{ll} s \in [1,\infty) & \mbox{if $d=2$},\\
s \in [1,6] & \mbox{if $d=3$},\end{array}\right.\nonumber
\end{eqnarray*}
independent of $L$ and $\Delta t$.

We shall use \eqref{equncon} to improve the bound \eqref{u-t-5}. To this end we first note that using \eqref{eqrhocon} in \eqref{equncon} yields, for all $\wt \in L^1(0,T;\Vt)$, that
\begin{align}
&\displaystyle\int_{0}^{T}\!\! \int_\Omega \left[
\frac{\partial}{\partial t} (\rho_L\,\ut_L)^{\Delta t}\cdot \wt
- \tfrac{1}{2}
\rho_L^{[\Delta t]} \, \ut_L^{\Delta t, -}\, \cdot \nabla_x (\ut_L^{\Delta t,+}\cdot
\wt) \right]
 \dx \dt
\nonumber\\
& \qquad+
\displaystyle\int_{0}^{T}\!\! \int_\Omega
\mu(\rho^{\Delta t, +}_L) \,\Dtt(\utaeDp)
:
\Dtt(\wt) \dx \dt
\nonumber \\
&\qquad\qquad+
\tfrac{1}{2} \int_{0}^T\!\! \int_{\Omega}
\rho^{\{\Delta t\}}_L \left[ \left[ (\utaeDm \cdot \nabx) \utaeDp \right]\cdot\,\wt
- \left[ (\utaeDm \cdot \nabx) \wt  \right]\cdot\,\utaeDp
\right]\!\dx \dt
\nonumber
\\
&\qquad\qquad=\int_{0}^T
\left[ \int_{\Omega}  \rho_L^{\Delta t,+}
\ft^{\Delta t,+} \cdot \wt \dx 
- k\,\sum_{i=1}^K \int_{\Omega}
\Ctt_i(M\,\zeta(\rho^{\Delta t,+}_L)\,\hpsiaet^{\Delta t,+}): \nabxtt
\wt \dx \right] \dt.
\end{align}
It follows that, for any $t', t'' \in [0,T]$
such that $0\leq t' < t'' \leq T$ and choosing $\wt = \chi_{[t',t'']} \vt$, with $\vt \in \Vt$, where, as in the discussion following \eqref{crhonLd-ini}, $\chi_{[t',t'']}$ denotes the characteristic function of the interval $[t',t'']$, we have
\begin{align*}
&\displaystyle\int_{t'}^{t''}\!\! \int_\Omega \left[
\frac{\partial}{\partial t} (\rho_L\,\ut_L)^{\Delta t}\cdot \vt
- \tfrac{1}{2}
\rho_L^{[\Delta t]} \, \ut_L^{\Delta t, -}\, \cdot \nabla_x (\ut_L^{\Delta t,+}\cdot
\vt) \right]
 \dx \dt
\nonumber\\
& \quad+
\displaystyle\int_{t'}^{t''}\!\! \int_\Omega
\mu(\rho^{\Delta t, +}_L) \,\Dtt(\utaeDp)
:
\Dtt(\vt) \dx \dt
\nonumber \\
&\quad\quad+
\tfrac{1}{2} \int_{t'}^{t''}\!\! \int_{\Omega}
\rho^{\{\Delta t\}}_L \left[ \left[ (\utaeDm \cdot \nabx) \utaeDp \right]\cdot\,\vt
- \left[ (\utaeDm \cdot \nabx) \vt  \right]\cdot\,\utaeDp
\right]\!\dx \dt
\nonumber
\\
&\quad\quad\quad=\int_{t'}^{t''}
\left[ \int_{\Omega}  \rho_L^{\Delta t,+}
\ft^{\Delta t,+} \cdot \vt \dx 
- k\,\sum_{i=1}^K \int_{\Omega}
\Ctt_i(M\,\zeta(\rho^{\Delta t,+}_L)\,\hpsiaet^{\Delta t,+}): \nabxtt
\vt \dx \right]\! \dt
\quad\forall \vt \in \Vt,\nonumber
\end{align*}
and hence, equivalently (cf. \eqref{zeta-time-average} for the definition of $\rho_L^{\{\Delta t\}})$,
\begin{align}\label{u-terms}
&\int_\Omega \left[(\rho_L\,\ut_L)^{\Delta t}(t'') - (\rho_L\,\ut_L)^{\Delta t}(t')\right]\cdot \vt \dx
= -\displaystyle\int_{t'}^{t''}\!\! \int_\Omega
\mu(\rho^{\Delta t, +}_L) \,\Dtt(\utaeDp)
:
\Dtt(\vt) \dx \dt
\nonumber \\
&\quad- \tfrac{1}{2} \int_{t'}^{t''} \int_{\Omega}
(\rho^{\{\Delta t\}}_L - \rho^{[\Delta t]}_L)\left[ (\utaeDm \cdot \nabx) \utaeDp \right]\cdot\,\vt \dx \dt\nonumber\\
&\qquad+\tfrac{1}{2} \int_{t'}^{t''} \int_{\Omega}
(\rho^{\{\Delta t\}}_L + \rho^{[\Delta t]}_L)\left[ (\utaeDm \cdot \nabx) \vt  \right]\cdot\,\utaeDp
\!\dx \dt
\nonumber
\\
&\quad\quad\quad+\int_{t'}^{t''} \int_{\Omega}  \rho_L^{\Delta t,+}
\ft^{\Delta t,+} \cdot \vt \dx \dt  
- k\,\sum_{i=1}^K \int_{t'}^{t''} \int_{\Omega}
\Ctt_i(M\,\zeta(\rho^{\Delta t,+}_L)\,\hpsiaet^{\Delta t,+}): \nabxtt
\vt \dx \dt
\nonumber\\
&=: {\tt U}_1 + {\tt U}_2 + {\tt U}_3 + {\tt U}_4 + {\tt U}_5 \hspace{3.2in}\forall \vt \in \Vt.
\end{align}
We shall suppose that $t'' = t' + \delta$, where $\delta \in (0,T-t']$, and bound each of the terms ${\tt U_i}$, $i=1,\dots, 5$, in turn.
We note in particular that, by the definition \eqref{zeta-time-average} of $\rho^{\{\Delta t\}}_L$,
the term ${\tt U}_2=0$ when $t',t'' \in \{0=t_0, t_1, \dots, t_{N-1}, t_N=T\}$.

For ${\tt U}_1$, by the Cauchy--Schwarz inequality, we have that
\begin{align*}
|{\tt U}_1| &\leq  \mu_{\rm max} \left[\int_{t'}^{t'+\delta} \left(\int_\Omega |\Dtt(\ut^{\Delta t,+}_L(t))|^2\dx\right)^{\frac{1}{2}} \dt\right] \|\Dtt(\vt)\|\\
&\leq \mu_{\rm max}\,\delta^{\frac{1}{2}}\, \|\Dtt(\ut^{\Delta t,+}_L)\|_{L^2(t', t'+\delta;L^2(\Omega))}\,\|\Dtt(\vt)\|\qquad \forall \vt \in \Vt.
\end{align*}

Further, for any $q \in (2,\infty)$ when $d=2$ and any $q \in [3,6]$ when $d=3$, we have that
\begin{eqnarray*}
|{\tt U}_2| \leq \rho_{\rm max} \left[\int_{t'}^{t'+\delta} \|\utaeDm(t)\|_{L^{\frac{2q}{q-2}}(\Omega)} \|\nabx \utaeDp(t)\| \dt\right] \|\vt\|_{L^q(\Omega)}\qquad \forall \vt \in \Vt.
\end{eqnarray*}
The Gagliardo--Nirenberg inequality \eqref{eqinterp} and Korn's inequality \eqref{Korn}
imply that
\[ \|\ut^{\Delta t,-}_L(t)\|_{L^{\frac{2q}{q-2}}(\Omega)} \leq C(d,q,\Omega) \,\|\ut^{\Delta t,-}_L(t)\|^{1- \frac{d}{q}}\, \|\Dtt(\ut^{\Delta t,-}_L(t))\|^{\frac{d}{q}},\qquad t \in (0,T],\]
for all $q \in (2,\infty)$ when $d=2$ and $q \in [3,6]$ when $d=3$. Therefore, by Korn's inequality
\eqref{Korn} again, and by Sobolev embedding and H\"older's inequality, we have that
\begin{align*}
|{\tt U}_2| & \leq  C(d,q,\Omega)\,\rho_{\rm max}\, \|\ut^{\Delta t,-}_L\|^{1- \frac{d}{q}}_{L^\infty(0,T;L^2(\Omega))} \left[\int_{t'}^{t'+\delta} \|\Dtt(\utaeDm(t))\|^{\frac{d}{q}}\, \|\Dtt(\utaeDp(t))\| \dt\right] \|\Dtt(\vt)\|\\
& \leq C(d,q,\Omega)\,\rho_{\rm max}\, \|\ut^{\Delta t,-}_L\|^{1- \frac{d}{q}}_{L^\infty(0,T;L^2(\Omega))}\, \delta^{\frac{q-d}{2q}}\\
& \qquad\qquad \times \|\Dtt(\ut^{\Delta t,-}_L)\|_{L^2(t',t'+\delta;L^2(\Omega))}^{\frac{d}{q}} \,
\|\Dtt(\ut^{\Delta t,+}_L)\|_{L^2(t',t'+\delta;L^2(\Omega))}\, \|\Dtt(\vt)\|\qquad \forall \vt \in \Vt.
\end{align*}
Hence,
\begin{align*}
|{\tt U}_2| & \leq C(d,q,\Omega)\,\rho_{\rm max}\, \|\ut^{\Delta t,-}_L\|^{1- \frac{d}{q}}_{L^\infty(0,T;L^2(\Omega))}\,\|\Dtt(\ut^{\Delta t,-}_L)\|_{L^2(0,T;L^2(\Omega))}^{\frac{d}{q}} \\
& \qquad\qquad \times  \delta^{\frac{q-d}{2q}}\,
\|\Dtt(\ut^{\Delta t,+}_L)\|_{L^2(t',t'+\delta;L^2(\Omega))}\, \|\Dtt(\vt)\|\qquad\forall \vt \in \Vt,
\end{align*}
for all $q \in (2,\infty)$ when $d=2$ and all $q \in [3,6]$ when $d=3$.
An identical argument yields that
\begin{align*}
|{\tt U}_3| & \leq C(d,q,\Omega)\,\rho_{\rm max}\, \|\ut^{\Delta t,-}_L\|^{1- \frac{d}{q}}_{L^\infty(0,T;L^2(\Omega))}\,\|\Dtt(\ut^{\Delta t,-}_L)\|_{L^2(0,T;L^2(\Omega))}^{\frac{d}{q}} \\
& \qquad\qquad \times  \delta^{\frac{q-d}{2q}}\,
\|\Dtt(\ut^{\Delta t,+}_L)\|_{L^2(t',t'+\delta;L^2(\Omega))}\, \|\Dtt(\vt)\|\qquad \forall
\vt \in \Vt,
\end{align*}
for all $q \in (2,\infty)$ when $d=2$ and all $q \in [3,6]$ when $d=3$.

For ${\tt U}_4$, on noting \eqref{inidata}, \eqref{fvkbd}, Korn's inequality \eqref{Korn} and H\"older's inequality (with respect to the variable $t'$) yield that
\begin{align*}
|{\tt U}_4| \leq C(c_0,\varkappa,\Omega)\,\rho_{\rm max}\, \delta^{\frac{1}{2}}\, \|\ft^{\Delta t,+}\|_{L^2(t',t'+\delta;L^{\!\!\!\!~^{\varkappa}}(\Omega))}\,\|\Dtt(\vt)\|\qquad \forall \vt \in \Vt,
\end{align*}
with $\varkappa>1$ if $d=2$ and $\varkappa=\frac{6}{5}$ if $d=3$.

Finally, for term ${\tt U}_5$, from \eqref{eqCtt}, \eqref{intbyparts}, the Cauchy--Schwarz
inequality, we have that
\begin{align*}
|{\tt U}_5| &= \left|- k\,\sum_{i=1}^K \int_{t'}^{t'+\delta} \int_{\Omega\times D} M\,\zeta(\rho^{\Delta t,+}_L(t))\,\hpsiaet^{\Delta t,+}(t) \,U'_i(\tfrac{1}{2}|\qt_i|^2)\,\qt_i\, \qt_i^{\rm T}: \nabxtt
\vt \dq \dx \dt  \right|\\
& = \left|-k\,\int_{t'}^{t'+\delta} \int_{\Omega\times D} M\,\zeta(\rho^{\Delta t,+}_L(t))\,\left[\sum_{i=1}^K \nabqi\hpsiaet^{\Delta t,+}(t) \cdot (\nabxtt
\vt)\,\qt_i \right]\dq \dx \dt  \right|\\
& = \left|-2k\,\int_{t'}^{t'+\delta} \int_{\Omega} \zeta(\rho^{\Delta t,+}_L(t))\int_D \left[M\,\sqrt{\hpsiaet^{\Delta t,+}(t)} \sum_{i=1}^K \nabqi\sqrt{\hpsiaet^{\Delta t,+}(t)} \cdot (\nabxtt
\vt)\,\qt_i \right]\dq \dx \dt  \right|\\
& \leq 2k \,\zeta_{\rm max}\,\int_{t'}^{t'+\delta} \int_{\Omega} |\nabxtt \vt| \left[\int_D M\,|\qt|^2\,\hpsiaet^{\Delta t,+}(t) \dq\right]^{\frac{1}{2}} \left[\int_D M \left|\nabq\sqrt{\hpsiaet^{\Delta t,+}(t)}\right|^2 \dq\right]^{\frac{1}{2}} \dx \dt
\nonumber\\
& \leq 2k\,\zeta_{\rm max}\,\int_{t'}^{t'+\delta} \int_{\Omega} |\nabxtt \vt|  \left[\int_D M \left|\nabq\sqrt{\hpsiaet^{\Delta t,+}(t)}\right|^2 \dq\right]^{\frac{1}{2}} \dx \dt\\
&\hspace{7cm}\times \mbox{ess.sup}_{(x,t) \in \Omega \times (0,T)} \left[\int_D M\, |\qt|^2\, \widetilde\psi^{\Delta t,+}_L \dq \right]^{\frac{1}{2}} \\
&\leq 2k\, \zeta_{\rm max}\,\delta^{\frac{1}{2}}\,
\|\nabq\sqrt{\hpsiaet^{\Delta t,+}}\|_{L^2(t',t'+\delta;L^2_M(\Omega \times D))}
\,\|\nabxtt \vt\|\\
&\hspace{5cm}\times \mbox{ess.sup}_{(x,t) \in \Omega \times (0,T)} \left[\int_D M\, |\qt|^2\,
\widetilde\psi^{\Delta t,+}_L \dq \right]^{\frac{1}{2}} \qquad \forall \vt \in \Vt.
\end{align*}
Since $|\qt|^2 = |\qt_1|^2 + \cdots + |\qt_K|^2 <
b_1 + \cdots + b_K=: b$, the inequality \eqref{lambda-bound} immediately implies that the final factor
is bounded by $\sqrt{\omega\,b}$.
Korn's inequality \eqref{Korn} then implies that
\[ |{\tt U}_5| \leq C(k,\omega, b, \zeta_{\rm max}, c_0) \,\delta^{\frac{1}{2}}\,
\|\nabq\sqrt{\hpsiaet^{\Delta t,+}}\|_{L^2(t',t'+\delta;L^2_M(\Omega \times D))}
\, \|\Dtt(\vt)\|\qquad \forall \vt \in \Vt.\]

By collecting the upper bounds on the terms ${\tt U_i}$, $i=1, \dots, 5$, and noting
the upper bounds on the first and the third term on the left-hand side of \eqref{eq:energy-u+psi-final2}
in terms of $[{\sf B}(\ut_0,\ft,\widetilde\psi_0)]^2$,
together with the uniform lower bounds $\rho_L^{\Delta t,+}(\xt,t) \geq \rho_{\rm min}$ and
$\mu(\rho^{\Delta t,+}_L(\xt,t)) \geq \mu_{\rm min}$ for $(\xt,t) \in \Omega \times [0,T]$,
we thus have from \eqref{u-terms} that
\begin{align}\label{frac-u}
&\left|\,\int_\Omega \left[(\rho_L\,\ut_L)^{\Delta t}(t'') - (\rho_L\,\ut_L)^{\Delta t}(t')\right]\cdot \vt \dx\,\right|
\nonumber\\
&\qquad
\leq C\, \|\Dtt(\vt)\| \left((\delta^{\frac{1}{2}}+ \delta^{\frac{q-d}{2q}})\,\|\Dtt(u^{\Delta t,+}_L)\|_{L^2(t',t'+\delta;L^2(\Omega))}
+ \delta^{\frac{1}{2}}\|\ft^{\Delta t,+}\|_{L^2(t',t'+\delta;L^{\varkappa}(\Omega))}\right.\nonumber\\
&\qquad\qquad \left.
 + \delta^{\frac{1}{2}}\,
\|\nabq\sqrt{\hpsiaet^{\Delta t,+}}\|_{L^2(t',t'+\delta;L^2_M(\Omega \times D))}\right),
\end{align}
where $C$ is a positive constant, independent of $\Delta t$, $L$ and $\delta$; and $\delta=t''-t'\in (0,T-t']$.

In what follows, we shall suppose that $t', t'' \in \{0=t_0, t_1, \dots, t_{N-1}, t_N=T\}$ with $t' < t''$,
so that $\delta:=t''-t'$ is an integer multiple of $\Delta t$; then
\begin{align}\label{inter-t-t}
(\rho_L\,\ut_L)^{\Delta t}(t'') - (\rho_L\,\ut_L)^{\Delta t}(t') &= \rho_L^{[\Delta t]}(t'')\,
\ut^{\Delta t}_L(t'') - \rho_L^{[\Delta t]}(t')\,\ut_L^{\Delta t}(t')\nonumber\\
& = \rho_L^{[\Delta t]}(t'')\,[\ut^{\Delta t}_L(t'') - \ut^{\Delta t}_L(t')] + [\rho^{[\Delta t]}_L(t'') - \rho^{[\Delta t]}_L(t')]\, \ut^{\Delta t}_L(t').
\end{align}

By selecting $\eta = \chi_{[t',t'']}(t)\, (\ut^{\Delta t}_L(t') \cdot \vt)$ in \eqref{eqrhocon},
with $t'' = t'+\delta$, we have using Korn's inequality \eqref{Korn} that,
for any $\vt \in \Vt$,
\begin{align}\label{inter-t-t1}
&\left|\,\int_\Omega [\rho^{[\Delta t]}_L(t'+\delta) - \rho^{[\Delta t]}_L(t')]\, (\ut^{\Delta t}_L(t') \cdot \vt) \dx\,\right| \nonumber\\
&\qquad\leq \rho_{\rm max}\,\delta^{\frac{1}{2}}\,\|\ut^{\Delta t,-}_L\|_{L^2(t',t'+\delta;L^q(\Omega))}\, \|\nabx(\ut^{\Delta t}_L(t') \cdot \vt)\|_{L^{\frac{q}{q-1}}(\Omega)}\nonumber\\
&\leq \rho_{\rm max}\,\delta^{\frac{1}{2}}\,\|\ut^{\Delta t,-}_L\|_{L^2(t',t'+\delta;L^q(\Omega))}\, \left(\|\nabx\ut^{\Delta t}_L(t')\|\, \|\vt\|_{L^{\frac{2q}{q-2}}(\Omega)} + \|\nabx \vt\|\, \|\ut^{\Delta t}_L(t')\|_{L^{\frac{2q}{q-2}}(\Omega)}\right)\nonumber\\
&\qquad\leq C\,\delta^{\frac{1}{2}}\,\|\ut^{\Delta t,-}_L\|_{L^2(t',t'+\delta;L^q(\Omega))}\nonumber\\
&\qquad\qquad\times\left(\|\Dtt(\ut^{\Delta t}_L(t'))\|\, \|\vt\|^{\frac{q-4}{q-2}}\,\|\vt\|_{L^q(\Omega)}^{\frac{q}{q-2}} +
\|\Dtt(\vt)\|\, \|\ut^{\Delta t}_L(t')\|^{\frac{q-4}{q-2}}\, \|\ut^{\Delta t}_L(t')\|_{L^q(\Omega)}^{\frac{2}{q-2}}\right),
\end{align}
where $q \in [4,\infty)$ when $d=2$ and $q \in [4,6]$ when $d=3$, and $C=C(c_0,\rho_{\rm max})$, where $c_0$ is the constant in Korn's inequality \eqref{Korn}.

By dotting \eqref{inter-t-t} with $\vt \in \Vt$ integrating over $\Omega$ and substituting
\eqref{frac-u} and \eqref{inter-t-t1} into the resulting identity, we deduce that
\begin{align}\label{frac-v}
&\left|\,\int_\Omega \rho_L^{[\Delta t]}(t'+\delta)\,[\ut^{\Delta t}_L(t'+\delta) - \ut^{\Delta t}_L(t')]\cdot \vt \dx\,\right|\nonumber\\
&\qquad \leq C\,\delta^{\frac{1}{2}}\,\|\ut^{\Delta t,-}_L\|_{L^2(t',t'+\delta;L^q(\Omega))}\nonumber\\
&\qquad\qquad\qquad\times\left(\|\Dtt(\ut^{\Delta t}_L(t'))\|\, \|\vt\|^{\frac{q-4}{q-2}}\,\|\vt\|_{L^q(\Omega)}^{\frac{2}{q-2}} +
\|\Dtt(\vt)\|\, \|\ut^{\Delta t}_L(t')\|^{\frac{q-4}{q-2}}\, \|\ut^{\Delta t}_L(t')\|_{L^q(\Omega)}^{\frac{2}{q-2}}\right)\nonumber\\
&\qquad\qquad + C\, \|\Dtt(\vt)\| \left(\!(\delta^{\frac{1}{2}}+ \delta^{\frac{q-d}{2q}})\,\|\Dtt(\ut^{\Delta t,+}_L)
\|_{L^2(t',t'+\delta;L^2(\Omega))}+\delta^{\frac{1}{2}}\|\ft^{\Delta t,+}\|_{L^2(t',t'+\delta;L^{\varkappa}(\Omega))}\!\!\!\right.\nonumber\\
&\qquad\qquad\qquad \left. + ~\delta^{\frac{1}{2}}\,
\|\nabq\sqrt{\hpsiaet^{\Delta t,+}}\|_{L^2(t',t'+\delta;L^2_M(\Omega \times D))}\right),
\end{align}
for all $\vt \in \Vt$, where $q \in [4,\infty)$ when $d=2$ and $q \in [4,6]$ when $d=3$,
and $C$ is a positive constant, independent of $\Delta t$, $L$ and $\delta$. Here $\delta = \ell\,  \Delta t$, where $\ell=1,\dots, N-m$, and $t'=m\,\Delta t$ for $m=0,\dots, N-1$.
The symbol $\delta$ will be understood to have the same meaning throughout the rest of this section, unless otherwise stated.


We now select $\vt = \ut^{\Delta t}_L(t'+\delta) - \ut^{\Delta t}_L(t')$ in \eqref{frac-v},
sum the resulting collection of inequalities over $t' \in \{0,t_1, \dots, T-\delta\}$ (denoting the sum
over all $t'$ contained in this set by $\sum_{t'=0}^{T-\delta}$), and note the following obvious inequalities:
\begin{align*}
\|\ut^{\Delta t, -}_L\|_{L^2(t',t'+\delta;L^q(\Omega))} &= \|\ut^{\Delta t, -}_L\|^{\frac{2}{q-2}}_{L^2(t',t'+\delta;L^q(\Omega))} \|\ut^{\Delta t, -}_L\|^{\frac{q-4}{q-2}}_{L^2(t',t'+\delta;L^q(\Omega))}\\
& \leq C\,\|\ut^{\Delta t, -}_L\|^{\frac{2}{q-2}}_{L^2(0,T;L^q(\Omega))} \|\Dtt(\ut^{\Delta t, -}_L)\|^{\frac{q-4}{q-2}}_{L^2(t',t'+\delta;L^2(\Omega))},
\end{align*}
\[ \|\ut^{\Delta t}_L(t'+\delta) - \ut^{\Delta t}_L(t')\|^{\frac{q-4}{q-2}}
\leq \left( 2\,\|\ut^{\Delta t}_L\|_{L^\infty(0,T;L^2(\Omega))}\right)^{\frac{q-4}{q-2}},\]
\[\| \ut^{\Delta t}_L(t')\|^{\frac{q-4}{q-2}} \leq \|\ut^{\Delta t}_L\|^{\frac{q-4}{q-2}}_{L^\infty(0,T;L^2(\Omega))}\qquad\mbox{and}\qquad\| \ut^{\Delta t}_L(t'+\delta)\|^{\frac{q-4}{q-2}} \leq \|\ut^{\Delta t}_L\|^{\frac{q-4}{q-2}}_{L^\infty(0,T;L^2(\Omega))},\]
the first of which follows by Sobolev embedding and Korn's inequality \eqref{Korn}, to
deduce from \eqref{frac-v} that
\begin{align*}
&\Delta t\,\sum_{t'=0}^{T-\delta}\int_\Omega \rho_L^{[\Delta t]}(t'+\delta)\,|\ut^{\Delta t}_L(t'+\delta) - \ut^{\Delta t}_L(t')|^2 \dx \nonumber\\
&\qquad\qquad \leq C \,\delta^{\frac{1}{2}}\,
\|\ut^{\Delta t,-}_L\|_{L^2(0,T;L^q(\Omega))}^{\frac{2}{q-2}}\,\|\ut^{\Delta t}_L\|^{\frac{q-4}{q-2}}_{L^\infty(0,T;L^2(\Omega))}\nonumber\\
&\times\left(\Delta t\,\sum_{t'=0}^{T-\delta}\|\Dtt(\ut^{\Delta t,-}_L)\|_{L^2(t',t'+\delta;L^2(\Omega))}^{\frac{q-4}{q-2}}\|\Dtt(\ut^{\Delta t}_L(t'))\|\, \,\|\Dtt(\ut^{\Delta t}_L(t'+\delta)) - \Dtt(\ut^{\Delta t}_L(t'))\|^{\frac{2}{q-2}} \right.
\nonumber \\
&\qquad\qquad\left. + \Delta t\,\sum_{t'=0}^{T-\delta}
\|\Dtt(\ut^{\Delta t,-}_L)\|_{L^2(t',t'+\delta;L^2(\Omega))}^{\frac{q-4}{q-2}}\|\Dtt(\ut^{\Delta t}_L(t'+\delta)) - \Dtt(\ut^{\Delta t}_L(t'))\|\, \|\Dtt(\ut^{\Delta t}_L(t'))\|^{\frac{2}{q-2}}\right)\nonumber\\
&\qquad\qquad\qquad+ C\, \Delta t\,\sum_{t'=0}^{T-\delta}\|\Dtt(\ut^{\Delta t}_L(t'+\delta)) - \Dtt(\ut^{\Delta t}_L(t'))\| \bigg((\delta^{\frac{1}{2}}+ \delta^{\frac{q-d}{2q}})\,\|\Dtt(u^{\Delta t,+}_L)\|_{L^2(t',t'+\delta;L^2(\Omega))} \nonumber\\
&\qquad\qquad\qquad \qquad+ \delta^{\frac{1}{2}}\|\ft\|_{L^2(t',t'+\delta;L^{\varkappa}(\Omega))} + ~\delta^{\frac{1}{2}}\,
\|\nabq\sqrt{\hpsiaet^{\Delta t,+}}\|_{L^2(t',t'+\delta;L^2_M(\Omega \times D))}\bigg)\nonumber
\end{align*}
~\vspace{-5mm}
\begin{equation}
\hspace{-2.9cm}=: {\tt V}_1 \left({\tt V}_2 + {\tt V}_3\right) + C\,\Delta t\sum_{t'=0}^{T-\delta} {\tt V}_4(t')
\bigg({\tt V}_5(t') + {\tt V}_6(t') + {\tt V}_7(t')\bigg),
\label{intermediate}
\end{equation}
where $q \in [4,\infty)$ when $d=2$ and $q \in [4,6]$ when $d=3$, and $C$ is a positive
constant, independent of $\Delta t$, $L$ and $\delta$; ${\tt V}_1$ denotes the expression in the first line
on the right-hand side of \eqref{intermediate}; ${\tt V}_2$ and ${\tt V}_3$ denote the two
terms in the bracketed expression multiplied by ${\tt V}_1$; ${\tt V}_4(t')$ is the factor in front
of the bracket in the fourth line on the right-hand side of \eqref{intermediate}; and ${\tt V}_5(t')$,
${\tt V}_6(t')$ and ${\tt V}_7(t')$ are the three terms in the bracketed expression multiplied by ${\tt V_4}(t')$.
We shall consider each of the terms ${\tt V}_1, {\tt V}_2, {\tt V}_3, {\tt V}_4(t'),\dots, {\tt V}_7(t')$ separately. We begin by noting that
by \eqref{u-t-4} and by the definition of $\ut^{\Delta t}_L$ in conjunction with the uniform
bounds $\|\ut^{\Delta t, \pm}_L\|_{L^\infty(0,T;L^2(\Omega))} \leq C$ (cf. the discussion in the paragraph between \eqref{rho-t-4} and \eqref{rho-t-5}),
\begin{equation}\label{interm-1}
{\tt V}_1 = C\, \delta^{\frac{1}{2}}\,\|\ut^{\Delta t,-}_L\|_{L^2(0,T;L^q(\Omega))}^{\frac{2}{q-2}}\,\|\ut^{\Delta t}_L\|^{\frac{q-4}{q-2}}_{L^\infty(0,T;L^2(\Omega))} \leq C \delta^{\frac{1}{2}},
\end{equation}
where $q$ is as above, and $C$ is a positive constant, independent of $L$, $\Delta t$ and $\delta$. In the rest of this section we shall assume that $q \in (4,\infty)$ when $d=2$ and $q \in (4,6]$ when $d=3$.

Next, for the term ${\tt V}_2$, H\"older's inequality with respective exponents $2(q-2)/(q-4)$, $2$ and $q-2$
for the three factors under the summation sign in this term yields that
\begin{align*}
{\tt V}_2 &\leq \left(\Delta t\,\sum_{t'=0}^{T-\delta}\|\Dtt(\ut^{\Delta t,-}_L)\|_{L^2(t',t'+\delta;L^2(\Omega))}^2\right)^{\frac{q-4}{2(q-2)}}
\left(\Delta t\,\sum_{t'=0}^{T-\delta}  \|\Dtt(\ut^{\Delta t}_L(t'))\|^2\right)^{\frac{1}{2}}
\\
&\qquad \times \left(\Delta t\,\sum_{t'=0}^{T-\delta} \|\Dtt(\ut^{\Delta t}_L(t'+\delta)) - \Dtt(\ut^{\Delta t}_L(t'))\|^2\right)^{\frac{1}{q-2}}\\
&=\left(\Delta t\,\sum_{t'=0}^{T-\delta}\int_{t'}^{t'+\delta}\|\Dtt(\ut^{\Delta t,-}_L(s))\|^2 \dd s \right)^{\frac{q-4}{2(q-2)}}
\left(\Delta t\, \|\Dtt(\ut^0)\|^2 + \Delta t\,\sum_{t'=\Delta t}^{T-\delta}  \|\Dtt(\ut^{\Delta t}_L(t'))\|^2\right)^{\frac{1}{2}}
\\
&\qquad \times \left(\Delta t\,\sum_{t'=0}^{T-\delta} \|\Dtt(\ut^{\Delta t}_L(t'+\delta)) - \Dtt(\ut^{\Delta t}_L(t'))\|^2\right)^{\frac{1}{q-2}}\\
&=: {\tt V}_{21} \, {\tt V}_{22} \, {\tt V}_{23}.
\end{align*}
We shall consider the three factors on the right-hand side of this inequality separately.
For the first factor, we shall use the following elementary result, whose proof is omitted.

\begin{lemma}\label{summation} Suppose that $g \in L^1(0,T)$, $g \geq 0$, is a piecewise
constant, left-continuous function on the partition $\{0=t_0, t_1, \dots, t_{N-1}, t_N = T\}$
of the interval $[0,T]$ with step size $\Delta t = T/N$, where $N \in \mathbb{N}_{\geq 1}$; and
let $\delta = \ell \, \Delta t$, where $\ell \in \{1, \dots, N\}$ (i.e., $g(t_k) = g(t_k-)$ for
$k=1,\dots, N$). Then,
\[ \Delta t \, \sum_{t'=0}^{T-\delta} \int_{t'}^{t'+\delta} g(s) \dd s =
(\Delta t)^2\, \sum_{k=1}^{N-\ell+1} \sum_{s=k}^{k + \ell-1} g(t_k)
\leq \ell (\Delta t)^2 \sum_{k=1}^N g(t_k) = \delta \int_0^T g(s) \dd s,\]
where the $\leq$ sign can be replaced by an equality sign when $\ell = 1$ and $\ell = N$.
\end{lemma}

On applying Lemma \ref{summation} with $g: s \in (0,T] \mapsto \|\Dtt(\ut^{\Delta t,-}_L(s))\|^2$ in conjunction with
\eqref{idatabd} and \eqref{eq:energy-u+psi-final2} we deduce that
\begin{align*}
{\tt V}_{21} &\leq \left(\delta \int_0^T\|\Dtt(\ut^{\Delta t,-}_L(s))\|^2 \dd s\right)^{\frac{q-4}{2(q-2)}}
\leq  \delta^{\frac{q-4}{2(q-2)}} \left(\Delta t \|\Dtt(\ut^0)\|^2 + \int_0^T
\|\Dtt(\ut^{\Delta t,+}_L(s))\|^2 \dd s \right)^{\frac{q-4}{2(q-2)}}\\
& \leq  C\,\delta^{\frac{q-4}{2(q-2)}},
\end{align*}
where $C$ is a positive constant, independent of $L$, $\Delta t$ and $\delta$. Analogously, since $u^{\Delta t}_L$
and $u^{\Delta t,+}_L$ coincide at all points $t'=\ell\,\Delta t$, $\ell = 1, \dots, N$, we have, again by
\eqref{idatabd} and \eqref{eq:energy-u+psi-final2}, that
\begin{align*}
{\tt V}_{22} & =  \left(\Delta t\, \|\Dtt(\ut^0)\|^2 + \Delta t\,\sum_{t'=\Delta t}^{T-\delta}
\|\Dtt(\ut^{\Delta t,+}_L(t'))\|^2\right)^{\frac{1}{2}} \\
& =  \left(\Delta t\, \|\Dtt(\ut^0)\|^2
+ \Delta t\,\int_0^{T-\delta} \|\Dtt(\ut^{\Delta t,+}_L(s))\|^2 \dd s\right)^{\frac{1}{2}}
\leq  C,
\end{align*}
where $C$ is a positive constant, independent of $L$, $\Delta t$ and $\delta$.
For the term ${\tt V}_{23}$, we have by the triangle inequality, shifting indices in the summation, and noting,
once again, \eqref{idatabd} and \eqref{eq:energy-u+psi-final2}, that
\begin{align*}
{\tt V}_{23} & \leq  2^{\frac{2}{q-2}} \left(\Delta t\,\sum_{t'=0}^{T}
\|\Dtt(\ut^{\Delta t}_L(t'))\|^2\right)^{\frac{1}{q-2}}
\\
&
= 2^{\frac{2}{q-2}}\left(\Delta t\, \|\Dtt(\ut^0)\|^2 + \Delta t\,\sum_{t'=\Delta t}^{T}
\|\Dtt(\ut^{\Delta t,+}_L(t'))\|^2\right)^{\frac{1}{q-2}} \leq  C,
\end{align*}
where $C$ is a positive constant, independent of $L$, $\Delta t$ and $\delta$.
Thus we deduce that
\[ {\tt V}_2 = {\tt V}_{21}\,{\tt V}_{22}\,{\tt V}_{23} \leq C\, \delta^{\frac{q-4}{2(q-2)}},\]
where $C$ is a positive constant, independent of $L$, $\Delta t$ and $\delta$.
An identical argument yields that
\[{\tt V}_3 \leq C\, \delta^{\frac{q-4}{2(q-2)}},\]
and therefore
\begin{equation}\label{u123}
  {\tt V}_1\, ({\tt V}_{2} + {\tt V}_3) \leq C\, \delta^{\frac{1}{2}}\, \delta^{\frac{q-4}{2(q-2)}},
\end{equation}
where $C$ is a positive constant, independent of $L$, $\Delta t$ and $\delta$.

Finally, by the Cauchy--Schwarz inequality applied to the sum starting in the fifth line of \eqref{intermediate},
the bound on the term ${\tt V}_{23}$ above, using Lemma \eqref{summation} with
\[ s \in (0,T]\mapsto g(s) = (\delta^{\frac{1}{2}}+ \delta^{\frac{q-d}{2q}})^2\,\|\Dtt(u^{\Delta t,+}_L(s))\|^2 + \delta\,
 \|\ft^{\Delta t,+}(s)\|^2_{L^{\varkappa}(\Omega)} + ~\delta\,\|\nabq\sqrt{\hpsiaet^{\Delta t,+}}\|^2_{L^2_M(\Omega \times D)},
\]
and the bound \eqref{eq:energy-u+psi-final2}, we deduce that
\begin{equation}\label{u4567}
C\,\Delta t\sum_{t'=0}^{T-\delta} {\tt V}_4(t')\, ({\tt V}_5(t') + {\tt V}_6(t') + {\tt V}_7(t')) \leq C\,
(\delta^{\frac{1}{2}}+ \delta^{\frac{q-d}{2q}})\, \delta^{\frac{1}{2}} + C \delta^{\frac{1}{2} + \frac{1}{2}}  + C \delta^{\frac{1}{2} + \frac{1}{2}},
\end{equation}
where $C$ is a positive constant, independent of $L$, $\Delta t$ and $\delta$,
with $\delta = \ell\, \Delta t$, $\ell=1,\dots, N$.

On substituting \eqref{u123} and \eqref{u4567} into \eqref{intermediate}, we thus have that
\begin{align*}
&\Delta t\,\sum_{t'=0}^{T-\delta}\int_\Omega \rho_L^{[\Delta t]}(t'+\delta)\,|\ut^{\Delta t}_L(t'+\delta) - \ut^{\Delta t}_L(t')|^2 \dx
\\
&\qquad\quad \leq C\, \delta^{\frac{1}{2} + \frac{q-4}{2(q-2)}} + C\, (\delta^{\frac{1}{2}}+ \delta^{\frac{q-d}{2q}})\, \delta^{\frac{1}{2}}
+ C \delta \leq C\, \delta^{1- \frac{1}{q-2}} \quad \mbox{with}\quad\!
\left\{\begin{array}{ll}
q \in (4,\infty) & \mbox{if $d=2$}\\
q \in (4,6]      & \mbox{if $d=3$,}
\end{array}
\right.
\end{align*}
where $C$ is a positive constant, independent of $L$, $\Delta t$ and $\delta$,
with $\delta = \ell\, \Delta t$, $\ell=1,\dots, N$.

On recalling that $\rho_L^{[\Delta t]}(t'+\delta)  \geq \rho_{\rm min}$ for all $t' \in [0,T-\delta]$, we finally
have that
\begin{align*}
&\Delta t\,\sum_{t'=0}^{T-\delta}\|\ut^{\Delta t}_L(t'+\delta) - \ut^{\Delta t}_L(t')\|^2 \leq C\, \delta^{1- \frac{1}{q-2}} \quad \mbox{with}\quad\!
\left\{\begin{array}{ll}
q \in (4,\infty) & \mbox{if $d=2$}\\
q \in (4,6]      & \mbox{if $d=3$,}
\end{array}
\right.
\end{align*}
where $C$ is a positive constant, independent of $\Delta t$, $L$ and $\delta$;
with $\delta =\ell\,\Delta t$, $\ell=1,\ldots,N$.
As $\ut^{\Delta t}_L(\ell\,\Delta t) = \ut_L^{\Delta t,+}(\ell\,\Delta t)$ and
$\ut^{\Delta t}_L((\ell-1)\,\Delta t) = \ut_L^{\Delta t,-}(\ell\,\Delta t)$, $\ell = 1,\dots, N$,
and $\ut^{\Delta t,\pm}_L$ are piecewise constant functions on the partition
$\{0=t_0, t_1, \dots, t_{N-1}, t_N=T\}$, we thus deduce that
\begin{align}\label{delta-0}
&\Delta t\,\sum_{t'=\Delta t}^{T-\delta}\|\ut^{\Delta t,\pm}_L(t'+\delta) - \ut^{\Delta t,\pm}_L(t')\|^2 \leq C\, \delta^{1- \frac{1}{q-2}} \quad \mbox{with}\quad\!
\left\{\begin{array}{ll}
q \in (4,\infty) & \mbox{if $d=2$}\\
q \in (4,6]      & \mbox{if $d=3$,}
\end{array}
\right.
\end{align}
where $C$ is a positive constant, independent of $\Delta t$, $L$ and $\delta$, with
$\delta=\ell\,\Delta t$, $\ell=1,\dots, N-1$.
By selecting $q$ as large as possible and taking the square root of the previous inequality,
we have that
\begin{align}\label{delta-00}
\|\ut^{\Delta t,\pm}_L(\cdot+\delta) - \ut^{\Delta t,\pm}_L(\,\cdot\,)\|_{L^2(0,T-\delta;L^2(\Omega))} \leq C\,\delta^{\gamma}
\end{align}
for all $\delta = \ell \, \Delta t$, $\ell = 1,\dots, N-1$,
where $C$ is a positive constant independent of step size
$\Delta t$, $L$ and $\delta$; $0<\gamma<1/2$ when
$d=2$ and $0 < \gamma \leq  3/8$ when $d=3$.

We shall now extend the validity of \eqref{delta-00} to values of $\delta \in (0,T]$ that are not necessarily integer multiples of $\Delta t$.
We shall therefore at this point alter our original notational convention for $\delta$, and will consider
$\delta = \upsilon\, \Delta t$, with $\upsilon \in (0,N]$.
Let us define to this end $\ell = [ \upsilon ] := \max\{k \in \mathbb{N}\,:\,
k \leq \upsilon\}$, $\vartheta := \upsilon - [\upsilon ] \in [0,1)$, and for $t \in (0,T]$
let  $m \in \{0,\dots, N-\ell-2\}$ be such that $t \in (m\,\Delta t, (m+1)\,\Delta t]$.
Hence,
\begin{equation}\label{delta-1}
\ut^{\Delta t,\pm}_L(t + \upsilon\, \Delta t) = \left\{\begin{array}{ll}
\ut^{\Delta t, \pm}(t + \ell\, \Delta t) & \quad \mbox{if $t \in (m\, \Delta t, m\, \Delta t
+ (1-\vartheta)\, \Delta t]$} \\
\ut^{\Delta t, \pm}(t + (\ell+1)\, \Delta t) & \quad \mbox{if $t \in (m\, \Delta t
+ (1-\vartheta)\, \Delta t, (m+1)\,\Delta t]$,}
\end{array}
  \right.
\end{equation}
which then implies on noting that
\[s \in \mathbb{R}_{\geq 0} \mapsto s^\gamma\in \mathbb{R}_{\geq 0}\]
is a concave function, that
\begin{align*}
&\|\ut^{\Delta t,\pm}_L(\cdot + \delta) -
\ut^{\Delta t,\pm}_L(\cdot)\|^2_{L^2(0,T-\delta;L^2(\Omega))} = \int_0^{T - \upsilon\, \Delta t} \|\ut^{\Delta t,\pm}_L(\cdot + \upsilon\, \Delta t) -
\ut^{\Delta t,\pm}_L(\cdot)\|^2 \dd t \nonumber\\
&\qquad\leq  (1-\vartheta)\,\Delta t\! \sum_{t'=\Delta t}^{T-\ell\,\Delta t}
\|\ut^{\Delta t,\pm}_L(t'+\ell\,\Delta t) - \ut^{\Delta t,\pm}_L(t')\|^2 \nonumber\\
&\qquad\qquad +
\vartheta\,\Delta t \!\sum_{t'=\Delta t}^{T-(\ell +1) \Delta t}
\|\ut^{\Delta t,\pm}_L(t'+(\ell + 1)\Delta t) - \ut^{\Delta t,\pm}_L(t')\|^2\nonumber\\
&\qquad\leq (1-\vartheta)\,C\,(\ell\, \Delta t)^{\gamma} +
\vartheta\,C\,((\ell+1)\, \Delta t)^{\gamma}\nonumber\\
&\qquad\leq C\left[(1-\vartheta)\, \ell\,\Delta t+
\vartheta\,(\ell+1)\, \Delta t\right]^{\gamma}\nonumber\\
&\qquad = C\left[(\ell + \vartheta)\,\Delta t\right]
= C\left(\upsilon\, \Delta t\right)^{\gamma} = C\, \delta^{\gamma},
\end{align*}
where $\delta = \upsilon\, \Delta t$, $\upsilon \in (0,N]$;
 $0<\gamma<1/2$ when $d=2$ and $0 < \gamma \leq  3/8$ when $d=3$; and $C$ is a
positive constant, independent of $\Delta t$, $L$ and $\delta$. The second inequality
in the chain of inequalities above follows by applying \eqref{delta-0} first with $\delta = \ell\, \Delta t$ and then with $\delta = (\ell + 1)\,\Delta t$.

Consequently,
\begin{align*}
\|\ut^{\Delta t,\pm}_L(\cdot+\delta) - \ut^{\Delta t,\pm}_L(\,\cdot\,)\|_{L^2(0,T-\delta;L^2(\Omega))} \leq C\,\delta^{\gamma}
\end{align*}
for all $\delta  \in (0,T]$, where $C$ is a positive constant independent of
$\Delta t$ and $L$; $0<\gamma<1/2$ when $d=2$ and $0 < \gamma \leq  3/8$ when $d=3$.

Thus we have established the following \textit{Nikol'ski\u{\i} norm} estimate:
\begin{align}\label{fractional-u}
\|\ut_L^{\Delta t,\pm}\|_{N^{\gamma,2}(0,T;L^2(\Omega))}:=\sup_{0<\delta <T}\delta^{-\gamma}\| \ut_L^{\Delta t,\pm}(\cdot +\delta) - \ut_L^{\Delta t,\pm}(\,\cdot\,)\|_{L^2(0, T-\delta ;L^2(\Omega))} \leq C,
\end{align}
where $C$ is a positive constant, independent of $L$ and $\Delta t$;  $0<\gamma<1/2$ when $d=2$ and
$0 < \gamma \leq 3/8$ when $d=3$.

\begin{remark}
We note in passing that in the special case of an incompressible Newtonian fluid
with variable density, when the extra stress tensor appearing on the right-hand side of \eqref{ns1a} is identically
zero, \eqref{fractional-u} continues to hold and improves the
Nikol'ski\u{\i} index $\gamma=1/4$ obtained in the work of Simon \cite[p.1100, Proposition 8 (ii)]{Simon-density} and \cite[p.1103, Theorem 9 (ii)]{Simon-density}
(under the hypothesis $\ft \in L^1(0,T;\Lt^2(\Omega))$ compared with $\ft
\in L^2(0,T;\Lt^\varkappa(\Omega))$ assumed here with $\varkappa>1$ when $d=2$ and
$\varkappa>\frac{6}{5}$ when $d=3$, and under the same assumptions on the initial data $\ut_0$ and $\rho_0$ as in \eqref{inidata} here).
\end{remark}


\subsection{Strong convergence of the sequences $\{\rho^{\Delta t(,\pm)}_L\}_{L>1}$,
$\{\ut^{\Delta t(,\pm)}_L\}_{L>1}$, and weak convergence
of $\{\psi^{\Delta t(,\pm)}_L\}_{L>1}$}
\label{passage}

We begin by collecting a number of relevant bounds on the sequences
$\{\rho^{\Delta t(,\pm)}_L\}_{L>1}$, 
$\{\ut^{\Delta t(,\pm)}_L\}_{L>1}$, and $\{\psi^{\Delta t(,\pm)}_L\}_{L>1}$.

First, we recall that $\rho_L^{\Delta t}(t) \in \Upsilon$ for all $t\in [0,T]$ and
from \eqref{eq:energy-rho-0} that
\begin{subequations}
\begin{align}
&\|\rho_L^{[\Delta t]}(t)\|_{L^p(\Omega)}  \leq \|\rho_0\|_{L^p(\Omega)}, \qquad t \in (0,T],
\label{eq:energy-rho-0-a}
\end{align}
and therefore, for each $p \in [1,\infty]$,
\begin{align}
&\|\rho_L^{\Delta t(,\pm)}(t)\|_{L^p(\Omega)}  \leq \|\rho_0\|_{L^p(\Omega)}, \qquad t \in (0,T].
\label{eq:energy-rho6}
\end{align}
\end{subequations}
%
%

Next, noting (\ref{ulin},b), a simple calculation yields that
[see (6.32)--(6.34) in \cite{BS2010} for details]:
\begin{align}
\int_0^T\int_{\Omega \times D} M\big|\nabx \sqrt{\psia^{\Delta t}}\big|^2\dq\dx \dt
&\leq 2 \int_0^T \int_{\Omega \times D} M\,
\left[ |\nabx\sqrt{\psia^{\Delta t,+}}|^2
+ |\nabx\sqrt{\psia^{\Delta t,-}}|^2 \right] \dq \dx \dt,
\label{nabxDTpm}
\end{align}
and an analogous result with $\nabx$ replaced by $\nabq$.
Then the bound \eqref{eq:energy-u+psi-final2},
on noting (\ref{inidata}), (\ref{fractional-u}), (\ref{Korn}),
(\ref{ulin},b), (\ref{idatabd}), (\ref{inidata-1}), (\ref{nabxDTpm})
and the convexity of ${\mathcal F}$,
imply the existence
of a constant ${C}_{\star}>0$, depending only on
${\sf B}(\ut_0,\ft,\widetilde\psi_0)$ and
on
$\epsilon$,  $\rho_{\rm min}$, $\rho_{\rm max}$, $\mu_{\rm \min}$,
$\zeta_{\rm min}$, $\zeta_{\rm max}$,
$T$, $|\Att|$, $a_0$, $c_0$, $C_\varkappa$, $k$,
$\lambda$, $K$ and  $b$, 
but {\em not} on $L$ or $\Delta t$, such that:
\begin{align}\label{eq:energy-u+psi-final5}
&\mbox{ess.sup}_{t \in [0,T]}\|\uta^{\Delta t(,\pm)}(t)\|^2
+ \|\ut_L^{\Delta t,\pm}\|_{N^{\gamma,2}(0,T;L^2(\Omega))}^2
\nonumber \\
& \;
+ \frac{1}{\Delta t} \int_{0}^T \|\uta^{\Delta t, +}- \uta^{\Delta t,-}\|^2
\dd t
+ \int_0^T \|\nabxtt \uta^{\Delta t(,\pm)}\|^2 \dd t
\nonumber \\
&\;
+ \mbox{ess.sup}_{t \in [0,T]}\!\!
\int_{\Omega \times D}\!\! M\, \mathcal{F}(\psia^{\Delta t(,\pm)}(t)) \dq \dx
 + \frac{1}{\Delta t\, L}
\int_0^T \!\! \int_{\Omega \times D}\!\! M\, (\psia^{\Delta t, +} - \psia^{\Delta t, -})^2
\dq \dx \dd t
\nonumber \\
&\; + \int_0^T\!\! \int_{\Omega \times D} M\,
\big|\nabx \sqrt{\psia^{\Delta t(, \pm)}} \big|^2 \dq \dx \dd t
+\, \int_0^T\!\! \int_{\Omega \times D}M\,\big|\nabq \sqrt{\psia^{\Delta t(, \pm)}}\big|^2
\,\dq \dx \dd t \leq C_\ast,
\end{align}
where $0 < \gamma < 1/2$ when $d = 2$ and $0 < \gamma \leq 3/8$
when $d = 3$.
%
%

Henceforth, we shall assume that
\begin{equation}\label{LT}
\Delta t = o(L^{-1})\qquad \mbox{as $L \rightarrow \infty$}.
 \end{equation}
Requiring, for example, that $0<\Delta t \leq C_0/(L\,\log L)$, $L > 1$,
with an arbitrary (but fixed)
constant $C_0$ will suffice to ensure that \eqref{LT} holds. The sequences
$\{\rho^{[\Delta t]}_L\}_{L>1}$,
$\{\rho^{\Delta t(,\pm)}_L\}_{L>1}$, $\{\rho^{\{\Delta t\}}_L\}_{L>1}$
$\{\uta^{\Delta t(,\pm)}\}_{L>1}$ and $\{\psia^{\Delta t(,\pm)}\}_{L>1}$
as well as all sequences of spatial and temporal derivatives of the entries of these sequences
will thus be, indirectly, indexed by $L$ alone, although for reasons of consistency
with our previous notation
we shall not introduce new, compressed, notation with $\Delta t$ omitted
from the superscripts. Instead, whenever $L\rightarrow \infty$
in the rest of this section, it will be understood that $\Delta t$ tends to $0$ according to \eqref{LT}.
We are now ready to embark on the passage to limit with $L\rightarrow \infty$.

\begin{theorem}
\label{convfinal} Suppose that the assumptions \eqref{inidata} and the condition \eqref{LT},
relating $\Delta t$ to $L$, hold. Then,
there exists a subsequence of $\{(\rho^{\Delta t}_L,\utae^{\Delta t}, \widetilde\psi^{\Delta t}_L)
\}_{L >1}$ (not indicated)
with $\Delta t = o(L^{-1})$, and functions $(\rho,\ut, \widetilde\psi)$,
with $\widetilde\psi \geq 0$ a.e. on $\Omega \times D \times [0,T]$,
such that
\[ \rho \in L^\infty(0,T;\Upsilon)\cap C([0,T];L^p(\Omega)), \qquad
\ut \in L^{\infty}(0,T;\Lt^2(\Omega))\cap L^{2}(0,T;\Vt),
\]
where $p \in [1,\infty)$, and
%
\[
\sqrt{\widetilde\psi} \in L^{2}(0,T;H^1_M(\Omega \times D)),
\]
%
such that, as $L\rightarrow \infty$ (and thereby $\Delta t \rightarrow 0_+$),
\begin{subequations}
\begin{alignat}{2}
\rho_L^{[\Delta t]} &\rightarrow \rho &&\qquad \mbox{weak$^\star$ in }
L^\infty(0,T;\Lt^\infty(\Omega)), \label{rho[]weak}\\
\rho_L^{[\Delta t]} &\rightarrow \rho &&\qquad \mbox{strongly in }
L^{\infty}(0,T;{L}^p(\Omega)), \label{rho[]sconL2}\\
\rho_L^{\Delta t (,\pm)}, \, \rho_L^{\{\Delta t\}},
&\rightarrow \rho &&\qquad \mbox{strongly in }
L^{\infty}(0,T;{L}^p(\Omega)), \label{rhosconL2}\\
\mu(\rho_L^{\Delta t(,\pm)}) &\rightarrow \mu(\rho) &&\qquad \mbox{strongly in }
L^{\infty}(0,T;{L}^p(\Omega)), \label{musconL2}\\
\zeta(\rho_L^{[\Delta t]}),\,\zeta(\rho_L^{\Delta t(,\pm)}),
\,
\zeta_L^{\{\Delta t\}}
 &\rightarrow \zeta(\rho) &&\qquad \mbox{strongly in }
L^{\infty}(0,T;{L}^p(\Omega)), \label{zetasconL2}
\end{alignat}
\end{subequations}
where $p \in [1,\infty)$;
\begin{subequations}
\begin{alignat}{2}
\utae^{\Delta t (,\pm)} &\rightarrow \ut &&\qquad \mbox{weak$^\star$ in }
L^{\infty}(0,T;{\Lt}^2(\Omega)), \label{uwconL2a}\\
\utae^{\Delta t (,\pm)} &\rightarrow \ut &&\qquad \mbox{weakly in }
L^{2}(0,T;\Vt), \label{uwconH1a}\\
\utae^{\Delta t (,\pm)} &\rightarrow \ut &&\qquad \mbox{strongly in }
L^{2}(0,T;\Lt^{r}(\Omega)), \label{usconL2a}
\end{alignat}
\end{subequations}
where $r \in [1,\infty)$ if $d=2$ and $r \in [1,6)$ if $d=3$;
and
\begin{subequations}
\begin{alignat}{2}
\hpsiaet^{\Delta t(,\pm)} & \rightarrow  \hpsiaet &&\qquad \mbox{weakly in } L^1(0,T;L^1_M(\Omega\times D)),\label{psiwconL1}\\
M^{\frac{1}{2}}\,\nabx \sqrt{\hpsiaet^{\Delta t(,\pm)}} &\rightarrow M^{\frac{1}{2}}\,\nabx \sqrt{\widetilde\psi}
&&\qquad \mbox{weakly in } L^{2}(0,T;\Lt^2(\Omega\times D)), \label{psiwconH1a}\\
M^{\frac{1}{2}}\,\nabq \sqrt{\hpsiaet^{\Delta t(,\pm)}} &\rightarrow M^{\frac{1}{2}}\,\nabq \sqrt{\widetilde\psi}
&&\qquad \mbox{weakly in } L^{2}(0,T;\Lt^2(\Omega\times D)). \label{psiwconH1xa}
\end{alignat}
\end{subequations}
\end{theorem}
\begin{proof}
The weak convergence results (\ref{uwconL2a},b) follow directly from
the first and fourth bounds in \eqref{eq:energy-u+psi-final5}.
We deduce the strong convergence result \eqref{usconL2a}
in the case of $\uta^{\Delta t,+}$ on noting
the second and fourth bounds in \eqref{eq:energy-u+psi-final5},
(\ref{compact1}), and the compact embedding of
$\Vt$ into $\Lt^r(\Omega)\cap \Ht$, with the values of $r$ as in the statement of the theorem.
In particular, with $r=2$, a subsequence of $\{\uta^{\Delta t, +}\}_{L >1}$ converges to $\ut$,
strongly in $L^2(0,T;\Lt^2(\Omega))$ as $L\rightarrow \infty$, with $\Delta t = o(L^{-1})$.
Then, by the third bound in \eqref{eq:energy-u+psi-final5},
we deduce that the three corresponding subsequences
$\{\uta^{\Delta t(,\pm)}\}_{L >1}$ converge to $\ut$, strongly in $L^2(0,T;\Lt^2(\Omega))$ as
$L \rightarrow \infty$ (and thereby $\Delta t \rightarrow 0_+$).
Since these subsequences 
are bounded in $L^2(0,T;\Ht^1(\Omega))$ (cf.\ the bound on the fourth term in \eqref{eq:energy-u+psi-final5})
and strongly convergent in $L^2(0,T;\Lt^2(\Omega))$, it follows from  \eqref{eqinterp} that
\eqref{usconL2a} holds, with the values of $r$ as in the statement of the theorem.
Thus we have proved (\ref{uwconL2a}--c).

The convergence result (\ref{rho[]weak}) and the fact that $\rho \in
L^\infty(0,T;\Upsilon)$ follow immediately from (\ref{eq:energy-rho-0-a}) and as
$\rho_L^{\Delta t}(t) \in \Upsilon$ for all $t\in [0,T]$.
Further \eqref{rhonL} implies that
\begin{align*}
-\int_{t_{n-1}}^{t_n} \int_\Omega \rho^{[\Delta t]}_L \, \frac{\partial \eta}{\partial t}
\dx \dt - \int_{t_{n-1}}^{t_n} \int_\Omega \rho^{[\Delta t]}_L\, \ut^{\Delta t,-}_L
\cdot \nabx \eta \dx \dt =& -\left[\left.\int_\Omega \rho^{[\Delta t]}_L\, \eta \dx\right]
\right|^{t_n}_{t_{n-1}}\\
&
\quad\forall \eta \in C^1([t_{n-1},t_n]; W^{1,\frac{q}{q-1}}(\Omega)),
\end{align*}
with $q \in (2,\infty)$ when $d=2$ and $q \in [3,6]$ when $d=3$. Upon summation through
$n=1,\dots, N$, and noting that $\rho^0_L = \rho_0$, we then deduce that
\begin{align*}
&-\int_{0}^{T} \int_\Omega \rho^{[\Delta t]}_L \, \frac{\partial \eta}{\partial t}
\dx \dt - \int_{0}^{T} \int_\Omega \rho^{[\Delta t]}_L\, \ut^{\Delta t,-}_L
\cdot \nabx \eta \dx \dt = \int_\Omega \rho_0\, \eta \dx
\\
&\hspace{2in}\forall \eta \in C^1([0,T]; W^{1,\frac{q}{q-1}}(\Omega)) \mbox{ s.t.\ }
\eta(\cdot,T)=0,
\end{align*}
with $q \in (2,\infty)$ when $d=2$ and $q \in [3,6]$ when $d=3$. Hence, on letting
$L \rightarrow \infty$, with $\Delta t = o(L^{-1})$, and noting \eqref{rho[]weak}
and \eqref{usconL2a}, we deduce that
\begin{align}
&-\int_{0}^{T} \int_\Omega \rho \, \frac{\partial \eta}{\partial t}
\dx \dt - \int_{0}^{T} \int_\Omega \rho\, \ut
\cdot \nabx \eta \dx \dt = \int_\Omega \rho_0\, \eta \dx
\nonumber \\
&
\hspace{2in}\forall \eta \in C^1([0,T]; W^{1,\frac{q}{q-1}}(\Omega)) \mbox{ s.t.\ }
\eta(\cdot,T)=0,
\label{eqrhoconP}
\end{align}
with $q \in (2,\infty)$ when $d=2$ and $q \in [3,6]$ when $d=3$. Thus we have shown
that $\rho$ is a weak solution to (\ref{ns0a},b).
One can now apply the theory of DiPerna \& Lions \cite{DPL} to (\ref{ns0a},b).
As $\ut \in L^2(0,T;\Vt)$, it follows from Corollaries II.1 and II.2, and p.546, in \cite{DPL},
that there exists a unique
solution to (\ref{ns0a},b) for this given $\ut$, which must therefore coincide with $\rho$.
In addition, $\rho \in C([0,T],L^p(\Omega))$ for $p \in [1,\infty)$,
and the following equality holds:
\begin{align}\label{energyrho}
&\|\rho(t)\|_{L^p(\Omega)}  = \|\rho_0\|_{L^p(\Omega)}, \qquad t \in (0,T].
\end{align}
Thanks to (\ref{eq:energy-rho-0-a}) and (\ref{rho[]weak}), by the weak$^\ast$
lower semicontinuity of the norm function, and (\ref{energyrho}) with $p=2$, we have that, for a.e.\  $t \in [0,T]$,
\[
\|\rho(t)\|^2 \leq \liminf_{L \rightarrow \infty} \|\rho^{[\Delta t]}_L(t)\|^2
\leq \|\rho_0\|^2 = \|\rho(t)\|^2.
\]
This then implies 
for a.e. $t \in [0,T]$,
\begin{align}
\|\rho(t)\|^2 = \lim_{L \rightarrow \infty} \|\rho^{[\Delta t]}_L(t)\|^2
= \|\rho_0\|^2.
\label{rhosL2}
\end{align}
Thus we have proved \eqref{rho[]sconL2} in the case of $p=2$, which, on extracting a further
subsequence, implies that
\begin{align}
\lim_{L \rightarrow \infty} \rho^{[\Delta t]}_L =
\rho \qquad \mbox{a.e. on $\Omega \times (0,T)$.}
\label{rho-ae}
\end{align}
By recalling \eqref{rho-lower-upper}, we then deduce \eqref{rho[]sconL2} for all $p \in [1,\infty)$
by Lebesgue's dominated convergence theorem.

It follows from \eqref{rhonL-2}, with $\eta = \chi_{[t_{n-1},t)}\,\varphi$, $t \in (t_{n-1},t_n]$,
and $\eta = \chi_{(t,t_n]}\,\varphi$, $t \in [t_{n-1},t_n)$, and $\varphi \in W^{1,\frac{q}{q-1}}(\Omega)$, where $q \in (2,\infty)$ when $d=2$ and $q \in [3,6]$ when $d=3$,
that
\begin{equation}\label{rho-rho1}
\| \rho^{[\Delta t]}_L - \rho^{\Delta t,\pm}_L \|_{L^\infty(0,T;W^{1,\frac{q}{q-1}}(\Omega)')}
\leq
C \max_{n=1,\ldots, N} \int_{t_{n-1}}^{t_n} \|\nabxtt \ut^{\Delta t, -}_L \| \dt
\leq
C \,(\Delta t)^{\frac{1}{2}},
\end{equation}
where we have noted (\ref{eq:energy-rho-0-a}), (\ref{eqinterp}) and (\ref{eq:energy-u+psi-final5}).
Similarly, on recalling (\ref{zeta-time-average}), we obtain
\begin{equation}\label{rho-rho2}
\| \rho^{[\Delta t]}_L - \rho^{{\Delta t}}_L \|_{L^\infty(0,T;W^{1,\frac{q}{q-1}}(\Omega)')}
+ \| \rho^{[\Delta t]}_L - \rho^{\{\Delta t\}}_L \|_{L^\infty(0,T;W^{1,\frac{q}{q-1}}(\Omega)')}
\leq C \,(\Delta t)^{\frac{1}{2}},
\end{equation}
and on noting (\ref{Gequn-trans00})
\begin{equation}\label{zeta-zeta2}
\| \zeta(\rho^{[\Delta t]}_L) - \zeta^{\{\Delta t\}}_L \|_{L^\infty(0,T;W^{1,\frac{q}{q-1}}(\Omega)')}
\leq C \,(\Delta t)^{\frac{1}{2}}.
\end{equation}
Applying \eqref{rho[]sconL2} with $p=q$, we have that $\rho^{[\Delta t]}_L$
converges to $\rho$ strongly in $L^\infty(0,T;L^q(\Omega))
= L^\infty(0,T;L^{\frac{q}{q-1}}(\Omega)') \subset L^\infty(0,T;W^{1,\frac{q}{q-1}}(\Omega)')$, where $q \in (2,\infty)$ when $d=2$ and $q \in [3,6]$ when $d=3$; thus,
$\rho^{[\Delta t]}_L$ converges to $\rho$ strongly in $L^\infty(0,T;W^{1,\frac{q}{q-1}}(\Omega)')$, where $q \in (2,\infty)$ when $d=2$ and $q \in [3,6]$ when $d=3$. By means of a triangle inequality in
the norm of ${L^\infty(0,T;W^{1,\frac{q}{q-1}}(\Omega)')}$ and noting \eqref{rho-rho1}, we thus
deduce that $\rho^{\Delta t,\pm}_L$ converges to $\rho$ strongly in $L^\infty(0,T;W^{1,\frac{q}{q-1}}(\Omega)')$, where $q \in (2,\infty)$ when $d=2$ and $q \in [3,6]$ when $d=3$; analogously, using \eqref{rho-rho2} this time, $\rho^{\Delta t}_L$
and $\rho^{\{\Delta t\}}_L$,  converge to $\rho$ strongly in $L^\infty(0,T;W^{1,\frac{q}{q-1}}(\Omega)')$, where $q \in (2,\infty)$ when $d=2$ and $q \in [3,6]$ when $d=3$. Since, thanks to \eqref{eq:energy-rho6}, $\{\rho^{\Delta t,\pm}_L\}_{L>1}$, $\{\rho^{{\Delta t}}_L\}_{L>1}$
and $\{\rho^{\{\Delta t\}}_L\}_{L>1}$
are weak$^\ast$ compact in $L^\infty(0,T;L^\infty(\Omega))
\subset L^\infty(0,T;W^{1,\frac{q}{q-1}}(\Omega)')$, where $q \in (2,\infty)$ when $d=2$ and $q \in [3,6]$ when $d=3$, it follows that the weak$^\ast$ limits of the weak$^\ast$ convergent subsequences extracted from $\{\rho^{\Delta t,\pm}_L\}_{L>1}$, $\{\rho^{{\Delta t}}_L\}_{L>1}$
and $\{\rho^{\{\Delta t\}}_L\}_{L>1}$
have the same limit
as $\rho^{[\Delta t]}_L$: the element $\rho \in L^\infty(0,T;L^\infty(\Omega))$.
The weak$^*$ lower semicontinuity of the norm function and inequality \eqref{eq:energy-rho6}
together imply that
\[
\|\rho(t)\|^2 \leq \liminf_{L \rightarrow \infty} \|\rho^{\Delta t (,\pm)}_L(t)\|^2
\leq \|\rho_0\|^2 = \|\rho(t)\|^2.
\]
Hence, proceeding as above in the case of $\rho^{[\Delta t]}$, we deduce \eqref{rhosconL2}.

Concerning \eqref{musconL2} and \eqref{zetasconL2}, these follow from \eqref{rhosconL2}
and (\ref{zeta-zeta2})
via Lebesgue's dominated convergence theorem by possibly extracting a further subsequence,
thanks to our assumptions in \eqref{inidata} on $\mu$ and $\zeta$.

We complete the proof by establishing (\ref{psiwconL1}--c). According to \eqref{eq:energy-u+psi-final5} and \eqref{inidata},
\[ 2k\,\zeta_{\rm min} \int_{\Omega \times D} M\, \mathcal{F}(\widetilde\psi^{\Delta t(,\pm)}_L(t)) \dq \dx \leq [{\sf B}(\ut_0, \ft,\widetilde\psi_0)]^2\]
for all $t \in [0,T]$;
hence, on noting that $s \log(s+1)< 2\,[\mathcal{F}(s) + 1]$ for all $s\in \mathbb{R}_{\geq 0}$, we have that
\begin{equation}\label{psi-weak-0}
\max_{t \in [0,T]}\int_{\Omega \times D} M\, \widetilde\psi^{\Delta t(,\pm)}_L(t)
\log (\widetilde\psi^{\Delta t(,\pm)}_L(t)+1) \dq \dx \leq \frac{1}{k\,\zeta_{\rm min}}[{\sf B}(\ut_0, \ft,\widetilde\psi_0)]^2 + 2|\Omega|.
\end{equation}
As $s \in \mathbb{R}_{\geq 0} \mapsto s \log(s+1) \in \mathbb{R}_{\geq 0}$ is nonnegative, strictly monotonic increasing and convex, it follows from de la Vall\'ee-Poussin's theorem that the sequence
of nonnegative functions $\{\widetilde\psi^{\Delta t(,\pm)}_L\}_{L>1}$, with $\Delta t = o(L^{-1})$ is uniformly integrable on $\Omega \times D\times(0,T)$ with respect to the measure $\dd \nu:= M(\qt) \dq \dx \dt$.
Hence, by the Dunford--Pettis theorem, $\{\widetilde\psi^{\Delta t(,\pm)}_L\}_{L>1}$, with $\Delta t = o(L^{-1})$, is weakly relatively compact in $L^1(\Omega \times D\times (0,T);\nu)=L^1(0,T;L^1_M(\Omega \times D))$; i.e., there exists a nonnegative function
$\widetilde \psi \in L^1(0,T;L^1_M(\Omega \times D))$ and a subsequence (not indicated) such that
\begin{equation}\label{psi-weak-1}
\widetilde\psi^{\Delta t(,\pm)}_L \rightarrow \widetilde\psi\qquad \mbox{weakly in $L^1(0,T;L^1_M(\Omega \times D))$},
\end{equation}
as $L \rightarrow \infty$, where $\Delta t = o(L^{-1})$.
The fact that the limits of the subsequences of $\{\widetilde\psi^{\Delta t(,\pm)}_L\}_{L>1}$
are the same follows from the sixth bound in \eqref{eq:energy-u+psi-final5}.
Thus we have shown that \eqref{psiwconL1}
holds,
and that the limiting function $\widetilde\psi$ is nonnegative.

The extraction of the convergence results
(\ref{psiwconH1a},c) from \eqref{eq:energy-u+psi-final5} can be found in Step 2 in the
proof of Theorem 6.1 in \cite{BS2011-fene}.
\end{proof}

In the next section we shall strengthen \eqref{psiwconL1} by showing that
(a subsequence of) the sequence $\{\widetilde\psi_L^{\Delta t(,\pm)}\}_{L>1}$ is strongly convergent to $\widetilde \psi$ in $L^1(0,T;L^1_M(\Omega \times D))$ as $L \rightarrow \infty$, with
$\Delta t = o(L^{-1})$.


\subsection{Strong convergence of $\widetilde\psi_L^{\Delta t(,\pm)}$ in $L^1(0,T;L^1_M(\Omega \times D))$}
\label{Lindep-time}

In Section \ref{utsec} we derived an $L$-independent bound on the Nikol'ski\u{\i} norm,
based on time-shifts, of the sequence $\{\ut^{\Delta t (,\pm)}_L\}_{L>1}$ of approximate velocities. The Nikol'ski\u{\i} norm bound \eqref{fractional-u}
was used in the previous section
in conjunction with the bounds on spatial derivatives of $\{\ut^{\Delta t (,\pm)}_L\}_{L>1}$
established in Section \ref{Lindep-space} to deduce, via Simon's extension of the Aubin--Lions theorem \cite{Simon},
strong convergence of $\{\ut^{\Delta t (,\pm)}_L\}_{L>1}$ in $L^2(0,T;\Lt^r(\Omega))$ as $L \rightarrow \infty$, with $\Delta t=o(L^{-1})$, for
$1\leq r<\infty$ when $d=2$ and $1\leq r < 6$ when $d=2$, which we shall then use to pass to the limit
in nonlinear terms in \eqref{equncon} in conjunction with weak convergence results for the sequence,
which suffice for passage to the limit in those terms in \eqref{equncon} that depend linearly on
$\{\ut^{\Delta t (,\pm)}_L\}_{L>1}$.

In \cite{BS2011-fene} we used a similar argument for the sequence of approximations to the solution of the Fokker--Planck equation, except that due to the form of the Kullback--Leibler relative entropy and the associated Fisher information in the bounds on spatial norms of the sequence resulting from our entropy-based testing, (which, in turn, was motivated by the natural energy balance between the Navier--Stokes and Fokker--Planck equations in the coupled system, that manifests itself in
a fortuitous cancellation of the extra-stress tensor in the Navier--Stokes equation with the drag term in the Fokker--Planck equation in the course of the entropy-testing), we had to appeal to Dubinski\u{\i}'s extension to seminormed cones in Banach spaces of the original Aubin--Lions theorem to deduce strong convergence of the approximating sequence of probability density functions.

Unfortunately, in the present setting, the appearance of the nonlinear drag $\zeta(\rho)$ in the Fokker--Planck equation obstructs the application of Dubinski\u{\i}'s compactness theorem, and the approach based on Nikol'ski\u{\i} norm estimates, that was used in Section \ref{utsec}
in the density-dependent Navier--Stokes equation, also fails,
because --- in order to compensate for the rather
weak spatial control in \eqref{eq:energy-u+psi-final2} of the Kullback--Leibler relative entropy and the Fisher information --- its application ultimately requires a uniform $L^\infty(0,T;L^\infty(\Omega \times D))$ bound on the sequence of approximations to the probability density function, which is not available. We shall therefore adopt a different approach here.
Since the argument below that finally delivers the desired compactness of the sequence
$\{\widetilde\psi_L^{\Delta t(,\pm)}\}_{L>1}$ (with $\Delta t = o(L^{-1})$ as $L \rightarrow \infty$), in $L^1(0,T;L^1_M(\Omega \times D))$ is long and rather
technical, we begin with a brief overview of the key steps.

First, by \eqref{uwconL2a}
(a subsequence of) the sequence $\{\widetilde\psi_L^{\Delta t(,\pm)}\}_{L>1}$ is weakly convergent in the space $L^1(0,T;L^1_M(\Omega \times D))$
to $\widetilde\psi \in L^1(0,T;L^1_M(\Omega \times D))$ as $L \rightarrow \infty$, with $\Delta t = o(L^{-1})$. We shall then make use of the property that if $\Phi$
is a strictly convex weakly lower-semicontinuous function
defined on a convex open set $U$ of $\mathbb{R}$, and the weak limit of
$\Phi(\widetilde\psi_L^{\Delta t(,\pm)})$ is {\em equal} to
$\Phi(\widetilde\psi)$, then the sequence $\{\widetilde\psi_L^{\Delta t(,\pm)}\}_{L>1}$ converges almost everywhere on $(0,T)\times \Omega \times D$
as $L \rightarrow \infty$, with $\Delta t = o(L^{-1})$ (cf. Theorem 10.20 on p.339 of \cite{FN}).
According to \eqref{eq:energy-u+psi-final5},
%
%
%
\[ \max_{t \in [0,T]} \int_{\Omega \times D} M\, \mathcal{F}(\widetilde\psi_L^{\Delta t(,\pm)}) \dq \dx\]
is bounded, uniformly in $L$ and $\Delta t$; in addition $\mathcal{F}$ is strictly convex. Thus
$\mathcal{F}$ may appear as a logical first candidate for the choice of the function $\Phi$.
Unfortunately, we do not know at this point if the weak limit of the sequence
$\{\mathcal{F}(\widetilde\psi_L^{\Delta t(,\pm)})\}_{L>1}$ in $L^1(0,T;L^1_M(\Omega \times D))$ is {\em equal} to $\mathcal{F}(\widetilde\psi)$, and therefore the argument outlined in the previous paragraph is not directly applicable with the choice
$\Phi(s) = \mathcal{F}(s)$. We shall therefore make a different choice: we
select the {\em strictly convex} function
$\Phi(s) = (1+s)^{1+\alpha}$, $s \geq 0$, where $\alpha \in (0,1)$ is a suitable (small) positive real number. We note in passing, as this will be important in the argument that will follow,
that $s \mapsto (1+s)^{\alpha}$, $s \geq 0$, is a {\em strictly concave} function on $\mathbb{R}_{\geq 0}$ for $\alpha \in (0,1)$.
Although we do not know at this point if, with the latter choice of $\Phi$,  the weak limit of the sequence $\{\Phi(\widetilde\psi_L^{\Delta t(,\pm)})\}_{L>1}$ in $L^{1}(0,T;L^{1}_M(\Omega \times D))$ is {\em equal} to  $\Phi(\widetilde\psi)$,
and therefore this particular $\Phi$  may seem no better than the original suggestion of $\Phi(s) = \mathcal{F}(s)$, we note that by using estimates interior to $\Omega \times D$ on subdomains $\Omega_0 \times D_0 \Subset \Omega \times D$ on which the Maxwellian weight is bounded above and below by positive constants, and therefore the uniform bounds in Maxwellian-weighted norms that result from \eqref{eq:energy-u+psi-final2} become bounds in standard, unweighted, Lebesgue and Sobolev norms, one can use function space interpolation between these unweighted norms to deduce a
uniform bound on the $L^{1+\delta}(0,T;L^{1+\delta}(\Omega_0 \times D_0))$ norm of
$\widetilde\psi_L^{\Delta t(,\pm)}$, for a suitable (small) value of $\delta$, which, with $\alpha \in (0,\delta)$ and an application of the Div-Curl lemma, then implies that the weak limit in $L^1(0,T;L^1(\Omega_0 \times D_0))$ of
$\Phi(\widetilde\psi_L^{\Delta t(,\pm)})$ is {\em equal} to $\Phi(\widetilde\psi)$ as $L \rightarrow \infty$, with $\Delta t=o(L^{-1})$. Hence, by the argument, outlined in the previous paragraph, we
deduce almost everywhere convergence of a subsequence on $(0,T) \times \Omega_0 \times D_0$, and finally, using an increasing sequence of nested Lipschitz subdomains $(0,T) \times \Omega_k \times D_k$, $k=1,2, \dots$, and extracting a diagonal sequence from $\widetilde\psi_L^{\Delta t(,\pm)}$, we arrive at a subsequence of
$\{\widetilde\psi_L^{\Delta t(,\pm)}\}_{L>1}$ that converges almost everywhere on $(0,T)\times \Omega \times D$ to $\widetilde\psi$ as $L \rightarrow \infty$, with $\Delta t=o(L^{-1})$.
Since the set $\mathfrak{D}:=(0,T) \times \Omega \times D$ has finite measure, according to
Egoroff's theorem (cf. Theorem 2.22 on p.149 of \cite{Fonseca_Leoni}) almost everywhere convergence implies almost uniform convergence, and in particular
convergence in measure. Thus, by Vitali's convergence theorem (cf. Theorem 2.24 on p.150 of \cite{Fonseca_Leoni}), and thanks to the uniform integrability of the sequence $\{\widetilde\psi_L^{\Delta t(,\pm)}\}_{L>1}$ in $L^1(0,T;L^1_M(\Omega \times D))$,
 we finally deduce the desired strong convergence of the sequence $\{\widetilde\psi_L^{\Delta t(,\pm)}\}_{L>1}$ in $L^1(0,T;L^1_M(\Omega \times D))$ as $L \rightarrow \infty$, with $\Delta t=o(L^{-1})$. We will further
strengthen this by using Lemma \ref{le:supplementary} below to strong convergence in the $L^p(0,T;L^1_M(\Omega \times D))$ norm, for any $p \in [1,\infty)$.

\smallskip

We now embark on the programme outlined above by observing that, since on each compact subset of $D$ the Maxwellian $M$ is bounded above and below by positive constants (depending on the choice of the compact subset), it follows from \eqref{psiwconL1} that
$\{\widetilde\psi^{\Delta t(,\pm)}_L\}_{L>1}$,
with $\Delta t = o(L^{-1})$, is weakly relatively compact
in $L^1_{\rm loc}(\Omega \times D \times (0,T))$. Hence, by uniqueness of the weak limit,
\begin{equation}\label{psi-weak-2}
\widetilde\psi^{\Delta t(,\pm)}_L \rightarrow \widetilde\psi\qquad
\mbox{weakly in $L^1_{\rm loc}(0,T;L^1(\Omega \times D))$}.
\end{equation}
We shall show that in fact
\begin{equation}\label{psi-weak-3}
\widetilde\psi^{\Delta t(,\pm)}_L \rightarrow \widetilde\psi\qquad \mbox{a.e. on $(0,T)\times
\Omega \times D$}.
\end{equation}

Let $\mathcal{O}:=\Omega \times D$, and suppose that $\mathcal{O}_0$ is a Lipschitz subdomain of
 $\mathcal{O}$ such that $\mathcal{O}_0\Subset \mathcal{O}$. As $s\log(s+1) + 1 > s$ for all $s \in \mathbb{R}_{\geq 0}$ we have from \eqref{psi-weak-0}, the bounds on the seventh and the eighth term on the left-hand side of \eqref{eq:energy-u+psi-final5}, and noting once again that $M$ is bounded below
on $\mathcal{O}_0$ by a positive constant (which may depend on $\mathcal{O}_0$), that
\begin{equation}\label{psi-weak-4}
\max_{t \in [0,T]}\|\sqrt{\widetilde\psi^{\Delta t(,\pm)}_L(t)}\|^2_{L^2(\mathcal{O}_0)}
+ \int_0^T \|\sqrt{\widetilde\psi^{\Delta t(,\pm)}_L(t)}\|^2_{H^1(\mathcal{O}_0)} \dt  \leq
C(\mathcal{O}_0),
\end{equation}
where $C(\mathcal{O}_0)$ is a positive constant, which may depend on $\mathcal{O}_0$ but is
independent of $L$ and $\Delta t$. It then
follows from the bound on the second term on the left-hand side of \eqref{psi-weak-4} by Sobolev embedding applied on the bounded Lipschitz domain $\mathcal{O}_0 \Subset \mathcal{O} \subset \mathbb{R}^{(K+1)d}$ that
\begin{equation}\label{psi-weak-5}
\int_0^T \|\sqrt{\widetilde\psi^{\Delta t(,\pm)}_L(t)}\|^2_{L^{\frac{2(K+1)d}{(K+1)d - 2}}(\mathcal{O}_0)} \dt  \leq C(\mathcal{O}_0).
\end{equation}
Interpolation between the first inequality in \eqref{psi-weak-4} and the inequality \eqref{psi-weak-5} then yields that
\begin{equation}\label{psi-weak-6}
\int_0^T \|\widetilde\psi^{\Delta t(,\pm)}_L(t)\|^{\frac{(K+1)d + 2}{(K+1)d}}_{L^{\frac{(K+1)d + 2}{(K+1)d}}(\mathcal{O}_0)} \dt = \int_0^T \|\sqrt{\widetilde\psi^{\Delta t(,\pm)}_L(t)}\|^{2\,\frac{(K+1)d + 2}{(K+1)d}}_{L^{2\,\frac{(K+1)d + 2}{(K+1)d}}(\mathcal{O}_0)} \dt \leq C(\mathcal{O}_0).
\end{equation}
By writing $\widetilde\psi^{\Delta t(,\pm)}_L(t) = [\sqrt{\widetilde\psi^{\Delta t(,\pm)}_L}]^2$ and applying H\"older's inequality we then deduce for any $p \in [1,2)$ that
\begin{eqnarray}\label{psi-weak-7}
&&\int_0^T |\widetilde\psi^{\Delta t(,\pm)}_L(t)|_{W^{1,p}(\mathcal{O}_0)}^p \dt\nonumber\\
&&\qquad \leq 2^p \left(\int_0^T |\sqrt{\widetilde\psi^{\Delta t(,\pm)}_L(t)}|_{H^1(\mathcal{O}_0)}^2\dt \right)^{\frac{p}{2}}
\left(\int_0^T \|\sqrt{\widetilde\psi^{\Delta t(,\pm)}_L(t)}
\|^{\frac{2p}{2-p}}_{L^{\frac{2p}{2-p}}(\mathcal{O}_0)}\dt\right)^{\frac{2-p}{p}}
 \leq C(\mathcal{O}_0),
\end{eqnarray}
provided that
\begin{equation}\label{psi-weak-7a}
\frac{2p}{2-p} \leq 2\,\frac{(K+1)d + 2}{(K+1)d}.
\end{equation}
Let $p_0$ be the largest number in the range $[1,2)$ that satisfies \eqref{psi-weak-7a}; we thus deduce from \eqref{psi-weak-7} that
\begin{eqnarray}\label{psi-weak-8}
\int_0^T |\widetilde\psi^{\Delta t(,\pm)}_L(t)|^{p_0}_{W^{1,p_0}(\mathcal{O}_0)} \dt \leq C(\mathcal{O}_0),\qquad \mbox{with}\quad p_0:=\frac{(K+1)d +2}{(K+1)d + 1} \in (1,2).
\end{eqnarray}

Thanks to \eqref{lambda-bound} we also have that
\[
\|\widetilde\psi^{\Delta t(,\pm)}_L\|_{L^\infty((0,T)\times \Omega; L^1_M(D))} \leq C,
\]
and therefore,
\begin{equation}\label{psi-weak-9}
\|\widetilde\psi^{\Delta t(,\pm)}_L\|_{L^\infty((0,T)\times \Omega_0; L^1(D_0))} \leq C
\end{equation}
for any two Lipschitz
subdomains\footnote{
Let us suppose that $\mathcal{O} := \Omega \times D$, where $\Omega$ and $D$
are bounded open Lipschitz domains in $\mathbb{R}^m$ and $\mathbb{R}^n$ respectively.
Then, $\mathcal{O}$ is a bounded open Lipschitz domain in $\mathbb{R}^{m+n}$. This follows on noting
that the Cartesian product of two bounded open sets is itself is open. That $\mathcal{O}$
is a Lipschitz domain follows by combining Theorem 3.1 in the Ph.D. Thesis of Hochmuth \cite{RH},
which implies that the Cartesian product of a finite number of bounded domains,
each satisfying the uniform cone property, is a bounded domain satisfying the uniform cone property;
and Theorem 1.2.2.2 in the book of Grisvard \cite{PG}, which states that a bounded open set in
$\mathbb{R}^n$ has the uniform cone property if, and only if, its boundary is Lipschitz.}
$\Omega_0 \Subset \Omega$
and $D_0 \Subset D$; here $C$ is a positive constant, independent of $L$ and $\Delta t$. Fixing $\mathcal{O}_0 = \Omega_0 \times D_0$ and interpolating between
\eqref{psi-weak-9} and \eqref{psi-weak-6}, which states that
\[ \|\widetilde\psi^{\Delta t(,\pm)}_L\|_{L^{\frac{(K+1)d + 2}{(K+1)d}}((0,T)\times \Omega_0 \times D_0)} \leq C(\mathcal{O}_0),\]
we deduce that for any two real numbers $q_1$ and $q_2$, with
\begin{equation}\label{psi-weak-10aa}
 1+ \frac{2}{(K+1)d}\leq q_1<\infty \quad \mbox{and}\quad 1< q_2 \leq 1+ \frac{2}{(K+1)d}
\end{equation}
and satisfying the relation
\begin{equation}\label{psi-weak-10bb}
q_1\left(1-\frac{1}{q_2}\right) = \frac{2}{(K+1)d}\,,
\end{equation}
we have that
\begin{equation}\label{psi-weak-10}
\|\widetilde\psi^{\Delta t(,\pm)}_L\|_{L^{q_1}((0,T)\times \Omega_0; L^{q_2}(D_0))}
\leq C(\mathcal{O}_0).
\end{equation}

Note further that since $\rho^{\Delta t,+}_L \geq \rho_{\rm min}$ a.e. on $\Omega \times [0,T]$
and $\mu(\rho^{\Delta t,+}_L) \geq \mu_{\rm min}$ a.e. on $\Omega \times [0,T]$, interpolation
between the bounds (cf. \eqref{eq:energy-u+psi-final5})
\begin{equation}\label{psi-weak-10a}
\max_{t \in [0,T]} \int_\Omega |\ut^{\Delta t(,\pm)}_L(t)|^2 \dx
\leq \frac{[{\sf B}(\ut_0, \ft, \widetilde\psi_0)]^2}{\rho_{\rm min}};
\quad \int_0^T \int_\Omega |\Dtt(\ut^{\Delta t(,\pm)}_L)(t)|^2 \dx \dt \leq
\frac{[{\sf B}(\ut_0, \ft, \widetilde\psi_0)]^2}{\mu_{\rm min}}
\end{equation}
yields, using Korn's inequality \eqref{Korn}, that
\begin{align} \int_0^T \|\ut^{\Delta t(,\pm)}_L(t)\|^s_{L^s(\Omega)} \dt \leq C,
\qquad
 \mbox{where}\quad \left\{\begin{array}{ll} s < \frac{2(d+2)}{d} = 4 & \mbox{when $d=2$}\\
                                     s = \frac{2(d+2)}{d} = \frac{10}{3} & \mbox{when $d=3$,}
                   \end{array}\right.
 \label{psi-weak-11}
                   \end{align}
and $C$ is a positive constant, independent of $L$ and $\Delta t$.
%
%
Hence, by H\"older's inequality,
\eqref{psi-weak-11} and \eqref{psi-weak-10}, we get that
\begin{eqnarray}\label{psi-weak-11a}
&&\left[\int_0^T \int_{\Omega_0 \times D_0} \left[|\ut^{\Delta t, -}_L|\,\zeta^{\{\Delta t\}}_L\, \widetilde\psi^{\Delta t,+}_L\right]^{1+\delta} \dq \dx \dt \right]^{\frac{1}{1+\delta}}\nonumber\\
&&\qquad\leq \zeta_{\rm max} \left[\int_0^T \int_{\Omega_0} |\ut^{\Delta t, -}_L|^{1+\delta}\left[\int_{D_0}\left[\widetilde\psi^{\Delta t,+}_L\right]^{1+\delta} \dq\right] \dx\dt \right]^{\frac{1}{1+\delta}}
\nonumber\\
&&\qquad\leq \zeta_{\rm max} \left[\int_0^T \int_{\Omega_0}
|\ut^{\Delta t, -}_L|^{(1+\delta){\mathfrak{a}}}\dx \dt\right]^{\frac{1}{(1+\delta){\mathfrak{a}}}}
\left[\int_0^T \int_{\Omega_0}\left(\int_{D_0}\left[\widetilde\psi^{\Delta t,+}_L\right]^{1+\delta}\dq \right)^{\!\!\mathfrak{b}} \dx \dt\right]^{\frac{1}{(1+\delta){\mathfrak{b}}}},
\end{eqnarray}
where $\delta>0$ is to be chosen, and $1/{\mathfrak{a}} + 1/{\mathfrak{b}} = 1$, $1 < {\mathfrak a}, {\mathfrak{b}} < \infty$, with
\[ (1+\delta)\,{\mathfrak{a}} < \frac{2}{d}\,(d+2).\]

To this end, we define
\[r:= \frac{2}{(K+1)\, d}\]
and we select any
\[q_1 > \left(1 + \frac{d}{d+4}\right) \left(1 + r\right).\]
Note that $q_1> 1 + r = 1 + 2/((K+1)d)$; and in particular $q_1>r$.
We then define $q_2 := 1+\delta$, where $\delta := r/(q_1-r)$; hence $q_1>q_2$. We let
${\mathfrak{a}} := q_1/(q_1-q_2)$, ${\mathfrak{b}} := q_1/q_2$. Clearly, with such
a choice of $q_1$ and $q_2$,  we have that $q_1 (1- (1/{q_2}))=r$ and $1< q_2 < 1+r = 1 + 2/((K+1)d)<q_1$. We thus deduce from \eqref{psi-weak-11a} using \eqref{psi-weak-11} and \eqref{psi-weak-10} that
\begin{equation}\label{psi-weak-12}
\| \ut^{\Delta t, -}_L\,\zeta^{\{\Delta t\}}_L\, \widetilde\psi^{\Delta t,+}_L\|_{L^{1+\delta}((0,T)\times \mathcal{O}_0)} \leq C(\mathcal{O}_0),
\end{equation}
where $\delta>0$ is as defined above;
$\mathcal{O}_0= \Omega_0 \times D_0$; and
$C(\mathcal{O}_0)$ is a positive constant, independent of $L$ and $\Delta t$.

Analogously,
\begin{equation}\label{psi-weak-13}
 \left\|\sum_{i=1}^K
\left[\sigtt(\utae^{\Delta t,+})
\,\qt_i\right]\,\zeta(\rho_L^{\Delta t,+})\,\beta^L(\hpsiaet^{\Delta t,+}) \right\|_{L^{1+\delta}((0,T) \times \mathcal{O}_0)} \leq C(\mathcal{O}_0),
\end{equation}
this time with $0<\delta \leq r/(r+2)$, where, as above, $r= 2/((K+1)d)$. This follows on noting
the second inequality in \eqref{psi-weak-10a}, (\ref{Korn}) and that, thanks to \eqref{psi-weak-10},
\[
\|\widetilde\psi^{\Delta t(,\pm)}_L\|_{L^{\hat{q}_1}((0,T)\times \Omega_0; L^{\hat{q}_2}(D_0))}
\leq \|\widetilde\psi^{\Delta t(,\pm)}_L\|_{L^{q_1}((0,T)\times \Omega_0; L^{q_2}(D_0))}
\leq C(\mathcal{O}_0),
\]
with $\hat{q}_1 = 2(1+\delta)/(1-\delta)$; $\hat{q}_2 = 1 + \delta$ with $0<\delta \leq r/(r+2)$;
$q_2 = \hat{q}_2$; $q_1$ related to $q_2$ via \eqref{psi-weak-10bb}, and
noting that since $0<r/(r+2)<r<1$, we have $1+r<\hat{q}_1\leq q_1 = r(1+\delta)/\delta<\infty$, $1<\hat{q}_2=q_2 <1+r$. Here, again, $C(\mathcal{O}_0)$ is a positive constant, independent of $L$ and $\Delta t$.

With the bounds we have established on $\widetilde\psi^{\Delta t(,\pm)}_L$, we are now in a position to
apply the Div-Curl lemma, which we next state (cf., for example, \cite{FN}, p.343, Theorem 10.21).

\begin{theorem}\label{div-curl}
Suppose that $\mathfrak{D} \subset \mathbb{R}^{\mathfrak{N}}$ is a bounded open Lipschitz domain and $\mathfrak{N} \in \mathbb{N}_{\geq 2}$.
Let, for any real number $s>1$, $W^{-1,s}(\mathfrak{D})$ and $W^{-1,s}(\mathfrak{D};\mathbb{R}^{{\mathfrak{N}} \times {\mathfrak{N}}})$ denote the
duals of the Sobolev spaces $W^{1,\frac{s}{s-1}}_0(\mathfrak{D})$ and
$W^{1,\frac{s}{s-1}}_0(\mathfrak{D};\mathbb{R}^{{\mathfrak{N}} \times {\mathfrak{N}}})$, respectively.
Assume that
\begin{align*}
\left.
\begin{array}{rl}
\undertilde{H}_n \rightarrow \undertilde{H}\qquad & \mbox{weakly in $L^p(\mathfrak{D}; \mathbb{R}^{\mathfrak{N}})$},\\
\undertilde{Q}_n \rightarrow \undertilde{Q}\qquad & \mbox{weakly in $L^q(\mathfrak{D}; \mathbb{R}^{\mathfrak{N}})$},
\end{array}
\right\}
\end{align*}
where
\[ \frac{1}{p} + \frac{1}{q} = \frac{1}{r}<1.\]
Suppose also that there exists a real number $s>1$ such that
\begin{align*}
\left. \begin{array}{rll}    {\rm div }~ \undertilde{H}_n
&\equiv ~\nabla \cdot \undertilde{H}_n
& \mbox{is precompact in $W^{-1,s}(\mathfrak{D})$, and}\\
                             {\rm curl }~ \undertilde{Q}_n
&\equiv ~\left(\underdtilde{\nabla} \undertilde{Q}_n - (\underdtilde{\nabla} \undertilde{Q}_n)^{\rm T}\right)
& \mbox{is precompact in $W^{-1,s}(\mathfrak{D};\mathbb{R}^{{\mathfrak{N}} \times {\mathfrak{N}}})$.}
       \end{array}\right\}
\end{align*}
Then,
\[ \undertilde{H}_n\cdot \undertilde{Q}_n �\rightarrow  \undertilde{H}\cdot \undertilde{Q}\quad
 \mbox{weakly in $L^r(\mathfrak{D})$}.\]
\end{theorem}

\begin{remark}
Concerning the extension of Theorem \ref{div-curl} to the case of $r=1$, we refer to the recent paper by Conti, Dolzmann and M\"{u}ller \cite{CDM}.
For the our purposes here the version of the Div-Curl lemma as stated above will suffice.
\end{remark}

\bigskip

Consider the following sequences of $\mathfrak{N} = 1 + d  + Kd$ component vector functions defined on the Lipschitz
domain $\mathfrak{D}:=(0,T)\times \Omega_0 \times D_0 \subset \mathbb{R}^{\mathfrak{N}}$,
\[H_{\Delta t, L} := \bigg(H^{(t)}_{\Delta t, L}\,;\, H^{(x,1)}_{\Delta t, L}, \dots, H^{(x,d)}_{\Delta t, L}\,;\, H^{(q_1,1)}_{\Delta t, L}, \dots, H^{(q_1,d)}_{\Delta t, L}, \dots, H^{(q_K,1)}_{\Delta t, L}, \dots, H^{(q_K,d)}_{\Delta t, L}\bigg)\]
where,
\begin{eqnarray*}
H^{(t)}_{\Delta t, L} &:=& M\, (\zeta(\rho_L)\,\widetilde\psi_L)^{\Delta t},\\
(H^{(x,1)}_{\Delta t, L}, \dots, H^{(x,d)}_{\Delta t, L}) &:=&
\utae^{\Delta t,-}\,M\,\zeta^{\{\Delta t\}}_L\,\hpsiaet^{\Delta t,+}-\epsilon\, M\,
\nabx \hpsiaet^{\Delta
t,+},\\
(H^{(q_i,1)}_{\Delta t, L},\dots, H^{(q_i,d)}_{\Delta t, L}) &:=&
M\,\left[\sigtt(\utae^{\Delta t,+})
\,\qt_i\right]\,\zeta(\rho_L^{\Delta t,+})\,\beta^L(\hpsiaet^{\Delta t,+})-\frac{1}{4\,\lambda}
\,\sum_{j=1}^K A_{ij}\,
 M\, \nabqj \hpsiaet^{\Delta t,+},\\
&&\hspace{8cm}\quad i=1,\dots,K,
\end{eqnarray*}
and
\[Q_{\Delta t, L} := \bigg((1+\widetilde \psi^{\Delta t}_L)^\alpha,
\underbrace{0, \dots, 0}_{
\begin{smallmatrix}
  \text{$d + Kd$} \\
  \text{times}
\end{smallmatrix}}\,\bigg),\qquad \mbox{with $\alpha \in (0,\textstyle{\frac{1}{2}})$ fixed}.\]

Thanks to \eqref{psi-weak-10}, \eqref{psi-weak-7}, \eqref{psi-weak-12} and \eqref{psi-weak-13},
there exists a real number $p_\ast \in (1,2)$ such that the sequence $\{H_{\Delta t, L}\}_{L>1}$, with $\Delta t=o(L^{-1})$,
is bounded in $L^{p_\ast}(\mathfrak{D},\mathbb{R}^{\mathfrak{N}})$; hence there exists an element  $H \in L^{p_\ast}(\mathfrak{D},
\mathbb{R}^{\mathfrak{N}})$ and a subsequence, not indicated, such that $H_{\Delta t, L} \rightarrow H$, weakly in $L^{p_\ast}(\mathfrak{D},\mathbb{R}^{\mathfrak{N}})$. Also, the sequence
$\{Q_{\Delta t, L}\}_{L>1}$, with $\Delta t=o(L^{-1})$,
is bounded in $L^{q_\ast}(\mathfrak{D},\mathbb{R}^{\mathfrak{N}})$ with $q_\ast=1/\alpha$; hence there exists a $Q \in L^{q_*}(\mathfrak{D},
\mathbb{R}^{\mathfrak{N}})$ and a subsequence, not indicated, such that $Q_{\Delta t, L} \rightarrow Q$,
weakly in $L^{q_\ast}(\mathfrak{D},\mathbb{R}^{\mathfrak{N}})$.

With our definition of $\{H_{\Delta t, L}\}_{L>1}$, we have that
\[ \mbox{div}_{t,x,q} \,H_{\Delta t, L} = 0.\]
Therefore the sequence $\{\mbox{div}_{t,x,q}\, H_{\Delta t, L}\}_{L>1}$, with $\Delta t = o(L^{-1})$, is precompact in $W^{-1,s}(\mathfrak{D})$ for all $s>1$.

Further, since $\alpha \in (0,\textstyle{\frac{1}{2}})$, it follows from \eqref{psi-weak-4} that $\{Q_{\Delta t, L}\}_{L>1}$ satisfies
\begin{eqnarray*}
\int_{(0,T)\times \mathcal{O}_0}|\,\mbox{curl}_{t,x,q}  \,Q_{\Delta t, L}\,|^2 \dq \dx \dt &\leq & C \int_{(0,T)\times \mathcal{O}_0} |\nabla_{x,q}\,(1+\widetilde\psi^{\Delta t}_L)^\alpha|^2 \dq \dx \dt\\
&\leq & C \int_{(0,T)\times \mathcal{O}_0} |\nabla_{x,q}\sqrt{\widetilde\psi^{\Delta t}_L}|^2 \dq \dx \dt \\
&\leq & C(\mathcal{O}_0).
\end{eqnarray*}
Therefore, the sequence $\{\mbox{curl}_{t,x,q}  \,Q_{\Delta t, L}\}_{L>1}$, with $\Delta t=o(L^{-1})$,  is precompact in the function space $W^{-1,2}(\mathfrak{D};\mathbb{R}^{\mathfrak{N}\times \mathfrak{N}})$.

We thus deduce from Theorem \ref{div-curl} that
\[
 \overline{H_{\Delta t, L} \cdot Q_{\Delta t, L}} = \overline{H_{\Delta t, L}}\cdot\overline{Q_{\Delta t, L}},
\]
where the overline $\overline{~\cdot~}$ signifies the weak limit in $L^1(\mathfrak{D})$ of the sequence appearing under the overline; thus,
\begin{equation}\label{weak-1}
\overline{(\zeta(\rho_L)\,\widetilde\psi_L)^{\Delta t}\, (1+\widetilde{\psi}^{\Delta t}_L)^\alpha}
= \overline{(\zeta(\rho_L)\,\widetilde\psi_L)^{\Delta t}}\,\, \overline{(1+\widetilde{\psi}^{\Delta t}_L)^\alpha}.
\end{equation}
As
\[ (\zeta(\rho_L)\,\widetilde\psi_L)^{\Delta t}(\cdot,t)  =
\zeta^n_L(\cdot) \left[\frac{t-t_{n-1}}{\Delta t} \widetilde\psi^n_L(\cdot) + \frac{t_n  - t}{\Delta t} \widetilde\psi^{n-1}_L(\cdot) \right]
+ \left(\zeta^{n-1}_L(\cdot) - \zeta^n_L(\cdot)\right) \widetilde\psi^{n-1}_L(\cdot)\,\frac{t_n - t}{\Delta t}\]
for all $t \in [t_{n-1},t_n]$ and $n=1,\dots,N$, which in turn implies that
\begin{equation}\label{weak-2} (\zeta(\rho_L)\,\widetilde\psi_L)^{\Delta t} = \zeta(\rho^{\Delta t,+}_L) \,\widetilde\psi_L^{\Delta t}
+ \big(\zeta(\rho^{\Delta t,-}_L) - \zeta(\rho^{\Delta t,+}_L)\big)\, \widetilde\psi^{\Delta t,-}_L\, \theta_{\Delta t},
\end{equation}
where 
$\theta_{\Delta t}$ is the nonnegative discontinuous piecewise linear function defined on $(0,T]$ by
\[ \theta_{\Delta t}(t) = \frac{t_n - t}{\Delta t}, \qquad t \in (t_{n-1},t_n], \quad n=1,\dots,N,\]
the fact that, by \eqref{zetasconL2},
\begin{align*}
&\|\zeta(\rho^{\Delta t,-}_L) - \zeta(\rho^{\Delta t,+}_L)\|_{L^\infty(0,T;L^p(\Omega))}\\
&\hspace{2cm}\leq  \|\zeta(\rho^{\Delta t,-}_L) - \zeta(\rho)\|_{L^\infty(0,T;L^p(\Omega))} +
 \|\zeta(\rho^{\Delta t,+}_L) - \zeta(\rho)\|_{L^\infty(0,T;L^p(\Omega))}\rightarrow 0
\end{align*}
as $L \rightarrow \infty$ (with $\Delta t = o(L^{-1})$), $1<p<\infty$, implies,
on noting (\ref{psi-weak-6}), that
\begin{equation}\label{weak-3} \|\big(\zeta(\rho^{\Delta t,-}_L) - \zeta(\rho^{\Delta t,+}_L)\big)\, \widetilde\psi^{\Delta t,-}_L\, \theta_{\Delta t}\|_{L^1(\mathfrak{D})} \rightarrow 0,
\end{equation}
as $L \rightarrow \infty$ (with $\Delta t = o(L^{-1})$). Further, by \eqref{zetasconL2},
$\zeta(\rho^{\Delta t,+}_L)\rightarrow
\zeta \in L^\infty(0,T;L^\infty(\Omega))$ strongly in $L^\infty(0,T; L^p(\Omega))$, $1\leq p<\infty$,  and,
by \eqref{psi-weak-2}, $\widetilde\psi_L^{\Delta t}
\rightarrow \widetilde\psi$ weakly in $L^1_{{\rm loc}}(0,T;L^1(\Omega\times D))$;
thus we deduce, on noting (\ref{psi-weak-6}), that
$\zeta(\rho^{\Delta t,+}_L) \,\widetilde\psi_L^{\Delta t}$ converges to $\zeta(\rho)\, \widetilde\psi$, weakly in $L^1(\mathfrak{D})$. Hence we have shown that
\[ \overline{(\zeta(\rho_L)\,\widetilde\psi_L)^{\Delta t}} = \overline{\zeta(\rho^{\Delta t,+}_L) \,\widetilde\psi_L^{\Delta t}} = \zeta(\rho)\, \widetilde\psi \in L^1(\mathfrak{D}).\]
Consequently, we have from \eqref{weak-1} that
\[
(\zeta(\rho_L)\,\widetilde\psi_L)^{\Delta t}\, (1+\widetilde{\psi}^{\Delta t}_L)^\alpha
\rightarrow  \zeta(\rho)\,\widetilde\psi\,\, \overline{(1+\widetilde{\psi}^{\Delta t}_L)^\alpha}
\qquad \mbox{weakly in $L^1(\mathfrak{D})$.}
\]
Noting \eqref{weak-2} and \eqref{psi-weak-6} then yields that
\[
\zeta(\rho_L^{\Delta t,+})\,\widetilde\psi_L^{\Delta t}\, (1+\widetilde{\psi}^{\Delta t}_L)^\alpha
\rightarrow  \zeta(\rho)\,\widetilde\psi\,\, \overline{(1+\widetilde{\psi}^{\Delta t}_L)^\alpha}
\qquad \mbox{weakly in $L^1(\mathfrak{D})$.}
\]
Thus, by the strong convergence $\zeta(\rho_L^{\Delta t,+}) \rightarrow \zeta(\rho)
\in L^\infty(0,T;L^\infty(\Omega))$ in $L^\infty(0,T;L^p(\Omega))$, $1<p<\infty$, which, thanks
to our assumptions on the function $\zeta$ stated in \eqref{inidata}
implies that $1/\zeta(\rho_L^{\Delta t,+}) \rightarrow 1/\zeta(\rho) \in L^\infty(0,T;L^\infty(\Omega))$ in $L^\infty(0,T;L^p(\Omega))$, $1<p<\infty$, we finally have, on noting (\ref{psi-weak-6}),
that
\begin{equation}\label{weak-3a}
\widetilde\psi_L^{\Delta t}\, (1+\widetilde{\psi}^{\Delta t}_L)^\alpha
\rightarrow \widetilde\psi\,\, \overline{(1+\widetilde{\psi}^{\Delta t}_L)^\alpha}
\qquad \mbox{weakly in $L^1(\mathfrak{D})$.}
\end{equation}
As, by definition, $(1+\widetilde{\psi}^{\Delta t}_L)^\alpha \rightarrow \overline{(1+\widetilde{\psi}^{\Delta t}_L)^\alpha}$, weakly in $L^1(\mathfrak{D})$, by adding this to
\eqref{weak-3a} we have that
\[
(1+\widetilde{\psi}^{\Delta t}_L)^{\alpha+1} = (1+\widetilde\psi_L^{\Delta t})\, (1+\widetilde{\psi}^{\Delta t}_L)^\alpha
\rightarrow (1+\widetilde\psi)\, \overline{(1+\widetilde{\psi}^{\Delta t}_L)^\alpha}
\qquad \mbox{weakly in $L^1(\mathfrak{D})$.}
\]
Thanks to the weak lower-semicontinuity of the continuous convex function $s\in [0,\infty) \mapsto s^{\alpha+1} \in [0,\infty)$ it follows (cf. Theorem 10.20 on p.339 of \cite{FN}) that
\[
(1+\widetilde\psi)^{1+\alpha} \leq (1+\widetilde\psi)\,\overline{(1+\widetilde\psi_L^{\Delta t})^{\alpha}}.
\]
Consequently,
\begin{equation}\label{weak-3aa}
(1+\widetilde\psi)^{\alpha}\le \overline{(1+\widetilde\psi_L^{\Delta t})^\alpha}.
\end{equation}

On the other hand, the function $s\in [0,\infty) \mapsto s^\alpha \in [0,\infty)$ is continuous and concave, and therefore $s\in [0,\infty) \mapsto -s^\alpha \in (-\infty,0]$ is continuous and convex;
thus, once again by the weak lower-semicontinuity of continuous convex functions,
we immediately have (cf. Theorem 10.20 on p.339 of \cite{FN}) that
\begin{equation}\label{weak-3aaa}
- (1+\widetilde\psi)^{\alpha}\leq - \overline{(1+\widetilde\psi_L^{\Delta t})^\alpha}.
\end{equation}
We deduce from \eqref{weak-3aa} and \eqref{weak-3aaa} that
\begin{equation}
-(1+\widetilde\psi)^{\alpha}=-\overline{(1+\widetilde\psi_L^{\Delta t})^\alpha},
\end{equation}
and consequently, since the function $s\in [0,\infty) \mapsto -s^\alpha \in (-\infty,0]$ is continuous and {\em strictly} convex, and its domain of definition, $[0,\infty)$, is a convex set, Theorem 10.20 on p.339 of \cite{FN} implies that there exists a subsequence (not relabelled) such that
\begin{equation}
\widetilde\psi_L^{\Delta t} \to \widetilde\psi\qquad \mbox{ a.e. in  $\mathfrak{D}$}.
\label{weak-4}
\end{equation}

Next, we select an increasing nested sequence $\{\mathfrak{D}_0^k\}_{k=1}^{\infty}$ of bounded open Lipschitz domains $\mathfrak{D}_0^k = (0,T)\times \Omega_0^k \times D_0^k$,
where $\{\Omega_0^k\}_{k=1}^\infty$ and $\{D_0^k\}_{k=1}^\infty$ are increasing nested sequences of bounded open Lipschitz domains in $\Omega$ and $D$, respectively,
such that $\bigcup_{k=1}^{\infty}\mathfrak{D}_0^k = (0,T)\times \Omega \times D$.
Since for each $k$ we have pointwise convergence on $\mathfrak{D}_0^k$ of a subsequence of $\{\widetilde\psi_L^{\Delta t}\}_{L>1}$, by using a diagonal procedure, we can extract
from $\{\widetilde\psi_L^{\Delta t}\}_{L>1}$ a subsequence, (which is, once again, not relabelled) such that
\[ \widetilde\psi_L^{\Delta t} \to \widetilde\psi\qquad \mbox{ a.e. in  $(0,T)\times \Omega \times D$}
\]
with respect to the Lebesgue measure on $(0,T)\times \Omega \times D$. Let, for any Borel subset
$A$ of $(0,T)\times\Omega\times D$,
\[ \nu(A):= \int_A M \dq \dx \dt.\]
 Since $M \in L^1((0,T)\times \Omega \times D)$, the measure $\nu$ is absolutely continuous with respect to the Lebesgue measure, which then
implies that $\widetilde\psi_L^{\Delta t} \to \widetilde\psi$, almost everywhere  with respect to the measure $\nu$ (or, briefly, $\nu$ almost everywhere) in $(0,T)\times \Omega \times D$.
Since $\nu((0,T) \times \Omega \times D)<\infty$, according to
Egoroff's theorem (cf. Theorem 2.22 on p.149 of \cite{Fonseca_Leoni}) $\nu$ almost everywhere convergence of $\widetilde\psi_L^{\Delta t}$ to $\widetilde\psi$ implies $\nu$ almost uniform convergence of $\widetilde\psi_L^{\Delta t}$ to $\widetilde\psi$, and in particular $\nu$ convergence in measure of  $\widetilde\psi_L^{\Delta t}$ to $\widetilde\psi$. Finally, by Vitali's convergence theorem (cf. Theorem 2.24 on p.150 of \cite{Fonseca_Leoni}), the uniform integrability of the sequence $\{\widetilde\psi_L^{\Delta t}\}_{L>1}$ in $L^1((0,T)\times\Omega \times D; \nu) = L^1(0,T;L^1_M(\Omega \times D))$ and $\nu$ convergence in measure of $\widetilde\psi_L^{\Delta t}$ to $\widetilde\psi$
together imply that
\begin{align}
\widetilde\psi_L^{\Delta t} \to \widetilde\psi \qquad \mbox{ strongly in } L^1(0,T; L^1_M(\Omega \times D)).\label{strong-1}
\end{align}
It follows from the (\ref{strong-1}) and the sixth bound in (\ref{eq:energy-u+psi-final5}) that
\begin{align}
\widetilde\psi_L^{\Delta t(,\pm)} \to \widetilde\psi \qquad \mbox{ strongly in } L^1(0,T; L^1_M(\Omega \times D)).\label{strong-1pm}
\end{align}
In fact, \eqref{strong-1pm} can be further strengthened: it follows from Lemma \ref{le:supplementary} below and \eqref{strong-1pm} that
\begin{align}
\widetilde\psi_L^{\Delta t(,\pm)} \to \widetilde\psi \qquad \mbox{ strongly in } L^p(0,T; L^1_M(\Omega \times D))\qquad \forall p \in [1,\infty).\label{strong-2}
\end{align}

\begin{lemma}\label{le:supplementary}
Suppose that a sequence $\{\varphi_n\}_{n=1}^\infty$ converges in
$L^1(0,T; L^1_M(\Omega \times D))$ to a function $\varphi \in L^1(0,T; L^1_M(\Omega \times D))$,
and is bounded in $L^\infty(0,T; L^1_M(\Omega \times D))$; i.e., there exists $K_0>0$ such that
$\|\varphi_n\|_{L^\infty(0,T; L^1_M(\Omega \times D))} \leq K_0$ for all $n \geq 1$.
Then, $\varphi \in L^p(0,T; L^1_M(\Omega \times D))$ for all $p \in [1,\infty)$,
and the sequence $\{\varphi_n\}_{n \geq 1}$ converges to $\varphi$ in
$L^p(0,T; L^1_M(\Omega \times D))$ for all $p \in [1,\infty)$.
\end{lemma}

\begin{proof}
See the proof of Lemma 5.1 in \cite{BS2010}.
\end{proof}

This then completes our proof of strong convergence of the sequence
$\{\widetilde\psi_L^{\Delta t(,\pm)}\}_{L>1}$  to the function $\widetilde\psi \in L^\infty(0,T;L^1_M(\Omega \times D))$ in the norm
of the space $L^p(0,T; L^1_M(\Omega \times D))$, for all $p \in [1,\infty)$.


\section{Passage to the limit $L \rightarrow \infty$:
existence of weak solutions to this FENE chain model with variable density and viscosity}
\label{5-sec:passage.to.limit}
\setcounter{equation}{0}

We are now ready to pass to the limit with $L\rightarrow \infty$.

\subsection{Passage to the limit $L\rightarrow \infty$}
\label{5-passage}

In this section we prove the central result of the paper.

\begin{theorem}
\label{5-convfinal} Suppose that the assumptions \eqref{inidata} and the condition \eqref{LT},
relating $\Delta t$ to $L$, hold. Then,
there exists a subsequence of $\{(\rho^{\Delta t}_L,\utae^{\Delta t}, \widetilde\psi^{\Delta t}_L)
\}_{L >1}$ (not indicated)
with $\Delta t = o(L^{-1})$, and functions $(\rho,\ut, \widetilde\psi)$ such that
\[ \rho \in L^\infty(0,T;\Upsilon)\cap C([0,T];L^p(\Omega)), \qquad
\ut \in L^{\infty}(0,T;\Lt^2(\Omega))\cap L^{2}(0,T;\Vt),
\]
where $p \in [1,\infty)$, and
\[\widetilde\psi \in L^1(0,T;L^1_M(\Omega \times D)),
\]
with $\widetilde\psi \geq 0$ a.e. on $\Omega \times D \times [0,T]$, satisfying
\begin{equation}\label{5-mass-conserved}
\int_D M(\qt)\,\widetilde\psi(\xt,\qt,t) \dq \leq
\mbox{\em ess.sup}_{x \in \Omega}\left(\frac{1}{\zeta(\rho_0(\xt))}
\int_D \psi_0(\xt,\qt) \dq \right)
\quad \mbox{for a.e. $(x,t) \in \Omega \times [0,T]$},
\end{equation}
whereby $\widetilde\psi \in L^\infty(0,T; L^1_M(\Omega \times D))$;
and finite relative entropy and Fisher information, with
\begin{equation}\label{5-relent-fisher}
\mathcal{F}(\widetilde\psi) \in L^\infty(0,T;L^1_M(\Omega\times D))\quad
\mbox{and}\quad \sqrt{\widetilde\psi} \in L^{2}(0,T;H^1_M(\Omega \times D)),
\end{equation}
such that, as $L\rightarrow \infty$ (and thereby $\Delta t \rightarrow 0_+$),
\begin{subequations}
\begin{alignat}{2}
\rho_L^{[\Delta t]} &\rightarrow \rho &&\qquad \mbox{weak$^\star$ in }
L^{\infty}(0,T;{L}^\infty(\Omega)), \label{5-rhowconL2}\\
\rho_L^{[\Delta t]},\,\rho_L^{\Delta t (,\pm)},\,
\rho_L^{\{\Delta t\}}
&\rightarrow \rho &&\qquad \mbox{strongly in }
L^{\infty}(0,T;{L}^p(\Omega)), \label{5-rhosconL2}\\
\mu(\rho_L^{\Delta t(,\pm)}) &\rightarrow \mu(\rho) &&\qquad \mbox{strongly in }
L^{\infty}(0,T;{L}^p(\Omega)), \label{5-musconL2}\\
\zeta(\rho_L^{[\Delta t]}),\,
\zeta(\rho_L^{\Delta t(,\pm)}),\,
\zeta_L^{\{\Delta t\}}
&\rightarrow \zeta(\rho) &&\qquad \mbox{strongly in }
L^{\infty}(0,T;{L}^p(\Omega)), \label{5-zetasconL2}
\end{alignat}
\end{subequations}
where $p \in [1,\infty)$;
\begin{subequations}
\begin{alignat}{2}
\utae^{\Delta t (,\pm)} &\rightarrow \ut &&\qquad \mbox{weak$^\star$ in }
L^{\infty}(0,T;{\Lt}^2(\Omega)), \label{5-uwconL2a}\\
\utae^{\Delta t (,\pm)} &\rightarrow \ut &&\qquad \mbox{weakly in }
L^{2}(0,T;\Vt), \label{5-uwconH1a}\\
\utae^{\Delta t (,\pm)} &\rightarrow \ut &&\qquad \mbox{strongly in }
L^{2}(0,T;\Lt^{r}(\Omega)), \label{5-usconL2a}
\end{alignat}
\end{subequations}
where $r \in [1,\infty)$ if $d=2$ and $r \in [1,6)$ if $d=3$;
and
\begin{subequations}
\begin{alignat}{2}
M^{\frac{1}{2}}\,\nabx \sqrt{\hpsiaet^{\Delta t(,\pm)}} &\rightarrow M^{\frac{1}{2}}\,\nabx \sqrt{\widetilde\psi}
&&\qquad \mbox{weakly in } L^{2}(0,T;\Lt^2(\Omega\times D)), \label{5-psiwconH1a}\\
M^{\frac{1}{2}}\,\nabq \sqrt{\hpsiaet^{\Delta t(,\pm)}} &\rightarrow M^{\frac{1}{2}}\,\nabq \sqrt{\widetilde\psi}
&&\qquad \mbox{weakly in } L^{2}(0,T;\Lt^2(\Omega\times D)), \label{5-psiwconH1xa}\\
\widetilde\psi^{\Delta t (,\pm)}_L &\rightarrow \widetilde\psi
&&\qquad \mbox{strongly in }
L^{p}(0,T;L^{1}_M(\Omega\times D)),\label{5-psisconL2a}
\\
\beta^L(\widetilde\psi^{\Delta t (,\pm)}_L) &\rightarrow \widetilde\psi
&&\qquad \mbox{strongly in }
L^{p}(0,T;L^{1}_M(\Omega\times D)),\label{5-psisconL2beta}
\end{alignat}
for all $p \in [1,\infty)$; and,
\begin{alignat}{2}
\nabx\cdot\sum_{i=1}^K\Ctt_i(M\,
\zeta(\rho^{\Delta t,+}_L)\,
\widetilde\psi^{\Delta t ,+}_L) &\rightarrow \nabx \cdot\sum_{i=1}^K
\Ctt_i(M\,\zeta(\rho)\,\widetilde\psi)
&&\qquad \mbox{weakly in }
L^{2}(0,T;\Vt').\label{5-CwconL2a}
\end{alignat}
\end{subequations}
%
The triple $(\rho,\ut,\widetilde\psi)$ is a global weak solution to problem (P), in the sense that
\begin{subequations}
\begin{align}\label{5-eqrhoconP}
&\displaystyle-\int_{0}^{T}
\int_{\Omega} \rho \, \frac{\partial \eta}{\partial t}
\dx \dt
- \int_{0}^T \int_{\Omega}
\rho\, \ut \cdot \nabx \eta \dx \dt
= \int_\Omega \rho_0(\xt) \,\eta(\xt ,0) \dx
\nonumber \\
& \hspace{2in} \forall \eta \in W^{1,1}(0,T;W^{1,\frac{q}{q-1}}(\Omega)) \mbox{ s.t.\ $\eta(\cdot,T)=0$},
\end{align}
with $q \in (2,\infty)$ when $d=2$ and $q \in [3,6]$ when $d=3$,
\begin{align}\label{5-equnconP}
&\displaystyle-\int_{0}^{T}
\int_{\Omega} \rho\,\ut \cdot \frac{\partial \wt}{\partial t}
\dx \dt
+ \int_{0}^T \int_{\Omega}
\left[ \,
\mu(\rho) \,\Dtt(\ut)
:
\Dtt(\wt)
- \rho\,(\ut \otimes \ut) : \nabxtt \wt
\right] \dx \dt
\nonumber
\\
&\quad  = \int_\Omega \rho_0(\xt)\,\ut_0(\xt) \cdot \wt(\xt ,0) \dx  + \int_{0}^T
\int_{\Omega} \left[ \rho\, \ft \cdot \wt
- k\,\sum_{i=1}^K
\Ctt_i(M\,\zeta(\rho)\,\widetilde\psi): \nabxtt
\wt \dx \right] \dx \dt
\nonumber
\\
& \hspace{2in}
\forall \wt \in W^{1,1}(0,T;\Vt) \mbox{ s.t.\ $\wt(\cdot,T)=0$},
\end{align}
and
\begin{align}
&-\int_{0}^T
\int_{\Omega \times D} M\,\zeta(\rho)\,\widetilde\psi\, \frac{\partial \varphi}{\partial t}
\dq \dx \dt
+ \int_{0}^T \int_{\Omega \times D} M \left[
\epsilon\,
\nabx \widetilde\psi - \ut \,\zeta(\rho)\,\widetilde\psi \right]\cdot\, \nabx
\varphi
\,\dq \dx \dt
\nonumber \\
&
\quad +\frac{1}{4\,\lambda}
\int_{0}^T \int_{\Omega \times D} \sum_{i=1}^K
 \,\sum_{j=1}^K A_{ij}\,
 M\,
 \nabqj \widetilde\psi
\cdot\, \nabqi
\varphi
\,\dq \dx \dt
\nonumber \\
&
\quad -
\int_{0}^T \int_{\Omega \times D}\! M\,\sum_{i=1}^K\,
\left[\sigtt(\ut)
\,\qt_i\right] \zeta(\rho)\,\widetilde\psi \,\cdot\, \nabqi
\varphi
\,\dq \dx \dt
\nonumber \\
& \hspace{1in}
= \int_{\Omega \times D}M(\qt)\, \zeta(\rho_0(\xt))\,\widetilde\psi_0(\xt,\qt)\,\varphi(\xt,\qt,0)
\dq \dx
\nonumber \\
& \hspace{2in}
\forall \varphi \in W^{1,1}(0,T;H^s(\Omega\times D))
\mbox{ s.t.\ $\varphi(\cdot,\cdot,T)=0$},
\label{5-eqpsinconP}
\end{align}
\end{subequations}
with $s > 1+ \frac{1}{2}\,(K+1)\,d$.
In addition, the weak solution $(\rho,\ut,\widetilde\psi)$
satisfies for all $t \in [0,T]$:
\begin{subequations}
\begin{align}
\int_{\Omega} |\rho(t)|^p \dx = \int_{\Omega} |\rho_0|^p \dx,
\label{5-energyrho}
\end{align}
for $p \in [1,\infty)$, and the following energy inequality
for a.e.\ $t \in [0,T]$:
\begin{align}\label{5-eq:energyest}
&\int_{\Omega} \rho(t)\,|\ut(t)|^2 \dx
+ \int_0^t \int_{\Omega} \mu(\rho)\,|\Dtt(\ut)|^2 \dx \dd s
+ \,2\,k\int_{\Omega \times D}\!\! M\, \zeta(\rho(t))\,\mathcal{F}(\widetilde\psi(t)) \dq \dx
\nonumber \\
&\qquad +\, 8\,k\,\varepsilon\, \int_0^t \int_{\Omega \times D} M\,
|\nabx \sqrt{\widetilde\psi} |^2 \dq \dx \dd s
+\, \frac{a_0 \,k}{\lambda}  \int_0^t \int_{\Omega \times D}\,
M\,
|\nabq \sqrt{\widetilde\psi}|^2 \,\dq \dx \dd s\nonumber\\
&\quad \leq \int_{\Omega} \rho_{0}\,|\ut_0|^2 \dx + \frac{\rho_{\rm max}^2 C_\varkappa^2}{
\mu_{\rm min}\,c_0}\,
\int_0^t\|\ft\|^2_{L^\varkappa(\Omega)}
\dd s + 2\,k \int_{\Omega \times D} M\,\zeta(\rho_0)\,
\mathcal{F}(\widetilde\psi_0) \dq \dx
\nonumber \\
& \quad \leq [{\sf B}(\ut_0,\ft, \widetilde\psi_0)]^2,~~~~~~~
\end{align}
\end{subequations}
with $\mathcal{F}(s)= s(\log s - 1) + 1$, $s \geq 0$,
and $[{\sf B}(\ut_0,\ft, \widetilde\psi_0)]^2$ as defined in
{\rm (\ref{eq:energy-u+psi-final2})}.
\end{theorem}
\begin{proof}
We split the proof into a number of steps.

\smallskip

\textit{Step A.}
The convergence results (\ref{5-rhowconL2}--d), (\ref{5-uwconL2a}--c)
and (\ref{5-psiwconH1a},b) were proved in Theorem \ref{convfinal}.
The strong convergence result (\ref{5-psisconL2a}) was established in Section \ref{Lindep-time},
see (\ref{strong-2}).
Next from the Lipschitz continuity of $\beta^L$, we obtain for any $p \in [1,\infty)$ that
\begin{align}
&\|\widetilde \psi- \beta^L(\widetilde \psi^{\Delta t(,\pm)}_L)\|_{L^p(0,T;L^1_M(\Omega \times D))}
\nonumber \\
& \hspace{1in} \leq
\|\widetilde \psi- \beta^L(\widetilde \psi)\|_{L^p(0,T;L^1_M(\Omega \times D))}
+
\|\beta^L(\widetilde \psi)
- \beta^L(\widetilde \psi^{\Delta t(,\pm)}_L)\|_{L^p(0,T;L^1_M(\Omega \times D))}
\nonumber \\
& \hspace{1in} \leq
\|\widetilde \psi- \beta^L(\widetilde \psi)\|_{L^p(0,T;L^1_M(\Omega \times D))}
+
\|\widetilde \psi- \widetilde \psi^{\Delta t(,\pm)}_L
\|_{L^p(0,T;L^1_M(\Omega \times D))}.
\label{betapsicon}
\end{align}
The first term on the right-hand side of (\ref{betapsicon})
converges to zero as $L \rightarrow \infty$
on noting that $\beta^L(\widetilde \psi)$ converges to $\widetilde \psi$ almost everywhere
on $\Omega \times D \times (0,T)$ and applying Lebesgue's dominated convergence theorem,
the second term converges to $0$ on noting (\ref{5-psisconL2a}).
Hence, we obtain the desired result (\ref{5-psisconL2beta}).

As the sequences $\{\psia^{\Delta t(,\pm)}\}_{L>1}$ converge to $\widetilde\psi$ strongly in
$L^p(0,T; L^1_M(\Omega \times D))$, it follows (upon extraction of suitable subsequences) that
they converge to $\widetilde\psi$ a.e. on
$\Omega \times D \times [0,T]$. This then, in turn, implies that the sequences
$\{\mathcal{F}(\psia^{\Delta t(,\pm)})\}_{L>1}$
converge to $\mathcal{F}(\widetilde\psi)$ a.e. on $\Omega \times D \times [0,T]$; in particular,
for a.e. $t \in [0,T]$, the sequences $\{\mathcal{F}(\psia^{\Delta t(,\pm)}(\cdot,\cdot,t))\}_{L>1}$
converge to $\mathcal{F}(\widetilde\psi(\cdot,\cdot,t))$ a.e. on $\Omega \times D$.
Since $\mathcal{F}$ is nonnegative, Fatou's lemma then implies that, for a.e. $t \in [0,T]$,
\begin{align}\label{5-fatou-app}
&\int_{\Omega \times D} M(\qt)\, \mathcal{F}(\widetilde\psi(\xt,\qt,t))\dx \dq\nonumber\\
&\hspace{1.5in} \leq \mbox{lim inf}_{L \rightarrow \infty}
\int_{\Omega \times D} M(\qt)\, \mathcal{F}(\psia^{\Delta t(,\pm)}(\xt,\qt,t)) \dx \dq \leq C_\ast,
\end{align}
where the second inequality in \eqref{5-fatou-app} stems from
the bound on the fifth term on the left-hand side of \eqref{eq:energy-u+psi-final5}.
Hence the first result in (\ref{5-relent-fisher}) holds,
and the second was established in Theorem \ref{convfinal}.
Similarly, (\ref{5-mass-conserved}) is established on noting (\ref{lambda-bound})
and that $\psi^{\Delta(,\pm)}_L \geq 0$.
Analogously to (\ref{5-fatou-app}), one can establish
for a.e. $t \in [0,T]$
that
\begin{align}\label{5-fatou-app-rho}
&\int_{\Omega \times D} M(\qt)\, \zeta(\rho(\xt))\,
\mathcal{F}(\widetilde\psi(\xt,\qt,t))\dx \dq\nonumber\\
&\hspace{0.5in} \leq \mbox{lim inf}_{L \rightarrow \infty}
\int_{\Omega \times D} M(\qt)\, \zeta(\rho^{\Delta t(,\pm)}_L(\xt))\,
\mathcal{F}a(\psia^{\Delta t(,\pm)}(\xt,\qt,t)) \dx \dq.
\end{align}

Finally, it is shown in Step 3.7 in the proof of Theorem 6.1 in \cite{BS2011-fene}
on noting (\ref{intbyparts}), (\ref{5-psiwconH1xa},c) and
(\ref{5-mass-conserved}), that in the case $\zeta \equiv 1$
\begin{align}
k\int_{0}^T\!\sum_{i=1}^K \int_{\Omega}
\Ctt_i(M\,\zeta(\rho^{\Delta t,+}_L)\,\psia^{\Delta t,+}): \nabxtt \wt \dx \dt
\rightarrow
 k\int_{0}^T\!\sum_{i=1}^K \int_{\Omega}
\Ctt_i(M\,\zeta(\rho)\,\widetilde\psi): \nabxtt \wt \dx \dt\,
\label{5-stressconv}
\end{align}
as $L \rightarrow \infty$,
for any divergence-free function $\wt \in C^1([0,T];\Ct^\infty_0(\Omega))$.
The proof there is easily generalised to the present variable $\zeta$ on noting
(\ref{5-zetasconL2}).
This implies \eqref{5-CwconL2a}, thanks
to the denseness of
these smooth divergence functions
in the function space $L^2(0,T;\Vt)$,
and on showing that the right-hand side of (\ref{5-stressconv}) is well-defined for
$\wt \in L^2(0,T;\Vt)$,
on noting (\ref{intbyparts}), (\ref{5-relent-fisher}) and (\ref{5-mass-conserved}).
\smallskip

\textit{Step B.}
We are now ready to return to
(\ref{eqrhocon}--c)
and pass to the limit $L\rightarrow \infty$ (and thereby also $\Delta t \rightarrow
0_+$). 
We shall discuss them one at a time, starting with equation
\eqref{eqrhocon}. 
We have already passed to the limit $L \rightarrow \infty$ in (\ref{eqrhocon})
using (\ref{5-rhosconL2}) and (\ref{5-usconL2a}) to obtain (\ref{eqrhoconP})
in the proof of Theorem \ref{convfinal}.
The desired result (\ref{5-eqrhoconP}) then follows from (\ref{eqrhoconP}) on noting
the denseness of the set of all functions contained in $C^1([0,T];W^{1,\frac{q}{q-1}}(\Omega))$
and vanishing at $t=T$ in the set of all functions contained in
$W^{1,1}(0,T;W^{\frac{q}{q-1}}(\Omega))$
and vanishing at $t=T$.
In addition, 
the energy equality (\ref{5-energyrho}) was proved in the proof of Theorem \ref{convfinal},
see (\ref{energyrho}).

\smallskip

\textit{Step C.}
Having dealt with \eqref{eqrhocon}, we now turn to \eqref{equncon},
with the aim to pass to the limit with $L$ (and $\Delta t$).
We choose as our test function
\begin{align}
\wt \in C^1([0,T];\Ct^\infty_0(\Omega)) \quad
\mbox{with $\wt(\cdot,T)=0$, and $\nabx \cdot \wt = 0$ on $\Omega$
for all $t \in [0,T]$}.
\label{5-smoothtest}
\end{align}
Clearly, any such $\wt$ belongs to $L^1(0,T;\Vt)$ and is therefore a legitimate choice of
test function in \eqref{equncon}.
Integration by parts with respect to $t$ on the first term in (\ref{equncon}),
and noting (\ref{Gequn-trans}) and (\ref{zeta-time-average}) for the second term yields that
\begin{align}
&\displaystyle - \int_{0}^{T} \int_\Omega
(\rho_L\,\ut_L)^{\Delta t}
\cdot \frac{\partial \wt }{\partial t} \dx \dt
- \tfrac{1}{2}
\int_0^T \int_\Omega
\rho^{\{\Delta t\}}_L \,\ut^{\Delta t,-}_L
\cdot \nabx (\ut_L^{\Delta t,+}\cdot \wt) \dx \dt
\nonumber  \\
&\hspace{1cm}
+ \displaystyle\int_{0}^{T} \int_\Omega
\mu(\rho^{\Delta t, +}_L) \,\Dtt(\utaeDp)
:
\Dtt(\wt) \dx \dt
\nonumber \\
& \hspace{1cm}+
\tfrac{1}{2} \int_{0}^T\!\! \int_{\Omega}
\rho^{\{\Delta t\}}_L \left[ \left[ (\utaeDm \cdot \nabx) \utaeDp \right]\cdot\,\wt
- \left[ (\utaeDm \cdot \nabx) \wt  \right]\cdot\,\utaeDp
\right]\!\dx \dt
\nonumber
\end{align}
\begin{align}
&\hspace{0.5cm} =
\int_{\Omega} \rho_0\,\ut^0 \cdot \wt(0) \dx +
\int_{0}^T \int_{\Omega}  \rho_L^{\Delta t,+}
\, \ft^{\Delta t,+} \cdot \wt \dx \dt 
\nonumber \\
& \hspace{1.5cm}
- k\,\sum_{i=1}^K \int_0^T \int_{\Omega}
\Ctt_i(M\,\zeta(\rho^{\Delta t,+}_L)\, \widetilde\psi^{\Delta t,+}_L): \nabxtt
\wt \dx \dt.
\label{5-equnconint}
\end{align}
Next we note from (\ref{ulin},b) that for $t \in (t_{n-1},t_n]$, $n=1,\ldots,N$,
\begin{align}
(\rho_L \, \ut_L)^{\Delta t} = \rho^{\Delta t}_L \,\ut^{\Delta t,+}_L
-\frac{(t_n-t)}{\Delta t}\,\rho^{\Delta t,-}_L\,(\ut^{\Delta t,+}_L - \ut^{\Delta t, -}_L).
\label{5-rhoutdiff}
\end{align}
It follows
that we can pass to the limit $L \rightarrow \infty$, $(\Delta t \rightarrow 0)$, in (\ref{5-equnconint}),
on noting (\ref{5-rhoutdiff}),
(\ref{5-rhosconL2},c), (\ref{5-uwconH1a},c), (\ref{5-CwconL2a}),
(\ref{fncon}) and that $\ut^0$ converges to $\ut_0$ weakly in $\Ht$, to obtain
\begin{align}
&\displaystyle - \int_{0}^{T} \int_\Omega
\rho\,\ut
\cdot \frac{\partial \wt }{\partial t} \dx \dt
+ \displaystyle\int_{0}^{T} \int_\Omega
\mu(\rho) \,\Dtt(\ut)
:
\Dtt(\wt) \dx \dt
- \tfrac{1}{2}
\int_0^T \int_\Omega
\rho \,\ut \cdot \nabx (\ut\cdot \wt) \dx \dt
\nonumber
\\
& \hspace{1cm}+
\tfrac{1}{2} \int_{0}^T \int_{\Omega}
\rho \left[ \left[ (\ut \cdot \nabx) \ut \right]\cdot\,\wt
- \left[ (\ut \cdot \nabx) \wt  \right]\cdot\,\ut
\right]\!\dx \dt
\nonumber
\\
&\hspace{0.5cm} =
\int_{\Omega} \rho_0\,\ut_0 \cdot \wt(0) \dx +
\int_{0}^T \int_{\Omega}  \rho\,
\ft \cdot \wt \dx \dt 
- k\,\sum_{i=1}^K \int_0^T \int_{\Omega}
\Ctt_i(M\,\zeta(\rho)\,\widetilde \psi): \nabxtt
\wt \dx \dt.
\label{5-equnconintconv}
\end{align}
The desired result (\ref{5-equnconP}) then follows from (\ref{5-equnconintconv})
on noting (\ref{ruwiden}), the denseness of the test functions (\ref{5-smoothtest})
in $W^{1,1}(0,T;\Vt)$ and that all the terms in (\ref{5-equnconP}) are well-defined.

\smallskip

\textit{Step D.}
Similarly to (\ref{5-equnconint}), we obtain from performing integration by parts
with respect to time on the first term in
(\ref{eqpsincon}) that
\begin{align}
&\int_{0}^T \!\!\int_{\Omega \times D}
\left[
-M\, (\zeta(\rho_L)\,\hpsiaet)^{\Delta t}\,
\frac{ \partial \varphi}{\partial t} 
+
\frac{1}{4\,\lambda}
\,\sum_{i=1}^K
 \,\sum_{j=1}^K A_{ij}\,
 M\,
 \nabqj \hpsiaet^{\Delta t,+}
\cdot\, \nabqi
\varphi
\right]
\dq \dx \dt
\nonumber \\
& \qquad \qquad 
+ \int_{0}^T \!\!\int_{\Omega \times D} \left[
\epsilon\, M\,
\nabx \hpsiaet^{\Delta
t,+} - \utae^{\Delta t,-}\,M\,\zeta^{\{\Delta t\}}_L\,\hpsiaet^{\Delta t,+} \right]\cdot\, \nabx
\varphi
\dq \dx \dt
\nonumber \\
&
\qquad\qquad 
- \int_{0}^T \!\!\int_{\Omega \times D} M\,\sum_{i=1}^K
\left[\sigtt(\utae^{\Delta t,+})
\,\qt_i\right]\,\zeta(\rho_L^{\Delta t,+})\,\beta^L(\hpsiaet^{\Delta t,+}) \,\cdot\, \nabqi
\varphi
\,\dq \dx \dt
\nonumber\\
&\qquad = \int_{\Omega \times D} M\,\zeta(\rho_0)\,\widetilde \psi^0\,
\varphi(0) \dq \dx
\qquad \forall \varphi \in C^1([0,T];C^\infty(\overline{\Omega \times D}))
\mbox{ s.t.\ } \varphi(\cdot,\cdot,T)=0.
\label{eqpsiconibyp}
\end{align}
We now pass to the limit $L \rightarrow \infty$ (and $\Delta t \rightarrow 0_+$)
in \eqref{eqpsiconibyp} to obtain
(\ref{5-eqpsinconP}) for the smooth $\varphi$ of (\ref{eqpsiconibyp})
using the convergence results
(\ref{5-zetasconL2}), (\ref{5-uwconH1a},c), (\ref{5-psiwconH1a}--d)
and (\ref{psi0conv}).
The desired result (\ref{5-eqpsinconP}) then follows
on noting the denseness of the test functions $\varphi \in C^1([0,T],C^\infty(\overline{\Omega \times D}))$
with $\varphi(\cdot,\cdot,T)=0$
in $W^{1,1}(0,T;H^s(\Omega \times D))$ with $\varphi(\cdot,\cdot,T)=0$, for $s > 1+ \frac{1}{2}\,(K+1)\,d$,
and that all the terms in (\ref{5-eqpsinconP}) are well-defined.

\smallskip

\textit{Step E.}
The energy inequality \eqref{5-eq:energyest} is a direct consequence of
the convergence results (\ref{5-rhowconL2}--d),
(\ref{5-uwconL2a},b)
and (\ref{5-psiwconH1a},b), on noting
\eqref{fncon},
\eqref{5-fatou-app-rho}
and the (weak) lower-semicontinuity
of the terms on the left-hand side
of \eqref{eq:energy-u+psi-final2}. For example, it follows from
(\ref{eq:energy-u+psi-final2}) for a.e.\ $t \in [0,T]$ that
\begin{align}
&[\rho^{\Delta t,+}_L(t)]^{\frac{1}{2}}\,\ut_L^{\Delta t,+}(t)
\rightarrow \gt(t) \quad \mbox{weakly in } L^2(\Omega) \quad \mbox{as } L \rightarrow \infty
\nonumber \\
\Rightarrow \qquad &\liminf_{L \rightarrow \infty} \int_{\Omega} \rho^{\Delta t,+}_L(t)\,
|\ut^{\Delta t,+}_L(t)|^2 \dx \ge \int_{\Omega} |\gt(t)|^2 \dx.
\end{align}
It follows from (\ref{5-rhosconL2}) and (\ref{5-uwconL2a}) that
$\gt(t) = [\rho(t)]^{\frac{1}{2}}\,\ut(t)$.
\end{proof}

\noindent
\textbf{Acknowledgement}
ES was supported by the EPSRC Science and Innovation award to the Oxford Centre
for Nonlinear PDE (EP/E035027/1). We are grateful to Leonardo Figueroa for
supplying the image in Figure 1.

\bibliographystyle{siam}

\bibliography{polyjwbesrefs}

\appendix

\section{Equivalence of notions of solution to the continuity equation}\label{sec:apx}

The purpose of this Appendix is to show that the notion of solution that was used in \eqref{rho-reg}, \eqref{rhonL} coincides
with the notion of distributional solution to the continuity equation used by DiPerna \& Lions \cite{DPL}, whose work we referred to
following \eqref{fncon-1} for a proof of the existence of a unique solution to \eqref{rho-reg}, \eqref{rhonL}.

As previously, we shall assume throughout this Appendix that $\Omega$ is a bounded open Lipschitz domain in $\mathbb{R}^d$, $d \in \{2,3\}$. We shall suppose that $\bt \in L^2(0,T; \Vt)$ and $\rho_0 \in \Upsilon$ and we extend $\bt$ by $\zerot$ outside $\overline\Omega \times (0,T)$ to the whole of $\mathbb{R}^d \times (0,T)$.
The resulting function, denoted by  $\bt^*$, then belongs to $L^2(0,T;\Ht^1(\mathbb{R}^d))$
and is divergence-free on the whole of $\mathbb{R}^d\times (0,T)$. We define $\rho_0^*$ as being equal to $\rho_0$
on $\Omega$ and to $0$ in $\mathbb{R}^d\setminus \Omega$. Clearly, $\rho_0^* \in L^p(\mathbb{R}^d)$ for all $p \in [1,\infty]$.

We consider the transport equation
\[ \frac{\partial \rho^*}{\partial t} + \nabx \cdot (\bt^* \rho^*) = 0 \qquad \mbox{in $\mathbb{R}^d \times (0,T)$},\]
subject to the initial condition $\rho^*(\cdot,0) = \rho^*_0(\cdot)$. This initial-value problem is to
be understood in the sense of distributions of $\mathbb{R}^d \times (0,T)$; i.e.,
\begin{equation}\label{apx1}
-\int_0^T \int_{\mathbb{R}^d} \rho^\ast\, \frac{\partial \eta^*}{\partial t} \dx \dd t - \int_{\mathbb{R}^d}\rho^*_0\,\eta^*(\xt,0) \dx - \int_0^T \int_{\mathbb{R}^d} \bt^*\, \rho^* \cdot \nabx \eta^* \dx \dd t = 0\qquad
\end{equation}
for all test functions $\eta^* \in C^\infty(\mathbb{R}^d\times [0,T])$ with compact support in $\mathbb{R}^d\times [0,T)$;
we will denote this test space by $\mathcal{D}(\mathbb{R}^d \times [0,T))$ and note that $\mathcal{D}(\mathbb{R}^d \times [0,T))
= C^\infty_0([0,T); C^\infty_0(\mathbb{R}^d))$.

According to Proposition II.1 in the work of DiPerna \& Lions \cite{DPL}, there exists a solution
$\rho^\ast$ to the initial-value problem \eqref{apx1} in $L^\infty(0,T;L^p(\mathbb{R}^d))$ corresponding to the
initial condition $\rho^*_0 \in L^p(\mathbb{R}^d)$ for each $p \in [1,\infty]$.
Furthermore, by Corollary II.1 in \cite{DPL} the solution of \eqref{apx1} is unique in $L^\infty(0,T;L^p(\mathbb{R}^d))$
for all $p \in [1,\infty]$, and by Corollary II.2 in \cite{DPL}, $\rho^\ast \in C([0,T];L^p(\mathbb{R}^d))$ for all
$p \in [1,\infty)$.

Since $\bt^\ast = \zerot$ in $(\mathbb{R}^d \setminus\overline\Omega) \times (0,T)$ and $\rho^*_0=0$ in $\mathbb{R}^d \setminus\Omega$, trivially, $\rho^\ast = 0$ in $(\mathbb{R}^d\setminus\Omega) \times (0,T)$. Therefore, \eqref{apx1}
is equivalent to the following:
\begin{equation}\label{apx2}
-\int_0^T \int_{\Omega} \rho\, \frac{\partial \eta}{\partial t} \dx \dd t - \int_{\Omega}\rho_0\,\eta(\xt,0) \dx - \int_0^T \int_{\Omega} \rho\,\bt  \cdot \nabx \eta \dx \dd t = 0\qquad
\end{equation}
for all test functions $\eta \in C^\infty_0([0,T); C^\infty(\overline \Omega))$, where $\rho$ denotes the restriction of $\rho^*$
to $\Omega \times (0,T)$. Thus, the solution $\rho$ to \eqref{apx2}, subject to the initial datum $\rho_0 \in \Upsilon$, exists and is unique in $L^\infty(0,T;L^p(\Omega))$ for all $p\in [1,\infty]$, and it belongs to $C([0,T];L^p(\Omega))$ for all $p \in [1,\infty)$.

By selecting $\eta \in C^\infty_0((0,T); C^\infty(\overline \Omega))$ in \eqref{apx2}, the second term on the left-hand side
vanishes, and by noting that, thanks to the Sobolev embedding theorem,  $H^1(\Omega)$ is continuously embedded in $L^q(\Omega)$ for $q \in (2,\infty)$ when $d=2$ and for $q \in [3,6]$ when $d=3$, we then have that
\begin{equation}\label{apx3}
\left|-\int_0^T \int_{\Omega} \rho\, \frac{\partial \eta}{\partial t} \dx \dd t\right| \leq c(q)\,
\|\rho\|_{L^\infty(0,T;L^\infty(\Omega))}\,\|\bt\|_{L^2(0,T;L^q(\Omega))} \, \|\nabx \eta\|_{L^2(0,T;L^{\frac{q}{q-1}}(\Omega))},
\end{equation}
for all $\eta \in C^\infty_0((0,T); C^\infty(\overline \Omega))$, where $c(q)$ is a positive constant;
and hence, by the density of $C^\infty(\overline\Omega)$ in $W^{1,\frac{q}{q-1}}(\Omega)$, also for all
$\eta \in C^\infty_0((0,T); W^{1,\frac{q}{q-1}}(\Omega))$.

Let us consider, for $\bt \in L^2(0,T;\Vt)$ and $\rho_0 \in \Upsilon$ fixed,
and therefore $\rho \in L^\infty(0,T;L^\infty(\Omega))$ also fixed,
\[ \ell(\eta): = -\int_0^T \int_{\Omega} \rho\, \frac{\partial \eta}{\partial t} \dx \dd t,\qquad \eta \in
C^\infty_0((0,T); W^{1,\frac{q}{q-1}}(\Omega)).\]
By \eqref{apx3},
\[ |\ell(\eta)| \leq c(q,\rho,\bt)\, \|\eta\|_{L^2(0,T;W^{1,\frac{q}{q-1}}(\Omega))} \qquad \forall \eta \in
C^\infty_0((0,T); W^{1,\frac{q}{q-1}}(\Omega)).\]
By the Hahn--Banach theorem the continuous linear functional $\ell$ can be extended from the linear subspace
$C^\infty_0((0,T); W^{1,\frac{q}{q-1}}(\Omega))$ of $L^2(0,T;W^{1,\frac{q}{q-1}}(\Omega))$ to a continuous linear
functional on the whole of $L^2(0,T;W^{1,\frac{q}{q-1}}(\Omega))$, and since $C^\infty_0((0,T); W^{1,\frac{q}{q-1}}(\Omega))$ is dense in $L^2(0,T;W^{1,\frac{q}{q-1}}(\Omega))$ the extension is unique; we denote it by $\ell^*$. Hence,
$\ell^\ast(\eta) = \ell(\eta)$ for all $\eta \in C^\infty_0((0,T); W^{1,\frac{q}{q-1}}(\Omega))$, and
\[ |\ell^\ast(\eta)| \leq c(q,\rho,\bt)\, \|\eta\|_{L^2(0,T;W^{1,\frac{q}{q-1}}(\Omega))} \qquad \forall \eta \in
L^2(0,T; W^{1,\frac{q}{q-1}}(\Omega));\]
and therefore $\ell^* \in L^2(0,T; W^{1,\frac{q}{q-1}}(\Omega))'$. Since for the range of $q$ under consideration
both $W^{1,\frac{q}{q-1}}(\Omega)$ and $W^{1,\frac{q}{q-1}}(\Omega)'$ are reflexive Banach spaces, it
follows from a result of Bochner \& Taylor (cf. Remark 2.22.5 on p.125 in Kufner, John \& Fu\v{c}ik
\cite{KJF}; and a result of Gel'fand cited in item (b) on the same page in \cite{KJF})
that $L^2(0,T; W^{1,\frac{q}{q-1}}(\Omega))$ is a reflexive Banach space and
there exists an element $g \in L^2(0,T;W^{1,\frac{q}{q-1}}(\Omega)')$ such that
\[ \ell^\ast(\eta) = \int_0^T \left \langle g , \eta \right \rangle_{W^{1,\frac{q}{q-1}}(\Omega)} \dd t
\qquad \forall \eta \in L^2(0,T; W^{1,\frac{q}{q-1}}(\Omega)).\]
In particular since $\rho\, \frac{\partial \eta}{\partial t}$ is an integrable function on $\Omega$, we have that
\begin{eqnarray*}
&&\int_0^T \left\langle \rho , \frac{\partial \eta}{\partial t} \right\rangle_{W^{1,\frac{q}{q-1}}(\Omega)}
\dd t  = \int_0^T \left[\int_{\Omega} \rho\, \frac{\partial \eta}{\partial t} \dx \right] \dd t = -\ell(\eta)
= -\ell^\ast(\eta)
\nonumber\\
&&\qquad\qquad\qquad =-\int_0^T \left \langle g , \eta \right \rangle_{W^{1,\frac{q}{q-1}}(\Omega)} \dd t
\qquad \forall\eta \in C^\infty_0((0,T); W^{1,\frac{q}{q-1}}(\Omega)).
\end{eqnarray*}
Consequently, by selecting test functions $\eta$ of the form $\eta = \varphi(t)\, \xi(x)$,
where $\varphi \in C^\infty_0(0,T)$ and $\xi \in  W^{1,\frac{q}{q-1}}(\Omega)$, we deduce that
\begin{eqnarray*}
\int_0^T \left\langle \rho , \xi \right\rangle_{W^{1,\frac{q}{q-1}}(\Omega)} \frac{\dd \varphi}{\dd t}
\dd t  =-\int_0^T \left \langle g , \xi \right \rangle_{W^{1,\frac{q}{q-1}}(\Omega)} \varphi \dd t
\qquad \forall\varphi \in C^\infty_0(0,T),\; \forall \xi \in  W^{1,\frac{q}{q-1}}(\Omega),
\end{eqnarray*}
and therefore, by the bilinearity of the duality pairing $\langle \cdot , \cdot \rangle_{W^{1,\frac{q}{q-1}}(\Omega)}$, also
\begin{eqnarray*}
\int_0^T \left\langle \rho\,  \frac{\dd \varphi}{\dd t} , \xi \right\rangle_{W^{1,\frac{q}{q-1}}(\Omega)}
\dd t  =-\int_0^T \left \langle g\,  \varphi , \xi \right \rangle_{W^{1,\frac{q}{q-1}}(\Omega)}\dd t
\qquad \forall\varphi \in C^\infty_0(0,T),\; \forall \xi \in  W^{1,\frac{q}{q-1}}(\Omega).
\end{eqnarray*}
By applying Corollary \ref{Bochner4} to both sides of the last equality, noting in particular that $\rho$ and $g$ both belong to
$L^1(0,T;W^{1,\frac{q}{q-1}}(\Omega)')$, the latter since
$g \in L^2(0,T;W^{1,\frac{q}{q-1}}(\Omega)') \subset L^1(0,T;W^{1,\frac{q}{q-1}}(\Omega)')$, and the former since
$\rho \in L^\infty(0,T; L^\infty(\Omega))$, and $W^{1,\frac{q}{q-1}}(\Omega)
\subset L^1(\Omega)$ implies $L^\infty(\Omega) \subset W^{1,\frac{q}{q-1}}(\Omega)'$,
we deduce that
\begin{eqnarray*}
\left\langle \int_0^T \rho\,  \frac{\dd \varphi}{\dd t} \dd t, \xi \right \rangle_{W^{1,\frac{q}{q-1}}(\Omega)}
=\left \langle -\int_0^T  g\,  \varphi \dd t, \xi \right \rangle_{W^{1,\frac{q}{q-1}}(\Omega)}
\qquad \forall\varphi \in C^\infty_0(0,T),\; \forall \xi \in  W^{1,\frac{q}{q-1}}(\Omega).
\end{eqnarray*}
Hence,
\begin{eqnarray*}
\int_0^T \rho(t)\, \frac{\dd \varphi}{\dd t}(t) \dd t
=  -\int_0^T  g(t)\, \varphi(t) \dd t \qquad  \forall\varphi \in C^\infty_0(0,T),
\end{eqnarray*}
as an equality in the Banach space $W^{1,\frac{q}{q-1}}(\Omega)'$. By the implication (ii) $\Rightarrow$ (i) in Lemma
\ref{Bochner}, we then deduce that $\rho$ is almost everywhere on $(0,T)$ equal to the primitive of $g$; i.e., $g = \dd \rho/\dt$.
Therefore,
\[ \ell^*(\eta) = \int_0^T \left \langle \frac{\dd \rho}{\dt} , \eta \right \rangle_{W^{1,\frac{q}{q-1}}(\Omega)} \dd t
\qquad \forall \eta \in L^2(0,T; W^{1,\frac{q}{q-1}}(\Omega)),\]
and (reverting from the notation $\dd \rho/\dd t$ to $\partial \rho/\partial t$), in particular,
\[ - \int_0^T \int_\Omega \rho\, \frac{\partial \eta}{\partial t} \dx \dd t = \ell(\eta) = \ell^\ast(\eta) =
\int_0^T \left \langle \frac{\partial \rho}{\partial t} , \eta \right \rangle_{W^{1,\frac{q}{q-1}}(\Omega)} \dd t
\qquad \forall \eta \in C^\infty_0(0,T; W^{1,\frac{q}{q-1}}(\Omega)).\]
We now substitute this into \eqref{apx2} to deduce that
\begin{equation}\label{apx4}
\int_0^T \left \langle \frac{\partial \rho}{\partial t} , \eta \right \rangle_{W^{1,\frac{q}{q-1}}(\Omega)} \dd t - \int_0^T \int_{\Omega} \rho\,\bt \cdot \nabx \eta \dx \dd t = 0\qquad \forall \eta \in C^\infty_0(0,T; W^{1,\frac{q}{q-1}}(\Omega)).
\end{equation}
Hence, by density,
\begin{equation}\label{apx5}
\int_0^T \left \langle \frac{\partial \rho}{\partial t} , \eta \right \rangle_{W^{1,\frac{q}{q-1}}(\Omega)} \dd t - \int_0^T \int_{\Omega} \rho\,\bt \cdot \nabx \eta \dx \dd t = 0\qquad \forall \eta \in L^2(0,T; W^{1,\frac{q}{q-1}}(\Omega)).
\end{equation}
In particular if $\bt \in L^\infty(0,T;\Vt)$ (as is the case when $\bt\in \Vt$ is independent of $t$; see, for
  example \eqref{rhonL} where $\ut^{n-1} \in \Vt$ was used as $\bt$), the choice of the test function in \eqref{apx4} can be further relaxed thanks to
  the density of $C^\infty_0(0,T)$ in $L^1(0,T)$; so \eqref{apx4} then implies that
\begin{equation*}
\int_0^T \left \langle \frac{\partial \rho}{\partial t} , \eta \right \rangle_{W^{1,\frac{q}{q-1}}(\Omega)} \dd t - \int_0^T \int_{\Omega} \rho\,\bt \cdot \nabx \eta \dx \dd t = 0\qquad \forall \eta \in L^1(0,T; W^{1,\frac{q}{q-1}}(\Omega)),
\end{equation*}
and in this case $\partial\rho/\partial t \in L^\infty(0,T;W^{1,\frac{q}{q-1}}(\Omega)')$; i.e.,
$\rho \in W^{1,\infty}(0,T;W^{1,\frac{q}{q-1}}(\Omega)')$. This is the notion of solution that was used in \eqref{rho-reg}, \eqref{rhonL}.

Conversely, if $\rho \in L^\infty(0,T;L^\infty(\Omega)) \cap C([0,T]; L^p(\Omega)) \cap W^{1,\infty}(0,T;W^{1,\frac{q}{q-1}}(\Omega)')$ solves \eqref{apx5} subject to the initial
condition $\rho(\cdot,0) = \rho_0(\cdot) \in \Upsilon$, then, by choosing test functions $\eta$ of the form
$\eta(\xt,t) = \varphi(t)\, \xi(x)$, with $\varphi \in C^\infty_0(0,T)$ and $\xi \in C^\infty_0(\Omega)$,
we have from \eqref{apx4} that
\begin{equation}\label{apx7}
\int_0^T \left \langle \frac{\partial \rho}{\partial t} , \xi \right \rangle_{W^{1,\frac{q}{q-1}}(\Omega)} \varphi \dd t - \int_0^T \varphi\int_{\Omega} \rho\,\bt \cdot \nabx \xi  \dx \dd t = 0\quad \forall\varphi \in C^\infty_0(0,T),\; \forall \xi \in  C^\infty_0(\Omega).
\end{equation}
Since $\rho \in W^{1,\infty}(0,T;W^{1,\frac{q}{q-1}}(\Omega)')$ and so both $\rho$ and $g = \dd \rho/\dd t$ belong to $L^1(0,T;W^{1,\frac{q}{q-1}}(\Omega)')$, it follows from Corollary \ref{Bochner2} that
\[ \rho(t) = \rho_0 + \int_0^t \frac{\dd \rho}{\dd t}(s) \dd s,\qquad \rho_0 \in L^\infty(\Omega) \subset
W^{1,\frac{q}{q-1}}(\Omega)',\quad \mbox{for a.e. $t \in [0,T]$}.\]
The implication (i) $\Rightarrow$ (iii) of Lemma \ref{Bochner} applied to the first term in \eqref{apx7} then implies that \eqref{apx7} can be rewritten as
\begin{equation}\label{apx8}
\int_0^T \frac{\dd}{\dd t}\left \langle \rho , \xi \right \rangle_{W^{1,\frac{q}{q-1}}(\Omega)} \varphi \dd t - \int_0^T \varphi\int_{\Omega} \rho\,\bt \cdot \nabx \xi  \dx \dd t = 0 \quad\forall\varphi \in C^\infty_0(0,T),\; \forall \xi \in  C^\infty_0(\Omega).
\end{equation}
After partial integration in the first term in \eqref{apx8} we have that
\begin{equation}\label{apx9}
-\int_0^T\left \langle \rho , \xi \right \rangle_{W^{1,\frac{q}{q-1}}(\Omega)}  \frac{\dd \varphi}{\dd t} \dd t - \int_0^T \varphi\int_{\Omega} \rho\,\bt \cdot \nabx \xi  \dx \dd t = 0 \quad\forall\varphi \in C^\infty_0(0,T),\; \forall \xi \in  C^\infty_0(\Omega),
\end{equation}
and then, by rewriting the duality pairing in the first term of \eqref{apx9}
as an integral over $\Omega$ using that the product $\rho\, \xi$ is integrable on $\Omega$,
\begin{equation}\label{apx10a}
-\int_0^T\left[\int_\Omega \rho \, \xi \dd x \right] \frac{\dd \varphi}{\dd t} \dd t - \int_0^T \varphi\int_{\Omega}
\rho\,\bt \cdot \nabx \xi  \dx \dd t = 0\quad  \forall\varphi \in C^\infty_0(0,T),\; \forall \xi \in  C^\infty_0(\Omega).
\end{equation}
Upon rearrangement \eqref{apx10a} becomes
\begin{equation}\label{apx10}
-\int_0^T \int_\Omega \rho \, \frac{\partial (\varphi\, \xi)}{\partial t} \dd t - \int_0^T \int_{\Omega}
\rho\,\bt \cdot \nabx (\varphi\, \xi)  \dx \dd t = 0\quad  \forall\varphi \in C^\infty_0(0,T),\; \forall \xi \in  C^\infty_0(\Omega).
\end{equation}
The set
\begin{eqnarray*}
\bigg\{\eta \in C^\infty_0(\Omega \times (0,T))\,:\,\eta(\xt,t)=\sum_{1 \leq i \leq N}\varphi_i(t)\, \xi_i(\xt),\quad
\mbox{for some $N \in \mathbb{N}$}\nonumber\\
\mbox{and some $\varphi_i \in C^\infty_0(0,T)$, $\xi_i(x) \in C^\infty_0(\Omega)$, $i=1, \dots, N$}\bigg\}
\end{eqnarray*}
is dense in $C^\infty_0(\Omega \times (0,T))$ (cf. Lemma in Ch~1., \P3, Sec.~2 of Vladimirov \cite{Vladimirov}, for example). We thus
deduce from \eqref{apx10} with $\varphi$ and $\xi$ replaced by $\varphi_i$ and $\xi_i$ respectively, forming
finite linear combinations of the resulting equations, and then using a density argument, that
\begin{equation}\label{apx11}
-\int_0^T \int_\Omega \rho \, \frac{\partial \eta}{\partial t} \dd t - \int_0^T \int_{\Omega}
\rho\,\bt \cdot \nabx \eta  \dx \dd t = 0\qquad  \forall\eta \in C^\infty_0(\Omega \times (0,T)).
\end{equation}
This implies that
$\rho \in L^\infty(0,T;L^\infty(\Omega)) \cap C([0,T]; L^p(\Omega)) \cap W^{1,\infty}(0,T;W^{1,\frac{q}{q-1}}(\Omega)')$
satisfies
\[ \frac{\partial \rho}{\partial t} + \nabx \cdot (\bt \rho) = 0 \qquad \mbox{in $\mathcal{D}'(\Omega \times (0,T))$},
\]
together with the initial condition $\rho(\cdot,0) = \rho_0(\cdot)$, where  $\rho_0\in \Upsilon$,
i.e., that $\rho \in L^\infty(0,T;L^\infty(\Omega)) \cap C([0,T]; L^p(\Omega)) \cap W^{1,\infty}(0,T;W^{1,\frac{q}{q-1}}(\Omega)')$ is a solution to the continuity equation in the sense of distributions on $\Omega \times (0,T)$.

Alternatively, on extending $\bt$ by $\zerot$ to $\bt^*$ and $\rho_0$ and $\rho$ by $0$ to $\rho^*_0$ and $\rho^*$, respectively,
as at the start of this Appendix, we deduce from \eqref{apx11} that $\rho^*$ satisfies
\[ \frac{\partial \rho^*}{\partial t} + \nabx \cdot (\bt^* \rho^*) = 0 \qquad \mbox{in $\mathcal{D}'(\mathbb{R}^d \times (0,T))$},
\]
together with the initial condition $\rho^*(\cdot,0) = \rho_0^*(\cdot)$, where  $\rho_0^*\in L^p(\mathbb{R}^d)$ for all
$p \in [1,\infty]$,
and $\rho^\ast \in L^\infty(0,T;L^p(\mathbb{R}^d))$ with $p \in [1,\infty]$; $\rho^*\in C([0,T]; L^p(\mathbb{R}^d)) \cap W^{1,\infty}(0,T;W^{1,\frac{q}{q-1}}(\mathbb{R}^d)')$ for $p \in [1,\infty)$ and for the range of $q$ under consideration.
Hence, $\rho^\ast$ is a solution to the continuity equation in the sense of distributions on $\mathbb{R}^d \times (0,T)$, as in \cite{DPL}.

~\\

~\hfill{\footnotesize \textit{London \& Oxford}

~\hfill{\footnotesize \textit{20th December 2011}}

~\\

~\\


\end{document}
